\documentclass[twoside,onecolumn,11pt,a4paper]{book}
\usepackage{anuthesis}
\usepackage{float}
\usepackage{asymptote}
\usepackage{fancyhdr}
\usepackage{amsmath}
\usepackage{amsthm}
\usepackage[utf8]{inputenc}
\usepackage{amsfonts}
\usepackage{amssymb}
\usepackage{graphicx}
\usepackage{hyperref}
\hypersetup{
    colorlinks,
    citecolor=black,
    filecolor=black,
    linkcolor=black,
    urlcolor=black
}
\usepackage[page]{appendix}
\usepackage{chngcntr}
\begin{document}
\newtheorem{thm}{Theorem}[section]
\newtheorem{prop}[thm]{Proposition}
\newtheorem{lem}[thm]{Lemma}
\newtheorem{cor}[thm]{Corollary}
\newtheorem{con}[thm]{Conjecture}
\newtheorem{defn}[thm]{Definition}
\newtheorem{example}[thm]{Example}
\newtheorem{remark}[thm]{Remark}
\newcommand{\R}{\mathbb{R}}  
\newcommand{\C}{\mathbb{C}}
\newcommand{\N}{\mathbb{N}}
\newcommand{\T}{\mathbb{T}}
\newcommand{\Z}{\mathbb{Z}}
\newcommand{\TA}{\partial_\theta^\alpha}
\newcommand{\OB}{\partial_\omega^\beta}
\newcommand{\TD}{\partial_t^\delta}
\newcommand{\IB}{\partial_I^\beta}
\newcommand{\Ek}{E_\kappa}
\newcommand{\Cm}{\mathcal{C}^\infty_{M,\tilde{M}}(X\times Y)}
\newcommand{\CM}{\mathcal{C}^\infty_{M,\tilde{M}}}
\numberwithin{equation}{section}

  \renewcommand{\bottomfraction}{0.7}

\pagenumbering{roman}  

\begin{titlepage}
\title{\textbf{Quantum ergodicity in mixed and KAM Hamiltonian systems}\\[2cm]}
 \author{\textbf{Se\'{a}n P Gomes}\\[6cm]
 \textbf{A thesis submitted for the degree of}\\
 \textbf{Doctor of Philosophy of} \\
 \textbf{The Australian National University}\\[1cm]}
 \date{\textbf{May, 2017}}
\maketitle
 \end{titlepage}
 
 \sloppy
 
\chapter*{Declaration}
\addcontentsline{toc}{chapter}{Declaration}

This thesis is an account of research undertaken between August 2012 and 
May 2017 at the Mathematical Sciences Institute,
The Australian National University, Canberra, Australia.

Except where acknowledged in the customary manner, the material 
presented in this thesis is, to the best of my knowledge, original and 
has not been submitted in whole or part for a degree in any 
university.

The bulk of the original work in this thesis is contained in Chapters 3 and 6. Chapters 1 and 2 are a summary of known background material and Chapters 4 and 5 closely follow the work in \cite{popovkam} and \cite{popovquasis}.

\vspace{20mm}  

\hspace{80mm}\rule{40mm}{.15mm}\par   
\hspace{80mm} Se\'{a}n P Gomes\par
\hspace{80mm} May, 2016


\chapter*{Acknowledgements}
\addcontentsline{toc}{chapter}{Acknowledgements}

First and foremeost I would like to express my immense appreciation and gratitude to my thesis adviser, Professor Andrew Hassell, you have been a fantastic mentor and role model in these early stages of academic life. Your advice and support on matters both related to mathematics and career progress has been of immeasurable value to me. \\

\noindent I would also like to thank my committee members, Professors Ben Andrews and Xu-Jia Wang and the the HDR convenor Associate Professor Scott Morrison for making my thesis defence an enjoyable experience, and for providing useful feedback.\\

\noindent Thanks also go to Professor Georgi Popov for several enlightening email exchanges elaborating on aspects of his work in the KAM setting, and to Assistant Professor Semyon Dyatlov for a fruitful discussion that motivated a weakening of the ``slow torus" condition in the the results of Chapter 6.\\

\noindent Thanks to the Australian federal government, whose funding via the Australian postgraduate award and research training program stipend have made this research possible. \\

\noindent Thanks to all of my friends and colleagues, for always encouraging me to strive for my goals.\\

\noindent A special thanks to Frank, Ronette, Karen and Tanya. The unconditional love, support, and encouragement  that a family like ours provides is something that cannot be overstated.\\

\noindent Last but certainly not least, I must acknowledge my partner, Adeline. You have been a boundless source of personal support at the times when it was needed most. Being a thesis-widow is no easy burden, yet you have done everything in your power to support me on this journey and have rode the highs and lows alongside me, making even the most challenging obstacles seem surmountable.  

\chapter*{Abstract}  
\addcontentsline{toc}{chapter}{Abstract} 

In this thesis, we investigate quantum ergodicity for two 
classes of Hamiltonian systems satisfying intermediate dynamical hypotheses between the well 
understood extremes of ergodic flow and quantum  completely integrable flow. These two classes are 
mixed Hamiltonian systems and KAM Hamiltonian systems.

Hamiltonian systems with mixed phase space decompose into finitely many invariant subsets, only some of which are of ergodic character. It has been conjectured by Percival that the 
eigenfunctions of the quantisation of this system decompose into associated families of analogous 
character. The first project in this thesis proves a weak form of this 
conjecture for a class of dynamical billiards, namely the mushroom billiards of Bunimovich for a full measure subset of a shape parameter $t\in (0,2]$.

KAM Hamiltonian systems arise as perturbations of completely integrable Hamiltonian systems. The 
dynamics of these systems are well understood and have near-integrable character. The classical-quantum correspondence suggests that the quantisation of KAM systems will not have quantum ergodic 
character. The second project in this thesis proves an initial negative quantum ergodicity result for a class of positive Gevrey perturbations of a Gevrey Hamiltonian that satisfy a mild \emph{slow torus} condition.  

\tableofcontents
\listoffigures 

\pagenumbering{arabic} 
\setcounter{page}{1}  

\pagenumbering{arabic} 
\setcounter{page}{1}  

\chapter{Introduction to Quantum Ergodicity}
\label{chpqe}

The central objective in quantum chaos is to understand how chaotic dynamical assumptions about a classical mechanical system manifest themselves in the behaviour of its quantum mechanical analogue.

A natural setting for studying this correspondence is that of Hamiltonian flow on a compact Riemannian manifold $M$, and this is the setting of the original work in this thesis. In this setting, our dynamical assumption is based on the measure-theoretic concept of \emph{ergodicity}.

We shall begin in Section \ref{hamflowsec} by summarising the aspects of the Hamiltonian formalism relevant to our work. A more comprehensive treatment can be found in the book \cite{arnoldmechanics}. In particular, we shall highlight the opposing concepts of ergodicity and complete integrability.

In Section \ref{qdynsec} we introduce Schr\"{o}dinger's equation, the quantum mechanical counterpart to Hamilton's equations. We shall then discuss the semiclassical formalism and its relevance to studying the classical-quantum correspondence.

In Section \ref{qerdsec} we define the quantum mechanical analogue to ergodicity of Hamiltonian flow and survey the major results and conjectures in this field.

In Section \ref{nressec} we discuss the quantisations of Hamiltonian systems that are either completely integrable, or are small perturbations of completely integrable systems. As the Hamiltonian flow in these settings is far from ergodic, intuition suggests that the eigenfunctions for such a system will be far from equidistributed.

\section{Hamiltonian flow}
\label{hamflowsec}
Suppose that we have a smooth $n$-dimensional compact Riemannian manifold $(M,g)$ (possibly with boundary). Given a smooth Hamiltonian function $H:T^*M\rightarrow \rightarrow \mathbb{R}$ which we interpret as an energy, we obtain the Hamiltonian flow $\Phi^t:T^*M\times [0,\infty)\rightarrow T^*M$ generated by Hamilton's equations
\begin{equation}
\label{hamfloweq}
\dot{\xi}=-\nabla_x H(x,\xi)\qquad
\dot{x}=\nabla_\xi H(x,\xi)
\end{equation}
with coordinates $(x,\xi)$ corresponding to the cotangent vector $\sum_j \xi_j \, dx_j$. In this work we shall assume that our Hamiltonians are such that the flow $\Phi^t$ does not blow up in finite time. We denote the Hamiltonian vector field given by \eqref{hamfloweq} as $X_H$.

The primary Hamiltonians of interest in this thesis will be Schr\"{o}dinger Hamiltonians of the form
\begin{equation}
\label{schrodhamil}
H(x,\xi)=\|\xi\|_g^2+V(x,\xi),
\end{equation}
however the results of Chapter 6 could be generalised to symbols of more general classes of self-adjoint pseudodifferential operators.
In Chapter 5, the function $V$ is a compactly supported symbol in the Gevrey class $S_\ell$ from Definition \ref{selldef} with self-adjoint quantisation. In the special case of $V=0$ as in Chapter 3, this Hamiltonian system is referred to as \emph{billiards} on $M$. The trajectories of billiard flow can be identified with the geodesics of $M$ under the canonical isomorphism between the tangent and co-tangent bundles of Riemannian manifolds.

A major advantage of the Hamiltonian formulation of mechanics over the Newtonian and Lagrangian formulations is the duality of the variables $x$ and $\xi$, best highlighted through the lens of symplectic geometry.

\begin{defn}[Symplectic Form]
A \emph{symplectic form} $\omega$ on a smooth manifold $M$ is a closed non-degenerate differential $2$-form.
\end{defn}

\begin{defn}[Symplectomorphism]
A \emph{symplectomorphism} $\chi:N_1\rightarrow N_2$ between two symplectic manifolds $(N_1,\omega_1)$ and $(N_2,\omega_2)$ is a diffeomorphism such that $\chi^*\omega_2 = \omega_1$.
\end{defn}

\begin{defn}
An \emph{exact symplectic form} $\omega$ on a smooth manifold $M$ is an exact non-degenerate differential $2$-form. 
\end{defn}

\begin{defn}[Exact Symplectomorphism]
A \emph{symplectomorphism} $\chi:N_1\rightarrow N_2$ between two exact symplectic manifolds $(N_1,d\alpha_1)$ and $(N_2,d\alpha_2)$ is a diffeomorphism such that $\chi^*\alpha_2 - \alpha_1$ is an exact $1$-form.
\end{defn}

Given a smooth function $H$ on a smooth manifold $N$ equipped with a symplectic form $\omega$, we can obtain a Hamiltonian vector field $X_H$ on $N$ defined implicitly by

\begin{equation}
dH(Y)=\omega(X_H,Y).
\end{equation}

In the case $N=T^*M$ there is a canonical choice of symplectic form 
\begin{equation}
\label{canonicalcoords}
\omega=d\xi\wedge dx=\sum_j d\xi_j \wedge dx_j
\end{equation}
and the vector field $X_H$ generates the flow given by Hamilton's equations \eqref{hamfloweq}.

In fact for a general symplectic manifold $(N,\omega)$, Darboux's theorem asserts that local coordinates $(x,\xi)$ can be chosen such that \eqref{canonicalcoords} holds. Such coordinates are said to be \emph{canonical coordinates}.

Writing $z=(x,\xi),w=(y,\eta)$ in a canonical coordinate system on $(N,\omega)$ allows us to write 
\begin{equation}
\omega(z,w)=\langle Jz,w\rangle
\end{equation}
where
\begin{equation}
J:=\begin{pmatrix}
0 & I \\ 
-I & 0
\end{pmatrix} \in \R^{2n\times 2n}.
\end{equation}

Indeed we say that a matrix $A$ is symplectic if 
\begin{equation}
\label{sympmatrix}
JA^TJ=A^{-1},
\end{equation}
and a diffeomorphism on a symplectic manifolds is a symplectomorphism if and only if its Jacobian with respect to canonical coordinates is a symplectic matrix.

Symplectomorphisms are the natural class of coordinate transformations of a Hamiltonian system to work with as they preserve Hamilton's equations.

\begin{prop}
If $\chi:N_1\rightarrow N_2$ is a symplectomorphism, and $H:N_2\rightarrow \R$ is a smooth Hamiltonian then the transformed Hamiltonian $\tilde{H}:=H\circ \chi:N_1\rightarrow \R$ generates a Hamiltonian flow in the coordinates $(y,\eta)=\chi^{-1}(x,\xi)$ given by
\begin{equation}
\dot{\eta}=-\nabla_y \tilde{H}(y,\eta)\qquad
\dot{y}=\nabla_\eta \tilde{H}(y,\eta)
\end{equation}

The Hamiltonian vector field $X_{\tilde{H}}$ on $N_1$ is the pullback of the Hamiltonian vector field $X_{H}$ on $N_2$. 
\end{prop}

Hamiltonian flows give rise to symplectomorphisms in a natural way.
\begin{prop}
\label{flowissymp}
If $H$ is a Hamiltonian on the symplectic manifold $N$, then the Hamiltonian flow $\Phi^t$ on $N$ is a one-parameter family of symplectomorphisms.
\end{prop}

Another useful method of constructing symplectomorphisms is through the use of a generating function.
\begin{prop}
\label{genfunc}
If $\Phi(x,\xi)\in \mathcal{C}^\infty(N)$ is such that the Hessian $\partial^2_{x,\xi}\Phi$ is non-singular, then a solution to the implicit equation
\begin{equation}
\chi(\partial_\xi \Phi(x,\xi),\xi)=(x,\partial_x\Phi(x,\xi))
\end{equation}
is symplectic on its domain. 
\end{prop}
For the particularly simple symplectic manifold $\T^n\times D$, with $D\subset \R^n$, we can make this construction global, provided that $\Phi(x,\xi)-\langle x,\xi \rangle$ is $2\pi$-periodic in $x$. The resulting symplectomorphism $\chi$ is then an exact symplectomorphism.

Provided that all $E\in[a,b]$ are regular values for the Hamiltonian $H$, the canonical symplectic form $d\xi\wedge dx$ on $T^*M$ determines a family of measures $\mu_E$ on each of the energy hypersurfaces
\begin{equation}
\Sigma_E=\{z=(x,\xi)\in T^*M:H(z)=E\}
\end{equation}
defined implicitly by
\begin{equation}
\int_a^b\int_{\Sigma_E} f\, d\mu_E \, dE = \int_{|(x,\xi)|_{g^{-1}}\in [a,b]}f\, d\xi\, dx
\end{equation}
for $f\in \mathcal{C}_c^\infty(T^* M)$. In the special case of of $H=\|\xi\|_g^2$ and $E=1$, upon normalisation we obtain the Liouville measure $\mu_L$ on $S^*M$. 

The measures $\mu_E$ allow us to study the ergodic properties of the Hamiltonian flow $\phi_t$.
\begin{defn}
\label{ergdefn}
If $\phi^t:X\rightarrow X$ denotes a measure preserving flow on a finite measure space $(X,\mathcal{A},\mu)$, we say that $\phi^t$ is ergodic if
\begin{equation}
\lim_{T\rightarrow\infty}\frac{m(\{t\in [0,T]:\phi^t(x)\in A\})}{T}=\frac{\mu(A)}{\mu(X)}\quad \textrm{for each }A\in\mathcal{A}
\end{equation}
for $\mu$-almost all $x\in X$.
\end{defn}

That is to say, a flow is ergodic if and only if almost all trajectories equidistribute in the measure space. An equivalent characterisation can be made in terms of flow invariant subsets.

\begin{prop}
A measure preserving flow $\phi_t$ on a finite measure space $(X,\mathcal{A},\mu)$ is ergodic if and only if the only $\phi_t$-invariant measurable sets are of full measure or null measure.
\end{prop}

In particular, we say that the Hamiltonian flow $\Phi^t$ generated by $H:T^*M\rightarrow \R$ is ergodic on the energy surface $\Sigma_E=H^{-1}(E)$ if $\Phi^t$ satisfies \ref{ergdefn} on $\Sigma_E$ with respect to the measure $\mu_E.$

Two particularly famous examples of ergodic Hamiltonian systems are the Bunimovich stadium billiard \cite{bunimovichstadium} and the Sinai billiard \cite{sinai}.

\begin{figure}[h]
\begin{center}
\includegraphics[scale=0.5]{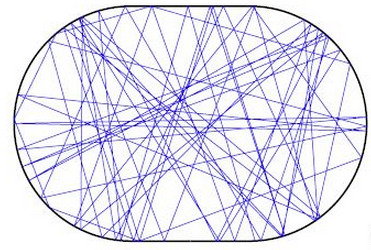}
\caption[]{An ergodic trajectory on the Bunimovich stadium billiard. Image from \cite{scholar}.}
\end{center}
\end{figure}

From Definition \ref{ergdefn}, we can see that for ergodic flows, the time average of a smooth classical observable $a:T^*M \rightarrow \mathbb{R}$ tends to its space average. That is, we have
\begin{equation}
\lim_{T\rightarrow \infty}\frac{1}{T}\int_0^T (a\circ\phi_t)(x_0,\xi_0)\, dt \rightarrow \int_{\Sigma_E} a \, d\mu_E
\end{equation}
for $\mu_E$-almost all $(x_0,\xi_0)\in\Sigma_E$.

A strictly stronger property of a flow is the mixing property which asserts that for smooth classical observables $a,b:T^*M\rightarrow \mathbb{R}$ we have
\begin{equation}
\label{mixing}
\int_{\Sigma_E} (a\circ\phi_t)\cdot b\, d\mu_E\rightarrow\left(\int_{\Sigma_E}a\, d\mu_E\right)\cdot\left(\int_{\Sigma_E}b\, d\mu_E\right)
\end{equation}
as $t\rightarrow\infty$.
The strong assumption of Anosov flow leads to \eqref{mixing} with an exponential rate of convergence. Thus, manifolds with non-positive sectional curvature also give rise to ergodic billiards. In this thesis we shall not discuss the stronger property of \eqref{mixing}, and restrict ourself to the study of ergodicity.

In order to study the Hamiltonian evolution of functions on our phase space $T^*M$, we define the Poisson bracket.
\begin{defn}[Poisson Bracket]
If $f,g\in\mathcal{C}^\infty (T^*M)$, we define
\begin{equation}
\{f,g\}:=\omega(X_f,X_g)=\sum_{j=1}^n \frac{\partial f}{\partial x_j}\frac{\partial g}{\partial \xi_j}-\frac{\partial f}{\partial \xi_j}\frac{\partial g}{\partial x_j}
\end{equation}
\end{defn}

An immediate consequence of the chain rule is that if $(\xi(t),x(t))$ is a trajectory of the Hamiltonian flow, then for a smooth function $f:T^*M\rightarrow \mathbb{C}$, we have 
\begin{equation}
\label{poissonderiv}
\frac{d}{dt}(f(\xi(t),x(t)))=\{f,H\}(\xi(t),x(t)).
\end{equation}

Motivated by this calculation we can define invariants of our flow.
\begin{defn}
An \emph{invariant} or \emph{first integral} of the Hamiltonian flow $X_H$ is a smooth function $f$ such that $\{f,H\}=0$.
\end{defn}

The Hamiltonian $H$ itself is of course an invariant of the flow that it generates. Hence Hamiltonian flow is constrained to energy shells $\Sigma_E=H^{-1}(\{E\})$. Often we can find additional invariants that are mathematical manifestations of conservation laws from physics, such as that of angular momentum. Symmetries in a system lead to an abundance of such flow invariants, as follows from Noether's theorem (See Chapter 4, Section 20 of \cite{arnoldmechanics}).

We say that a collection of flow invariants $\{f_i\}$ are in involution if their pairwise Poisson brackets vanish and we say that they are independent if their differentials $df_i$ are linearly independent.

\begin{defn}
If an invariant subset of Hamiltonian flow admits $n$ independent invariants that are in involution, we say that the corresponding Hamiltonian system is \emph{completely integrable.}
\end{defn}
In a completely integrable Hamiltonian system, trajectories are constrained to $n$-dimensional submanifolds and are thus far from equidistributed on the $(2n-1)$-dimensional energy surfaces. The Liouville--Arnold theorem from classical mechanics asserts that for completely integrable systems, we can find a neighbourhood $U$ of an arbitrary invariant manifold $\Lambda\subset T^*M$ and a symplectomorphism $\chi:\Lambda\rightarrow \T^n \times D$ for some $D\subset \R^n$ such that the transformed Hamiltonian $\tilde{H}(\theta,I)=H(\chi(\theta,I))$ is independent of $\theta$. Thus the invariant manifolds are diffeomorphic to $n$-dimensional tori.
Moreover, the Hamiltonian flow is quasiperiodic, with trajectories given by
\begin{equation}
\label{quasiperiodicflow}
I(t)=I(t_0);\qquad \theta(t)=\theta(t_0)+t\nabla_{I}\tilde{H}
\end{equation}
in the $(\theta,I)$-coordinates, referred to as \emph{action-angle} variables. A construction of these coordinates can be found in Section 50 of \cite{arnoldmechanics}.

We note that the invariant tori of a completely integrable Hamiltonian system are \emph{Lagrangian}, that is the restriction of the symplectic form to any of the invariant tori $\Lambda_{\mathbf{c}}\{(x,\xi):\mathbf{f}(x,\xi)=\mathbf{c}\in\R^n\}$ vanishes. 

An example of a completely integrable Hamiltonian system is geodesic billiards on an ellipsoid, pictured in Figure 1.2.

\begin{figure}[h]
\label{ellipsoidflow}
\begin{center}
\includegraphics[scale=0.5]{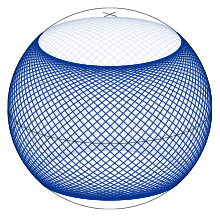}
\caption[]{Geodesic flow on an ellipsoid. Image generated using GeographicLib.}
\end{center}
\end{figure}

We can extend the above discussion to manifolds with boundary by extending Hamiltonian flow by reflection at non-tangential boundary collisions. We shall postpone this somewhat technical discussion until we require it in Chapter 3.

\section{Quantum dynamics}
\label{qdynsec}
The quantum mechanical analogue of the system \eqref{hameq} is the evolution of a \emph{wave-function} $\psi\in L^2(M)$ governed by Schr\"{o}dinger's equation.

\begin{equation}
\label{refschrod}
(ih\frac{\partial}{\partial t}-h^2\Delta)\psi=0
\end{equation}
where 
\begin{equation}
h\approx 6.626\times 10^{-34}\, m^2 \, kg \, s^{-1}
\end{equation} 
is Planck's constant and 
\begin{equation}
\Delta=-\textrm{div}\,\textrm{grad}
\end{equation}
is the Laplace--Beltrami operator with the positive sign convention.

The Bohr correspondence principle asserts that a classical Hamiltonian system is in a vague sense the macroscopic or high-energy limit of the corresponding quantum dynamical system.
By scaling the units in \eqref{refschrod}, we may instead consider $h$ to be a small parameter in our problem. This is known as the \emph{semiclassical formalism}. In the semiclassical formalism we replace differential operators with semiclassical differential operators $h\frac{\partial}{\partial x_j}$ to account for the scaling factor. By the Bohr correspondence principle, we then expect chaotic behaviour of the classical system to be manifest in the quantum system in the limit $h\rightarrow 0$.

In solving \eqref{refschrod}, we can expand $\psi$ in terms of the basis of $L^2(M)$ comprised by eigenfunctions of $h^2\Delta$, and so localisation of quantum dynamics can be understood by the study of the localisation of these eigenfunctions.

For Hamiltonian systems that are more general than billiards, such as the Schr\"{o}dinger type Hamiltonians in \eqref{schrodhamil}, we consider eigenfunction of the semiclassical Schr\"{o}dinger operator
\begin{equation}
P_h=H(x,hD)
\end{equation}
obtained by formally applying the Hamiltonian function to the operators $x$ and $hD=\frac{h\partial}{i}$.

One can then ask how the localisation properties of the Hamiltonian flow on $M$ are reflected by the spectral theory of the associated Schrodinger operator $P_h$ in the limit $h\rightarrow 0$.

Typically it is not possible to find exact, or even approximate expressions for chaotic eigenfunctions. Nevertheless, the machinery of microlocal analysis allows us to rigorously prove state and prove phase space equidistribution properties. The key tools here are pseudodifferential operators and Fourier integral operators, which correspond to quantisations of classical observables and symplectomorphisms respectively.

\section{Quantum ergodicity}
\label{qerdsec}

Under the assumption of ergodic Hamiltonian flow, one can make a remarkable statement of phase space equidistribution of the steady-state solutions to the corresponding Schr\"{o}dinger's equation \eqref{schrod}.

A primitive version of this theorem asserts that the sequence of probability measures $|u_j|^2$ on $M$ must have a full density subsequence which tends to the uniform measure.

The stronger statement of phase space equidistribution requires some additional machinery to state, such as the pseudodifferential calculus which we introduce in Chapter 2.

\begin{thm}[Quantum Ergodicity]
\label{qethm}
If the Hamiltonian $H(x,hD)=h^2\Delta_g+V(x)$ on the smooth compact boundaryless Riemannian manifold $(M,g)$ generates ergodic flow on the regular energy band $H^{-1}([a,b])$ for a smooth real potential $V(x)$ that is bounded below, then there exists a family of subsets $\Lambda(h)\subset [a,b]$ of eigenvalues of $P_h$ such that 
\begin{equation}
\lim_{h\rightarrow 0}\frac{\#\Lambda(h)}{\# \{E_j \in [a,b]\}}=1
\end{equation} 
and
\begin{equation}
\label{equidistributionqe}
\langle A(x,hD) u_j(h),u_j(h) \rangle \rightarrow \int_{H^{-1}([a,b])} A(x,\xi)\, d\xi\, dx
\end{equation}
uniformly for $j\in S(h)$ for any zero-th order semiclassical pseudodifferential operator $A(x,hD)$ with the property that
\begin{equation}
\frac{1}{\mu_E}\int_{\Sigma_E} A(x,\xi)\, d\mu_E
\end{equation}
is independent of $E$.
\end{thm}
A semiclassical proof of this theorem can be found in \cite{zworski}, whilst the initial result goes back to \cite{schnirelman},\cite{zelditch},\cite{verdiere}. For billiards on manifolds with boundary, there are additional technical considerations even from the purely dynamical perspective. Nevertheless, the quantum ergodicity theorem generalises to this setting \cite{zelditch-zworski},\cite{gerard-leichtnam}.

We can also define a notion of quantum ergodicity localised to an individual energy surface that is motivated by the results of \cite{HRM}. We will make use of this definition in the proof of the negative quantum ergodicity result  Theorem \ref{nonqe}, in Chapter 6.

Suppose the semiclassical pseudodifferential operator $P_h=h^2\Delta+V(x,hD)$ on the smooth manifold $M$ has principal symbol $p(x,\xi)\in \mathcal{C}^\infty(T^*M)$ and has purely point spectrum, with eigenpairs $(u_j(h),E_j(h))$ in increasing order. 

If $E\in \R$ is a regular value of $p$ with nonempty preimage, then we can define quantum ergodicity localised to the energy surface $\Sigma_E$ as follows.

\begin{defn}
\label{qesurface}
We say that $P_h$ is quantum ergodic at energy $E$ if for each sufficiently small $h>0$, there exists a family $S(h)\subset \{j\in \N:E_j(h)\in [E-h,E+h]\}$ such that
\begin{equation}
\frac{\#S(h)}{\# \{j\in \N:E_j(h)\in [E-h,E+h\}}\rightarrow 1
\end{equation}
and
\begin{equation}
\langle A(x,hD) u_j(h),u_j(h) \rangle \rightarrow \int_{\Sigma_E} A(x,\xi)\, d\mu_E
\end{equation}
uniformly for $j\in S(h)$ for any zero-th order semiclassical pseudodifferential operator $A(x,hD)$.
\end{defn}

\begin{remark}
We could replace the $O(h)$ width energy windows in Definition \ref{qesurface} with $O(h^\beta)$ for $0<\beta<1$. 
The $O(h)$ size energy windows in Definition \ref{qesurface} are the smallest windows in which quantum ergodicity results are possible because the Weyl asymptotics break down in smaller energy windows.
\end{remark}

An alternate formulation of quantum ergodicity can be made in the language of \emph{semiclassical measures}. We shall state this version of quantum ergodicity in the special case $H=\|\xi\|_g^2$, as we only make use of it in this setting of billiards in Chapter 3.

To each subsequence of $(u_j)$, we can associate at least one non-negative Radon measure $\mu$ on $S^*M$ which provides a notion of phase space concentration in the semiclassical limit. 
We say that the eigenfunction subsequence $(u_{j_k})$ has unique \emph{semiclassical measure} $\mu$ if 
\begin{equation}
\label{semimeaseqn}
\lim_{k\rightarrow \infty}\langle a(x,E_{j_k}^{-1/2}D)u_{j_k},u_{j_k} \rangle=\int_{S^*M} a(x,\xi)\, d\mu
\end{equation}
for each semiclassical pseudodifferential operator with principal symbol $a$ compactly supported supported away from the boundary of $S^*M$. In Chapter 5 of \cite{zworski}, the existence and basic properties of semiclassical measures are established using the calculus of semiclassical pseudodifferential operators (see also \cite{gerard-leichtnam}).

A billiard $M$ can then be said to be \emph{quantum ergodic} if there is a full density subsequence of eigenfunctions $(u_{n_k})$ such that the the Liouville measure on $S^*M$ is the unique semiclassical measure associated to the sequence $(u_{n_k})$. This statement can be interpreted as saying that the sequence of eigenfunctions equidistributes in phase space with the possible exception of a sparse subsequence.

Under the stronger assumption of Anosov flow, it is conjectured that the full sequence of eigenfunctions equidistributes in the sense of Theorem \ref{qethm}. This is known as the quantum unique ergodicity conjecture.

The prize-winning work of Lindenstrauss \cite{linden} verified this conjecture in certain arithmetic cases where we work with the Hecke joint eigenfunctions. The study of quantum ergodicity where we have this additional arithmetic structure is known as arithmetic QE. Sarnak \cite{sarnak} has written a survey on the recent developments in this field.

On the other hand, it is known that quantum ergodicity is strictly weaker than quantum unique ergodicity. Indeed, Hassell \cite{hassellque} showed that on the Bunimovich stadium there exist semiclassical measures that have positive mass on the union of the bouncing ball trajectories in phase space. 

\section{Negative results}
\label{nressec}

In the extreme case of quantum unique ergodicity, there is a unique semiclassical measure, which is the Liouville measure. It is natural to ask what we can say about the semiclassical measures associated with sequences of eigenfunctions of Hamiltonian systems that are not quantum uniquely ergodic.

For the Bunimovich stadium, the quantum ergodicity theorem implies that any non-uniform limit can only arise from a density-zero subsequence. 

Whilst Burq--Zworski \cite{burqzworski} showed that concentration in a strict subrectangle is not possible, numerical evidence suggests that there could well be a sparse sequence of eigenfunctions with semiclassical limit supported in the rectangle itself. Rigorous proof of this phenomenon remains an open problem, with the most notable progress being Hassell's proof that a semiclassical measure exists with positive mass on the union of bouncing ball trajectories \cite{hassellque}.

\begin{figure}[h]
\begin{center}
\includegraphics[scale=0.5]{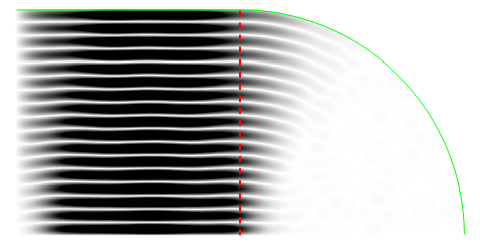}
\caption[]{An apparent ``bouncing ball" eigenfunction in the quarter stadium corresponding to the eigenvalue $\approx 2859.47$. Image courtesy of Dr. Barnett.}
\end{center}
\end{figure}
On the other hand, if a Hamiltonian system is assumed to be completely integrable, any trajectory is constrained to a single invariant torus corresponding to the intersection of the level sets of the $n$ conserved quantities. The intuition stemming from the classical-quantum correspondence suggests that this extreme concentration of trajectories should manifest itself in a statement about the concentration of eigenfunctions onto the Lagrangian tori.

Such a result is proven in \cite{toth} for systems satisfying a stronger notion of quantum integrability and tori satisfying a certain nonresonance condition, however rigorous results in the general setting of complete integrability seem to be elusive. 

At this point we introduce the notion of approximate eigenfunctions, or quasimodes.
\begin{defn}
Given a semiclassical pseudodifferential operator $P(h)$, a $O(h^\beta)$ \emph{quasimode} is a family of functions $u_h\in L^2(M)$ such that
\begin{equation}
\|(P(h)-E)u_h\|_{L^2(M)}=O(h^\beta)
\end{equation}
for some $\beta>0$ and some real $E$, referred to as the quasi-eigenvalue.
\end{defn}

\begin{remark}
As a consequence of the semiclassical rescaling, it should be noted that $O(h^2)$ for the semiclassical Laplacian $h^2\Delta$ correspond to $O(1)$ quasimodes of the Laplacian $\Delta$.
\end{remark}

We can of course replace the $O(h^\beta)$ with a stronger bound in this definition. The uses of quasimodes are plentiful. Most results about eigenfunctions apply just as well to quasimodes, and it is easier to construct quasimodes than exact eigenfunctions. In fact we can often construct quasimodes with desirable localisation properties, as they are better behaved with respect to taking cutoffs than exact eigenfunctions which are generally destroyed. The 

In \cite{colin}, Colin de Verdi\`{e}re established that for completely integrable Hamiltonian systems, there exist quasimodes with exponentially small error term that localise onto the certain individual invariant Langrangian tori. This result relies on the construction of a quantum Birkhoff normal form.

\section{Mixed and KAM systems}

Between the extremes of completely integrable Hamiltonian dynamics and ergodic Hamiltonian dynamics, results are rather sparse. The original work in this thesis explores two intermediate classes of Hamiltonian dynamical systems for which questions of quantum ergodicity are tractable.

The two main original results in this thesis, Theorem \ref{mainthm} and Theorem \ref{nonqe}, both make use of known quasimode constructions and perturbation arguments in order to prove eigenfunction localisation statements in the  cases of mixed billiards and KAM system respectively. We now outline these two classes of Hamiltonian systems.

\subsection*{Mixed billards}

If a dynamical billiard can be separated into multiple invariant subsets, only some of which are ergodic, it is said that the billiard is \emph{mixed}. In this case, it is conjectured that we can divide the sequence of eigenfunctions into corresponding families, with the eigenfunctions corresponding to an ergodic invariant subset satisfying a suitable equidistribution property. For the sake of simplicity, we shall state the conjecture in the case of billiards with exactly two invariant subsets, one ergodic and one completely integrable.

\begin{con}[Percival's Conjecture]
\label{percivalintrostrong}
For every compact Riemannian manifold $M$ such that $S^*M$ is the disjoint union of two invariant subsets $U,S^*M\setminus U$, with $U$ ergodic and $S^*M\setminus U$ completely integrable, we can find two subsets $A,B\subset \N$ such that

\begin{enumerate}
\item $A\cup B$ has density $1$ ;
\item $(u_k)_{k\in A}$ equidistributes in the ergodic region $U$ ;
\item Each semiclassical measure associated to the subset $B$ is supported in the completely integrable region $\mathcal{D} \setminus U$;
\item The density of $A$ is equal to $\mu_L(U)$.
\end{enumerate}
\end{con}
Numerical evidence \cite{barnett} strongly supports this conjecture, but no rigorous proofs have been discovered, even for concrete examples.

A weaker version of Percival's conjecture is formulated by slightly relaxing the density requirements of the subsets $A$ and $B$.
\begin{con}
[Weak Percival's Conjecture]
\label{percivalintro}
For every compact Riemannian manifold $M$ such that $S^*M$ is the disjoint union of two invariant subsets $U,S^*M\setminus U$, with $U$ ergodic and $S^*M\setminus U$ completely integrable, we can find two subsets $A,B\subset \N$ such that

\begin{enumerate}
\item $A\cup B$ has upper density $1$ 
\item $(u_k)_{k\in A}$ equidistributes in the ergodic region $U$ 
\item Each semiclassical measure associated to the subset $B$ is supported in the completely integrable region $S^*M \setminus U$
\item The upper densities of $A$ and $B$ are equal to $\mu_L(U)$ and $1-\mu_L(U)$ respectively.
\end{enumerate}
\end{con}

In Chapter 3, I provide the first verification of the weak Percival's conjecture for a family of ``mushroom" billiards, defined in Section \ref{chp3introsec}. The main result is
\begin{thm}
Conjecture \ref{percivalintro} holds for the mushroom billiard $M_t$ for any fixed inner and outer radii, and almost all ``stalk lengths" $t\in(0,2]$.
\end{thm}

\subsection*{KAM Hamiltonian systems}
A particularly interesting class of Hamiltonian systems arise if we apply a small perturbation to a completely integrable real analytic Hamiltonian in action-angle coordinates.
\begin{equation}
H(\theta,I):=H^0(I)+\epsilon H^1(\theta,I).
\end{equation}

Motivated by the geometry of completely integrable systems, where our phase space $\T^n\times D$ is foliated by the Lagrangian tori 
\begin{equation}
\Lambda_\omega:= \{(\theta,I)\in \T^n\times D:\theta\in \T^n\}
\end{equation}
where $\nabla_IH^0=\omega$, it is natural to ask whether there are any such invariant Lagrangian tori that survive the perturbation. This problem is one of real physical significane, as one application is the study of celestial stability by viewing the dynamics of the solar system as a small perturbation of the completely integrable system that results from neglecting forces between pairs of planets. This perturbation is of course ``small" because of the considerably greater mass of the sun compared to the planets.

The initial significant breakthrough in this problem was due to Kolmogorov \cite{kolmogorov}, with the conclusion that although a dense set of tori is indeed generally destroyed by such a perturbation, a large measure collection of the invariant tori survive, precisely those whose frequency $\omega=\nabla_I H^0$ of quasiperiodic flow \eqref{quasiperiodicflow} satisfy the Diophantine condition
\begin{equation}
|\langle\omega,k\rangle|\geq \frac{\kappa}{|k|^\tau}
\end{equation}
for all nonzero $k\in \mathbb{Z}^n$ and fixed $\kappa >0$ and $\tau>n-1$.
The tori satisfying this Diophantine condition are said to be nonresonant.

The field of KAM theory developed from this problem as a broad class of techniques applicable to perturbation problems in classical mechanics, founded by Kolmogorov, Arnold and Moser.

More recent work by Popov \cite{popovkam} proved a version of the KAM theorem for Hamiltonian systems in the Gevrey regularity class, with the purpose of constructing a Birkhoff normal form. This led to a quantum Birkhoff normal form construction, and a proof of the existence of quasimodes with exponentially small error localising onto the nonresonant tori in \cite{popovquasis}.

The details of Popov's construction are summarised in Chapter 4 and Chapter 5 for families of Hamiltonians of the form
\begin{equation}
\label{1paramfamyo}
H(x,\xi;t)=\|\xi\|_g^2+V(x,\xi)+tQ(x,\xi)
\end{equation}
for smooth real-valued symbols $V,Q$ in a suitable Gevrey class $S_\ell(T^*M)$ in the notation of \eqref{selldef}.

In Chapter 6, we prove the following main result. The formal statement is Theorem \ref{nonqe}.
\begin{thm}
Suppose $M$ is a compact boundaryless Gevrey smooth Riemannian manifold, and the perturbation $Q$ is such that $Q(x,hD)$ is a positive operator and there exists a \emph{slow torus} in the energy band $[a,b]$.

Then there exists $\delta>0$ such that for almost all $t\in[0,\delta]$ the quantisation $P_h(t)$ of $H(x,\xi;t)$ is non-quantum ergodic over the energy surface $\Sigma_E$ for a positive Lebesgue measure subset of energies $E\in [a,b]$.
\end{thm}

A \emph{slow torus}, defined formally in \ref{slowtorusdef}, is an invariant nonresonant Lagrangian torus $\Lambda$ in the energy surface $\mathcal{E}:=\{(x,\xi)\in T^*M:\|\xi\|_g^2+V(x,\xi)=E\}$ such that the average of $Q$ over $\Lambda$ is strictly smaller than the average of $Q$ over $\mathcal{E}$. The assumption of the existence of such a torus is a mild one, and will typically be satisfied by perturbations whose symbols are nonconstant on energy surfaces.  

\chapter{Semiclassical Analysis}

In this chapter, we briefly collect some of the necessary machinery of the semiclassical pseudodifferential calculus necessary for the results in the remainder of this thesis. The Fourier transform on $\R^n$ from classical harmonic analysis allows us to pass from the $n$-dimensional position space to the $n$-dimensional frequency space. Pseudodifferential operators allow us to formulate and prove statements in the full $2n$-dimensional phase space, and generalise the procedure of using a Fourier multiplier as a frequency cutoff. Some standard references for the classical pseudodifferential calculus include \cite{gs} \cite{shubin} \cite{ho3}. Our presentation shall be in the semiclassical formalism, for which an extensive account can be found in \cite{zworski} and \cite{sjostrand}.

A key application of the pseudodifferential calculus to spectral theory is Weyl's law, which provides an asymptotic for eigenvalues of a Schr\"{o}dinger operator.

\section{Semiclassical pseudodifferential operators}
\label{pseudosec}
We begin by presenting the semiclassical pseudodifferential calculus on $\R^{2n}$.

The semiclassical pseudodifferential calculus provides a correspondence between \emph{classical observables} (smooth functions on the phase space $\R^{2n}$) and \emph{quantum observables} (integral operators on position space $\R^n$).

The classical observables in this correspondence are traditionally referred to as symbols, and estimates on their derivatives are required to obtain desirable mapping properties for their quantisations.

For such operators, we can define semiclassical pseudodifferential operators on $\R^n$.
\begin{equation}
\label{standard}
a(x,hD)u(x):=(2\pi h)^{-n}\int\int e^{i(x-y)\cdot \xi/h}a(x,\xi;h)u(y)\, dy \, d\xi
\end{equation}

The quantisation \eqref{standard} is referred to as the standard quantisation. It is sometimes more convenient to work with a formally self-adjoint operator however. This motivates the definition of the Weyl quantisation.
\begin{equation}
a^w(x,hD)u(x)=(2\pi h)^{-n}\int\int e^{i(x-y)\cdot \xi/h}a(\frac{x+y}{2},\xi;h)u(y)\, dy \, d\xi
\end{equation}
which is a formally self-adjoint operator if $a$ is real.

These formally defined integral operators are clearly convergent if $a$ and $u$ are of Schwartz class, but is otherwise understood in the sense of oscillatory integrals (\cite{zworski} Theorem 3.8). In this fashion, the Kohn--Nirenberg symbol class leads to a class of semiclassical pseudodifferential operators bounded on semiclassical Sobolev spaces.

The standard class of $m$-th order Kohn--Nirenberg symbols on $\R^{2n}$ is given by:
\begin{equation}
\label{kohndef}
S^m:= \{a(x,\xi)\in \mathcal{C}^\infty(\R^{2n}):|\partial_x^\alpha\partial_\xi^\beta a|\leq C_{\alpha,\beta}(1+|\xi|^2)^{\frac{m-|\beta|}{2}}\textrm{ for all }\alpha,\beta\}.
\end{equation}
\begin{remark}
It is also sometimes useful to consider classes $S_\delta^m$, where the right-hand side of \eqref{kohndef} is multiplied by $h^{-\delta}$ with each differentiation, but we shall not require these symbol classes.
\end{remark}

This is a sufficiently broad symbol class for most applications, and includes semiclassical differential operators
\begin{equation}
\sum_{|\alpha|\leq n} a_\alpha(x)(hD)^\alpha 
\end{equation}
as a special case by quantising polynomials in $\xi$. 

The index $m$ in Definition \ref{kohndef} corresponds to the mapping properties of the associated pseudodifferential operator.
\begin{prop}
If $a\in S^m$, then
\begin{equation}
a(x,hD):H^{s}\rightarrow H^{s-m}
\end{equation}
is a bounded operator for any $s\in \R$, where $H^s$ denotes the Sobolev space of order $s$. In particular, zero-th order semiclassical pseudodifferential operators are bounded on $L^2$, and negative order semiclassical pseudodifferential operators are compact on $L^2$.
\end{prop}

In practice, symbols of semiclassical pseudodifferential operators are often constructed using formal power series in $h$. Indeed, if $a\in S^m(\R^{2n})$ and $a_j\in S^m(\R^{2n})$ for each $j$, we introduce the notation
\begin{equation}
\label{symbolexp}
a(x,\xi)\sim\sum_{j=0}^\infty a_j(x,\xi)h^j,
\end{equation}
to mean that
\begin{equation}
\left|\partial_x^\alpha\partial_\xi^\beta \left(a-\sum_{j=0}^N a_jh^j\right)\right|\leq C_{\alpha,\beta,N}h^{N+1}(1+|\xi|^2)^{\frac{m-|\beta|}{2}}
\end{equation}
for each $n$ uniformly in some interval $h\in (0,h_{N,\alpha,\beta}]$.

The key point is that for an arbitrary formal series $\sum_j a_jh^j$  with symbols $a_j\in S^m$, we can find a $\sim$-equivalent symbol $a\in S^m$.
\begin{prop}
\label{borelprop}
Given an arbitrary sequence of symbols $a_j\in S^m(\R^{2n})$, there exists a symbol $a\in S^m(\R^{2n})$ satisfying \eqref{symbolexp}. We call $a$ the \emph{Borel resummation} of the formal series $\sum_j a_jh^j$.
\end{prop}
A proof of Proposition \ref{borelprop} can be found in Theorem 4.15 of \cite{zworski}. 

We refer the the leading term $a_0$ in \eqref{symbolexp} as the \emph{principal symbol} of $a$, and write $a_0=\sigma(a)$. This is of course only well-defined modulo $O(h)$.

From repeated integration by parts in \eqref{standard}, we see that if the $\textrm{dist}(\textrm{spt}(u),\textrm{spt}(a))=c>0$, then the function $(a(x,hD)u)(x)$ is of size $O(h^n)$ for any $n\in\N$. We denote such a size estimate by $O(h^\infty)$, and note that these terms can be regarded as negligible in the semiclassical limit $h\rightarrow 0$.

At this point we introduce the notion of a semiclassical wavefront set for $L^2$ functions.
\begin{defn}
\label{microsupport}
Suppose $u(x;h)$ is a collection of smooth functions on $\R^n$ for $h\in (0,h_0]$. Then the  semiclassical wavefront set $WF_h(u)\subset \R^{2n}$ is defined as follows.
$(x_0,\xi_0)\in (WF_h(U))^c$ if there exists a $v\in \mathcal{C}_0^\infty(\R^{2n})$ with $|v(x_0,\xi_0)|>0$ such that we have
\begin{equation}
v(x,hD)u=O_{L^2}(h^k).
\end{equation}
for any $k\in \N$.
\end{defn}
Such a definition is possible in considerably more general classes of distributions (See Section 8.4 of \cite{zworski}), but we shall not require it in this generality.
\begin{remark}
In fact, it suffices to prove that $v(x,h_jD)u=O_{L^2}(h_j^k)$ for any $k\in \N$ and a single sequence $h_j\rightarrow 0$.
\end{remark}

A crucial formula in the pseudodifferential calculus is the composition formula, which assets that if $a\in S^{m_1}$ and $b\in S^{m_2}$, then the composition 
\begin{equation}
a(x,hD)\circ b(x,hD)
\end{equation}
is a semiclassical pseudodifferential operator of order $m_1+m_2$, and its symbol is given by
\begin{equation}
\label{ogcompform}
a\circ b:= \exp(\frac{h}{i}\langle\partial_\xi,\partial_y  \rangle)(a(x,\xi)b(y,\eta))_{(y,\eta)=(x,\xi)}.
\end{equation}
as is shown in Theorem 9.5 of \cite{zworski}.

Expanding the symbols $a,b,$ and $a\circ b$ in \eqref{ogcompform} as semiclassical series yields the following
\begin{prop}
Given two symbols $a\in S^{m_1}$ and $b\in S^{m_2}$, their composition $p\circ q\in S_{m_1+m_2}$ as the Borel  resummation of
\begin{equation}
a\circ b \sim \sum_{j=0}^\infty c_j h^j
\end{equation}
where
\begin{equation}
\label{compform}
\sum_{r+s+|\gamma|=j}\frac{1}{\gamma!}\partial_\xi^\gamma a_r(x,\xi)\partial_x^\gamma b_s(x,\xi).
\end{equation}
\end{prop}


A key feature of the pseudodifferential calculus that immediately lends itself to PDE applications is that of the invertibility of elliptic operators. 
\begin{prop}
\label{ellipticinverse}
If we have a symbol $a\in S^m$ with
\begin{equation}
a(x,\xi)\geq c(1+|\xi|^2)^{m/2}
\end{equation}
for some $c>0$, then there exists a symbol $b\in S^{-m}$ with
\begin{equation}
a(x,hD)\circ b(x,hD)=O(h^\infty).
\end{equation}
\end{prop}
The proof of Proposition \ref{ellipticinverse} is an application of \eqref{ogcompform} and can be found in  Proposition 2.6.10 of \cite{martinez}.

Importantly, for an arbitrary diffeomorphism from $\gamma:U\rightarrow V$ with $U,V\subset \R$ open, the symbols classes $S^m$ and invariant under the pullback of the lift of $\gamma$ to a symplectomorphism $\tilde{\gamma}:U\times \R^n\rightarrow V\times \R^n$. This invariance allows for the construction of semiclassical pseudodifferential operators on compact manifolds, as is done in (\cite{zworski} Chapter 14).

\begin{defn}
We write $S^m(T^*M)$ to denote the class of $m$-th order Kohn-Nirenberg symbols on a compact manifold $M$ and we write $S^{m,k}(T^*M):=h^{-k}S^{m,k}(T^*M):=h^{-k}S^m(T^*M)$ to denote the class of Kohn-Nirenberg symbols of differential order $m$ and semiclassical order $k$.
\end{defn}

\begin{defn}
We write $\Psi^m(M)$ to denote the class of $m$-th order semiclassical pseudodifferential operators on $M$ in the sense of (\cite{zworski} Chapter 14).
We write $\Psi^{m,k}(M):=h^{-k}\Psi^m(M)$ to denote the class of semiclassical pseudodifferential operators of differential order $m$ and semiclassical order $k$.
\end{defn}
One significant difference between the calculus on compact manifolds and on Euclidean space however, is that the symbol of a semiclassical pseudodifferential operator $A_h\in \Psi^{m,k}(M)$ is only invariantly defined modulo $S^{m-1,k-1}$.

One can also define semiclassical pseudodifferential operators on the space of half-densities on a compact manifold $M$.

\begin{defn}
A half-density $\rho$ on an $n$ dimensional vector space $V$ is a map $\rho:V^n\rightarrow \R$ such that
\begin{equation}
\rho(Av_1,Av_2,\ldots,Av_n)=|\det(A)|^{1/2}\rho(v_1,v_2,\ldots,v_n).
\end{equation}
for any $A$ is a linear transformation on $V$.
We denote the space of half-densities on $V$ by $\Omega^{1/2}(TM)$.
\end{defn}

\begin{defn}
The space $\mathcal{C}^\infty(M,\Omega^{1/2})$ of smooth half-densities on a compact Riemannian manifold is given by the collection of maps $u:M\rightarrow \cup_{x\in M}\Omega^{1/2}(TM)$ such that $u(x)\in \Omega^{1/2}(T_xM)$ for each $x\in M$, and $u(X_1,X_2,\ldots,X_n)\in \mathcal{C}^\infty(M)$ for any $n$ smooth vector fields $X_1,X_2,\cdots,X_n$.
\end{defn}

Since half-densities are given in local coordinates on a Riemannian manifold by
\begin{equation}
\mathcal{C}^\infty(M;\Omega^{1/2}):= \{f|dx|^{1/2}:f\in \mathcal{C}^\infty(M)\}
\end{equation}
where $dx$ is the Riemannian volume form, we can identify half-densities with functions in this setting, however note that their pullbacks as half-densities will involve a Jacobian factor.

Thus we can locally define semiclassical pseudodifferential operators on half-densities by setting
\begin{equation}
P(x,hD)(u|dx|^{1/2})=(P(x,hD)u)|dx|^{1/2},
\end{equation}
and they can be defined globally in a similar fashion to semiclassical pseudodifferential operators acting on functions. (See Section 14.2.5 of \cite{zworski}).

An advantage of working with half-densities is that principal symbols of semiclassical pseudodiffential operators on half-densities are invariantly in $S^{m,k}/S^{m-2,k-2}$, and subprincipal symbols of operators are thus invariantly defined. (See Section 1.3 of \cite{sternberg} for a further discussion of this invariance).

In Section 5.2, we work with semiclassical \emph{Fourier integral operators}, which are generalisations of semiclassical pseudodifferential operators obtained locally by replacing the phase function $i(x-y)\cdot \xi/h$ in the oscillatory integral expression \eqref{standard} with more general phase functions $i\psi(x,y,\xi)/h$.

The kernels of such operators are then special cases of Fourier integrals
\begin{equation}
\label{lagdist}
u(x;h)=(2\pi h)^{-m}\int_{\R^m} e^{i\phi(x,\xi)/h}a(x,\xi)\, d\xi
\end{equation}
with $m=2n$. For such a phase function, we can associate a Lagrangian submanifold of $\R^{2m}$ given by
\begin{equation}
\Lambda_\phi=\{(x,\partial_x\phi(x,\xi)):\partial_\xi(x,\xi)=0\}.
\end{equation}
Indeed, stationary phase asymptotics show that $WF_h(u)\subset \Lambda_\phi$ as in \cite{hormander}.

Defining a canonical relation $\chi:T^*M_1\rightarrow T^*M_2$ to be a relation with flipped graph
\begin{equation}
\Gamma_\chi:=\{(x,\xi,y,-\eta):(x,\xi,y,\eta)\in \chi\}
\end{equation}
a Lagrangian submanifold of $T^*(M_1\times M_2)$, we can then define a Fourier integral operator associated to a given relation $\mathcal{C}^\infty(M_1)\rightarrow\mathcal{C}^\infty(M_2)$ to be a finite sum of Fourier integrals associated to $\Gamma_\chi$.

The global theory of Fourier integrals is complicated by the fact that different phase functions $\phi$ can parametrise the same Lagrangian manifold $\Lambda$ locally, yet for different $\phi$ a different symbol $a$ will be required in \eqref{lagdist} in order to represent the same Fourier integral $u(x;h)$. In order to invariantly define the notion of a principal symbol for a Fourier integral operator, it must be defined as an object on a certain line bundle over the Lagrandigan submanifold $\Lambda$, known as the Maslov bundle.

A thorough account of Fourier integral operators can be found in the seminal paper \cite{hormander} in the classical setting, and in \cite{sternberg} in the semiclassical setting. We shall summarise the relevant details in our exposition of Popov's construction of the quantum Birkhoff normal form for KAM Hamiltonians \cite{popov2}\cite{popovquasis} in Section \ref{fioconj}.

\section{Gevrey class symbols}
\label{gevsymbsec}
Our application of the semiclassical pseudodifferential calculus in Chapter 5 involves working with Gevrey class  symbols. We outline the relevant differences from the theory in Section \ref{pseudosec} here.

We suppose $D$ is a bounded domain in $\mathbb{R}^n$, and take $X=\mathbb{T}^n$ or a bounded domain in $\mathbb{R}^m$.
We fix the parameters $\sigma,\mu>1$ and $\varrho\geq \sigma+\mu-1$, and denote the triple $(\sigma,\mu,\varrho)$ by $\ell$.

\begin{defn}
\label{formalsymboldefn}
A formal Gevrey symbol on $X\times D$ is a formal sum
\begin{equation}
\label{formalsymbol}
\sum_{j=0}^\infty p_j(\theta,I)h^j
\end{equation}
where the $p_j\in\mathcal{C}_0^\infty(X\times D)$ are all supported in a fixed compact set and there exists a $C>0$ such that
\begin{equation}
\sup_{X\times D} |\partial_\theta^\beta \partial_I^\alpha p_j(\theta,I)|\leq C^{j+|\alpha|+|\beta|+1}\beta!^\sigma\alpha!^\mu j!^\varrho.
\end{equation}
\end{defn}

\begin{defn}
A \textsf{realisation} of the formal symbol \eqref{formalsymbol} is a function $p(\theta,I;h)\in\mathcal{C}_0^\infty(X\times D)$ for $0<h\leq h_0$ with
\begin{equation}
\sup_{X\times D \times (0,h_0]} \left|\partial_\theta^\beta \partial_I^\alpha \left(p(\theta,I;h)-\sum_{j=0}^N  p_j(\theta,I)h^j\right)\right|\leq h^{N+1}C_1^{N+|\alpha|+|\beta|+2}\beta!^\sigma\alpha!^\mu (N+1)!^\varrho.
\end{equation}
\end{defn}

\begin{lem}
Given a formal symbol \eqref{formalsymbol}, one choice of realisation is
\begin{equation}
p(\theta,I;h):= \sum_{j\leq \epsilon h^{-1/\varrho}} p_j(\theta,I)h^j
\end{equation}
where $\epsilon$ depends only on $n$ and $C_1$.
\end{lem}

\begin{defn}
We define the residual class of symbols $S_\ell^{-\infty}$ as the collection of realisations of the zero formal symbol.
\end{defn}

\begin{defn}
\label{selldef}
We write $f\sim g$ if $f-g\in S_\ell^{-\infty}$. It then follows that any two realisations of the same formal symbol are $\sim$-equivalent.
We denote the set of equivalence classes by $S_\ell(X\times D)$.
\end{defn}

An important feature of the Gevrey symbol calculus is that the symbol class $S_\ell(X\times D)$ is closed under composition. 

We can now discuss the pseudodifferential operators corresponding to these symbols.
\begin{defn}
\label{gevpseudo}
To each symbol $p\in S_\ell(X\times D)$, we associate a semiclassical pseudodifferential operator defined by
\begin{equation}
(2\pi h)^{-n}\int_{X\times \mathbb{R}^n}e^{i(x-y)\cdot \xi/h}p(x,\xi;h)u(y)\, d\xi\, dy.
\end{equation}
for $u\in \mathcal{C}_0^\infty(X)$.
\end{defn}

The above construction is well defined modulo $\exp(-ch^{-1/\varrho})$, as for any $p\in S_\ell^{-\infty}(X\times D)$ we have
\begin{equation}
\|P_hu\|=O_{L^2}(\exp(-ch^{-1/\varrho}))
\end{equation}
for some constant $c>0$.
\begin{remark}
The exponential decay of residual symbols is a key strengthening that comes from working in a Gevrey symbol class.
\end{remark}

The operations of symbol composition and conjugation then correspond to composing operators and taking adjoints respectively.
Moreover, if $p\in S_{(\sigma,\sigma,2\sigma-1)}$, then $G^\sigma$-smooth changes of variable preserve the symbol class of $p$.
This coordinate invariance allows us to extend the Gevrey pseudodifferential calculus to compact Gevrey manifolds.\\

At this point we introduce the notion of a microsupport in the Gevrey sense.
\begin{defn}
\label{gevmicro}
Suppose $u(x;h)$ is a collection of smooth functions on the $G^\sigma$-manifold $M$ for $h\in (0,h_0]$. Then the $G^\varrho$ microsupport $MS^\varrho(u)\subset T^*M$ is defined as follows.\\ 

$(x_0,\xi_0)\in (MS^\varrho)^c$ if there exists a product of compact sets $U\times V\subset X\times \mathbb{R}^n$ with $(x_0,\xi_0)\in U\times V$ inside a single coordinate chart and there exists a $c>0$ such that for any $v\in G^\varrho\cap\mathcal{C}_0^\infty(M)$ we have
\begin{equation}
\int_{\mathbb{R}^n}e^{ix\cdot \xi /h}v(x)u(x;h)\, dx=O(e^{-ch^{-1/\varrho}})
\end{equation}
uniformly in $V$.
\end{defn}

It follows from stationary phase that if a symbol $p$ is $S_\ell^{-\infty}$ in a neighbourhood of a point $(x_0,\xi_0)$, then the point $(x_0,\xi_0,x_0,-\xi_0)$ lies outside the $G^\varrho$ microsupport of the distribution kernel of $P_h$.\\

\section{Weyl law}

An application of the semiclassical pseudodifferential calculus that is particularly important to us is the semiclassical Weyl law, which provides asymptotics for the counting functions of eigenvalues for suitable semiclassical pseudodifferential operators $P(h)$ in fixed energy bands $[a,b]$ or shrinking energy bands $[a,a+\epsilon(h)]$ with $\epsilon(h)\rightarrow 0$ as $h\rightarrow 0$. 

We consider semiclassical pseudodifferential operators of the form
\begin{equation}
P(h):=h^2\Delta_g+Q(x,hD)\in \Psi^2(M)
\end{equation}
on a compact Riemannian manifold $M$, where $Q(x,\xi)\in S^0(M)$ is real valued.

For each fixed $h>0$, the operator $P(h)$ is a self-adjoint operator
\begin{equation}
P(h):H^2(M)\subset L^2(M)\rightarrow L^2(M)
\end{equation}
with compact inverse, where $H^2(M)$ is the Sobolev space of order $2$.

Basic spectral theory then tells us that the spectrum of $P(h)$ is real and discrete, consisting of a countable orthonormal basis of eigenpairs $(u_j(h),E_j(h))$, with $E_j\rightarrow \infty$ as $j\rightarrow \infty$.

Weyl's law is then the statement that
\begin{equation}
\label{zworskiweyl}
(2\pi h)^{-n}\#\{j\in \N:E_j(h)\in [a,b]\}\rightarrow \mu(\{(x,\xi)\in T^*M:P(x,\xi)\in [a,b]\}).
\end{equation}
where $\mu$ denotes the symplectic measure $d\xi \, dx$ on $T^*M$.

A standard proof relies on a trace formula for a Schwartz class functional calculus for semiclassical  pseudodifferential operators.
If $f\in\mathcal{S}(\R)$, then we can define
\begin{equation}
f(P(h))(u):=\sum_{j=1}^\infty f(E_j)\langle u,u_j\rangle u_j.
\end{equation}
The rapid decay of $f$ in fact implies that $f(P(h))$ is a semiclassical pseudodifferential operator in the class
\begin{equation}
\Psi^{-\infty}(M)=\bigcap_{m\in \Z} \Psi^m(M).
\end{equation}

In fact it can be shown that $f(P(h))$ is a trace-class operator on $L^2(M)$, with principal symbol
\begin{equation}
\sigma(f(P(h)))=f(P(x,\xi))
\end{equation}
and trace
\begin{equation}
\label{traceform}
\textrm{tr}(f(P(h)))=(2\pi h)^{-n}\int_{T^*M} f(P(x,\xi))\, d\xi\,dx .
\end{equation}

The equation \eqref{zworskiweyl} then follows from \eqref{traceform} and regularisation of the indicator functions. Full details can be found in Chapter 14 of \cite{zworski}.
\begin{remark}
In the special case of the Laplace--Beltrami operator $p(x,hD)=h^2\Delta_g$, rescaling yields the classical Weyl law which gives counting asymptotics for the Laplacian eigenvalues.
\end{remark}

The Weyl law can also be localised in phase space by a semiclassical pseudodifferential operator. That is, for any $B\in \Psi^0(M)$, we have
\begin{equation}
(2\pi h)^n\sum_{E_j(h)\in [a,b]}\langle B(x,hD)u_j(h),u_j(h)\rangle\rightarrow \int_{P^{-1}([a,b])} B(x,\xi) \, d\xi\,dx.
\end{equation}
as $h\rightarrow 0$.
The proof of this generalisation again makes use of \eqref{traceform}, and can be found in Section 15.3 of \cite{zworski}.

A version of the semiclassical Weyl law was proven by Petkov and Robert \cite{robert} for boundaryless manifolds  that is localised to $O(h)$ sized energy bands. That is, for regular values $E$ of the Hamiltonian $P$, we have
\begin{equation}
\#\{E_j(h)\in [E-ch,E-ch]\} \sim\int_{p^{-1}([E-ch,E+ch])} P(x,\xi)\, d\xi\,dx.
\end{equation}
\begin{remark}
This result requires the dynamical assumption that the set of trapped trajectories is of measure zero. Without this assumption, we only obtain a uniform upper bound for $\#\{E_j(h)\in [E-ch,E-ch]\}$.
\end{remark}


\chapter{Quantum Ergodicity in Mixed Systems}

\section{Introduction}
\label{chp3introsec}



In this chapter, we turn our attention to mixed billiards. We begin by recalling the relevant definitions.

If $(M,g)$ is a compact boundaryless Riemannian manifold, we define \emph{dynamical billiards} on $M$ to be the Hamiltonian flow $\phi_t$ on the cotangent  bundle $T^*M$ of the manifold given by Hamilton's equations
\begin{equation}
\label{hameq}
\dot{x}_j=\frac{\partial H}{\partial \xi_j},\quad \dot{\xi}_j=\frac{\partial H}{\partial x_j}
\end{equation}
for the Hamiltonian $H(x,\xi):=|(x,\xi)|_{g^{-1}}^2$ where $g^{-1}$ is the dual metric tensor.

Since the Hamiltonian is an invariant of motion for the flow $\phi_t$, it is natural to restrict the domain of this flow to the cosphere bundle \begin{equation} {S^*M:=\{z=(x,\xi)\in T^*M:|z|_{g^{-1}}=1\}}.\end{equation}

More generally, one can define billiards on compact Riemannian manifolds with piecewise smooth boundary in the sense of Chapter 6 of \cite{cfs}, see also \cite{zelditch-zworski}.

To be precise, we assume that we can smoothly embed $M$ in a boundaryless manifold $\tilde{M}$ of the same dimension and that there exist finitely many smooth functions $f_j\in \mathcal{C}^\infty(\tilde{M})$ such that the following conditions are satisfied.
\begin{enumerate}
\item $df_i|_{f_i^{-1}(0)}\neq 0$,
\item $df_i, df_j$ are linearly independent on $f_i^{-1}(0)\cap f_j^{-1}(0)$,
\item $M=\{x\in \tilde{M}:f_j(x)\geq 0 \textrm{ for all }j\}.$
\end{enumerate}

We can then write $$\partial M=\cup_j \partial M_j:=\cup_j (f_j^{-1}(0)\cap M)$$ and denote by $\mathcal{S}\subset \partial M$ the set of points that lie in $\partial M_j$ for multiple $j$.\\

We define the broken Hamiltonian flow $\phi_t$ on $S^*M$ locally by extending the boundaryless Hamiltonian flow by reflection at non-tangential and non-singular boundary collisions.

That is, if $\phi_{t_0}(z)=(x,\xi_+)$ with $x\in\partial M\setminus \mathcal{S}$ and $\langle\xi_+, N_x\rangle >0$ where $N_x$ is the outgoing unit normal covector, we extend $\phi_t$ to sufficiently small $t>t_0$ by defining $\phi_t(z)=\phi_{t-t_0}(x,\xi_-)$, where $\xi_-\in S^*_x M$ is the unique covector such that $\xi_+ + \xi_- \in T^*\partial M$ and $\pi(\xi_+)=\pi(\xi_-)$ where $\pi: T^*_{\partial M}M\rightarrow T^*\partial M$ is the canonical projection. We terminate all trajectories that meet $\partial M$ in any other manner.\\

There are four subsets $\{\mathcal{B}_j\}_{j=1}^4$ of phase space for this class of manifolds which present an obstruction to obtaining a globally defined broken Hamiltonian flow or to the application of tools from microlocal analysis. We enumerate these sets below.
\begin{enumerate}
\item $\mathcal{B}_1=\{z\in S^*M: \phi_t(z)\in\mathcal{S}\}$
\item $\mathcal{B}_2=\{z\in S^*M: \phi_t(z) \in \partial M \textrm{ for infinitely many }t\textrm{ in a bounded interval}\}$
\item $\mathcal{B}_3=\{z\in S^*M: \phi_t(z)\notin \partial M \textrm{ for any }t>0 \textrm{ or }\phi_t(z)\notin \partial M \textrm{ for any }t<0\}$
\item $\mathcal{B}_4=\{z\in S^*M: \phi_t(z)\textrm{ meets }\partial M \textrm{ tangentially for some }t\in \mathbb{R}.\}$
\end{enumerate}
Removing these sets from our flow domain, we then obtain a globally defined billiard flow on $\mathcal{D}=S^*M\setminus (\cup_{j=1}^4\mathcal{B}_j)$. For manifolds without boundary, we simply take $\mathcal{D}=S^*M$.\\

The canonical symplectic form $d\xi\wedge dx$ on $T^*M$ determines a family of measures $\mu_c$ on each of the energy hypersurfaces
\begin{equation}
\Sigma_c=\{z=(x,\xi)\in T^*M:|z|_{g^{-1}}=c\}
\end{equation}
defined implicitly by
\begin{equation}
\int_a^b\int_{\Sigma_c} f\, d\mu_c \, dc = \int_{|(x,\xi)|_{g^{-1}}\in [a,b]}f\, d\xi\wedge dx
\end{equation}
for $f\in \mathcal{C}_c^\infty(T^* M)$.

Upon normalisation of $\mu_1$ we then obtain the \emph{Liouville measure} $\mu_L$ on $S^*M$, which allows us to define ergodicity for the billiard flow $\phi_t$ as in Definition \ref{ergdefn}.
\begin{remark}
It is shown in Section 6.2 of \cite{cfs} that the sets $\mathcal{B}_1,\mathcal{B}_2$ are of Liouville measure zero, and it is shown in \cite{zelditch-zworski} that the set $\mathcal{B}_4$ is of Liouville measure zero for the class of manifolds considered. That the remaining set $\mathcal{B}_3$ is Liouville null is usually taken as an assumption. In particular, it is clear that this assumption is satisfied by bounded domains in $\mathbb{R}^n$.
\end{remark}




The quantum mechanical analogue of the system \eqref{hameq} is the evolution of a wave function $\psi\in L^2(M)$ according to the rescaled Schrodinger's equation
\begin{equation}
\label{schrod}
-\Delta_g \psi=i\frac{\partial \psi}{\partial t}
\end{equation}
with boundary conditions to ensure self-adjointness of the Laplacian. We shall work with the most studied and technically easiest choice of Dirichlet boundary conditions.

Since the boundary of $M$ is Lipschitz, it follows that the Laplacian $-\Delta_g$ is self adjoint on $L^2$ when given the standard domain $H^2(M)\cap H_0^1(M)$. Standard spectral theory then shows that $-\Delta_g$ has purely discrete spectrum (counting multiplicity) $\{0<E_1 \leq E_2 \leq \ldots \}\subset \mathbb{R}^+$.

 
The phase space localisation of the high energy eigenfunctions of $-\Delta_g$ can then be described using the calculus of semiclassical pseudodifferential operators, as defined in Chapters 4 and 14 of \cite{zworski}.

For the reader's sake, we recap the definition of semiclassical measures from Section 1.3. To each subsequence of $(u_j)$, we can associate at least one non-negative Radon measure $\mu$ on $S^*M$ which provides a notion of phase space concentration in the semiclassical limit. 

We say that the eigenfunction subsequence $(u_{j_k})$ has unique \emph{semiclassical measure} $\mu$ if 
\begin{equation}
\lim_{k\rightarrow \infty}\langle a(x,E_{j_k}^{-1/2}D)u_{j_k},u_{j_k} \rangle=\int_{S^*M} a(x,\xi)\, d\mu
\end{equation}
for each semiclassical pseudodifferential operator with principal symbol $a$ compactly supported supported away from the boundary of $S^*M$. In Chapter 5 of \cite{zworski}, the existence and basic properties of semiclassical measures are established using the calculus of semiclassical pseudodifferential operators (see also \cite{gerard-leichtnam}).\\

A billiard $M$ is then said to be \emph{quantum ergodic} if there is a full density subsequence of eigenfunctions $(u_{n_k})$ such that the the Liouville measure on $S^*M$ is the unique semiclassical measure associated to the sequence $u_{n_k}$. This statement can be interpreted as saying that the sequence of eigenfunctions equidistributes in phase space with the possible exception of a sparse subsequence.


In this chapter we consider the family of mushroom billiards $M_{t}=R_t\cup S\subset \R^2$ where $R_t=[-r_1,r_1]\times[-t,0]$ and $S$ is the closed upper semidisk of radius $r_2>r_1$ centred at the origin. We denote the area of $M_t$ by $A(t)$.

\vspace{15pt}
\begin{figure}
\begin{center}
\includegraphics[scale=0.5]{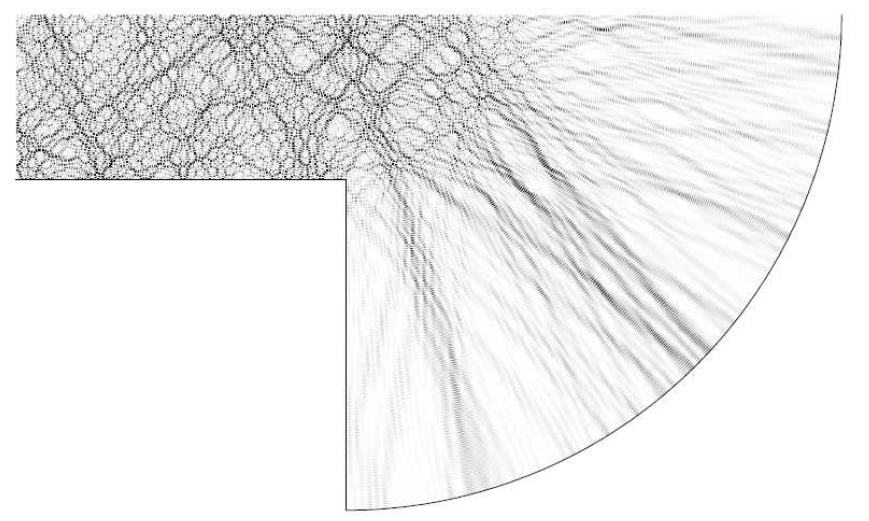}
\caption[]{The half-mushroom billiard, with a high energy eigenfunction that extends by odd symmetry to the mushroom billiard. This particular eigenfunction appears to live in the ergodic region of phase space. Image courtesy of Dr Barnett.}
\end{center}
\end{figure}
\vspace{15pt}

This billiard, proposed by Bunimovich \cite{bunimovich} is neither classically ergodic nor completely integrable for $t >0$ and is rather one of the simplest billiards that satisfies the following mixed dynamical assumptions. 

\begin{itemize}
\label{conditions}
\item $M$ is a smooth Riemannian manifold with piecewise smooth boundary
\item The flow domain $\mathcal{D}$ is the union of two invariant subsets, each of positive Liouville measure and one of which, $U$, has ergodic geodesic flow
\item The billiard flow is completely integrable on $\mathcal{D}\setminus U$.
\end{itemize}

In the mushroom billiard, $U_t$ consists of $\mu_L$-almost all trajectories that enter $R_t\cup \overline{B(0,r_1)}$ before their first boundary collision. The trajectories that do not enter $R_t\cup \overline{B(0,r_1)}$ before their first boundary collision lie entirely within the upper semi-annulus $S\setminus \overline{B(0,r_1)}$ and are just reflected trajectories of the disk billiard. The integrability of the geodesic flow on $\mathcal{D}\setminus U_t$ then follows from the integrability of the disk billiard.

In the case of such mixed systems, we do not yet have a satisfactory analogue to the quantum ergodicity theorem. It is a long-standing conjecture of Percival \cite{percival} that a full density  subset of a complete system of eigenfunctions of the Laplace--Beltrami operator can be divided into two disjoint subsets, one corresponding to the ergodic region of phase space and the other corresponding to the completely integrable region. Moreover, the natural density of these subsets is conjectured to be in proportion to the Liouville measures of the corresponding flow-invariant subsets of $\mathcal{D}$. 

\begin{con}[Percival's Conjecture]
\label{strongpercival}
For every compact Riemannian manifold $M$ such that $\mathcal{D}$ is the disjoint union of two invariant subsets $U,\mathcal{D}\setminus U$, with $U$ ergodic and $\mathcal{D}\setminus U$ completely integrable, we can find two subsets $A,B\subset \N$ such that

\begin{enumerate}
\item $A\cup B$ has density $1$ 
\item $(u_k)_{k\in A}$ equidistributes in the ergodic region $U$ 
\item Each semiclassical measure associated to the subset $B$ is supported in the completely integrable region $\mathcal{D} \setminus U$
\item The density of $A$ is equal to $\mu_L(U)$.
\end{enumerate}
\end{con}

Numerical evidence due to Barnett-Betcke \cite{barnett} has strongly supported this conjecture for the mushroom billiard, yet rigorous proof has remained elusive.

A weaker version of the conjecture can be formulated by slightly relaxing the density requirements of the subsets $A$ and $B$.

\begin{con}
[Weak Percival's Conjecture]
\label{percival}
For every compact Riemannian manifold $M$ such that $\mathcal{D}$ is the disjoint union of two invariant subsets $U,\mathcal{D}\setminus U$, with $U$ ergodic and $\mathcal{D}\setminus U$ completely integrable, we can find two subsets $A,B\subset \N$ such that

\begin{enumerate}
\item $A\cup B$ has upper density $1$ 
\item $(u_k)_{k\in A}$ equidistributes in the ergodic region $U$ 
\item Each semiclassical measure associated to the subset $B$ is supported in the completely integrable region $\mathcal{D} \setminus U$
\item The upper densities of $A$ and $B$ are equal to $\mu_L(U)$ and $1-\mu_L(U)$ respectively.
\end{enumerate}
\end{con}

In this chapter, we prove Conjecture \ref{percival} is indeed true for the mushroom billiard, at least for almost all $t\in (0,2]$. Essential in our work is the following result due to Galkowski \cite{galkowski}.

\begin{thm}
\label{galkowski}
For any compact Riemannian manifold with boundary satisfying \eqref{conditions}, there exists a full density subsequence of $(u_j)$, such that every associated semiclassical measure $\mu$ satisfies
\begin{equation}
\mu|_U = a\mu_L|_U
\end{equation}
for some constant $a$.
\end{thm}

Our strategy for this proof is motivated by that used by Hassell in constructing the first known example of a non-QUE ergodic billiard \cite{hassellque}. \\

We begin in Section 2 by using the Dirichlet eigenfunctions on the semicircle to construct a family $(v_n,\alpha_n^2)$ of $O(n^{-\infty})$ quasimodes that are almost orthogonal and are microlocally supported in the completely integrable region $S^*M_t\setminus U_t$.

Using the well-known asympotics of the Bessel function zeroes, we obtain a lower bound \eqref{quasiasymptoticeqn} for the counting function of this quasimode family.\\

In Section 3, the main result is Proposition \ref{spectral}, an abstract spectral theoretic result that allows us to approximate certain  eigenfunctions by linear combinations of quasimodes of similar energy given that the numbers of each are comparable. This is the essential ingredient for passing from localisation properties about our explicit family of quasimodes to localisation properties of a family of eigenfunctions with asymptotically equivalent counting function.\\

In Section 4 we commence our study of the variation of eigenvalues as the stalk length $t$ varies in $(0,2]$. In order to simplify the nomenclature, we often interpret $t$ as a time parameter. 

The Hadamard variational formula asserts that
\begin{equation}
\label{classhad}
\dot{E}(t)=-\int_{\partial M_t} \rho_t(s)(d_n u(t)(s))^2 \, ds
\end{equation}
where $\rho_t(s)$ is the unit normal variation of the domain at a boundary point $s$. 
For normally expanding domains such as ours, \eqref{classhad} directly implies that individual eigenvalues are non-increasing in $t$.

However, using an interior formulation of the Hadamard variational formula from Proposition \ref{hadamard}, we can also quantify the variation of the  eigenvalue $E_j(t)$ by
\begin{equation}
E_j^{-1}(t)\dot{E_j}(t)=\langle Qu_j(t),u_j(t)\rangle
\end{equation}
for an appropriate pseudodifferential operator $Q$ supported in the stalk $R_t\subset M_t$. 

Proposition \ref{heatkernelprop} then establishes that for a full density subset of the eigenvalues, the quantity $\langle Qu_j(t),u_j(t)\rangle$ can be approximated up to an error of $o(E_j)$ by cutting off $Q$ sufficiently close to the boundary $\partial M_t$. This result is shown by using analysis of the wave kernel to establish the key spectral projector estimates \eqref{projest} and \eqref{gradprojest}. 

We can then use the equidistribution result of Galkowski's Theorem \ref{galkowski} to asymptotically  control $\langle Qu_j(t),u_j(t)\rangle$ and hence provide us with an upper bound \eqref{flowspeedeq} on the speed of eigenvalue variation for almost all eigenvalues.\\

Section 5 completes the argument in two parts. 

In the first of these parts, we define a set $\mathcal{G}\subset (0,2]$ such that for $t\in\mathcal{G}$, we have a certain spectral non-concentration property on $M_t$. This property implies that that 
\begin{equation}\begin{gathered}
\text{
the number of eigenvalues lying in the union $\cup_{j=1}^n [\alpha_j^2-c,\alpha_j^2+c]$} \\ 
\text{can exceed $n$ by at most a small proportion, for infinitely many $n$.}
\end{gathered}\label{snc}\end{equation}

Proposition \ref{spectral} then implies that for $t\in\mathcal{G}$ and $n$ such that \eqref{snc} is satisfied, a large proportion of the corresponding eigenfunctions are asymptotically well-approximated by linear combinations of the previously constructed family of quasimodes $(v_n)$, which are microlocally supported in the completely integrable region $S^*M_t\setminus U_t$ of phase space.


In fact, the explicit computation \eqref{quasiasymptoticeqn} of the counting function of these quasimodes leads to a proof that the corresponding family of eigenfunctions must have the maximal upper density allowed by the contraint of Weyl's law. We show this in Theorem \ref{mainthm}.


Consequently, we show in Proposition \ref{ergodicprop} that a subset of the complementary family of eigenfunctions with full upper density must have all semiclassical mass in the ergodic region $U_t$. From Theorem \ref{galkowski}, this family must then equidistribute in $U_t$ as required.\\

The final part of the chapter establishes via contradiction that $(0,2]\setminus \mathcal{G}$ is Lebesgue-null. 
As in \cite{hassellque} we can choose the eigenvalue branches $E_j(t)$ to be in increasing order and piecewise smooth in $t$. The crucial ingredient here is then the asymptotic bound \eqref{flowspeedeq} on the speed of eigenvalue variation. 

If $\mathcal{G}$ is not of full measure, we can construct a small interval $\mathcal{I}=[t_1,t_2]$ in which the average number of eigenvalues $E_j(t)$ lingering near quasi-eigenvalues $\alpha_i^2$ exceeds $d=d(t_1)=1-\mu_L(U_{t_1})$ by using the negation of \eqref{snc}. 

Now Weyl's law
\begin{equation}
N_t(\lambda^2)\sim \frac{\lambda^2 A(t)}{4\pi}
\end{equation}
implies that the decrease of eigenvalues over $\mathcal{I}$ is asymptotically given by
\begin{equation}
\label{weyl}
E_j(t_1)-E_j(t_2)\sim 4\pi j (A(t_1)^{-1}-A(t_2)^{-1})
\end{equation}
in $\mathcal{I}$ as $j\rightarrow \infty$ where $A(t)$ denotes the area of the mushroom $M_t$.

We can use \eqref{weyl} together with the fact that the small windows about quasi-eigenvalues are comparatively sparse in the interval $[E_j(t_2),E_j(t_1)]$ to show that the upper bound \eqref{flowspeedeq} on eigenvalue speed provides a lower bound of $(1-d)$ on the time they must spend travelling outside of quasi-eigenvalue windows. 

This implies that the average proportion of time spent by large eigenvalues lingering near quasi-eigenvalues for $t\in\mathcal{I}$ cannot exceed $d$, and consequently that the proportion of lingering eigenvalues cannot exceed $d$. This contradiction concludes the proof.\\

\section{Quasimodes}

In polar coordinates, the Dirichlet eigenfunctions for the semidisk are given by
\begin{equation}
u_{n,k}:=\sin(n\theta)J_n(\alpha_{n,k}r/r_2)
\end{equation}
where $\alpha_{n,k}$ is the $k$-th positive zero of the $n$-th order Bessel function $J_n$.

\begin{prop}
If we define
\begin{equation}
v_{n,k}:=\frac{\chi(r)u_{n,k}}{\|\chi u_{n,k}\|_{L^2}},
\end{equation}
where
\begin{equation}
\chi(r)=\begin{cases}
0 &\mbox{for }r \leq r_1\\
1 &\mbox{for }r\geq (r_1+\epsilon)\sqrt{1-\epsilon^2}>r_1.
\end{cases}
\end{equation}
then the family
\begin{equation}
\label{family}
\{(v_{n,k},\alpha_{n,k}^2/r_2^2):\alpha_{n,k}<\frac{nr_2}{r_1+\epsilon}\}
\end{equation}
forms an $O(n^{-\infty})$ family of quasimodes, with all semiclassical mass contained in the completely integrable region $S^{*}M_t\setminus U_t$ of the billiard.

Moreover, these quasimodes are almost orthogonal, in the sense that
\begin{equation}
|\langle v_{n,k},v_{m,l}\rangle |=O(\min(n,m)^{-\infty})=O(\min(\alpha_{n,k},\alpha_{m,l})^{-\infty}).
\end{equation}
\end{prop}
\begin{proof}

The restriction on $k$ in our family implies that the error incurred in cutting off only depends on the values of the Bessel function $J_n(x)$ for $x\in[0,n\sqrt{1-\epsilon^2}]$.

Then from \cite{absteg}, we have the estimates
\begin{equation}
|J_n(nx)|\leq \frac{x^ne^{\sqrt{1-x^2}}}{(1+\sqrt{1-x^2})^n}\quad \textrm{ for $x\leq 1$}
\end{equation}
and
\begin{equation}
|J_n'(nx)|\leq \frac{(1+x^2)^{1/4}x^ne^{\sqrt{1-x^2}}}{x\sqrt{2\pi n}(1+\sqrt{1-x^2})^n}\quad \textrm{ for $x\leq 1$}
\end{equation}
for bounding the Bessel function near $0$.

Together these estimates imply that the error incurred by cutting off is $O(n^{-\infty})$. Furthermore, as the $u_{n,k}$ are pairwise orthogonal, these bounds also show that the $v_{n,k}$ are almost orthogonal in the sense claimed.\\

Now for any smooth compactly supported symbol $a$ spatially supported in $R_t\cup B(0,r_1)$, the disjointness of supports from our family of quasimodes implies that
\begin{equation}
\langle a(x,(r_2/\alpha_{n,k})D)v_{n,k},v_{n,k}\rangle=O(n^{-\infty}).
\end{equation}
In particular, we have that any semiclassical measure $\mu$ associated to these quasimodes cannot have mass in the region $\{(x,\xi)\in S^*M_t:x\in R_t\cup B(0,r_1)\}$.

Moreover, by the flow invariance of semiclassical measures (See Theorem 5.4 in \cite{zworski}), this implies that any corresponding semiclassical measure cannot have mass in the ergodic region $U_t$ because the pre-images under geodesic flow of \\ \noindent ${\{(x,\xi)\in \mathcal{D}_t:x\in R_t\cup B(0,r_1)\}}$ cover $U_t$.
\end{proof}

\begin{prop}
\label{decay}
We can index these quasimodes as $(v_n,\alpha_n^2)$ so that the quasi-eigenvalues are in increasing order, whilst having
\begin{equation}
(\Delta+\alpha_n^2)v_n=O(n^{-\infty})=O(\alpha_n^{-\infty})
\end{equation}
and
\begin{equation}
|\langle v_n,v_k\rangle|=O(\min(n,k)^{-\infty})=O(\min(\alpha_n,\alpha_k)^{-\infty}).
\end{equation}
\end{prop}

Moreover, as $\epsilon\rightarrow 0$, the counting function of these quasimodes has the following asymptotic bound.

\begin{prop}
\label{quasiasymptotic}
\begin{equation}
\label{quasiasymptoticeqn}
\lim_{\epsilon\rightarrow 0}\liminf_{\lambda\rightarrow \infty}\frac{\#\{(n,k):\alpha_{n,k}/r_2 < \lambda,\alpha_{n,k}<\frac{nr_2}{r_1+\epsilon}\}}{\lambda^2}\geq \left(1-\frac{\mu_L(U_t)}{\mu_L(\mathcal{D}_t)}\right)\cdot \frac{A(t)}{4\pi}.
\end{equation}
where $A(t)$ is the area of the mushroom $M_t$.
\end{prop}
\begin{proof}
To simplify our calculations, we scale $\mu_L$ so that $\mu_L(M_t)=2\pi A(t)$.

From an arbitrary point $(r,\theta)$ in the annulus, the trajectories that never enter the stalk have measure $(2\pi-4\sin^{-1}(r_1/r))$ out of the full measure $2\pi$ of the unit cosphere at that point.

Hence

\begin{eqnarray*}
\mu_L(\mathcal{D}_t)-\mu_L(U_t)&=&\int_0^\pi\int_{r_1}^{r_2} r(2\pi-4\sin^{-1}(r_1/r))\, dr \, d\theta\\
&=& \pi^2(r_2^2-r_1^2)-4\pi\int_{r_1}^{r_2} r\sin^{-1}(r_1/r)\, dr \\
&=& \pi^2r_2^2-2\pi r_1^2\sqrt{C^2-1}-2\pi r_2^2\sin^{-1}(C^{-1})
\end{eqnarray*}
where $C=r_2/r_1$.

This implies that 
\begin{equation}
\left(1-\frac{\mu_L(U_t)}{\mu_L(\mathcal{D}_t)}\right)\cdot \frac{A(t) \lambda^2}{4\pi}=\frac{r_2^2}{8}\left(1-\frac{2}{\pi C^2}\sqrt{C^2-1}-\frac{2}{\pi}\sin^{-1}(C^{-1})\right)\lambda^2.
\end{equation}

To estimate the left hand side of \eqref{quasiasymptoticeqn}, we use the leading order uniform asymptotics for Bessel function zeros found in \cite{absteg}. 

As $n\rightarrow \infty$, we have 
\begin{equation}
\alpha_{n,k}=nz(n^{-2/3}a_k)+o(n)
\end{equation}

uniformly in $k$, where $z:(-\infty,0]\rightarrow [1,\infty)$ is defined implicitly by 

\begin{equation}\frac{2}{3}(-\zeta)^{3/2}=\sqrt{z(\zeta)^2-1}-\sec^{-1}(z(\zeta))\end{equation}

and the $a_k$ are the negative zeros of the Airy function, which have asymptotic
\begin{equation}
a_k=-\left(\frac{3\pi k}{2}\right)^{2/3}+O(k^{-1/3}).
\end{equation} 

We now write $C_\epsilon=r_2/(r_1+\epsilon)$.

We count the left hand side of \eqref{quasiasymptoticeqn} by separating into two regimes based on the size of $n/\lambda$. In each of these two regimes, a single one of the inequalities defining our family \eqref{family} implies the other. More precisely, we have
\begin{eqnarray}\nonumber & &|\{(n,k):\alpha_{n,k}/r_2\leq\lambda,\alpha_{n,k}\leq C_\epsilon n\}|\\
\nonumber &= &|\{(n,k):n\leq r_2\lambda/C_\epsilon,\alpha_{n,k}\leq C_\epsilon n\}|\\ 
\nonumber &+&|\{(n,k):r_2\lambda/C_\epsilon < n \leq r_2\lambda,\alpha_{n,k}/r_2 \leq \lambda\}|\\
\label{regime} &=&N_A(\lambda;\epsilon)+N_B(\lambda;\epsilon)\end{eqnarray}
where
\begin{equation}
N_A(\lambda;\epsilon):=|\{(n,k):n\leq r_2\lambda/C_\epsilon,\alpha_{n,k}\leq C_\epsilon n\}|
\end{equation}
and
\begin{equation}
N_B(\lambda;\epsilon):=|\{(n,k):r_2\lambda/C_\epsilon < n \leq r_2\lambda,\alpha_{n,k}/r_2 \leq \lambda\}|
\end{equation}
respectively.

For $n,k$ sufficiently large, a sufficient condition for being in regime $A$ of \eqref{regime} is to have
\begin{equation}
z(n^{-2/3}a_k)\leq C_\epsilon-\epsilon=\hat{C_\epsilon}
\end{equation}
and
\begin{equation}
n\leq r_2\lambda/C_\epsilon.
\end{equation}

Also, from the Airy function asymptotics we have
\begin{equation}
\frac{2}{3}(-n^{-2/3}a_k)^{3/2}=\frac{2}{3n}\left(\left(\frac{3\pi k}{2}\right)^{2/3}+O(k^{-1/3})\right)^{3/2}=\frac{\pi k}{n}+O(n^{-1})
\end{equation}
where the error is uniform in $k$.

Hence from the monotonicity of $z$, for all $n,k$ sufficiently large with $n\leq r_2\lambda/C_\epsilon$, a sufficient condition for being in regime $A$ of \eqref{regime} is
\begin{equation}
k \leq \frac{\sqrt{\hat{C_\epsilon}^2-1}-\sec^{-1}(\hat{C_\epsilon})-\epsilon}{\pi}n.
\end{equation}

Noting that the contribution from small $n$ and $k$ is finite, we can conclude that
\begin{eqnarray*}
\liminf_\lambda\frac{N_A(\lambda^2)}{\lambda^2}&\geq &\liminf_\lambda (\frac{1}{\lambda^2}\sum_{n\leq r_2\lambda/C_\epsilon} n)\cdot \frac{\sqrt{\hat{C_\epsilon}^2-1}-\sec^{-1}(\hat{C_\epsilon})-\epsilon}{\pi}\\
&=& \frac{r_2^2(\sqrt{\hat{C_\epsilon}^2-1}-\sec^{-1}(\hat{C_\epsilon})-\epsilon)}{2C_\epsilon^2\pi}.
\end{eqnarray*}

Similarly, for sufficiently large $n,k$, a sufficient condition for being in regime $B$ of \eqref{regime} is to have
\begin{equation}
z(n^{-2/3}a_k)\leq \lambda r_2/n -\epsilon=D_\epsilon(\lambda,n).
\end{equation}
and
\begin{equation}
r_2\lambda/C_\epsilon < n \leq r_2\lambda.
\end{equation}

Hence, for all $n,k$ sufficiently large with $r_2\lambda/C_\epsilon < n \leq r_2\lambda/(1+\epsilon)$, a sufficient condition for being in regime $B$ of \eqref{regime} is
\begin{equation}
k\leq \frac{\sqrt{D_\epsilon(\lambda,n)^2-1}-\sec^{-1}(D_\epsilon(\lambda,n))-\epsilon}{\pi}n.
\end{equation}

Again throwing away a finite number of small pairs, we obtain
\begin{eqnarray*}
& & \liminf_\lambda \frac{N_B(\lambda^2)}{\lambda^2}\\
&\geq & \liminf_\lambda \frac{1}{\pi\lambda^2} \sum_{\frac{r_2\lambda}{C_\epsilon} < n \leq \frac{r_2\lambda}{1+\epsilon}}\left(n\sqrt{D_\epsilon(\lambda,n)^2-1}-  n\sec^{-1}(D_\epsilon(\lambda,n)) -\epsilon n\right)\\
&=& \liminf_\lambda\frac{1}{\pi\lambda^2}\int_{\frac{r_2\lambda}{C_\epsilon}}^{\frac{r_2\lambda}{1+\epsilon}} \left(t\sqrt{D_\epsilon(\lambda,t)^2-1} -  t\sec^{-1}(D_\epsilon(\lambda,t))-\epsilon t\right)\, dt.
\end{eqnarray*}
Each of the three summands in the integrand has elementary primitive, so we can explicitly compute this quantity.

Noting that $C_\epsilon,\hat{C_\epsilon}\rightarrow C$, we compute
\begin{eqnarray*}
& &\lim_{\epsilon\rightarrow 0}\liminf_\lambda \frac{N_A(\lambda^2)+N_B(\lambda^2)}{\lambda^2}\\
&\geq & \left(\frac{r_1^2\sqrt{C^2-1}}{2\pi}-\frac{r_1^2}{2\pi}(\frac{\pi}{2}-\sin^{-1}(C^{-1}))\right)\\
&+& \frac{1}{\pi \lambda^2}\int_{r_1\lambda}^{r_2\lambda} \sqrt{\lambda^2 r_2^2-t^2}\, dt-\frac{1}{\pi\lambda^2}\int_{r_1\lambda}^{r_2\lambda}t\sec^{-1}(\frac{\lambda r_2}{t})\, dt\\
&=& \frac{r_2^2}{8}\left(1-\frac{2}{\pi C^2}\sqrt{C^2-1}-\frac{2}{\pi}\sin^{-1}(C^{-1})\right).\\
\end{eqnarray*}
as required.

\end{proof}

\section{Spectral theory}

We next establish the following key spectral theoretic result.

\begin{prop}
\label{spectral}
Let $\mathcal{H}$ be a Hilbert space. Suppose $T\in\mathcal{L}(\mathcal{H})$ has a complete orthonormal system of eigenvectors $(u_i,E_i)_{i\in\N}$ with the sequence $(E_i)$ non-negative and increasing without bound. Suppose further that we have a family of normalised quasimodes $(v_i,E_i')_{i=1}^n$ with 
\begin{equation}\label{epsilonone}\|(T-E_i')v_i\| < \epsilon_1\end{equation}
and
\begin{equation}\label{epsilontwo}|\langle v_i,v_j\rangle| < \epsilon_2 \quad \textrm{for }i\neq j\end{equation}
for some positive $\epsilon_1,\epsilon_2>0.$

We write 
\begin{equation}
V=\textrm{Span}\{v_i\}_{i=1}^n
\end{equation}
and
\begin{equation}
U=\textrm{Span}\{u_j:E_j\in\bigcup_{i=1}^n [E_i'-c,E_i'+c]\}.
\end{equation}

We denote the orthogonal projection onto a subspace $S\subseteq \mathcal{H}$ by $\pi_S$.

If for some $c>0$ and some $0<\epsilon,\delta<1/2$ we have

\begin{equation}
m=\#\{j\in\N:E_j\in\bigcup_{i=1}^n [E_i'-c,E_i'+c] \}<n(1+\epsilon)
\end{equation}
and
\begin{equation}
\label{errormatrixbound}
\frac{\epsilon_1^2}{c^2}+\epsilon_2 < \frac{\delta}{n}
\end{equation}

then at least $n(1-\sqrt{\epsilon})$ of the corresponding eigenvectors $u_i$ satisfy

\begin{equation}
\label{eigenestimate}
\|u_i-\pi_V(u_i)\|<\epsilon^{1/4}+2\delta^{3/2}. 
\end{equation}

\end{prop}
\begin{proof}
The idea behind the proof of the estimate \eqref{eigenestimate} consists of several successive approximations. 

We first show that the projections $\pi_U(v_i)$ are almost orthogonal and can be transformed into an orthonormal basis $(w_i)_{i=1}^n$ of their span by a matrix $A$ that is approximately the identity. 

We then show that excluding some exceptional eigenvectors, the remaining eigenvectors are necessarily rather close to the space $W$. This implies that the non-exceptional eigenvectors can be well approximated by their projections, which leads us to conclude $u\approx\pi_W(u)=Bw=BA\pi_U(v)\approx BAv$ for some matrix $B$.\\

To begin, we reindex the eigenpairs $(u_i,E_i)$ so that $E_j\in \cup_{i=1}^n [E_i'-c,E_i'+c]$ precisely for $j=1,2,\ldots,m.$

The assumptions \eqref{epsilonone} and \eqref{epsilontwo} then imply

\begin{eqnarray*}
\|(T-E_i')\sum_{j\in\mathbb{N}}\langle v_i,u_j\rangle u_j\|^2 & < & \epsilon_1^2\\ 
\Rightarrow \sum_{j=m+1}^\infty |E_j-E_i'|^2|\langle v_i,u_j\rangle|^2 &<& \epsilon_1^2\\
\Rightarrow \sum_{j=m+1}^\infty |\langle v_i,u_j\rangle|^2 &<& \frac{\epsilon_1^2}{c^2}\\
\Rightarrow \|\pi_U(v_i)\|^2 &>& 1-\frac{\epsilon_1^2}{c^2}.
\end{eqnarray*}
and
\begin{eqnarray*}
|\langle\pi_U(v_i),\pi_U(v_j)\rangle|&\leq& |\langle v_i,v_j\rangle|+|\langle\pi_{U^\perp}(v_i),\pi_{U^\perp}(v_j)\rangle|\\
& < & \epsilon_2 + \sqrt{(1-\|\pi_U(v_i)\|^2)(1-\|\pi_U(v_j)\|^2)}\\
& < & \epsilon_2+\frac{\epsilon_1^2}{c^2}
\end{eqnarray*}
for $i\neq j$.

Together with \eqref{errormatrixbound} we obtain
\begin{equation}
\label{ineqone}
\|\pi_U(v_i)\|^2 > 1-\frac{\delta}{n}
\end{equation}
and
\begin{equation}
\label{ineqtwo}
|\langle \pi_U(v_i),\pi_U(v_j)\rangle|< \frac{\delta}{n}
\end{equation}
for $i\neq j$.

The Gram matrix $M$ with entries $M_{ij}=\langle\pi_U(v_i),\pi_U(v_j)\rangle$ satisfies

\begin{equation}
\label{approxid}
\|M-I\|_{HS}=:\|E\|_{HS}<\delta<1/2.
\end{equation}

Note that if the collection  $\{\pi_U(v_i)\}$ were linearly dependent, then the matrix $M$ would be singular. The estimate \eqref{approxid} precludes this possibility, because we can invert $M=I-(I-M)$ as a Neumann series. In particular, this implies that $m\geq n$.\\

We now write $W=\textrm{Span}\{\pi_U(v_i)\}_{i=1}^n$ and suppose that $(w_i)_{i=1}^n$ is an orthonormal basis for $W$ which can be given by the transformation $w=A\pi_U(v)$, where $A$ is an $n\times n$ real matrix that acts on the Hilbert space $\mathcal{H}^n$ via matrix multiplication.

Expanding out the matrix equation $\langle w_i,w_j\rangle=\delta_{ij}$ we obtain
\begin{equation}
AMA^*=\sum_{k=1}^n\sum_{l=1}^na_{ik}\overline{a_{jl}}\langle \pi_U(v_k),\pi_U(v_l)\rangle=I
\end{equation}
which has a solution
\begin{equation}
A=M^{-1/2}=(I+E)^{-1/2}=\sum_{k=0}^\infty \binom{-1/2}{k}E^k=\sum_{k=0}^\infty (-1)^k \binom{2k}{k}4^{-k}E^k.
\end{equation}

From \eqref{approxid}, we then deduce
\begin{equation}
\|A-I\|_{HS}\leq \sum_{k=1}^\infty \|E\|^k = \|E\|(1-\|E\|)^{-1}<2\delta.
\end{equation}

In the case $m=n$, that is when $W=U$, we can find a unitary matrix $B$ with $Bw=u$.

We now assume $m>n$, recalling that the assumptions of the proposition imply that this excess is small as a proportion of $n$.

We have
\begin{eqnarray*}
\sum_{i=1}^m \|\pi_{W^\perp}(u_i)\|^2=m-\sum_{i=1}^m \|\pi_{W}(u_i)\|^2=m-\sum_{i=1}^m\sum_{j=1}^n |\langle u_i,w_j\rangle|^2=m-n<n\epsilon
\end{eqnarray*}

which implies that
\begin{equation}
\#\{i:\|\pi_{W^\perp}(u_i)\|^2\geq\sqrt{\epsilon}\}<n\sqrt{\epsilon}
\end{equation}
and consequently
\begin{equation}
\label{closest}
\#\{i:\|\pi_{W}(u_i)\|^2>1-\sqrt{\epsilon}\}\geq m-n\sqrt{\epsilon}>n(1-\sqrt{\epsilon}).
\end{equation}

We again re-index the eigenpairs for convenience, so that the first $n'=\lceil n(1-\sqrt{\epsilon}) \rceil$ eigenvectors $u_i$ satisfy the estimate in \eqref{closest}.

In this case, we define the $n' \times n$ matrix $B$ to have entries
\begin{equation}
B_{ij}=\langle\pi_W(u_i),w_j\rangle
\end{equation}
and the vector $u\in\mathcal{H}^{n'}$ by
\begin{equation}
(u_i)_{i=1}^{n'}.
\end{equation}

We then have
\begin{equation}
\pi_W(u)=Bw
\end{equation}
and the $i$-th row $B_i$ of $B$ has $\ell^2$ norm trivially bounded by $1$.

This leaves us with
\begin{equation}
u=(u-\pi_W(u))+BA\pi_U(v)=(u-\pi_W(u))+BAv+BA(\pi_U(v)-v).
\end{equation}

which implies
\begin{eqnarray*}
\|(u-BAv)_i\|_{\mathcal{H}}&\leq& \|u_i-\pi_W(u_i)\|_{\mathcal{H}}+\|B_i\|_{\ell^2}\|A\|_{HS}\|\pi_U(v)-v\|_{\mathcal{H}^n}\\
&<& \epsilon^{1/4}+\cdot (2\delta) \cdot \sqrt{\delta}\\
&<& \epsilon^{1/4}+2\delta^{3/2}.
\end{eqnarray*}

This estimate shows that each $u_i$ has distance less than $\epsilon^{1/4}+2\delta^{3/2}$ to some element of $V$. 

Consequently 
\begin{equation} 
\|u_i-\pi_V(u_i)\| < \epsilon^{1/4}+2\delta^{3/2} 
\end{equation}
as required.
\end{proof}

Our strategy to prove the main theorem is to control the number of eigenvalues in most clusters formed by finite unions of overlapping intervals of the form ${[\alpha_i^2-c,\alpha_i^2+c]}$, and then repeatedly employ Proposition \ref{spectral} to establish the existence of a subsequence of these eigenfunctions with upper density $d$ that localises in the semidisk.

\section{Results on eigenvalue flow}
\label{eigflowsec}

Central to the argument is the analysis of how eigenvalues flow as we vary $t$. 

Weyl's law provides us with the asymptotic
\begin{equation}
N_t(\lambda^2)\sim \frac{\lambda^2|M_t|}{4\pi}
\end{equation}

where $N_t$ is the counting function of the Dirichlet eigenvalues on $M_t$.

To obtain a more precise statement about the change of individual eigenvalues, we employ an interior version of the Hadamard variational formula and Theorem \ref{galkowski}.

In order to make use of Theorem \ref{galkowski}, we choose $\phi(y)\in \mathcal{C}^\infty_c (\mathbb{\R})$ non-negative, supported near $y=-1/2$, and with integral $1$, and we define the family of metrics
\begin{equation}g_t=dx^2+(1+(t-1)\phi)^2 dy^2\end{equation}

on $M_1$.

This metric induces a natural isometry $I_t: (M_1,g_t)\rightarrow (M_t,g_1)$.

If we define $R_t=(1+(t-1)\phi)^{-1/2}$, then we have the following result from Proposition 7 of the appendix of \cite{hassellque}.

\begin{prop}
\label{hadamard}
Let $u(t)$ be an $L^2$-normalised real Dirichlet eigenfunction of $\Delta$ on $M_t$ with corresponding eigenvalue $E(t)$. We then have
\begin{equation}
\dot{E}(t)=-\frac{1}{2}\langle Qu(t),u(t)\rangle
\end{equation}

where the operator $Q$ is given by
\begin{equation}
Q=-4\partial_y\phi_t\partial_y + [\partial_y,[\partial_y,\phi_t]]=\phi_t''-4(\phi_t'\partial_y+\phi_t\partial_y^2)
\end{equation}

on $M_t$.

Here, $\phi_t:M_t\rightarrow \R$ is given by:
\begin{equation}
\phi_t=(\phi R_t^2)\circ I_t^{-1}.
\end{equation}
\end{prop}

We now cut $Q$ off away from the vertical sides of the stalk so that we can use the interior equidistribution result Theorem \ref{galkowski} to control the quantity $E_k^{-1}\dot{E_k}$.

We do this by defining
\begin{equation}Q_\delta=\chi_\delta Q \end{equation}

where $\chi_\delta\in\mathcal{C}^\infty$ satisfies
\begin{equation}
\chi_\delta(x)=\begin{cases}
0 &\mbox{for }x\in [-r_1,-r_1+\delta]\cup [r_1-\delta, r_1]\\
1 &\mbox{for }x \in [-r_1+2\delta,r_1-2\delta].
\end{cases}
\end{equation}

\begin{prop}
\label{heatkernelprop}
For any $\epsilon > 0$ and any $t\in (0,2]$, there exists $\delta>0$ such that
\begin{equation}
\label{heatkernel}
|E_{n_k}^{-1}\langle (Q_\delta-Q)u_{n_k}(t),u_{n_k}(t)\rangle|<\epsilon
\end{equation}
for all $k$, where $(n_k)$ is a $t$-dependent subsequence of the positive integers with lower density bounded below by $1-\epsilon$.
\end{prop}
\begin{proof}
First we show that it suffices for each $t$ to establish the spectral projector estimates
\begin{equation}
\label{projest}
\|\eta_t 1_{[\lambda,\lambda+1)}(\sqrt{-\Delta})\|_{L^2(M_t)\rightarrow L^\infty(M_t)}=O(\lambda^{1/2})
\end{equation}

and
\begin{equation}
\label{gradprojest}
\|\eta_t \nabla 1_{[\lambda,\lambda+1)}(\sqrt{-\Delta})\|_{L^2(M_t)\rightarrow L^\infty(M_t)}=O(\lambda^{3/2}).
\end{equation}

Here $\eta_t=\eta\circ I_t^{-1}$ where $\eta:M_1\rightarrow \mathbb{R}$ is an fixed smooth cutoff function supported and equal to $1$ in a neighbourhood of ${\partial M_1 \cap \textrm{spt}(\phi)}$ such that $\eta$ vanishes in a neigbourhood of the semidisk.


Applying $\eta_t\nabla 1_{[\lambda,\lambda+1)}(\sqrt{\Delta})$ to $\sum_{\lambda_j\in[\lambda,\lambda+1)} a_j u_j$ and using the estimate \eqref{gradprojest} then yields

\begin{eqnarray*}
\left|\eta_t(x)\sum_{\lambda_j\in[\lambda,\lambda+1)}a_j\nabla u_j(x)\right|&\leq& C\lambda^{3/2}\left\|\sum_{\lambda_j\in[\lambda,\lambda+1)}a_j u_j\right\|_{L^2}\\ 
&\leq& C\lambda^{3/2}\left(\sum_{\lambda_j\in[\lambda,\lambda+1)} |a_j|^2\right)^{1/2}
\end{eqnarray*}

for each $x\in M_t$.

Setting $a_j=\nabla u_j(x)$ then yields the estimate
\begin{equation}
\label{projest2}
\eta_t(x)^2\sum_{\lambda_j\in[\lambda,\lambda+1)} |\nabla u_j(x)|^2\leq C \lambda^3.
\end{equation}

Similarly, we obtain
\begin{equation}
\label{gradprojest2}
\eta_t(x)^2\sum_{\lambda_j\in[\lambda,\lambda+1)} |u_j(x)|^2\leq C \lambda.
\end{equation}

The estimates \eqref{projest2} and \eqref{gradprojest2} then allow us to control each term of \eqref{heatkernel} in an average sense.

For example, as only the the horizontal component $1-\chi_\delta$ of the cutoff function in \eqref{heatkernel} is $\delta$-dependent, we can integrate by parts in the second order term in \eqref{heatkernel} without loss. 

Then, by writing $\eta_{\delta}$ to denote the cutoff function in the new second order term and choosing $\delta >0$ sufficiently small so that $\eta_\delta=\eta\eta_\delta$, the contribution of these terms to \eqref{heatkernel} is controlled by
\begin{eqnarray*}
& &E^{-1}\sum_{E_j\leq E} E_j^{-1}\int_M \eta_\delta(x)|\nabla u_j(x)|^2 \, dx\\
&\sim &E^{-1}\sum_{k=1}^{E^{1/2}-1}\sum_{\lambda_j\in [k,k+1)} E_j^{-1}\int_M \eta_\delta(x)|\nabla u_j(x)|^2 \, dx\\
&\leq & E^{-1}\sum_{k=1}^{E^{1/2}-1}k^{-2} \int_M \eta_\delta(x)\eta(x)\sum_{\lambda_j\in [k,k+1)}|\nabla u_j(x)|^2 \, dx\\
& \leq & CE^{-1}\left(\sum_{k=1}^{E^{1/2}-1} k\right) \int_M \eta_\delta(x)\, dx\\
&\leq & \hat{C_\delta}
\end{eqnarray*}
for sufficiently small $\delta>0$, where $\hat{C_\delta}\rightarrow 0$ as $\delta\rightarrow 0$.

Together with analogous estimates for lower order terms, we obtain the estimate
\begin{equation}
\frac{1}{n}\sum_{j=1}^n E_j^{-1}|\langle(Q_\delta-Q)u_j,u_j)\rangle|<C_\delta
\end{equation}
where $C_\delta\rightarrow 0$ as $\delta\rightarrow 0$.

By taking $\delta$ sufficiently small that $C_\delta < \epsilon^2$, we ensure that the collection of $j$ with $E_j^{-1}|\langle(Q_\delta-Q)u_j,u_j)\rangle|\geq \epsilon$ has upper density at most $\epsilon$.\\

The estimate \eqref{projest} follows from Proposition 8.1 in \cite{restrict}. Note that the finite propagation speed of the operator $\cos(t\sqrt{-\Delta})$, the post-composition with a cutoff near a flat boundary, and the small-time nature of the argument together imply that $M_t$ can be treated as the half-plane, which certainly satisfies the geometric assumptions of the cited result.

By inserting the gradient operator in the dual estimate, it remains to control  $\|1_{[\lambda,\lambda+1)}(\sqrt{-\Delta})\nabla \eta_t\|_{L^1(M_t)\rightarrow L^2(M_t)}$ in order to give us \eqref{gradprojest}. 

The argument in the proof of Proposition 8.1 in \cite{restrict} allows us to replace the spectral projector by a smooth spectral projector $\rho_\lambda^{ev}(\sqrt{-\Delta})$ where ${\rho_\lambda^{ev}(s)=\rho(s - \lambda)+\rho(-s-\lambda)}$ and $\rho\in\mathcal{S}(\R)$ has non-negative Fourier transform supported in $[\epsilon/2,\epsilon]$ for some sufficiently small $\epsilon$.

So it suffices to estimate the $L^1\rightarrow L^2$ norm of the operator
\begin{equation}
\rho_\lambda^{ev}(\sqrt{-\Delta})\nabla=\frac{1}{\pi}\int_\R \cos(t\sqrt{-\Delta})(e^{-it\lambda}\hat{\chi}(t)+e^{it\lambda}\hat{\chi}(-t))\nabla \eta(x) \, dt.
\end{equation}

This integral is supported close to $t=0$, and hence by finite propagation speed, the kernel of the wave equation solution operator $\cos(t\sqrt{-\Delta})$ on $M$ is identical to that of the half-plane.

Moreover, the kernel of the wave equation solution operator on the half-plane can be obtained from the free space wave kernel by the reflection principle, and their $L^1\rightarrow L^2$ norms are identical.

This implies that it suffices to prove the estimate with the kernel for $\cos(t\sqrt{-\Delta})$ replaced by the free space wave kernel.

So the kernel to be estimated is
\begin{eqnarray*}
& &\frac{1}{4\pi^3}\int_\R \int_{\R^2}\int_{\R^2} e^{i(x-y)\cdot \xi}(e^{-it\lambda}\hat{\chi}(t)+e^{it\lambda}\hat{\chi}(-t))\xi \cos(|\xi|t)\eta(x)\,dy \,d\xi \, dt\\
&=& \nabla_x(K_\lambda(x,y))\eta(x)\\
&=& \nabla_x(\lambda^{(n-1)/2}a_\lambda(x,y)e^{i\lambda\psi(x,y)})\eta(x)\\
&=& O(\lambda^{(n+1)/2})
\end{eqnarray*}
where $K_\lambda,a_\lambda,\psi$ are as in Lemma 5.13 from \cite{sogge}, which we make use of in our penultimate line.

Duality then completes the proof of \eqref{gradprojest} and the proposition.
\end{proof}
Proposition \ref{hadamard} allows us to use Theorem \ref{galkowski} and Proposition \ref{heatkernelprop} to control the flow speed of a full density subsequence of the eigenfunctions for any fixed $t$.

\begin{prop}
\label{flowspeed1}
For each $t\in(0,2]$ there exists a full density subsequence $(n_k)$ of the positive integers such that
\begin{equation}
\liminf_{k\rightarrow\infty} E_{n_k}^{-1}(t)\dot{E}_{n_k}(t) \geq -\frac{\dot{A}(t)}{A(t)(1-d(t))}
\end{equation}
where $d(t)$ denotes the proportion of the phase space volume that is in the completely integrable region $S^{*}M_t\setminus U_t$.
\end{prop}
\begin{proof}
From Proposition \ref{hadamard}, Proposition \ref{heatkernelprop} and Theorem \ref{galkowski}, for all $\epsilon>0$ we may choose a $\delta>0$ and a subsequence of eigenfunctions with lower density bounded below by $1-\epsilon$ such that we have the estimate \eqref{heatkernel} and such that we have a unique semiclassical measure $\mu$ with $\mu|_{U_t}=a\mu_L|_{U_t}$ for some non-negative constant $a$.
We immediately have
\begin{equation}
a=\frac{\mu(U_t)}{1-d(t)}\leq \frac{1}{1-d(t)}.
\end{equation}

The definition of semiclassical measures then implies
\begin{eqnarray*}
\liminf_{k\rightarrow\infty} E_{n_k}^{-1}(t)\dot{E}_{n_k}  &\geq & -\frac{1}{2}\int_{S^*M} \sigma(Q_\delta) \, d\mu-\epsilon\\
&\geq & -\frac{1}{2}\int_{S^*M} 4\phi_t\xi_2^2 \, d\mu-\epsilon \\
&\geq & -\frac{2}{1-d}\int_{S^*M} \phi_t\xi_2^2 \, d\mu_L-\epsilon\\
&\geq & -\frac{1}{\pi(1-d(t))A(t)}\int_M \int_0^{2\pi}\phi_t(x)\sin^2(\theta)\,d\theta\, dx -\epsilon\\
&\geq & -\frac{\dot{A}(t)}{(1-d)A(t)}-\epsilon.
\end{eqnarray*}

We can then apply Lemma \ref{densitylemma} to obtain a full density subsequence of eigenfunctions satisfying the estimate \eqref{flowspeedeq}.
\end{proof}

Moreover, we can strengthen the above to an almost-uniform result.

\begin{prop}
\label{flowspeed}
For any $\epsilon>0$, there exists a full density subsequence $(n_k)$ of positive integers and a family of sets $B_{n_k}\subseteq (0,2]$ with $m(B_{n_k})\rightarrow 0$ such that
\begin{equation}
\label{flowspeedeq}
E_{n_k}^{-1}(t)\dot{E}_{n_k}(t) > -\frac{\dot{A}(t)}{A(t)(1-d(t))}-\epsilon
\end{equation}
for each $t\in (0,2]\setminus B_{n_k}$.
\end{prop}
\begin{proof}
For each $\delta>0$ we define the subset $S_\delta\subseteq \mathbb{N}$ as the collection of $n\in\mathbb{N}$ such that
\begin{equation}
m(\{t\in (0,2]:E_{n}^{-1}(t)\dot{E}_{n}(t) \leq -\frac{\dot{A}(t)}{A(t)(1-d(t))} -\epsilon\})>\delta.
\end{equation}

If every $S_\delta$ were of zero density, we could write $S'_\delta:=\mathbb{N}\setminus S_\delta$ and use Lemma \ref{densitylemma} to assemble a full density set satisfying the claims of the proposition.

Now suppose that $S_\delta$ has positive upper density for some $\delta>0$. 

Since for every $n\in S_\delta$, the sets $B_n=\{t:E_{n}^{-1}(t)\dot{E}_{n}(t) \leq -\frac{\dot{A}(t)}{A(t)(1-d(t))} -\epsilon\}$ have measure bounded below and are subsets of a set with finite measure, there must exist a further positive density subset $\hat{S}_\delta\subseteq S_\delta$ such that 
\begin{equation}
\bigcap_{n\in\hat{S}_\delta} B_n\neq \emptyset.
\end{equation} 
This can be seen for instance by applying the bounded convergence theorem to the function
\begin{equation}
\frac{1}{n}\sum_{j=1}^n 1_{B_j}.
\end{equation}
The existence of a $t$ in this intersection contradicts Proposition \ref{flowspeed1} and hence completes the proof.
\end{proof}

The almost-uniform result in Proposition \ref{flowspeed} can for our purposes be treated as a uniform bound on speed for large $E_j$, in light of the following weaker bound for $t$ in the sets $B_j$ of diminishing measure for which \eqref{flowspeedeq} does not hold.

\begin{prop}
\label{almostuniform}
There exists a positive constant $C$ such that for every $t\in (0,2]$ and every $j$, we have
\begin{equation}
\dot{E}_j(t)=-\frac{1}{2}\langle Qu_j,u_j\rangle\geq -CE_j(t)
\end{equation}
\end{prop}
\begin{proof}
Integration by parts and Cauchy--Schwartz on the left-hand side provides us with a lower bound of 
\begin{eqnarray*}
& &-\hat{C} \iint_M |(\partial_y^2 u_j)u_j| +|(\partial_y u_j)u_j|+|u_j|^2 \, dx\, dy\\
&\geq & -\hat{C}(\langle -\Delta u_j,u_j \rangle+\langle -\Delta u_j, u_j \rangle^{1/2}+1 )\\
& \geq & -CE_j
\end{eqnarray*}
for a positive constant $C$ that is uniform in time.
\end{proof}

\begin{cor}
\label{speedcor}
For any $\epsilon>0$, there exists $\delta>0$ and a full density subsequence $(n_k)$ such that we have
\begin{equation}
-\int_S \dot{E}_{n_k}(t)\, dt\leq E_{n_k}(t_1)\left(\frac{\dot{A}(t_1)}{A(t_1)(1-d(t_1))}+\epsilon\right)(t_2-t_1)
\end{equation}
for any measurable set $S\subseteq (0,2]$ with measure greater than $\delta$.
\end{cor}
\begin{proof}
This follows from Proposition \ref{flowspeed} and Proposition \ref{almostuniform} by removing finitely many elements from the subsequence constructed in Proposition \ref{flowspeed}.
\end{proof}

\section{Main results}
\label{mushymainresultsec}

We are now in a position to prove the main result of the chapter.
\begin{thm}
\label{percivalconclusion}
The weak Percival's conjecture \ref{percival} holds for the mushroom billiard $M_t$ for any fixed inner and outer radii, and almost all ``stalk lengths" $t\in(0,2]$.
\end{thm}

We prove Theorem \ref{percivalconclusion} by establishing the claim for $t\in \mathcal{G}\subset(0,2]$ satisfying a certain spectral nonconcentration property, and then prove that this set is of full measure.

For a given $c>0$ we define \emph{$c$-clusters} to be the connected components of $\cup_{i\in\N} [\alpha_i^2-c,\alpha_i^2+c]$, indexed $\mathcal{C}_k$ in increasing order.

\begin{defn}
We call $t\in(0,2]$ \emph{good} if for every $\epsilon>0$, there exists some $c>0$ with
\begin{equation}
\label{goodtime}
\liminf_{n\rightarrow \infty} \frac{\#\{j\in\N:E_j(t)\in \cup_{k=1}^n \mathcal{C}_k\}}{\#\{j\in\N:\alpha_j^2 \in \cup_{k=1}^n \mathcal{C}_k\}}<1+\epsilon^2
\end{equation}
We denote the set of good times by $\mathcal{G}$.
\end{defn}

We begin by proving the claim for fixed $t\in\mathcal{G}$.

We write $N_\textrm{semidisk}(\mathcal{C}),N_\textrm{mushroom}(\mathcal{C})$ to denote the number of quasi-eigenvalues and eigenvalues respectively contained in a given $c$-cluster $\mathcal{C}$. 

The assumption \eqref{goodtime} then implies the following Proposition.

\begin{prop}
\label{throwaway}
Suppose $t$ and $0<c<2/r_2^2$ are such that \eqref{goodtime} holds, and all but finitely many $c$-clusters contain at least as many eigenvalues as quasi-eigenvalues. Then there exists a subset $J$ of quasi-eigenvalues with upper  density at least $(1-\epsilon)$ such that each quasi-eigenvalue in $J$ is contained in a cluster $\mathcal{C}$ with 
\begin{equation}
\label{throwawayeq}
N_\textrm{semidisk}(\mathcal{C})\leq N_\textrm{mushroom}(\mathcal{C}) \leq (1+\epsilon)N_\textrm{semidisk}(\mathcal{C})
\end{equation}
\end{prop}
\begin{proof}
Index the $c$-clusters $\mathcal{C}_k$ in increasing order. 

Let $$S=\{k:N_\textrm{semidisk}(\mathcal{C}_k)\leq N_\textrm{mushroom}(\mathcal{C}_k)\leq(1+\epsilon)N_\textrm{semidisk}(\mathcal{C}_k)\}$$ and $$F=\{k:N_\textrm{mushroom}(\mathcal{C}_k)<N_\textrm{semidisk}(\mathcal{C}_k)\}.$$ 
From the defining property \eqref{goodtime} of $t\in\mathcal{G}$, we have:
\begin{equation}
\liminf_{n\rightarrow\infty}\frac{\sum_{k\leq n} N_\textrm{mushroom}(k)}{\sum_{k\leq n}N_\textrm{semidisk}(k)} < 1+\epsilon^2
\end{equation}
The definition of $S$ implies
\begin{multline*}
\liminf_{n\rightarrow\infty}\left(1-\frac{\sum_{k\leq n} N_\textrm{semidisk}(k)1_F(k)}{\sum_{k\leq n} N_\textrm{semidisk}(k)}\right.\\
+\left.\epsilon\frac{\sum_{k\leq n} N_\textrm{semidisk}(k)(1-1_F(k)-1_S(k))}{\sum_{k\leq n} N_\textrm{semidisk}(k)}\right)< 1+\epsilon^2.
\end{multline*}
The second term on the left-hand side is $o(1)$ from the finiteness assumption. Hence we obtain that the lower  density of $\N\setminus(S\cup F)$ is bounded above by $\epsilon$. As $F$ is finite, and consequently of density $0$, we can conclude that the upper density of $S$ is bounded below by $1-\epsilon$ as required. 
\end{proof}

\begin{thm}[Main Theorem]
\label{mainthm}
For each $t\in\mathcal{G}$, there exists $B_t\subset \N$ of density $d(t)$ such that any semiclassical measure associated to the eigenfunctions $(u_n)_{n\in B_t}$ is supported inside the completely integrable region.
\end{thm}
\begin{proof}

First, we fix $\epsilon >0$ and choose $c>0$ small enough so that the inequality \eqref{goodtime} holds.

Proposition \ref{decay} implies that we may choose the $\epsilon_1,\epsilon_2$ in  applications of Proposition \ref{spectral} to the increasing sequence of $c$-clusters to decay faster than any polynomial in the energy infima of these $c$-clusters.

By Weyl's law, this ensures that for all but possibly finitely many $c$-clusters, we have \eqref{errormatrixbound} with $\delta$ decaying faster than any polynomial in energy. We remove the exceptional $c$-clusters, without any loss in density of our subset.\\

In light of Proposition \ref{throwaway}, we can then select a subset of the remaining $c$-clusters such that the included subset of quasi-eigenvalues has upper density exceeding $1-\epsilon$ and such that \eqref{throwawayeq} holds for each cluster.

We can now apply Proposition \ref{spectral} on a cluster-by-cluster basis, with parameter $\delta\rightarrow 0$ faster than any polynomial in energy.

From the $L^2$ boundedness of pseudodifferential operators with compactly supported symbols, this implies that for each $\epsilon$, we get a subsequence of eigenfunctions $u_{j_k}$ such that any associated semiclassical measure $\mu$ satisfies
\begin{equation}
\mu(\mathcal{D}_t\setminus U_t)\geq 1-\epsilon^{1/4}.
\end{equation}

Moreover, by comparison to Proposition \ref{quasiasymptotic} and Weyl's law for the mushroom, we obtain a lower bound of $d(t)-h(\epsilon)$ for the upper density of this eigenfunction subsequence with  $h(\epsilon)\rightarrow 0$ as $\epsilon \rightarrow 0$. So for each $\epsilon>0$, we can find a $c$ and a subsequence of $(u_n)$ with upper  density at least $d(t)-h(\epsilon)$ which concentrates in the completely integrable region up to $\epsilon^{1/4}$ of its semiclassical mass.

We now take a sequence $\epsilon_j\rightarrow 0$ and denote the corresponding eigenvalue window widths by $c_j$. We write $B_{j,t}$ to denote the corresponding  concentrating eigenfunction subsequences.

Lemma \ref{densitylemma2} then allows us to obtain a subsequence $B_t$ of lower density at least $d(t)$ such that any associated semiclassical measure $\mu$ satisfies
\begin{equation}
\label{local}
\mu(\mathcal{D}_t\setminus U_t)= 1.
\end{equation}

To show that the upper density of $B_t$ cannot exceed $d(t)$, we choose a function $\chi_\epsilon\in\mathcal{C}_c^{\infty}(\R^2\times \R^2)$ supported in the interior of $M$ such that the following properties are satisfied.
\begin{itemize}
\item $0\leq \chi_\epsilon \leq 1$
\item $\chi_\epsilon | _{\mathcal{D}_t\setminus U_t}=0$
\item $\int_{\mathcal{D}_t} \chi_\epsilon \, d\mu_L > (1-\epsilon)\mu_L(U_t)$.
\end{itemize}

Applying the local Weyl law (Lemma 4 from \cite{zelditch-zworski}) to the corresponding semiclassical  pseudodifferential operator $\chi_\epsilon(x,hD)$, we obtain

\begin{equation}
\frac{1}{n}\sum_{j\in [1,n]\cap B_t} \langle \chi_\epsilon(x,E_j^{-1/2}D) u_j,u_j \rangle+\frac{1}{n}\sum_{j\in [1,n]\cap B_t^c}\langle \chi_\epsilon(x,E_j^{-1/2}D) u_j,u_j \rangle > (1-\epsilon)\mu_L(U_t) \end{equation}

for sufficiently large $n$.

The localisation property \eqref{local} implies that the first summand is $o(1)$ in $n$. Hence we have
\begin{equation}
\frac{1}{n}\sum_{j\in [1,n]\cap B_t^c}\langle \chi_\epsilon(x,E_j^{-1/2}D) u_j,u_j \rangle > (1-2\epsilon)\mu_L(U_t)
\end{equation}
for sufficiently large $n$.

From Theorem \ref{galkowski} and the bound $a\leq \mu_L(U_t)^{-1}$ that is immediate from semiclassical measures being probability measures, it follows that a full density subset of the remaining summands must be bounded above by $1$.

This implies that
\begin{equation}
\frac{\#\{j\leq n: j\in B_t^c\}}{n} > (1-2\epsilon)\mu_L(U_t)
\end{equation}
for sufficiently large $n$.

Rearranging and passing to the limit $n\rightarrow \infty$ and then $\epsilon\rightarrow 0$, we obtain the required upper bound of
\begin{equation}
\limsup_{n\rightarrow\infty} \frac{\#\{j\leq n: j\in B_t\}}{n}\leq 1-\mu_L(U_t)=d(t).
\end{equation}

Hence $B_t$ is a sequence of eigenfunctions of upper density $d(t)$ with semiclassical mass supported in the completely integrable region.
\end{proof}

\begin{prop}
\label{ergodicprop}
Let $A_t=\N\setminus B_t$. Then for each $t\in \mathcal{G}$, a full density subsequence of $(u_n)_{n\in A_t}$ equidistributes in $U_t$.
\end{prop}

\begin{proof}
From Theorem \ref{mainthm}, the sequence of eigenfunctions $(u_n)_{n\in B_t}$ has all semiclassical mass in the completely integrable region and $B_t$ has natural density $d(t)$.

Applying the local Weyl law again with the function $\chi_\epsilon$ from the proof of Theorem \ref{mainthm}, we obtain
\begin{equation}
\label{jt}
\frac{1}{n}\sum_{j\in [1,n]\cap B_t^c}\langle \chi_\epsilon(x,E_j^{-1/2}D) u_j,u_j \rangle > (1-2\epsilon)\mu_L(U_t)
\end{equation}
for sufficiently large $n$.

Then, splitting the summation into the set 
\begin{equation}
A_{\epsilon,t}=\{j\in B_t^c:\langle \chi_\epsilon(x,E_j^{-1/2}D) u_j,u_j \rangle < 1-\sqrt{\epsilon}\}
\end{equation} 
and its complement, the upper bound of $1$ for a full density subset of the summands in \eqref{jt} then implies:
\begin{equation}
\liminf_{n\rightarrow \infty} d_n(A_{\epsilon,t})< 2\sqrt{\epsilon}\mu_L(U_t)
\end{equation}
using the notation $d_n$ from Lemma \ref{densitylemma}.
Thus we obtain a subset of upper density exceeding $1-O(\sqrt{\epsilon})$ of $A_t$ with at least $\mu(U_t)>1-O(\sqrt{\epsilon})$ for any corresponding semiclassical measure.

An application of Lemma \ref{densitylemma2} then gives us a subsequence of $A_t$ with full upper density and all semiclassical mass in $U_t$.

Together with Theorem \ref{galkowski}, this implies that we can find a subsequence $(u_{n_k})$ of $A_t$ with full  upper density such that every associated semiclassical measure is of the form $1_{U_t}\cdot \mu_L(U_t)^{-1}\mu_L$.
\end{proof}

We now show that the set $(0,2]\setminus \mathcal{G}$ has Lebesgue measure $0$.

\begin{prop}
\label{propo}
If $(0,2]\setminus \mathcal{G}$ has positive Lebesgue measure, then there exists some $\epsilon>0$ and some interval $\mathcal{I}=[t_1,t_2]\subset (0,2]$ such that
\begin{equation}
\label{play1}
\frac{1}{|\mathcal{I}|}\int_\mathcal{I}\liminf_{n\rightarrow\infty}\left(\frac{\#\{j\in\N:E_j(t)\in \cup_{i=1}^n [\alpha_i^2-c,\alpha_i^2+c]\}}{n}\right) \, dt > 1+\epsilon
\end{equation}
for all $c>0$. Moreover, we can find such $\mathcal{I}$ with arbitrarily small length.
\end{prop}
\begin{proof}
By the monotone convergence property of measures, if $m((0,2]\setminus \mathcal{G})>0$ then there must exist $\epsilon>0$ and a positive measure set $S\subseteq (0,2]$ on which we have
\begin{equation}
\liminf_{n\rightarrow\infty}\left(\frac{\#\{j\in\N:E_j(t)\in \cup_{i=1}^n [\alpha_i^2-c,\alpha_i^2+c]\}}{n}\right)>1+2\epsilon
\end{equation}
for all $t\in S$ and for all $0<c<2/r_2^2$.
From the regularity of the Lebesgue measure, we can find an open set $S\subseteq U\subseteq (0,2]$ with $m(U)<m(S)+\delta$ for an arbitrarily small $\delta$.
We then have
\begin{equation}
\frac{1}{|S|}\int_S\liminf_{n\rightarrow\infty}\left(\frac{\#\{j\in\N:E_j(t)\in \cup_{i=1}^n [\alpha_i^2-c,\alpha_i^2+c]\}}{n}\right) \, dt > 1+2\epsilon
\end{equation}
and
\begin{equation}
\frac{1}{|U\setminus S|}\int_{U\setminus S}\liminf_{n\rightarrow\infty}\left(\frac{\#\{j\in\N:E_j(t)\in \cup_{i=1}^n [\alpha_i^2-c,\alpha_i^2+c]\}}{n}\right) \, dt \geq 1.
\end{equation}
from our pointwise bounds on the integrands.
By choosing $\delta$ sufficiently small, we are thus guaranteed the estimate
\begin{equation}
\frac{1}{|U|}\int_U \liminf_{n\rightarrow\infty}\left(\frac{\#\{j\in\N:E_j(t)\in \cup_{i=1}^n [\alpha_i^2-c,\alpha_i^2+c]\}}{n}\right) \, dt >1+\epsilon.
\end{equation}
Writing the open set $U$ as a countable union of disjoint open intervals, the average of the integrand over one such interval must exceed $1+\epsilon$, as claimed.

To complete the proof, we observe if we partition $\mathcal{I}$ into arbitrarily many intervals of equal length, at least one of them must also satisfy \eqref{play1}.
\end{proof}

To culminate the argument, we seek out a contradiction coming from the upper bound  \eqref{flowspeedeq} on speed of eigenvalue variation and the lower bound \eqref{play1} on the average proportion of eigenvalues lying in $c$-clusters.

\begin{prop}
For any $\epsilon>0$, there exists $c>0$ such that
\begin{equation}
\limsup_{m\rightarrow \infty} \frac{1}{m}\sum_{j=1}^m\frac{|\{t\in \mathcal{I}:E_j\in \cup_i [\alpha_i^2-c,\alpha_i^2+c]\}|}{|\mathcal{I}|}<d(t_1)+\epsilon
\end{equation}
for any sufficiently small interval $\mathcal{I}$.
\end{prop}
\begin{proof}
Note that we have the flow speed bound $\eqref{flowspeedeq}$ for a full density subsequence of eigenvalues, so if we can establish the claimed inequality for each summand with a sufficiently large index that obeys the flow speed bound, density will allow us to draw the desired conclusion.

We now suppose $E_j$ is a large eigenvalue that lies in this full density subsequence.

Writing $X=(A(t_1)^{-1}-A(t_2)^{-1})$ for brevity, Weyl's law applied to the mushroom gives
\begin{equation}E_j(t_1)-E_j(t_2)> (4\pi X -2\delta)j \end{equation}
for $\delta>0$ and all sufficiently large $j$.

Weyl's law for the semidisk (recalling that we constructed the completely integrable region quasimodes from semidisk eigenfunctions) gives us an upper bound of
\begin{equation}
\left(\frac{\pi r_2^2X}{2}+\delta\right) j
\end{equation}
for the number of quasi-eigenvalues in $[E_j(t_2),E_j(t_1)]$
and hence an upper bound of
\begin{equation}
2c\left(\frac{\pi r_2^2X}{2}+\delta\right) j
\end{equation}
for the length of $[E_j(t_2),E_j(t_1)]$ that lies within $\cup_i [\alpha_i^2-c,\alpha_i^2+c]$.

Now suppose that $E_j(t)$ spends proportion $q_j$ of $t\in\mathcal{I}$ in $\cup_i [\alpha_i^2-c,\alpha_i^2+c]$. 

From Proposition \ref{almostuniform}, it follows that the $q_j$ are uniformly bounded above by some $1-\delta$.

This means that we apply Corollary \ref{speedcor} to find a lower bound for the time taken by an eigenvalue $E_j$ in our full density subsequence to traverse the set \\ \noindent${[E_j(t_2),E_j(t_1)]\setminus \cup_i [\alpha_i^2-c,\alpha_i^2+c]}$. Heuristically, we can think of this as dividing the size of this set by an upper bound for the speed of the eigenvalue's variation.


Precisely, we have
\begin{eqnarray}
(1-q_j)(t_2-t_1)&>& \frac{jX(4\pi-\pi r_2^2 c)-j\delta(1+2c) }{E_j(t_1)(\frac{\dot{A}(t_1)}{A(t_1)(1-d(t_1))}+\delta)}\\
&=&\frac{j}{E_j(t_1)}\cdot\left(\frac{X(4\pi-\pi r_2^2 c)-\delta(1+2c) }{\frac{\dot{A}(t_1)}{A(t_1)(1-d(t_1))}+\delta}\right)\\
&>& \left(\frac{4\pi}{A(t_1)}+\delta\right)^{-1}\cdot\left(\frac{X(4\pi-\pi r_2^2 c)-\delta(1+2c) }{\frac{\dot{A}(t_1)}{A(t_1)(1-d(t_1))}+\delta}\right)\\
&>& \frac{XA(t_1)^2(1-d(t_1))}{\dot{A}(t_1)}-\frac{\epsilon}{2}
\end{eqnarray}
where the final two lines follow from Weyl's law and passing to sufficiently small $\delta$ and $c$ respectively.

Additionally, since $A(t)$ is a linear polynomial in $t$, we have 
\begin{equation}
X=\frac{A(t_2)-A(t_1)}{A(t_1)A(t_2)}=\frac{(t_2-t_1)\dot{A}(t_1)}{A(t_1)A(t_2)}
\end{equation}
which implies that
\begin{eqnarray*}
1-q_j &>& \frac{A(t_1)}{A(t_2)}(1-d(t_1))-\frac{\epsilon}{2}\\
q_j &<& d(t_1)+(d(t_1)-1)\left(\frac{A(t_1)}{A(t_2)}-1\right)+\frac{\epsilon}{2}\\
&<& d(t_1)+\epsilon
\end{eqnarray*}
for sufficiently small $|\mathcal{I}|$, using the uniform continuity of $A$.

Thus we have the required inequality for all sufficiently small intervals $\mathcal{I}$ and all sufficiently large $j$ in a full density subsequence on which \eqref{flowspeedeq} holds.

\end{proof}

\begin{prop}
\label{propo2}
For any $\epsilon>0$ there exists $c>0$, such that
\begin{equation}
\frac{1}{|\mathcal{I}|}\int_\mathcal{I}\liminf_{n\rightarrow\infty}\left(\frac{\#\{j\in\N:E_j(t)\in \cup_{i=1}^n [\alpha_i^2-c,\alpha_i^2+c]\}}{n}\right)\, dt<1+\epsilon
\end{equation}
for all sufficiently small $|\mathcal{I}|$.
\end{prop}
\begin{proof}
By Fatou's lemma, it suffices to show that we can find $c$ such that
\begin{equation}
\frac{1}{|\mathcal{I}|}\int_\mathcal{I} \frac{\#\{j:E_j(t)\in \cup_{i=1}^n[\alpha_i^2-c,\alpha_i^2+c]\}}{n}\, dt < 1+\frac{\epsilon}{2}
\end{equation}
for sufficiently large $n$.
This quantity is bounded above by
\begin{equation}
\frac{1}{n}\sum_{j:E_j(t_2)<\alpha_n^2+c} \frac{|\{t\in \mathcal{I}:E_j\in \cup_i [\alpha_i^2-c,\alpha_i^2+c]\}|}{|\mathcal{I}|}=\frac{1}{n}\sum_{j:E_j(t_2)<\alpha_n^2+c} q_j.
\end{equation}
The sum is controlled by the previous proposition, giving us an upper bound of
\begin{equation}
\frac{1}{n}\cdot \max \{j:E_j(t_2)<\alpha_n^2+c\}(d(t_1)+\delta)
\end{equation}
for sufficiently large $n$.

From Weyl's law, we have
\begin{equation}
\max \{j:E_j(t_2)<\alpha_n^2+c\}<(\alpha_n^2+c)\left( \frac{A(t_2)}{4\pi}+\delta\right)
\end{equation}
for sufficiently large n.

By taking $\delta$, $c$, and $\mathcal{I}$ sufficiently small, we then obtain 
\begin{equation}
\frac{1}{n}\sum_{j:E_j(t_2)<\alpha_n^2+c} q_j < \left(\frac{\alpha_n^2}{n}\cdot\frac{A(t_2)d(t_2)}{4\pi}\right)+\frac{\epsilon}{4}.
\end{equation}

Inverting the estimate \eqref{quasiasymptoticeqn} provides an upper bound of $1+\frac{\epsilon}{4}$ for the first summand on the right-hand side for all sufficiently large $n$, thus completing the proof.
\end{proof}

\begin{cor}
The set $\mathcal{G}$ has full measure in $(0,2]$.
\end{cor}
\begin{proof}
This is an immediate consequence of Propositions \ref{propo} and \ref{propo2}.
\end{proof}
This also completes the proof of Theorem \ref{percivalconclusion}.

\chapter{KAM Theory}
\label{chp4}






The KAM theory studied by Kolmogorov, Arnold and Moser in the 60's led to an improved understanding of the classical dynamics of a Hamiltonian $H$ that is a small perturbation of a completely integrable Hamiltonian $H^0$. In particular, they established that the Lagrangian invariant tori corresponding to all but a $o(1)$ measure subset of frequencies survive this perturbation as the size of the perturbation tends to zero. We denote the union of these tori by $\Lambda$.\\

The paper \cite{popovkam} uses a local version of the KAM theorem to construct a Birkhoff normal form for Gevrey class Hamiltonians $H$ about $\Lambda$. This normal form generalises the notion of ``action-angle" variables of a completely integrable Hamiltonian as discussed in \cite{arnoldmechanics}, and as a consequence of the normal form construction, Popov obtains an effective stability result for the Hamiltonian flow near the union of remaining invariant tori. The natural setting for the estimates is that of Gevrey regularity. This work generalises Popov's earlier work in \cite{popov1} and \cite{popov2} where he constructs a Birkhoff normal form for real analytic Hamiltonians.\\

The paper \cite{popovquasis} uses semiclassical Fourier integral operators to construct a corresponding quantum normal form for a class of semiclassical pseudodifferential operators $\mathcal{P}_h$, and uses this normal form to obtain a family of quasimodes microlocalised near the union of preserved invariant tori of a Hamiltonian associated to $\mathcal{P}_h$. In particular, this result is applicable to the Schr\"{o}dinger operator. \\

This can be regarded as an extension of the main result in \cite{colin}, which establishes the existence of quasimodes microlocalised near the Lagrangian tori of a completely integrable Hamiltonian on a compact smooth manifold. These results were obtained by a similar quantum normal form.\\

In this chapter we construct a family of Birkhoff normal forms corresponding to the convex family of Gevrey smooth Hamiltonians $H=H^0+tQ$, real-analytic in the parameter $t\in (-1,1)$. We shall complete the construction of quasimodes in Chapter 5 and use them to obtain results about the localisation of exact eigenfunctions.

Our treatment runs along the same lines as that of Popov \cite{popovkam}, yet we shall be fairly explicit in our presentation to ensure that the presence of the parameter $t$ does not affect any of the essential details in \cite{popovkam}.

\section{Notations}
We begin by introducing some notation that will be frequently used throughout Chapters 4 and 5.

The following function spaces will be the natural ones for much of this topic.
\begin{defn}
\label{gevrey}
For $\rho\geq 1$ and $X\subset \mathbb{R}^n$ open, the \textsf{Gevrey class of order }$\rho$ is given by
\begin{equation}
\label{gevreyeqn}
G^\rho_L(X):=\{f\in\mathcal{C}^\infty(X): \sup_\alpha\sup_{x\in X} |\partial_x^\alpha f(x)|L^{-|\alpha|}\alpha!^{-\rho}< \infty\}.
\end{equation}
\end{defn}

If $f\in G^\rho_L(X)$, the supremum in \eqref{gevreyeqn} is denoted by $\|f\|_L$. We will frequently suppress the $L$ in our notation. Equipped with this norm, $G^\rho_L(X)$ is a Banach space.

Gevrey regularity is generally weaker the real analyticity (they coincide when $\rho=1$ as can be seen by using the Cauchy--Hadamard theorem to characterise analytic functions by the growth of their Taylor coefficients) and importantly, there exist bump functions in the Gevrey class for $\rho>1$.  \\

An important property of the Gevrey class that follows from Taylor's theorem is that if a Gevrey function has vanishing derivatives, then locally it is super-exponentially small.

\begin{prop}
\label{flatissmall}
Suppose $f\in G^\rho(X)$, and $\rho>1$.
Then there exist positive constants $c,C,\eta$ and $r_0$ only dependent on the Gevrey constant $L$, the norm $\|f\|_L$, and the set $X$ such that
\begin{equation}
f(x_0+r)=\sum_{|\alpha|\leq \eta |r|^{1/(1-\rho)}}f_\alpha(x_0)r^\alpha +R(x_0,r)
\end{equation}
where $f_\alpha=(\partial^\alpha f)/\alpha!$ and
\begin{equation}
|\partial_x^\beta R(x_0,r)|\leq C^{1+|\beta|}\beta!^\rho e^{-c|r|^{-1/(\rho-1)}}\quad \forall 0<|r|\leq \min(r_0,d(x_0,\R^n\setminus X)).
\end{equation}
\end{prop}

In the study of KAM systems, we will need to work on functions on products of domains in Euclidean spaces with different degrees of Gevrey regularity in each variable.

To this end, we introduce the class of anisotropic Gevrey spaces.

\begin{defn}
\label{gevrey2}
Suppose $X$ and $Y$ are open subsets of Euclidean spaces. Suppose that $\rho_1,\rho_2\geq 1$ and $L_1,L_2>0$. Then
\begin{equation}
\label{anisotropicgevreyeqn}
G^{\rho_1,\rho_2}_{L_1,L_2}(X\times Y)=\{f\in\mathcal{C}^\infty(X\times Y):\sup_{(x,y)\in X\times Y} |\partial_x^\alpha\partial_y^\beta f|  L_1^{-|\alpha|} L_2^{-|\beta|}\alpha!^{-\rho_1} \beta!^{-\rho_2}<\infty \}.
\end{equation}
\end{defn}

If $f\in G^{\rho_1,\rho_2}_{L_1,L_2}$, then we denote the supremum in \eqref{anisotropicgevreyeqn} by $\|f\|_{L_1,L_2}$. Equipped with this norm, $G^{\rho_1,\rho_2}_{L_1,L_2}$ is a Banach space. This definition extends in the natural way to $k\geq 3$ variables. Furthermore, some of these variables might lie in complex domains.


For the statement of the KAM theorem in the next section, we record the following notation.
\begin{defn}
\label{domains1}
If $D\subset \mathbb{R}^n$ and $s,r>0$ we write
\begin{equation}
\mathbb{T}^n+s :=\{z\in\mathbb{C}^n/2\pi\mathbb{Z}^n:|\textrm{Im}(z)|\leq s\}
\end{equation}
and
\begin{equation}
D_{s,r}:=\{\theta\in\mathbb{C}^n / 2\pi\mathbb{Z}^n:|\textrm{Im}(\theta)|<s\}\times \{I\in \mathbb{C}^n:|I|<r\},
\end{equation}
where $|\cdot|$ denotes the sup-norm on $\C^n$ induced by the $2$-dimensional $\ell^\infty$ norm on $\C$.
\end{defn}

These domains arise from considering the analytic extension of real analytic Hamiltonians in action-angle variables. In this topic we frequently bound derivatives of analytic functions using Cauchy estimates, which requires the use of shrinking sequences of domains. 

For simplicity of nomenclature, we call an analytic function of several complex variables real analytic if its restriction to a function of $n$ real variables is real-valued.

As a final notational convenience, we use $|\cdot|$ to denote the $\ell^1$ norm when applied to elements of $\Z^n$ throughout this chapter, as well as the matrix norm induced by the sup norm on $\C^n$.

\section{Formulation of the KAM theorem}
Let $D^0\subset \mathbb{R}^n$ be a bounded domain, and consider a completely integrable Hamiltonian $H^0(I)=H^0(\theta,I):\mathbb{T}^n\times D^0 \rightarrow \mathbb{R}$ in action-angle coordinates. To begin with, we shall assume that this Hamiltonian is real analytic.

In addition, we assume the non-degeneracy condition $\det\left(\frac{\partial^2 H}{\partial I^2}\right)\neq 0$. This assumption implies that the map relating the action variable $I$ to the frequency $\omega=\nabla H^0(I)$ is locally invertible.

In fact, we assume that 
\begin{equation}
\label{nondegen}
I\mapsto \nabla H^0(I)
\end{equation}
is a diffeomorphism from $D^0$ to $\Omega^0\subset \mathbb{R}^n$. The inverse to this map is given by $\nabla g^0$, where $g^0$ is the Legendre transform of $H^0$.\\

We now let $D\subset D^0$ be a subdomain, and denote by $\Omega=\nabla H^0(D)$ the corresponding frequency set.

The phase space $\mathbb{T}^n\times D$ is then foliated by the family of Lagrangian tori $\{\mathbb{T}^n\times \{I\}:I\in D\}$ that are invariant under Hamiltonian flow associated to $H^0$.\\

The KAM theorem asserts that small perturbations of $H(\theta,I)=H^0(I)+H^1(\theta,I)$ on $\mathbb{T}^n\times D$ still possess a family of Lagrangian tori which fill up phase space up to a set of Liouville volume $o(1)$ in the size of the perturbation.

More precisely, if $\Omega:=\{\omega:\omega=\nabla_I H^0\}$ is the set of frequencies for the quasi-periodic flow of $H^0$, the frequencies satisfying:
\begin{equation}
\label{diophantine}
|\langle\omega,k\rangle|\geq \frac{\kappa}{|k|^\tau}
\end{equation}
for all nonzero $k\in \mathbb{Z}^n$ and fixed $\kappa >0$ and $\tau>n-1$ also correspond to Lagrangian tori for the perturbed Hamiltonian $H$, provided $\|H-H^0\| < \epsilon(\kappa)$ in a suitable norm.\\

Such frequencies are said to be non-resonant, and we denote the set of non-resonant frequencies by $\Omega^*_\kappa$, suppressing the dependence on $\tau$ from our notation. These sets are obtained by taking the intersection of the sets
\begin{equation}
\{\omega\in \Omega:|\langle\omega,k\rangle|\geq \frac{\kappa}{|k|^\tau}\}
\end{equation}
over all nonzero $k\in \Z^n$, and hence $\cap_{\kappa>0}\Omega^*_\kappa$ is closed and perfect like the Cantor set. Unlike the Cantor set however, $\cap_{\kappa>0}\Omega^*_\kappa$ is of full measure in $\Omega$, as can be seen from the observation that 
\begin{equation}
m(\{\omega\in \R^n: |\langle k,\omega\rangle|< \frac{\kappa}{|k|^\tau}\})=O(\frac{\kappa}{|k|^{\tau+1}}).
\end{equation}
\begin{remark}
Such sets occur frequently in perturbation theory as a method of handling the infamous ``small divisor" problem that we shall encounter shortly in Section \ref{kamstepsec}.
\end{remark}

We work with the sets
\begin{equation}
\Omega_\kappa:=\{\omega\in\Omega_\kappa^*:\textrm{dist}(\omega,\partial\Omega) \geq\kappa\}
\end{equation}
which has positive measure for sufficiently small $\kappa$.\\

It is convenient to work with the set of points of Lebesgue density in $\Omega_\kappa$, which we denote by
\begin{equation}
\tilde{\Omega}_\kappa:=\{\omega\in \Omega:\frac{m(B(\omega,r)\cap \Omega_\kappa}{m(B(\omega,r))})\rightarrow 1 \textrm{ as }r\rightarrow 0\}.
\end{equation}

From the Lebesgue density theorem we have that $m(\tilde{\Omega}_\kappa)=m(\Omega_\kappa)$. We also note that a smooth function vanishing on $\Omega_\kappa$ is necessarily flat on $\tilde{\Omega}_{\kappa}$\\

The construction of the Birkhoff normal form stems from Theorem \ref{kamwithparams}, which is a version of the KAM theorem localised around the frequency $\omega$ which is taken as an independent parameter. This version is particularly useful for our application as it makes it an easier task to check the regularity of the invariant tori with respect to the  frequency parameter.\\

To illustrate the setup of this theorem, we set 
\begin{equation}
\label{domains}
\Omega'=\{\omega\in \Omega: \textrm{dist}(\omega,\Omega_\kappa)\leq \kappa/2\},\quad D'=\nabla g^0(\Omega').
\end{equation}

Taking $z_0\in D'$ we let $I=z-z_0$ lie in a small ball of radius $R$ about $0$. That is, $R$ is chosen such that $B_R(z_0)\subset D$.\\

Taylor expanding gives us the expression
\begin{equation}
\label{taylorexpand}
H^0(z)=H^0(z_0)+\langle \nabla_z H^0(z_0),I\rangle+\int_0^1 (1-t)\langle \nabla_z^2 H^0(z_0+tI)I,I\rangle\, dt.
\end{equation}

We now take $\omega\in\Omega^0$ to be the corresponding frequency $\nabla H^0(z_0)$. The inverse of the frequency map is 
\begin{equation}
\label{psidef}
\psi_0(\omega)=\nabla g^0(\omega),
\end{equation}
where $g^0$ is the Legendre transform of $H^0$.\\

Hence we can write
\begin{equation}
H^0(z)=H^0(\psi_0(\omega))+\langle \omega, I\rangle +\langle P^0(I;\omega)I,I\rangle
\end{equation}
where $P^0$ is the quadratic remainder term in \eqref{taylorexpand}.\\

Expanding about the point $z_0=\nabla g^0(\omega)$, we can write our perturbation $H^1$ locally as
\begin{equation}
H^1(\theta,z)=H^1(\theta,\nabla g^0(\omega)+I)=P^1(\theta,I;\omega).
\end{equation}

This leads us to consider perturbed real analytic Hamiltonians in the form
\begin{equation}
\label{expnearfreq}
H(\theta,I;\omega)=H^0(\psi_0(\omega))+\langle\omega,I\rangle+P(\theta,I;\omega)=:N(I;\omega)+P(\theta,I;\omega).
\end{equation}

where
\begin{equation}
N(I;\omega)=H^0(\psi_0(\omega))+\langle \omega, I\rangle
\end{equation}
and
\begin{equation}
\label{defnofPzero}
P(\theta,I;\omega)=\langle P^0(I;\omega)I,I\rangle+P^1(\theta,I;\omega).
\end{equation}

The traditional formulations of the KAM theorem assert the existence of a Cantor family of tori that persist under small perturbations of a single Hamiltonian $H^0$ with domain $D$. In the framework laid out above, we now have a Cantor family of Hamiltonians parametrised by $\omega\in \Omega_\kappa$. Note that each of these Hamiltonians is only linear in $I$.

The essence of our version of the frequency localised KAM theorem in Theorem \ref{kamwithparams} is then that for sufficiently small $P$, we can find a symplectic change of variables that transforms $H$ to a linear normal form in $I$ with remainder quadratic in $I$ for $\omega\in \Omega_\kappa$. This establishes the persistence of the Lagrangian torus  with frequency $\omega$. In Section \ref{mainresultssec} we will pass from this result to a version of the KAM theorem on $\mathbb{T}^n\times D$ in Theorem \ref{kamcons} that establishes the existence of a Cantor family of invariant tori for the original Hamiltonian $H$.\\

The proof of Theorem \ref{kamwithparams} is based on a rapidly converging iterative procedure introduced by Kolmogorov \cite{kolmogorov}. \\

To work with Gevrey smooth Hamiltonians, we fix $L_0\geq 1$ and $A_0>1$, and assume that $H^0\in G^\rho_{L_0}(D^0),g^0\in G^\rho_{L_0}(\Omega^0)$ with the estimates
\begin{equation}
\label{gevhamest}
\|H^0\|_{L_0},\|g^0\|_{L_0}\leq A_0.
\end{equation} 

For $L_2\geq L_1\geq 1$ with $L_2\geq L_0$, we now consider the analytic family of Gevrey perturbation ${H^1\in G^{\rho,\rho,1}_{L_1,L_2,L_2}(\mathbb{T}^n\times D\times (-1,1))}$
with the perturbation norm
\begin{equation}
\label{gevpertest}
\epsilon_H:=\kappa^{-2}\|H^1\|_{L_1,L_2,L_2}.
\end{equation}

The estimate \eqref{gevhamest} implies that there is a constant $C(n,\rho)$ dependent only on $n$ and $\rho$ such that taking
\begin{equation}
\label{Radbound}
R\leq \frac{C(n,\rho)\kappa}{A_0L_0^2}
\end{equation}
is sufficient to ensure that $B_R(z_0)\subset D$ for any $z_0\in D'$.\\

At this point we introduce the notational convention for this chapter that $C$ represents an arbitrary positive constant, dependent only on $n,\tau,\rho$ and $L_0$. Similarly, $c$ will represent a positive constant strictly less than $1$, also only dependent on $n,\tau,\rho$ and $L_0$. We will be explicit when we stray from this convention.\\

The  estimates \eqref{gevhamest} and \eqref{gevpertest}, together with Proposition A.3 in \cite{popovkam} show that our constructed functions $P^0$ and $P^1$ are in the Gevrey classes ${G^\rho_{CL_0,CL_2}(B_R\times \overline{\Omega'})\subset G^\rho_{CL_2,CL_2}(B_R\times \overline{\Omega'})}$ and $G^{\rho,\rho,\rho,1}_{L_1,L_2,CL_2,L_2}(\mathbb{T}^n\times B_R\times \overline{\Omega'}\times (-1,1))$ respectively, where the $C$ in $G^\rho_{CL_0,CL_2}$ does not depend on $L_0$ or $L_2$.

Additionally we have the estimate
\begin{equation}
\|P^1\|_{L_1,CL_2,CL_2,CL_2}\leq \kappa^{-2}\epsilon_H.
\end{equation}

Dropping the factors in our Gevrey constants dependent only on $n,\tau,\rho,L_0$ for brevity of notation, we are in a position to state the local KAM theorem in terms of the weighted norm
\begin{equation}
\langle P\rangle_r:=r^2\|P^0\|_{L_2,L_2}+\|P^1\|_{L_1,L_2,L_2,L_2}
\end{equation}
for $0<r<R$.

Our proofs in this chapter shall largely follow those of Popov \cite{popovkam}, but we shall be fairly explicit in order to keep the presentation relatively self-contained as well as to demonstrate regularity in the analytic parameter $t\in (-1,1)$.

\begin{thm}
\label{kamwithparams}
Suppose $0<\zeta\leq 1$ is fixed and $\kappa < L_2^{-1-\zeta}$. Then there exists ${N(n,\rho,\tau)>0}$ and $\epsilon>0$ independent of $\kappa,L_1,L_2,R,\Omega$ such that whenever the Hamiltonian 
\begin{equation}
H(\theta,I;\omega,t)=H^0(\psi_0(\omega))+\langle \omega,I\rangle+\langle P^0(I,\omega)I,I\rangle +P^1(\theta,I;\omega,t)
\end{equation} 
and $0<r<R$ are such that 
\begin{equation}
\label{kamwithparamasassump}
\langle P\rangle_r <\epsilon\kappa r L_1^{-N}
\end{equation}
we can find $$\phi\in G^{\rho(\tau+1)+1,1}(\Omega\times (-3/4,3/4),\Omega)$$ and $$\Phi=(U,V)\in G^{\rho,\rho(\tau+1)+1,1}(\mathbb{T}^n\times \Omega\times (-3/4,3/4),\mathbb{T}^n\times B_R)$$ such that
\begin{enumerate}
\item For all $\omega \in \Omega_\kappa$ and all $t\in (-3/4,3/4)$, the map $\Phi_{\omega,t}=\Phi(\cdot;\omega,t):\mathbb{T}^n\rightarrow \mathbb{T}^n\times B_R$ is a $G^\rho$ embedding, with image $\Lambda_{\omega,t}$ an invariant Lagrangian torus with respect to the Hamiltonian $H_{\phi(\omega,t),t}(\theta,I)=H(\theta,I;\phi(\omega,t),t)$. The Hamiltonian vector field restricted to this torus is given by
\begin{equation}
X_{H_{\phi(\omega,t),t}}\circ \Phi_{\omega,t}=D\Phi_{\omega,t}\cdot \mathcal{L}_\omega
\end{equation}
where 
\begin{equation}
\mathcal{L}_\omega=\sum_{j=1}^n \omega_j\frac{\partial}{\partial \theta_j}\in T\mathbb{T}^n.\\
\end{equation}

\item There exist positive constants $A$ and $C$ dependent only on $n,\tau,\rho,L_0$ such that
\begin{eqnarray}
\label{kamwithparamsbound}
& &|\partial_\theta^\alpha \partial_\omega^\beta (U(\theta;\omega,t)-\theta)|+r^{-1}|\partial_\theta^\alpha \partial_\omega^\beta V(\theta;\omega,t)|+\kappa^{-1}| \partial_\omega^\beta(\phi(\omega;t)-\omega)|\nonumber \\
&\leq& A(CL_1)^{|\alpha|}(CL_1^{\tau+1}/\kappa)^{|\beta|}\alpha!^\rho\beta!^{\rho(\tau+1)+1}\frac{\langle P\rangle_r}{\kappa r}L_1^N
\end{eqnarray}
uniformly in $\mathbb{T}^n\times \Omega\times (-3/4,3/4)$.
\end{enumerate}

We remark that at the endpoint $t=0$, this result is trivial by taking $\phi(\omega,0)=\omega,U(\theta,\omega,0)=\theta$ and $V(\theta,\omega,0)=\nabla g^0(\omega)$.
\end{thm}

KAM theory can be viewed as a collection of techniques for dealing with perturbative problems, rather than just the specific theorem describing the dynamics of a perturbed completely integrable Hamiltonian system. In this spirit we shall first present an application to Arnold's theorem on analytic diffeomorphisms of the circle  before moving on to the proof of Theorem \ref{kamwithparams}. Arnold's theorem is simpler than the KAM theorem from a technical standpoint, nonetheless it serves as an excellent model to provide the intuition behind its proof. An exposition of this problem can be found in \cite{wayne}, from which this account is drawn.\\

\section{Arnold's theorem on analytic circle diffeomorphisms}

Suppose $f:\R\rightarrow \R$ is the lift of an orientation preserving analytic diffeomorphism $\tilde{f}$ of $\T$. Then we have
\begin{equation}
f(x+1)=f(x)+1 \textrm{ and } f'(x)>0
\end{equation}
for each $x\in  \R$.

The simplest examples of such functions are the rotation maps $R_\theta(x):=x+\theta$, whose dynamics are well understood. In particular, for rational $\theta$, the corresponding dynamical system obtained by iterating $\tilde{R}_\theta$ is a periodic map, whilst for irrational $\theta$ the orbits are dense and equidistributed on $\T$.

Denjoy's theorem \cite[p. 301]{rotnum} shows that these are in fact the only orientation preserving diffeomorphisms up to conjugation by a homeomorphic and $1$-periodic change of variables $\chi:\R\rightarrow\R$.

The question as to whether arbitrary orientation preserving diffeomorphisms were \emph{analytically diffeomorphic} to rotation maps remained wide open until the invention of KAM theory. Arnold used KAM techniques to answer this question in the affirmative for small perturbations of suitably irrational rotation maps.

To make this precise, we first introduce the notion of suitable irrationality or nonresonance. This is a one-dimensional analogue to the nonresonance condition \eqref{diophantine}. 

\begin{defn}
\label{1dnonres}
For $\kappa,\rho >0$, we define 
\begin{equation}
\Omega_{\kappa,\rho}:=\{\theta\in\R:|\theta-m/n|> \frac{\kappa}{|n|^\rho}\textrm{ for all }m,n\in \Z, n\neq 0\}.
\end{equation}
\end{defn}

Intuitively, $\Omega_{\kappa,\rho}$ consists of irrational numbers that can't be approximated well by rationals. Dirichlet's theorem on Diophantine approximation asserts that every irrational number can be approximated by a rational number $m/n$ up to an order $n^{-2}$ error, but it turns out that is the best approximation we can hope for from a typical irrational number.

\begin{prop}
For $\rho >2$ and any interval $[a,b]$, the set 
\begin{equation}
\cup_{\kappa >0}([a,b]\setminus \mathbb{Q})\cap \Omega_{\kappa,\rho}
\end{equation}
is of full measure.
\end{prop}
A proof of this fact can be found in \cite[p. 116]{arnold2}.

As we are interested in small $\kappa$, we assume that $\kappa \leq 1$. We also require the notion of rotation number for an arbitrary circle diffeomorphism, which we now introduce.
\begin{prop}
\label{rotationnum}
For $f$ the lift of a circle diffeomorphism, the limit
\begin{equation}
\theta=\lim_{n\rightarrow\infty} \frac{f^n(x)-x}{n}
\end{equation}
exists and is independent of $x$. We say that $\theta$ is the \emph{rotation number} of the diffeomorphism $\tilde{f}$.
\end{prop}
The proof of Proposition \ref{rotationnum} can be found in \cite[p. 296]{rotnum}.

It is easy to show that rotation numbers are preserved by homeomorphisms. Moreover, we have the following.
\begin{cor}
\label{etavanishes}
If $f(x)=x+\theta+\eta(x)$ is the lift of a circle diffeomorphism with rotation number $\theta$, then $\eta(x)$ vanishes at some $x_0\in [0,1]$.
\end{cor}

The benefit of working with real analytic functions is that have the Cauchy estimates from Proposition \ref{cauchyappendix} at our disposal. To this end, we define
\begin{equation}
S_\sigma=\{z\in\C:|\textrm{Im}(z)|\leq \sigma\}
\end{equation}
for $\sigma>0$, and
\begin{equation}
H_\sigma=\{f:S_\sigma\rightarrow \C:f \textrm{ holomorphic and $1$-periodic on }S_\sigma\textrm{ and real-valued on }\R\}.
\end{equation}
\begin{thm}[Arnold's Theorem]
\label{arnolds}
Suppose that $\theta\in \Omega_{\kappa,\rho}$ and $f\in H_\sigma$ for some $\sigma>0$. Then there exists $\epsilon>0$, dependent only on $\kappa,\rho$ and $\sigma$, such that for any $f(x)=x+\theta+\eta(x)$ with rotation number $\theta$ and with $\eta$ satisfying $\|\eta\|_\sigma<\epsilon$, there exists a real analytic diffeomorphism $\chi$ such that $\chi^{-1}\circ f \circ \chi=R_\theta$.
\end{thm}

The proof proceeds in three steps.

\begin{enumerate}
\item First we linearise the equation defining the sought diffeomorphism $\chi$, and use basic Fourier analysis to solve this linearised problem and explicitly bound the solution.
\item The next key step is to note that upon applying the coordinate change constructed from the linear problem,  the function $f(x)$ transforms to $x+\theta+\tilde{\eta}(x)$, with $\tilde{\eta} =O (\|\eta\|^2)$. These two steps are a Banach space analogue of Newton's method for root approximation.
\item Finally, we iterate the process above. The primary challenge is to ensure the the composition of diffeomorphisms converges. This turns out to be a consequence of the rapid (quadratic) convergence.
\end{enumerate}

If the perturbation $\eta(x)$ in the statement of \ref{arnolds} is small, we expect that the sought diffeomorphism $\chi(x)$ is of the form
\begin{equation}
\label{chidefn}
\chi(x)=x+\mu(x)
\end{equation}
with $\mu$ small.

Substituting this into the equation
\begin{equation}
f(\chi(x))=\chi(R_{\theta}(x))
\end{equation}
we arrive at
\begin{equation}
\mu(x+\theta)-\mu(x)=\eta(x+\mu(x)).
\end{equation}

Since $\eta$ and $\mu$ are small, it is natural to drop higher order terms and instead consider the equation
\begin{equation}
\label{lineardenjoy}
\mu(x+\theta)-\mu(x)=\eta(x)
\end{equation}
which is now linear in the unknown function $\mu$.

Additionally, since the functions $\eta$ and $\mu$ are periodic by construction, equation \ref{lineardenjoy} can almost be formally solved up to a constant by simply equating nonzero Fourier coefficients on both sides. Note that the zero-th Fourier coefficient of the left-hand side of \eqref{lineardenjoy} vanishes, and so the resulting Fourier series will be formally equal to $\eta(x)-\hat{\eta}(0)$.

Explicitly, with
\begin{equation}
\label{mufourier}
\mu(x):=\sum_{n\neq 0} \hat{\mu}(n) e^{2\pi i n x}=\sum_{n\neq 0}(\int_0^1 e^{-2\pi i nx}\mu(x)\, dx)e^{2\pi i n x}
\end{equation}
we obtain
\begin{equation}
\label{smalldiv}
\hat{\mu}(n)=\frac{\hat{\eta}(n)}{e^{2\pi i n \theta}-1}
\end{equation}
for nonzero $n$.

For general irrational $\theta$, the denominators in the series \eqref{smalldiv} can be very small. This is a common difficulty in perturbation theory, and one of the major innovations in KAM theory was the usage of Diophantine sets like $\Omega_{\kappa,\rho}$ to control these denominators.

Indeed, we have
\begin{lem}
\label{diophbound}
If $\theta\in \Omega_{\kappa,\rho}$, we have
\begin{equation}
|e^{2\pi i n \theta}-1|\geq 4\kappa |n|^{-(\rho-1)}.
\end{equation}
\end{lem}
\begin{proof}

\end{proof}

An immediate consequence of Lemma \ref{diophbound} is that the formal Fourier series \eqref{mufourier} for $\mu$ converges.

Using the estimate on Fourier coefficients of an analytic function from Proposition \ref{fourierappendix}, we obtain
\begin{eqnarray}
\label{denjoymubound}
\|\mu\|_{\sigma-\delta}&=&\sup\left|\sum_{n\neq 0}\frac{\hat{\eta}(n)e^{2\pi i n z}}{e^{2\pi i n \theta}-1}\right|\\&\leq & \sum_{n\neq 0} \frac{|n|^{\rho-1}}{4\kappa}e^{-2\pi \sigma |n|}e^{2\pi|n|(\sigma-\delta)}\|\eta\|_{\sigma}\\
&\leq & \label{lastlinedenjoy}\frac{\Gamma(\rho)}{\kappa(2\pi \delta)^\rho}\|\eta\|_\sigma.
\end{eqnarray}
for $\delta$ sufficiently small dependent only on $\rho$, where \eqref{lastlinedenjoy} follows from estimating the sum in $n$ by  the integral defining the Gamma function.

This estimate allows us to show that $\chi$ is invertible on a suitable domain.

\begin{prop}
\label{denjoyinversedomain}
If $\max\left(\delta,\frac{\|\eta\|_{\sigma}}{\delta^{\rho+1}}\right)$ is sufficiently small (dependent only on the fixed quantities $\rho$ and $\kappa$), then $\chi$ has an analytic inverse on $\chi(S_{\sigma-2\delta})$.

Moreover, $\chi(S_{\sigma-2\delta})\supseteq S_{\sigma-3\delta}$.
\end{prop}
\begin{proof}
The analytic inverse function theorem implies that it suffices to show that that $\|\textrm{Id}-\chi'\|_{\sigma-2\delta}=\|\mu'\|_{\sigma-2\delta}<1$, which follows directly from the Cauchy estimate from Proposition \ref{cauchyappendix} and \eqref{denjoymubound}.

Moreover, since the assumptions of this proposition imply that $\|\mu\|_{\sigma-2\delta}<\delta$, we have $\chi(S_{\sigma-2\delta})\supseteq S_{\sigma-3\delta}$ by a straightforward degree theoretic argument.
\end{proof}

\begin{prop}
Under the conditions of Proposition \ref{denjoyinversedomain}, we have
\begin{equation}
\label{chiinvdef}
\chi^{-1}(z)=z-\mu(z)+\nu(z),
\end{equation}
with the estimate
\begin{equation}
\label{denjoynubound}
\|\nu\|_{\sigma-4\delta}\leq \frac{C(\rho,\kappa)}{\delta^{2\rho+1}}\|\eta\|_\sigma^2
\end{equation}
\end{prop}
\begin{proof}
From \eqref{chiinvdef} and Proposition \ref{denjoyinversedomain}, we have
\begin{eqnarray}
z&=&(\chi^{-1}\circ \chi)(z)\\
&=& z+\mu(z)-\mu(z+\mu(z))+\nu(z+\mu(z))
\end{eqnarray}
for all $z\in S_{\sigma-2\delta}$.

Hence
\begin{equation}
\nu(z)=\mu(\chi^{-1}(z))\left(\int_0^1 \mu'(\chi^{-1}(z)+t\mu(\chi^{-1}(z)))\, dt\right)
\end{equation}
for $z\in S_{\sigma-3\delta}$.

As in the second part of Proposition \ref{denjoyinversedomain}, we obtain $\chi(S_{\sigma-3\delta})\supset S_{\sigma-4\delta}$.

Since we have $\|\mu\|_{\sigma-2\delta}<\delta$, it follows that
\begin{equation}
\chi^{-1}(z)+t\mu(\chi^{-1}(z))\in S_{\sigma-2\delta}
\end{equation}
for all $t\in [0,1]$ and all $z\in S_{\sigma-4\delta}$.

Thus
\begin{eqnarray}
\|\nu\|_{\sigma-4\delta}&\leq& \|\mu\|_{\sigma-3\delta} \cdot \|\mu'\|_{\sigma-2\delta}\\
&\leq  &\delta^{-1}\left(\frac{\Gamma(\rho)\|\eta\|_\sigma}{\kappa(2\pi \delta)^\rho}\right)^2\\
& = & \frac{C(\rho,\kappa)}{\delta^{2\rho+1}}\|\eta\|_\sigma^2.
\end{eqnarray}
\end{proof}

The two estimates \eqref{denjoymubound} and \eqref{denjoynubound} allow us to bound the new error term $\tilde{\eta}$.

\begin{prop}
\label{arnoldkamstep}
Under the conditions of Proposition \ref{denjoyinversedomain}, we can write
\begin{equation}
\tilde{f}(z)=(\chi^{-1}\circ f \circ \chi)(z)=z+\theta+\tilde{\eta}(z)
\end{equation}
with the estimate
\begin{equation}
\|\tilde{\eta}\|_{\sigma-6\delta}\leq \frac{C(\rho,\kappa)}{\delta^{2\rho+1}}\|\eta\|_{\sigma}^2
\end{equation}
\end{prop}
\begin{proof}
From \eqref{chidefn}, \eqref{chiinvdef}, and \eqref{lineardenjoy}, we have
\begin{eqnarray}
& &(\chi^{-1}\circ f \circ \chi)(z)\\
& = & z+\theta+\mu(z)+\eta(z+\mu(z))-\mu(z+\theta+\mu(z)+\eta(z+\mu(z)))\\
&+& \nonumber\nu(z+\theta+\mu(z)+\eta(z+\mu(z))).
\end{eqnarray}
Hence
\begin{eqnarray}
\tilde{\eta}(z)&=&\tilde{f}(z)-z-\theta\\
&=& \label{RHSdenjoy}\hat{\eta}_0+(\eta(z+\mu(z))-\eta(z))\\
&+&\nonumber(\mu(z+\theta)-\mu(z+\theta+\mu(z)+\eta(z+\mu(z))))\\
&+& \nonumber\nu(z+\theta+\mu(z)+\eta(z+\mu(z))).
\end{eqnarray}

It remains to bound the right-hand side terms in \eqref{RHSdenjoy} individually.

First, from \eqref{denjoymubound} and  the conditions of Proposition \ref{denjoyinversedomain} we bound
\begin{eqnarray}
\label{RHSest1}
\eta(z+\mu(z))-\eta(z) &= &\mu(z)\int_0^1 \mu'(z+t\mu(z))\, dt\\
\Rightarrow \|\eta(z+\mu(z))-\eta(z)\|_{\sigma-4\delta}& \leq &\delta^{-1}\|\mu\|_{\sigma-4\delta}\|\eta\|_{\sigma-2\delta}\\
&\leq & \frac{C(\rho,\kappa)}{\delta^{\rho+1}}\|\eta\|_\sigma^2. 
\end{eqnarray}

Similarly, we can bound
\begin{eqnarray}
\label{RHSest2}
& &\mu(z+\theta)-\mu(z+\theta+\mu(z)+\eta(z+\mu(z)))\\
&=& -(\mu(z)+\eta(z+\mu(z)))\int_0^1 \mu'(z+\theta+t(\mu(z)+\eta(z+\mu(z))))\, dt \\
\Rightarrow & & \|\mu(z+\theta)-\mu(z+\theta+\mu(z)+\eta(z+\mu(z)))\|_{\sigma-4\delta}\\
& \leq & \|\mu'\|_{\sigma-2\delta}(\|\mu\|_{\sigma-4\delta}+\|\eta\|_{\sigma-3\delta})\\
& \leq & \frac{C(\rho,\kappa)}{\delta^{2\rho+1}}\|\eta\|_\sigma^2.
\end{eqnarray}

Since $\|\eta\|_\sigma<\delta$ can be assumed, we have 
\begin{equation}
z+\theta+\mu(z)+\eta(z+\mu(z))\in S_{\sigma-4\delta}
\end{equation}
whenever $z\in S_{\sigma-6\delta}$. This allows us to apply the estimate \eqref{denjoynubound} for $\nu$ directly.

The only remaining term in \eqref{RHSdenjoy} is the zero-th Fourier coefficient $\hat{\eta}(0)$. To bound this, we observe that Corollary \ref{etavanishes} implies the existence of an $x_0\in [0,1]$ with $\tilde{\eta}(x_0)=0$. So setting $z=x_0$ in \eqref{RHSdenjoy} and rearranging yields

\begin{equation}
\label{RHSest3}
|\hat{\eta}(0)|\leq \frac{C(\rho,\kappa)}{\delta^{\rho+1}}\|\eta\|_\sigma^2+\frac{C(\rho,\kappa)}{\delta^{2\rho+1}}\|\eta\|_\sigma^2+\frac{C(\rho,\kappa)}{\delta^{2\rho+1}}\|\eta\|_\sigma^2.
\end{equation}

Putting the estimates \eqref{RHSest1},\eqref{RHSest2},\eqref{denjoynubound}, and \eqref{RHSest3} into \eqref{RHSdenjoy} completes the proof.
\end{proof}

What we have completed at this stage is frequently referred to as the KAM step, and the remainder of the proof of Theorem \ref{arnolds} is an induction formed by iterating the KAM step with a carefully chosen decreasing sequence of constants $\sigma_n$ and $\delta_n$.

The sequence of lifts of circle diffeomorphisms $(f_n)$ is defined inductively by
\begin{equation}
f_0(x):=f(x),\quad f_{n+1}(x):=\chi_n^{-1}\circ f_n \circ \chi_n.
\end{equation}
We also define
\begin{equation}
\eta_0(x):=\eta(x),\quad \eta_{n+1}(x):=f_{n+1}(x)-x-\theta
\end{equation}
and
\begin{equation}
\label{mudefn}
\mu_n(x)=\chi_n(x)-x.
\end{equation}
By construction, we have that $\mu_n$ satisfies
\begin{equation}
\mu_n(x+\theta)-\mu_n(x)=\eta_n(x)-\hat{\eta}_{n+1}(0).
\end{equation}
The corresponding sequences of $\delta_n,\sigma_n$ are defined by
\begin{equation}
\delta_n:=\frac{\sigma}{36(1+n^2)}
\end{equation}
and
\begin{equation}
\sigma_0:=\sigma,\quad \sigma_{n+1}=\sigma_n-6\delta_n
\end{equation}
for all $n\geq 0$.

The key feature of this choice is that $\lim_{n\rightarrow\infty }\sigma_n>\sigma/2>0$, so the decreasing sequence of domains $S_{\sigma_n}$ have width positively bounded below.

To control the errors, we introduce the sequence
\begin{equation}
\epsilon_0:=\|\eta\|_{\sigma},\quad \epsilon_n:=\epsilon_0^{(3/2)^n}.
\end{equation}

We can now formulate the induction of the KAM step.
\begin{prop}
For sufficiently small $\epsilon_0$, dependent only on $\sigma,\rho,\kappa$, we have $\eta_{n+1}\in H_{\sigma_{n+1}}$ and $\|\eta_{n+1}\|_{\sigma_{n+1}}\leq \epsilon_{n+1}$.
Moreover, we have the estimates
\begin{equation}
\label{indstepmu}
\|\mu_n\|_{\sigma_n-\delta_n}\leq \frac{C(\rho,\kappa)\epsilon_n}{\delta_n^\rho}
\end{equation}
and
\begin{equation}
\label{indstepnu}
\|\nu_n\|_{\sigma_n-4\delta_n}\leq \frac{C(\rho,\kappa)\epsilon_n^2}{\delta_n^{2\rho+1}}
\end{equation}
\end{prop}
\begin{proof}
The estimates \eqref{indstepmu} and \eqref{indstepnu} follow from \eqref{denjoymubound} and \eqref{denjoynubound} respectively, provided that  that $\|\eta_n\|_{\sigma_{n}}\leq \epsilon_{n}$ and $\delta_n^{-\rho-1}\epsilon_n<c(\rho,\kappa)$  for all $n\geq 0$, which we verify inductively.

If $\|\eta_n\|_{\sigma_{n}}\leq \epsilon_{n}$, then from Proposition \ref{arnoldkamstep}, we have
\begin{eqnarray}
\|\eta_{n+1}\|_{\sigma_{n+1}}&\leq& \frac{C(\rho,\kappa)}{\delta_n^{2\rho+1}}\epsilon_n^2\\
&\leq & C(\sigma,\rho,\kappa)n^{4\rho+2}\epsilon_n^2.
\end{eqnarray}
Hence it suffices to choose $\epsilon_0$ small enough that 
\begin{equation}
\epsilon_0^{(3/2)^n}\leq \frac{C(\sigma,\rho,\kappa)}{n^{8\rho+4}}
\end{equation}
which is possible due to the rapid decay of the left-hand side.

Similarly, $\delta_n^{-\rho-1}\epsilon_n<c(\rho,\kappa)$ can be ensured by taking $\epsilon_0$ sufficiently small, dependent only on $\rho,\kappa$ and $\sigma$, due to the rapid decay of $\epsilon_n$. 
\end{proof}

We can now complete the proof of Theorem \ref{arnolds} by analysing the convergence of $\chi^n(x):=\chi_0\circ\chi_1\circ\ldots\circ \chi_n(x)$.

\begin{proof}[Proof of Theorem \ref{arnolds}]
We first note that the definitions 

From \eqref{mudefn}, it follows that
\begin{equation}
\chi^n(x)=x+\sum_{k=0}^n \mu_{n-k}(h_k(x))
\end{equation}
where the sequence of functions $h_k$ on $S_{\delta_n}$ is defined recursively by 
\begin{equation}
h_0(x):=x
\end{equation}
and
\begin{equation}
h_m(x)=x+\sum_{k=0}^{m-1} \mu_{n-k}(h_k(x)).
\end{equation}

This gives us the estimate
\begin{equation}
\label{chinearid}
\|\chi^n(z)-z\|_{\sigma_n-2\delta_n}\leq \sum_{n=0}^\infty \frac{C(\rho,\kappa)\epsilon_n}{\delta_n^\rho}=\tilde{C}<\infty
\end{equation}
with convergence following from the rapid decay of $\epsilon_n$.

To prove convergence of the $\chi^n(z)$ on the limiting strip $S_{\sigma_\infty}$, we write
\begin{equation}
\chi^{n+1}(z)-\chi^n(z)=\mu_n(z)\int_0^1(\chi^n)'(z+t\mu_n(z))\, dt.
\end{equation}

Using the Cauchy estimate from Proposition \ref{cauchyappendix} and \eqref{chinearid}, we obtain
\begin{eqnarray}
\|\chi^{n+1}(z)-\chi^n(z)\|_{\sigma_{n+1}}&\leq& \|\mu_n\|_{\sigma_n-2\delta_n}\cdot(1+\|(\chi^n)'(z)-1\|_{\sigma_n-5\delta_n})\\
&\leq & \|\mu_n\|_{\sigma_n-2\delta_n}\cdot(1+\frac{3}{\delta_n}\|\chi^n(z)-z\|_{\sigma_n-2\delta_n})\\
&\leq & \left(1+\frac{3\tilde{C}}{\delta_n}\right)\cdot \frac{C(\rho,\kappa)\epsilon_n}{\delta_n^\rho}
\end{eqnarray}
which is summable, again due to the rapid decay of $\epsilon_n$.

This implies that the $\chi^n(z)$ converge uniformly to an analytic limit function $\chi(z)$ on $S_{\sigma_\infty}$.

From the Cauchy estimate from Proposition \ref{cauchyappendix} and \eqref{chinearid}, we have
\begin{equation}
\|\chi'(z)-1\|_{(1-r)\sigma_\infty}\leq \frac{\tilde{C}}{r\sigma_\infty}
\end{equation}
for arbitrary $r\in(0,1/2)$.

Having fixed $r$, we can choose $\epsilon_0$ sufficiently small so that $\tilde{C}<r\sigma_\infty$ from the rapid decay of $\epsilon_n$, and so as in Proposition \ref{denjoyinversedomain}, we obtain that $\chi$ is invertible on $S_{(1-r)\sigma_\infty}$ with $\chi(S_{(1-r)\sigma_\infty})\supseteq S_{(1-2r)\sigma_\infty}$.

Now, we have
\begin{eqnarray}
(f\circ \chi)(z)&=&\lim_{n\rightarrow\infty} (f\circ \chi^n)(z)\\
&= & \lim_{n\rightarrow\infty} \chi^n(z+\theta+\eta_n(z))\\
&=& \chi(z+\theta)
\end{eqnarray}
for all $z\in S_{(1-2r)\sigma_\infty}$. Since $\chi$ is invertible on this domain, we can apply $\chi^{-1}$ to both sides, demonstrating that $\chi$ is the sought diffeomorphism and we are done.
\end{proof}

In the next section, we return to the problem of dynamics in a small perturbation of a completely integrable Hamiltonian system. The largest technical difference from Arnold's theorem is the involvement of functions that are only of Gevrey regularity and are not necessarily analytic. This necessitates the approximation of this sequence of Gevrey functions by a sequence of analytic functions in order to exploit the Cauchy estimates that are essential to the method of proof. 

We begin by proving the KAM step.

\section{The KAM step}
We now prepare for the proof of Theorem $\ref{kamwithparams}$ by first proving the result that will comprise the steps of the iterative argument.
\label{kamstepsec}
Given a Hamiltonian in the form
\begin{equation}
H(\theta,I;\omega,t)=e(\omega;t)+\langle \omega,I \rangle +P(\theta,I;\omega,t)=N(I;\omega,t)+P(\theta,I;\omega,t),
\end{equation}
we aim to construct a $t$-dependent symplectic map $\Phi$ and a $t$-dependent frequency transformation $\phi$ such that for $\mathcal{F}=(\Phi,\phi)$, we have 
\begin{equation}
(H\circ \mathcal{F})(\theta,I;\omega,t)=N_+(I;\omega,t)+P_+(\theta,I;\omega,t)
\end{equation}
where $N_+(I,\omega,t) = e_+(\omega)+\langle I,\omega\rangle$ and with $|P_+|$ controlled by $|P|^r$ for some $r>1$.\\

This construction is analogous to that in \cite{poschel}, although our application requires working with families of Hamiltonians that are also real analytic in the additional parameter $t$.

\begin{thm}
\label{kamstep}
Suppose $\epsilon,h,v,s,r,\eta,\sigma,K$ are positive constants such that
\begin{equation}
\label{kamstepconsts}
s,r<1,\;v<1/6,\;\eta<1/8,\;\sigma<s/5,\;\epsilon\leq c\kappa \eta r \sigma^{\tau+1},\;\epsilon\leq cvhr,\;h\leq \kappa/2K^{\tau+1}.
\end{equation}
where $c$ is a constant dependent only on $n$ and $\tau$.

\noindent Suppose $H(\theta,I;\omega,t)=N(I;\omega,t)+P(\theta,I;\omega,t)$ is real analytic on $D_{s,r}\times O_h \times (-1,1)$, and $|P|_{s,r,h}\leq \epsilon$.
Here, $D_{s,r}$ is as in Definition \ref{domains1} and 
\begin{equation}
O_h:=\{\omega\in \C^n:\textrm{dist}(\omega,\Omega_\kappa)<h\}.
\end{equation}
Then there exists a real analytic map \begin{equation}\mathcal{F}=(\Phi,\phi):D_{s-5\sigma,\eta r}\times O_{(1/2-3v)h}\times (-1,1)\rightarrow D_{s,r}\times O_h \end{equation} where the maps
\begin{equation}
\Phi:D_{s-5\sigma,\eta r}\times O_h\times (-1,1) \rightarrow D_{s,r}
\end{equation}
and
\begin{equation}
\phi: O_{(1/2-3v)h}\times (-1,1)\rightarrow O_h
\end{equation}
are such that 
\begin{equation}
H\circ\mathcal{F}=e_+(\omega,t)+\langle \omega,I\rangle+P_+(\theta,I;\omega,t)=N_+(I;\omega,t)+P_+(\theta,I;\omega,t)
\end{equation}
and we have the new remainder estimate
\begin{equation}
\label{newerrbd}
|P_+|_{s-5\sigma,\eta r,(1/2-2v)h}\leq C\left(\frac{\epsilon^2}{\kappa r \sigma^{\tau+1}}+(\eta^2+K^ne^{-K\sigma})\epsilon\right).
\end{equation}
Moreover $\Phi$ is symplectic for each $(\omega,t)$ and has second component affine in $I$. 
Finally, we have the following uniform estimates on the change of variables.
\begin{equation}
|W(\Phi-id)|,|W(D\Phi-Id)W^{-1}|\leq \frac{C\epsilon}{\kappa r \sigma^{\tau+1}}
\end{equation}
\begin{equation}
|\phi-id|,vh|D\phi-Id|\leq \frac{C\epsilon}{r}
\end{equation}
where $W=\textrm{diag}(\sigma^{-1}Id,r^{-1}Id)$.
All estimates are uniform and analytic in the parameter $t\in (-1,1)$.
\end{thm}
\begin{proof}
We first linearise $P$ about $I=0$ and truncate its Fourier series to order $K$.
If we define
\begin{equation}
Q=P(\theta,0;\omega,t)+I\cdot\nabla_I P(\theta,0;\omega,t),
\end{equation}
then the Cauchy estimate from Proposition \ref{cauchyappendix} for analytic functions together with Taylor's theorem yields
\begin{equation}
|Q|_{s,r}\leq C\epsilon
\end{equation}
and
\begin{equation}
\label{PappQ}
|P-Q|_{s,2\eta r}\leq C\eta^2 \epsilon.
\end{equation}
Defining
\begin{equation}
R(\theta,I;\omega,t):=\sum_{|k|\leq K} \langle Q(\cdot,I;\omega,t),e^{i\langle k,\cdot\rangle} \rangle e^{i\langle k,\theta\rangle},
\end{equation}
the Fourier truncation result, Proposition \ref{truncateappendix} yields 
\begin{equation}
\label{RappQ}
|R-Q|_{s-\sigma,r}\leq CK^{-n}e^{-K\sigma}\epsilon.
\end{equation}
When we apply this KAM step, we can assume $K$ is sufficiently large (dependent on $\sigma$), so that we will in fact have
\begin{equation}
\label{Rbound}
|R|_{s-\sigma,r}\leq C\epsilon.
\end{equation}
All estimates thus far are uniform in $(\omega,t)\in O_h\times (-1,1)$, and the function $R(\theta,I;\omega,t)$ is still analytic in all variables.

Recalling the assumption $h\leq \kappa/(2K^{\tau+1})$, we next extend the nonresonance estimate \eqref{diophantine} from the Cantor set $\Omega_\kappa$ to the estimate
\begin{equation}
\label{diophlocal}
|\langle k,\omega\rangle|\geq \frac{\kappa}{2|k|^\tau}
\end{equation}
for $\omega\in O_h$ and $|k|\leq K$. This of course motivates our Fourier truncation of $Q$.

We now make the ansatz that our sought transformation $\Phi$ can be obtained as the time-$1$ Hamiltonian flow generated by an undetermined function $F$ dependent on the parameters $\omega,t$. From Proposition \ref{flowissymp}, such transformations are known to be symplectic, and so it suffices to show that we can choose $F$ in a way that yields a transformation bringing $H$ to normal form up to an error term that is quantifiably smaller than $\epsilon$.

\noindent To this end, we analyse the expression 
\begin{equation}
(N(I;\omega,t)+R(\theta,I;\omega,t))\circ \Phi^\tau_F.
\end{equation}
Recalling \eqref{poissonderiv}, Taylor's theorem yields
\begin{eqnarray}
\label{poisint}
\nonumber & &(N(I;\omega,t)+R(\theta,I;\omega,t))\circ \Phi^1_F\\
\nonumber &=& N+\{N,F\}+\int_0^1 \{(1-\tau)\{N,F\},F\}\circ \Phi^\tau_F \, d\tau+R+\int_0^1 \{R,F\}\circ \Phi_F^\tau \, d\tau\\
&=& N+\{N,F\}+R+\int_0^1 \{(1-\tau)\{N,F\}+R,F\}\circ \Phi^\tau_F \, d\tau.
\end{eqnarray}
The integral in \eqref{poisint} is of second order in $F$ and $R$, and is thus an ideal candidate to absorb into remainder term $P^+$, which will also contain the error from approximating $P$ by $R$. In other words, we need to choose $F$ in such a way that $N+\{N,F\}+R$ is in linear normal form.
This amounts to solving
\begin{equation}
\label{equatefourier}
\{F,N\}+(N_+-N)=R
\end{equation}
for $F$ and then finding a suitable frequency transformation so that $N_+=e_+(\omega;t)+\langle\omega,I\rangle$.

Taking the Fourier expansion
\begin{equation}
F=\sum_{k\in \mathbb{Z}^n} F_ke^{i\langle k,\theta\rangle},
\end{equation}
we formally obtain
\begin{equation}
\{F,N\}=\sum_{j=1}^n \omega_j \frac{\partial F}{\partial \theta_j}=\sum_{k\in \mathbb{Z}^n} i\langle k,\omega\rangle F_ke^{i\langle k,\theta\rangle}.
\end{equation}
Since $\omega$ is a nonresonant frequency, the factors $\langle k,\omega\rangle$ are nonzero for nonzero $k$ which allows us to choose 
\begin{equation}
\label{Ffourier}
F_k=\frac{R_k}{i\langle k,\omega\rangle}
\end{equation}
and define $F$ by the resulting Fourier series. This sum is finite, thanks to the earlier truncation of $R$.

To equate the zero-th Fourier coefficients in \eqref{equatefourier}, we take $N_+:=N+\hat{R}(0)$.
\begin{remark}
The small denominators $\langle k,\omega\rangle$ occurring in \eqref{Ffourier} are a frequent and problematic feature in the formal series that arise in  perturbation theory. The use of nonresonance conditions such as \eqref{diophlocal} to control such denominators in the KAM theorem was a key advance for the field.
\end{remark}
Using \eqref{Ffourier} and the estimate \eqref{diophlocal}, we obtain the bound
\begin{eqnarray}
|F|_{r,s-2\sigma}&\leq & \sum_{|k|\leq K} \frac{|R_k|_r|e^{i\langle k,\theta\rangle}|}{|\langle k,\omega\rangle|}\\
&\leq & \frac{C|R|_{s-\sigma}}{\kappa}\cdot\sum_{|k|\leq K} |k|^\tau e^{-|k|(s-\sigma)}e^{|k|(s-2\sigma)}\\
&= & \frac{C|R|_{s-\sigma}}{\kappa}\\
&\leq &\frac{\epsilon}{\kappa \sigma^\tau},
\end{eqnarray}
where the last line follows from \eqref{Rbound}.
Furthermore, $F$ is analytic in all variables so we may use the Cauchy estimate again to obtain
\begin{equation}
\label{Ftheta}
|\partial_\theta F|_{s-3\sigma,r} \leq \frac{C\epsilon}{\kappa \sigma^{\tau+1}}
\end{equation}
and
\begin{equation}
\label{FI}
|\partial_I F|_{s-2\sigma,r/2} \leq \frac{C\epsilon}{\kappa r \sigma^\tau}.
\end{equation}
We can combine these estimates as 
\begin{equation}
\max\{r^{-1}|\partial_\theta F|, \sigma^{-1}|\partial_I F|\} \leq \frac{C\epsilon}{\kappa r \sigma^{\tau+1}}
\end{equation}
uniformly on $D_{s-3\sigma,r/2}\times O_h \times (-1,1)$.
At this point we can also estimate
\begin{equation}
\label{Ravgbound}
N_+-N=\hat{R}(0)\leq |R|_{s-\sigma,r}\leq C\epsilon.
\end{equation}

Now the bounds on the $(\theta,I)$-derivatives of $F$ control the Hamiltonian flow $\Phi=(U,V)$.
Indeed, the estimates \eqref{kamstepconsts},\eqref{Ftheta} and \eqref{FI} imply that 
\begin{equation}
|\partial_\theta F|\leq \eta r \leq r/8, \;|\partial_I F|\leq \sigma 
\end{equation}
and consequently that the time-$1$ flow is well-defined as a map 
\begin{equation}
\label{Phimapping}
D_{s-4\sigma,r/4}\rightarrow D_{s-3\sigma,r/2},
\end{equation}
with the component bounds
\begin{equation}
|U-id|\leq |\partial_\theta F|\leq \frac{C\epsilon}{\kappa \sigma^{\tau+1}},\; |V-id|\leq |\partial_I F|\leq \frac{C\epsilon}{\kappa r\sigma^{\tau}}.
\end{equation}
By construction, we have that $F$ is affine linear in $I$. Consequently, $\partial_I F$ and $V$ are both $I$-independent, and $\partial_\theta F$ and $U$ are affine linear in $I$.
To complete the estimates of $\Phi$, we use the Cauchy estimate again to yield
\begin{equation}
|\partial_I U-Id|\leq \frac{C\epsilon}{\kappa r \sigma^{\tau+1}},\; |\partial_\theta U|\leq \frac{C\epsilon}{\kappa \sigma^{\tau+2}},\; |\partial_\theta V-Id|\leq \frac{C\epsilon}{\kappa r \sigma^{\tau+1}}
\end{equation}
uniformly on $D_{s-5\sigma,r/8}\supseteq D_{s-5\sigma,\eta r}$.

It remains to bound the new error term $P_+$ given by \eqref{poisint} and to construct the frequency map $\phi$ that transforms $N_+$ to normal form. First we treat the error term $P_+$. By using the Cauchy estimate, we can bound
\begin{equation}
\{R,F\}\leq |\partial_I R||\partial_\theta F|+|\partial_\theta R||\partial_I F|\leq \frac{C\epsilon^2}{\kappa r \sigma^{\tau+1}}
\end{equation}
uniformly on $D_{s-3\sigma,r/2}$.

In exactly the same way, recalling that $N_+-N=\hat{R}(0)$, we use \eqref{Ravgbound} to estimate
\begin{equation}
\{N_+-N,F\}\leq \frac{C\epsilon^2}{\kappa r \sigma^{\tau+1}}.
\end{equation}

Together with the mapping property \eqref{Phimapping}, the discussion of $P_+$ following \eqref{poisint}, and the bounds \eqref{RappQ} and \eqref{PappQ} we obtain
\begin{eqnarray}
& &\left|\int_0^1 \{(1-\tau)(R+\{N,F\})+\tau R,F\} \circ \Phi_F^\tau \, d\tau\right|_{s-5\sigma,\eta r}\\
& \leq & |\{(1-\tau)(R+\{N,F\})+\tau R,F\}|_{s-4\sigma,r/2} \\
&\leq & \frac{C\epsilon^2}{\kappa r \sigma^{\tau+1}},
\end{eqnarray}
and
\begin{equation}
|(P-R)\circ \Phi|_{s-5\sigma,\eta r}\leq |P-R|_{s-4\sigma,2\eta r}\leq C(\eta^2+K^ne^{-K\sigma})\epsilon
\end{equation}
which proves \eqref{newerrbd}.

Finally, we have 
\begin{equation}
\label{Nplus}
N_+=N+\hat{R}(0)=e(\omega,t)+\langle \omega,I\rangle+\hat{R}(0,0;\omega,t)
\end{equation}
which we need to re-write in normal form $e_+(\omega)+\langle \omega,I\rangle$ by finding a suitable frequency transformation $\phi$. Noting that $R$ is linear in $I$ by construction, this amounts to inverting the map ${\omega\mapsto\omega+q(\omega;t)}$.
where
\begin{equation}
q(\omega;t)=\hat{R_I}(0,0;\omega,t)
\end{equation}
which is bounded by $C\epsilon/r\leq vh$ from the Cauchy estimate, \eqref{Ravgbound}, and our assumption \eqref{kamstepconsts}.

An application of a version of the implicit function theorem, Proposition \ref{popovlemma}, then constructs $\phi:O_{(1/2-3v)h}\times (-1,1)\rightarrow O_{(1-4v)h}$ inverse to $q$ which satisfies the claimed estimates. 
\end{proof}
As in \cite{poschel},\cite{popov1}, Theorem \ref{kamstep} can be used to prove the KAM theorem for real analytic Hamiltonians $H(\theta,I;\omega,t)$. However, in order to treat the more general class of Gevrey smooth Hamiltonians $H\in G^{\rho,1}((\mathbb{T}^n\times D\times \Omega)\times (-1,1))$, we require the approximation result Proposition \ref{approxlemma}. This method was used to prove Theorem 2.1 in \cite{popovkam} without the presence of the parameter $t$.

\section{Approximations of Gevrey functions}
\label{approxsec}
It is convenient to extend the $P^j$ to Gevrey functions $\tilde{P}^j\in G^{\rho,\rho,1}_{CL_1,CL_2,CL_2}(\mathbb{T}^n\times \mathbb{R}^{2n}\times (-1,1))$ where $C$ depends only on $n$ and $\rho$ by making use of a Gevrey formulation of the Whitney extension theorem, from Theorem \ref{whitneythm}.
We thus obtain the estimate
\begin{equation}
\label{whitneyapplication}
\|\tilde{P}^j\|\leq AL_1^{n+1}\|P^j\|
\end{equation}
where $A$ also only depends on $n$ and $\rho$.
We can then cut-off $\tilde{P}^j$ without loss to have $(I,\omega)$ support in $B_1\times B_{\bar{R}}\subset \R^{2n}$, where $1\ll\bar{R}$ is such that $\Omega^0\subset B_{\bar{R}-1}$. 
From here, we suppress the tilde in our notation, as well as the factor $C$ in our Gevrey constant.
\begin{prop}
\label{approxlemma}
Suppose $P\in G^{\rho,\rho,1}_{L_1,L_2,L_2}(\mathbb{T}^n\times \R^{2n}\times (-1,1))$ satisfies ${\textrm{supp}_{(I,\omega)}(P)\subset B_1\times B_{\bar{R}}}$.
If $u_j,w_j,v_j$ are positive real sequences monotonically tending to zero such that
\begin{equation}
\label{uvw}
v_jL_2,w_jL_2\leq u_jL_1 \leq 1, \; v_0,w_0\leq L_2^{-1-\zeta} 
\end{equation}
where $1\leq L_1 \leq L_2$ and $0<\zeta \leq 1$ are fixed, then we can find a sequence of real analytic functions $P_j:U_j\rightarrow \mathbb{C}$ such that 
\begin{equation}
|P_{j+1}-P_j|_{U_{j+1}}\leq C(\bar{R}^n+1)L_1^n \exp\left(-\frac{3}{4}(\rho-1)(2L_1 u_j)^{-1/(\rho-1)}\right)\|P\|,
\end{equation}
\begin{equation}
|P_0|_{U_0}\leq C(\bar{R}^n+1)\left(1+L_1^n \exp\left(-\frac{3}{4}(\rho-1)(2L_1 u_0)^{-1/(\rho-1)}\right)\right),
\end{equation}
and
\begin{equation}
|\partial_x^\alpha (P-P_j)(\theta,I;\omega,t)|\leq C(1+\bar{R}^n)L_1^nL_2 \exp\left(-\frac{3}{4}(\rho-1)(2L_1 u_j)^{-1/(\rho-1)}\right)
\end{equation}
in $\mathbb{T}^n\times B_1 \times B_{\bar{R}}\times (-1,1)$ for $|\alpha|\leq 1$,
where 
\begin{multline}
U_j^m:=\{(\theta,I;\omega,t)\in \C^n / 2\pi \mathbb{Z}^n \times \C^n \times \C^n\times \C:\\
|\textrm{Re}(\theta)|\leq\pi,|\textrm{Re}(I)|\leq 2,|\textrm{Re}(\omega)|\leq \bar{R}+1,|\textrm{Re}(t)|\leq 1,\\|\textrm{Im}(\theta)|\leq 2u_j,|\textrm{Im}(I)|\leq 2v_j,|\textrm{Im}(\omega_k)|\leq 2w_j,|\textrm{Im}(t)|\leq (2L_2)^{-1} \}
\end{multline}
and 
\begin{equation}
U_j:= U_j^1
\end{equation}
where we have identified $[-\pi,\pi]^n$ with $\T^n$ for simplicity of notation.
\end{prop}
\begin{proof}
We first extend $P$ to functions $F_j:U_j^2\rightarrow \mathbb{C}$ that are almost analytic in $(\theta,I,\omega)$ and are analytic in $t$.
The Gevrey estimate on $t$-derivatives of $P$ imply that the Taylor expansions in $t$ have radius of convergence $L_2^{-1}$, and so the expression
\begin{equation}
\label{effjay}
F_j(\theta+i\tilde{\theta},I+i\tilde{I},\omega+i\tilde{\omega},t+i\tilde{t}):=\sum_{\mathcal{M}_j} \partial_\theta^\alpha \partial_I^\beta \partial_\omega^\gamma P(\theta,I;\omega,t)\frac{(i\tilde{\theta})^\alpha (i\tilde{I})^\beta(i\tilde{\omega})^\gamma(i\tilde{t})^\delta}{\alpha!\beta!\gamma!\delta!}
\end{equation}
is convergent on $U_j^2$ where the index set is given by 
\begin{equation}
\mathcal{M}_j=\{(\alpha,\beta,\gamma,\delta):\alpha_k\leq N_1,\; \beta_k\leq N_2,\; \gamma_k\leq N_3 \}
\end{equation}
where
\begin{equation}
N_1=\lfloor (2L_1u_j)^{-1/(\rho-1)}\rfloor+1 ,\; N_2= (2L_2v_2)^{-1/(\rho-1)}\rfloor+1,\; N_3=(2L_2w_j)^{-1/(\rho-1)}\rfloor+1.
\end{equation}
A simple consequence of Stirling's formula is that for $s\in (0,1]$ and $m\in \mathbb{Z}\cap [1,t^{-1/(\rho-1)}+1]$, we have the estimate
\begin{equation}
\label{stirling}
s^m m!^{\rho-1}\leq C(\rho)m^{\frac{\rho-1}{2}}e^{-m(\rho-1)}
\end{equation}
which will be used repeatedly in estimating the almost analyticity of $F_j$.
By using the Gevrey estimates of $P$ to control each summand on $F_j$, we obtain the bound
\begin{equation}
\label{summandsofalmostanalyticextension}
|F_j|_{U_j^2} \leq \|P\|\cdot \sum_{\mathcal{M}_j} (2)^{-\delta}(2L_1u_j)^{|\alpha|}(2L_2v_j)^{|\beta|}(2L_2 \omega_j)^{|\gamma|}(2L_2 t)^{\delta}(\alpha!\beta!\gamma!)^{\rho-1}.
\end{equation}

Now, for $k$ with $\alpha_k >0$, we can apply \eqref{stirling} to obtain
\begin{equation}
(2L_1u_j)^{\alpha_k} \alpha_k!^{\rho-1}\leq C\alpha_k^{\frac{\rho-1}{2}}e^{-\alpha_k(\rho-1)}.
\end{equation}
Bounding the factors $(2L_2v_j)^{\beta_k} \beta_k!^{\rho-1}$ and $(2L_2w_j)^{\gamma_k} \gamma_k!^{\rho-1}$ in the same way, we arrive at the bound
\begin{equation}
|F_j|_{U_j^2}\leq 2\|P\|\cdot\left(1+C(\rho)\sum_{m=1}^\infty m^{\frac{1-\rho}{2}}e^{-m(\rho-1)}\right)^{3n}=C\|P\|.  
\end{equation}

We now consider the Cauchy--Riemann operator $\bar{\partial}_{z_k}=\frac{1}{2}(\partial_{\theta_k}+\partial_{i\tilde{\theta}_k})$ where $z_k=\theta_k+i\tilde{\theta}_k$.
Applying $\bar{\partial}_{z_k}$ to $F_j$, we obtain
\begin{equation}
\label{thetaalmostanalytic}
2\bar{\partial}_{z_k}=\sum_{\mathcal{M}_j,\alpha_k=N_1} \partial_\theta^{\alpha+e_k} \partial_I^\beta \partial_\omega^\gamma P(\theta,I;\omega,t)\frac{(i\tilde{\theta})^\alpha (i\tilde{I})^\beta(i\tilde{\omega})^\gamma(i\tilde{t})^\delta}{\alpha!\beta!\gamma!\delta!}.
\end{equation}
From the Gevrey estimates for $P$, we can estimate each summand by
\begin{equation}
|F_j|_{U_j^2} \leq \|P\| L_1(2L_1u_j)^{|\alpha|}(2L_2v_j)^{|\beta|}(2L_2 \omega_j)^{|\gamma|}(2L_2 t)^{\delta}(\alpha!\beta!\gamma!)^{\rho-1}(\alpha_k+1)^\rho.
\end{equation}
Using the fact that $(2L_1 u_j)^{-1/(\rho-1)}\leq N_1\leq (2L_1 u_j)^{-1/(\rho-1)}+1$ together with \eqref{stirling}, we can bound
\begin{equation}
(2L_1)^{\alpha_k}\alpha_k!^{\rho-1}(\alpha_k+1)^\rho\leq C(L_1u_j)^{-\frac{3\rho-1}{2\rho-2}}\exp(-(\rho-1)(2L_1u_j)^{-1/(\rho-1)}).
\end{equation}
Hence
\begin{eqnarray}
\label{cauchyriemannbound}
|\bar{\partial}_{z_k} F_j|_{U_j^2}&\leq& CL_1(L_1u_j)^{-\frac{3\rho-1}{2\rho-2}}\exp(-(\rho-1)(2L_1u_j)^{-1/(\rho-1)})\|P\|\\ 
&\leq & CL_1\exp\left(-\frac{3}{4}(\rho-1)(2L_1u_j)^{-1/(\rho-1)}\right)\|P\|.
\end{eqnarray}
where the last line follows from absorbing the power term into the exponential. The constant $C$ depends only on $n$ and $\rho$.

Applying the same method of estimation to $F_j$ with an additional differentiation yields
\begin{eqnarray}
\label{betagammader}
& &\max(|\partial_{\theta_l}^\beta \partial_{\tilde{\theta}_l}^\gamma \bar{\partial}_{z_k}F_j|,|\partial_{I_l}^\beta \partial_{\tilde{I}_l}^\gamma \bar{\partial}_{z_k}F_j|,|\partial_{\omega_l}^\beta \partial_{\tilde{\omega}_l}^\gamma \bar{\partial}_{z_k}F_j|,|\partial_{t}^\beta \partial_{\tilde{t}}^\gamma \bar{\partial}_{z_k}F_j|)\\ &\leq& CL_1L_2\exp\left(-\frac{3}{4}(\rho-1)(2L_1u_j)^{-1/(\rho-1)}\right)\|P\|.
\end{eqnarray}
if $\beta+\gamma\leq 1$, where we have also used $1\leq L_1\leq L_2.$

We can also bound $\bar{\partial}_{I_k+i\tilde{I}_k}F_j$ and $\bar{\partial}_{\omega_k+i\tilde{\omega}_k}F_j$ in the same fashion as \eqref{cauchyriemannbound}. Indeed, since $L_2\leq \min((L_2v_j)^{-1/\zeta},(L_2 w_j)^{-1/\zeta})$, we obtain the stronger estimate
\begin{eqnarray}
|\bar{\partial}_{z_k}F_j|&\leq& CL_2(L_2u_j)^{-\frac{3\rho-1}{2\rho-2}}\exp(-(\rho-1)(2L_2u_j)^{-1/(\rho-1)})\|P\|\\
&\leq & C(L_2u_j)^{-\frac{3\rho-1}{2\rho-2}-\frac{1}{\zeta}}\exp(-(\rho-1)(2L_2u_j)^{-1/(\rho-1)})\|P\|\\
&\leq & C\exp\left(-\frac{3}{4}(\rho-1)(2L_2u_j)^{-1/(\rho-1)}\right)\|P\|\\
&\leq & C\exp\left(-\frac{3}{4}(\rho-1)(2L_1u_j)^{-1/(\rho-1)}\right)\|P\|
\end{eqnarray}
where $C$ now depends on $n,\rho,\zeta$ only.

Generalising, we now let $z=(\theta+i\tilde{\theta},I+i\tilde{I},\omega+i\tilde{\omega},t+i\tilde{t})\in \C^n\times \mathbb{C}^{2n}\times \C$. Then if $\alpha_k\leq 1$ for each $k$, applying the operator $\bar{\partial}_{z}^{\alpha}$ to $F_j$ amounts to restricting the index set $M_j$ to the multi-indices with $k$-th component maximal, for each $k$ with $\alpha_k=1$.
Hence we obtain
\begin{equation}
|\bar{\partial}_{z}^\alpha F_j|_{U_j^2} \leq CL_1^n\exp\left(-\frac{3}{4}(\rho-1)(2L_1u_j)^{-1/(\rho-1)}\right)\|P\|.
\end{equation}
We also obtain the derivative bound 
\begin{equation}
\label{almostanalyticend}
|\partial_x^\beta\partial_y^\gamma\bar{\partial}_{z}^\alpha F_j|_{U_j^2} \leq CL_1^nL_2^{|\beta+\gamma|}\exp\left(-\frac{3}{4}(\rho-1)(2L_1u_j)^{-1/(\rho-1)}\right)\|P\|.
\end{equation}
for $|\alpha|\geq 1, |\beta+\gamma|\leq 1$ as in \eqref{betagammader}.
Of course, if $\alpha_{3n+1}=1$, the operator $\bar{\partial}_{z}^\alpha$ will annihilate $F_j$ because of the analyticity in $t$.\\ 

Having constructed a family of almost-analytic extensions $F_j$ of $P$, the next step is to approximate the $F_j$ by functions that are real analytic in $U_j^2$. The key tool here is Green's formula.
\begin{equation}
\label{greens}
\frac{1}{2\pi i}\int_{\partial D} \frac{f(\eta)}{\eta-\zeta}\, d\eta+\frac{1}{2\pi i}\int\int_D \frac{\bar{\partial}f(\eta)}{\eta-\zeta}\, d\eta\wedge d\bar{\eta}=f(\zeta)
\end{equation}
if $D$ is a bounded domain symmetric with respect to the real axis with a piecewise smooth positively oriented boundary $\partial D$, $f\in \mathcal{C}^1(\overline{D})$, and $\zeta\in D$.
We also observe that the first summand in \eqref{greens} is real analytic in D if $f(\eta)=\overline{f(\bar{\eta})}$ on $\partial D$, a property satisfied by our $F_j$ by the nature of their construction.
We define open rectangles in $\mathbb{C}$ given by
\begin{equation}
D_k=\{z: |\textrm{Re}(z)|< a_k, \; |\textrm{Im}(x)|<b_k\}
\end{equation}
where 
$$a_k = 
\begin{cases} 
\pi &\mbox{if } 1\leq k \leq n \\ 
2 & \mbox{if } n+1\leq k \leq 2n\\
\bar{R}+1 &\mbox{if } 2n+1 \leq k \leq 3n
 \end{cases}$$
and
$$b_k = 
\begin{cases} 
2u_j &\mbox{if } 1\leq k \leq n \\ 
2v_j & \mbox{if } n+1\leq k \leq 2n\\
2w_j &\mbox{if } 2n+1 \leq k \leq 3n.
 \end{cases}$$
Additionally, we define the oriented union of line segments 
\begin{equation}
\Gamma=[-\pi-2iu_j,\pi-2iu_j]\cup[\pi+2iu_j,-\pi+2iu_j].
\end{equation}
We also introduce the $2\pi$-periodic meromorphic function
\begin{equation}
K(\eta,\zeta)=\lim_{N\rightarrow\infty}\sum_{|k|\leq N} \frac{1}{\eta-\zeta+2\pi k}=\frac{1}{\eta-\zeta}+K_1(\eta,\zeta).
\end{equation}

Writing $F_{j,0}(z)=F_j(z)$, we define
\begin{equation}
F_{j,1}(z):=\frac{1}{2\pi i}\int_{\Gamma} F_{j,0}(\eta_1,z_2,\ldots,z_{3n+1})K(\eta_1,z_1)\, d\eta_1.
\end{equation}
This function is analytic and $2\pi$-periodic for $z_1$ in the strip $|\textrm{Im}(z_1)|<2u_j$. Moreover, it satisfies the identity $F_{j,1}(z)=\overline{ F_{j,1}(\bar{z})}$.
If $z_1\in D_1$, we can safely avoid poles by integrating about $\partial D_1$. Thus, periodicity gives us
\begin{equation}
F_{j,1}(z)=\frac{1}{2\pi i}\int_{\partial D_1} F_{j,0}(\eta_1,z_2,\ldots,z_{3n+1})K(\eta_1,z_1)\, d\eta_1.
\end{equation}
Green's formula \eqref{greens} then implies
\begin{equation}
F_{j,1}(z)=F_{j,0}(z)-\frac{1}{2\pi i}\int_{D_1} \bar{\partial}_{\eta_1}F_{j,0}(\eta_1,z_2,\ldots,z_{3n+1})K(\eta_1,z_1)\, d\eta_1\wedge d\bar{\eta}_1.
\end{equation}
which extends to $\textrm{Re}(z_1)=\pi$ by continuity.

Defining the set $U_{j,1}^2=U_j^2\cap \{|\textrm{Im}(z_1)|\leq u_j\}$, we claim that for any multi-indices $\alpha,\beta,\gamma$ with $\alpha_1=0$ and $|\beta|+|\gamma|\leq 1$, we have the estimate
\begin{equation}
\label{greensclaim}
|\partial_{z_k}^\beta \bar{\partial}_{z_k}^\gamma \bar{\partial}_z^\alpha(F_{j,1}-F_{j,0})|_{U_{j,1}^2}\leq C L_1^n L_2^{|\beta|+|\gamma|}\exp(-\frac{3}{4}(\rho-1)(2L_1 u_j)^{-1/(\rho-1)})\|P\|.
\end{equation}
For $k\neq 1$, this estimate follows directly from \eqref{almostanalyticend} by differentiating $F_{j,1}$ under the integral.
For $k=1$, we first use $\eqref{greens}$ to write 
\begin{eqnarray}
& &\frac{1}{2\pi i}\int_{D_1} \frac{\bar{\partial}_{\eta_1}F_{j,0}(\eta_1,z_2,\ldots,z_{3n+1})}{\eta_1-z_1}\, d\eta_1\wedge d\bar{\eta}_1\\
&=&\nonumber-\bar{z}_1\bar{\partial}_{z_1}F_j(z)+\frac{1}{2\pi i}\int_{D_1} \frac{\bar{\partial}_{\eta_1}F_{j,0}(\eta_1,z_2,\ldots,z_{3n+1})-\bar{\partial}_{z_1}F_j(z)}{\eta_1-z_1}\, d\eta_1\wedge d\bar{\eta}_1.
\end{eqnarray}
We can then differentiate under the integral as before and use \eqref{almostanalyticend} to establish \eqref{greensclaim}.
Together with \eqref{almostanalyticend}, we arrive at
\begin{equation}
|\partial_{z_k}^\beta \bar{\partial}_{z_k}^\gamma \bar{\partial}_z^\alpha F_{j,1}|_{U_{j,1}^2}\leq C L_1^n L_2^{|\beta|+|\gamma|}\exp(-\frac{3}{4}(\rho-1)(2L_1 u_j)^{-1/(\rho-1)})\|P\|
\end{equation}
for $|\alpha|\geq 1$.

We proceed by defining $F_{j,m}$ inductively using the same contour integral in the $m$-th variable, and taking $U_{j,m}^2=U_{j,m-1}^2\cap \{|\textrm{Im}(z_m)|\leq u_j\}$ for $m\leq n$.
With each step, $F_{j,m}$ becomes analytic in an additional variable, and we obtain the estimate \eqref{greensclaim} in $U_{j,m}^2$ for $\alpha$ with $\alpha_k=0$ for $k\leq m$.

For $n+1 \leq m \leq 3n$, we no longer requre $2\pi$ periodicity of our construction, and can instead define
\begin{equation}
F_{j,m}(z):=\frac{1}{2\pi i}\int_{\partial D_m} \frac{F_{j,0}(z_1,\ldots,\eta_m,\ldots,z_{3n+1})}{\eta_m-z_m}\, d\eta_m
\end{equation}
for $z\in U_j^2$.
We can then proceed as above to estimate $F_{j,m}$ in $U_{j,m}^2$, defined inductively by
\begin{equation}
U_{j,m}^2=U_{j,m-1}^2\cap \{|\textrm{Im}(z_m)|\leq p_m\}
\end{equation}
where $p_m=v_j,w_j$ for $n+1\leq m \leq 2n$ and $2n+1 \leq m \leq 3n$ respectively.

For $2n+1 \leq m \leq 3n$, the constant $C$ in the estimate $\eqref{greensclaim}$ has to be multiplied by $(\bar{R}+1)^{m-2n}$ to account for the measure of $D_m$, but this is the only dependence of $C$ on anything other than $n,\rho,\zeta$. We absorb this dependence into the $C$ for the rest of the proof.

We now set $P_j:=F_{j,3n}$, which is analytic in $U_j$ by construction.
We have shown that for $l\leq 1$ and for any $1\leq k\leq 3n+1$, we have
\begin{equation}
|\partial_{x_k}^l (P_j-F_j)|_{U_j}\leq CL_1^n L_2^l\exp(-\frac{3}{4}\left(\rho-1)(2L_1 u_j)^{-1/(\rho-1)}\right)\|P\|.
\end{equation}
In particular, this implies that
\begin{eqnarray}
|P_{j+1}-P_j|_{U_{j+1}}&\leq& |P_{j+1}-F_{j+1}|_{U_{j+1}}+|P_j-F_j|_{U_{j+1}}+|F_{j+1}-F_j|_{U_{j+1}}\\
&\leq& CL_1^n \exp\left(-\frac{3}{4}(\rho-1)(2L_1 u_j)^{-1/(\rho-1)}\right).
\end{eqnarray}
Since by construction $F_j(z)=P(z)$ for real $z$, we also obtain
\begin{equation}
|\partial_{x_k}^l(P_j(x)-P(x))|_{U_j\cap \mathbb{R}^{3n+1}} \leq CL_1^nL_2^l \exp\left(-\frac{3}{4}(\rho-1)(2L_1 u_j)^{-1/(\rho-1)}\right).
\end{equation}
The final claimed estimate arises as follows
\begin{equation}
|P_0|_{U_0}\leq |F_0|_{U_0}+|P_0-F_0|_{U_0}\leq C\left(1+L_1^n \exp\left(-\frac{3}{4}(\rho-1)(2L_1 u_0)^{-1/(\rho-1)}\right)\right)\|P\|.
\end{equation}
\end{proof}

Our next goal is to set up an iterative scheme based on Theorem \ref{kamstep} that converges in the Gevrey class $G^{\rho,\rho(\tau+1)+1,\rho(\tau+1)+1,1}(\mathbb{T}^n\times D\times \Omega\times (-1,1))$. This involves defining decreasing sequences of our parameters $s_j,r_j,h_j,\eta_j,\epsilon_j,\sigma_j,K_j$ such that the hypotheses of Theorem \ref{kamstep} are always satisfied, as well as decreasing sequences of the the parameters $u_j,v_j,w_j$ such that the hypotheses of the Proposition \ref{approxlemma} are always satisfied.

\section{The iterative scheme}

We begin by examining how the parameters are defined in the first iteration. It is convenient to choose a weighted error $0<E<1/64$ and a fixed $0<\hat{\epsilon}\leq 1$ such that
\begin{equation}
\eta=E^{1/2},\epsilon=\hat{\epsilon}\kappa Er\sigma^{\tau+1}.
\end{equation}
We can then define $K$ and $h$ by
\begin{equation}
K^ne^{-K\sigma}=E,h=\frac{\kappa}{2K^{\tau+1}}.
\end{equation}
The choice for $K$ is motivated by the form of the error term \eqref{newerrbd} and the choice for $h$ is motivated by \eqref{kamstepconsts}.

We define the subsequent values of the parameters $r,s,\sigma$ by
\begin{equation}
r_+=\eta r,s_+=s-5\sigma,\sigma_+=\delta\sigma
\end{equation}
where $\delta(\rho)\in (0,1)$ will subsequently be chosen in a convenient way.

The KAM step together with the definitions of our parameters yields the estimate
\begin{eqnarray*}
|P_+|_{s_+,r_+,(1/2-3v)h}&<& C\hat{\epsilon}\kappa r \sigma^{\tau+1}(E^2+(\eta^2+K^ne^{-K\sigma})E)\\
&=& C\hat{\epsilon}\kappa r \sigma^{\tau+1}E^2\\
&=& C\delta^{-\tau-1}\hat{\epsilon}\kappa r_+ \sigma_+^{\tau+1}E^{3/2}.
\end{eqnarray*}

Hence, having chosen $\delta(\rho)$, we will arrive at the estimate
\begin{equation}
\label{c1def}
|P_+|_{s_+,r_+,(1/2-3v)h}\leq \frac{1}{2}\hat{\epsilon}c_1^{1/2}\kappa r_+ \sigma_+^{\tau+1}E^{3/2}
\end{equation}
for some $c_1>1$ dependent only on $n,\rho,\tau$.

We can then choose our subsequent weighted and unweighted error as 
\begin{equation}
E_+=c_1^{1/2}E^{3/2},\epsilon_+=\hat{\epsilon}\kappa r_+ \sigma_+^{\tau+1}E_+.
\end{equation}

The constants $\eta_+,K_+,h_+$ can then be defined in terms of $E_+$ and $\sigma_+$ as before, and we have $c_1E_+=(c_1E)^{3/2}$ which will give us rapid convergence provided $c_1E<1$. Moreover, if we have
\begin{equation}
h_+ \leq (1/2-3v)h
\end{equation}
then we obtain
\begin{equation}
\label{Pplusbound}
|P_+|_{s_+,r_+,h_+}\leq \epsilon_+/2.
\end{equation}

For brevity of notation, we write $D_j=D_{s_j,r_j}, O_j=O_{h_j},V_j=D_j\times O_j\times (-1,1)$.

Now, in terms of $\delta$ we choose our sequences $s_j$ and $\sigma_j$ to decay at a geometric rate, by setting
\begin{equation}
s_j=s_0 \delta^j,\sigma_j=\sigma_0\delta^j, s_0=5\sigma_0/(1-\delta)
\end{equation}

where $\sigma_0\ll1$ remains to be chosen to be convenient for the subsequent estimates.

A direct consequence of this definition is that
\begin{equation}
s_{j+1}=s_j-5\sigma_j,\sigma_j=(1-\delta)s_j/5.
\end{equation}

For the parameter sequences required in our approximation lemma, we set
\begin{equation}
\label{uvwdefn}
u_j=2s_0\delta^jM(\rho),v_j=2r_0\delta^jM(\rho),w_j=2h_0\delta^jM(\rho)
\end{equation}
where $M(\rho)=(25(1-\delta)^{-2}+2)^{1/2}$ and assume for now that these sequences satisfy \eqref{uvw}.

We write $U_j=U_j^1\cap \{|I|<r\}$ where $U_j^1$ is defined as in Proposition \ref{approxlemma}. Now, by applying Proposition \ref{approxlemma} to the terms $P^0,P^1$ from \eqref{defnofPzero}, we obtain sequences $P_j^0,P_j^1$ of real analytic functions in $U_j^1$ that are good approximations to $P^0$ and $P^1$.

We set
\begin{equation}
P_j(\theta,I;\omega,t):=\langle P_j^0(I;\omega,t)I,I\rangle+P_j^1(\theta,I;\omega,t).
\end{equation}
Proposition \ref{approxlemma}, together with the factors picked up during the Whitney extension of $P^0,P^1$ in \eqref{whitneyapplication} then yields the estimates
\begin{equation}
\label{Pzerobound}
|P_0|_{U_0}\leq C(n,\rho,\zeta)(\bar{R}^n+1)L_1^{2n+1}\langle P\rangle_r
\end{equation}
and
\begin{eqnarray}
\label{Pjbound}
|P_j-P_{j-1}|_{U_j}&\leq& C(n,\rho,\zeta)(\bar{R}^n+1)L_1^{2n+1}\langle P \rangle_r e^{-B_0\sigma_j^{-1/(\rho-1)}}\\
&=&C(n,\rho,\zeta)(\bar{R}^n+1)L_1^{2n+1}\langle P \rangle_r e^{-\tilde{B}_0 s_j^{-1/(\rho-1)}}
\end{eqnarray}
where $B_0,\tilde{B}_0$ are constant multiples of $L_1^{-1/(\rho-1)}$ by a factors depend only on $\rho$ and $\delta$.

We now choose the constant
\begin{equation}
\hat{\epsilon}:=\langle R \rangle_r L_1^{N-n-2}(a\kappa r)^{-1}.
\end{equation} 
and set
\begin{equation}
\tilde{\epsilon}_j:=\hat{\epsilon}\kappa r_0 \sigma_0^{\tau+1}\exp(-B_0\sigma_j^{-1/(\rho-1)})
\end{equation}
where $N(n,\tau,\rho)$ and $a(n,\rho,\tau,\zeta,\bar{R})\in (0,1]$ are constants to be fixed later in such a way that we can bound $P_0$ and $P_j-P_{j-1}$ by $\tilde{\epsilon}_0$ and $\tilde{\epsilon}_j$ respectively, yielding rapid convergence of the $P_j$.

Note that we will have $\hat{\epsilon}\leq 1$ by taking the $\epsilon$ in Theorem \ref{kamwithparams} sufficiently small.

We then define the weighted error $E_j$ by 
\begin{equation}
\label{weightederrordefn}
E_j:=c_1^{-1}\exp(-B\sigma_j^{-1/(\rho-1)})
\end{equation}
where $c_1(n,\rho,\tau)$ is the constant from \eqref{c1def} and
\begin{equation}
B:=B_0(\delta^{-1/(\rho-1)}-1)/2.
\end{equation}
The desired recurrence $E_{j+1}=c_1^{1/2}E_j^{3/2}$ together with the recurrence for $\sigma_j$ then forces $\delta=(2/3)^{\rho-1}$, which in turn forces $B=B_0/4=A_0(\rho)L_1^{-1/{\rho-1}}$.

We now define the remaining parameters using $E_j$ as previously discussed. We set $\eta_j=E_j^{1/2},r_{j+1}=\eta_j r_j$ and write
\begin{equation}
\epsilon_j=\hat{\epsilon}\kappa r_j \sigma_j^{\tau+1}E_j.
\end{equation}

Our definition of $E_j$ in \eqref{weightederrordefn} was chosen precisely so that we can obtain the inequality $\tilde{\epsilon}_j\leq \epsilon_{j+1}/2$ which allows us to conveniently handle the error terms of the analytic approximations in our iterative scheme.

We define $K_j\geq 1$ implicitly by $K_j^ne^{-K_j\sigma_j}=E_j$.

Setting $x_j=K_j\sigma_j$, we get 
\begin{equation}
x_j^ne^{-x_j}=E_j\sigma_j^n=c_1^{-1}\sigma_j^n\exp(-B\sigma_j^{-1/(\rho-1)})
\end{equation}
and taking logarithms this becomes
\begin{equation}
\label{randomlog}
x_j-n\log(x_j)=B\sigma_j^{-1/(\rho-1)}-n\log(\sigma_j)+\log(c_1).
\end{equation}

For convenience in our numerous estimates that require $\sigma_j$ to be small, we now set
\begin{equation}
\label{sigmanaught}
\sigma_0=\sigma L_1^{-1}(\log(L_1+e))^{-(\rho-1)}
\end{equation}
where $\sigma\leq \tilde{\sigma}(n,\rho)\ll 1$.

We can then bound the right-hand side of \eqref{randomlog} using 
\begin{eqnarray}
B\sigma_j^{-1/(\rho-1)}-n\log(\sigma_j)+\log(c_1)&\geq& B\sigma_j^{-1/(\rho-1)}\\
 &\geq& B\sigma_0^{-1/(\rho-1)}\\
 &=&A_0(L_1\sigma_0)^{-1/(\rho-1)}\\
 &>&A_0\sigma^{-1/(\rho-1)}\gg1.
\end{eqnarray}
Hence for each $j\geq 1$ we indeed obtain a unique $x_j(\sigma)$ such that
\begin{equation}
x_j \geq x_j-n\log(x_j) \geq B\sigma_j^{-1/(\rho-1)}\gg1.
\end{equation}  
This implies that we uniformly have
\begin{equation}
x_j-n\log(x_j)=x_j(1+o(1))
\end{equation}
as $\sigma\rightarrow 0$.
On the other hand we can also bound this quantity above by again making use of \eqref{sigmanaught} to arrive at
\begin{eqnarray}
& &\nonumber x_j-n\log(x_j)\\
&\leq & \nonumber B\sigma_j^{-1/(\rho-1)}(1-nA_0^{-1}(L_1\sigma_j)^{1/(\rho-1)}\log(L_1\sigma_j)+nA_0^{-1}(L_1\sigma_0)^{1/(\rho-1)}(\log(L_1)+\log(c_1)))\\ 
&=& \nonumber B\sigma_j^{-1/(\rho-1)}(1+o(1))
\end{eqnarray}
as $\sigma\rightarrow 0$, again uniformly with respect to $j$. Hence
\begin{equation}
\label{xjasymp}
B\sigma_j^{-1/(\rho-1)}\leq x_j \leq B\sigma_j^{-1/(\rho-1)}(1+o(1)) 
\end{equation}
as $\sigma\rightarrow 0.$

Defining $h_j=\kappa/(2K^{\tau+1})$, we also fix $v=1/54$.

Having set up the parameters, our next task is to establish the remaining hypotheses in \eqref{kamstepconsts}, noting that the last is immediate from our definition of $h_j$. First we observe that
\begin{eqnarray}
\epsilon_j &=& \hat{\epsilon}\kappa r_j \sigma_j^{\tau+1}E_j\\
&\leq & \kappa r_j \sigma_j^{\tau+1}E_j \\
& \leq & c\kappa \eta_j r_j \sigma_j^{\tau+1}
\end{eqnarray}
since $\eta_j^2=E_j$ and $E_j=c_1^{-1}\exp(-B\sigma_j^{-1/(\rho-1)})$ is $o(1)$ as $\sigma\rightarrow 0$. This establishes the first remaining hypothesis in \eqref{kamstepconsts}. To prove the only remaining hypothesis, we compute
\begin{eqnarray}
\frac{\epsilon_j}{r_j h_j}&=& \hat{\epsilon} \kappa \sigma_j^{\tau+1}E_j/h_j\\
&\leq & 2E_j x_j^{\tau+1}\\
&\leq & 2c_1^{-1}\exp(-B\sigma_j^{-1/(\rho-1)})(B\sigma_j^{-1/(\rho-1)})^{\tau+1}(1+o(1))\\
&\leq & c(\rho,\tau)\exp(-\frac{A_0}{2}(L_1\sigma_j)^{-1/(\rho-1)})\\
&\leq & c(\rho,\tau)\exp(-\frac{A_0}{2}(\sigma \delta^j)^{-1/(\rho-1)})
\end{eqnarray}
which is $o(1)$ as $\sigma\rightarrow 0$ and hence we have verified the hypotheses of Theorem \ref{kamstep}.

Moreover, for $\tilde{\sigma}(n,\rho,\tau)$ sufficiently small, we have
\begin{equation}
\label{rjhj}
\prod_{j=0}^\infty (1+\frac{C\epsilon_j}{r_jh_j})\leq \exp(\sum_{j=0^\infty}\frac{C\epsilon_j}{r_jh_j})\leq 2
\end{equation}
where the constant $C(n,\tau)$ in \eqref{rjhj} comes from the estimates in Theorem \ref{kamstep}.

From \eqref{xjasymp}, we have $(x_j/x_{j+1})=(\sigma_{j+1}/\sigma_j)^{1/(\rho-1)}(1+o(1))$. Consequently, we have
\begin{eqnarray}
(h_{j+1}/h_j)&=&(x_j/x_{j+1})^{\tau+1}(\sigma_{j+1}/\sigma_j)^{\tau+1}\\
&=&\delta^{(\tau+1)\rho/(\rho-1)}(1+o(1))\\
&=&(2/3)^{\rho(\tau+1)}(1+o(1))\\
&<&(4/9)^\rho\\
&<& 1/2-3v
\end{eqnarray}
for $\tilde{\sigma}(n,\rho,\tau)$ sufficiently small, which verifies $h_+<(1/2-3v)h$.

We now inductively establish the claim
\begin{equation}
\label{epsilonvstilde}
\tilde{\epsilon}_j\leq \frac{1}{2}\epsilon_{j+1}.
\end{equation}
After doing this and choosing the constants $a,N,\tilde{\sigma}$ as required, we will at last be in a position to state and prove the iterative result that arises from $n$ applications of the KAM step.

Unravelling the definitions of $\epsilon_j,\tilde{\epsilon}_j$, we obtain
\begin{equation}
\tilde{\epsilon}_0/\epsilon_1 =(r_0\sigma_0^{\tau+1}\exp(-4B\sigma_0^{-1/(\rho-1)}))/(r_1\sigma_1^{\tau+1}E_1)=C(n,\rho,\tau)\exp(-2B\sigma_0^{-1/(\rho-1)})\leq 1/2
\end{equation}
for sufficiently small $\tilde{\sigma}\ll 1$.

Moreover
\begin{equation}
(\tilde{\epsilon}_j/\epsilon_{j+1})(\epsilon_j/\tilde{\epsilon}_{j-1})=\exp(-\frac{4}{3}B\sigma_j^{-1/(\rho-1)})\cdot \frac{\epsilon_j}{\epsilon_{j+1}}=C(n,\rho,\tau)\exp(-\frac{1}{3}B\sigma_j^{-1/(\rho-1)})\leq 1
\end{equation}
for sufficiently small $\tilde{\sigma}\ll1$. This establishes \eqref{epsilonvstilde}.

Next we show that show that the sequences $u_j,v_j,w_j$ defined in \eqref{uvwdefn} satisfy the hypotheses \eqref{uvw} of Proposition \ref{approxlemma}. First we observe from the definition of $\sigma_0$ that
\begin{equation}
4s_0L_1 =20\sigma_0 (1-\delta)^{-1}L_1 \leq 20 (1-\delta^{-1})\tilde{\sigma}\leq 1
\end{equation}
for sufficiently small $\tilde{\sigma}\ll 1$. At this point, we fix $\tilde{\sigma}(n,\rho,\tau)\ll 1$. For the sequence $w_j$, we estimate
\begin{equation}
h_0=\frac{\kappa \sigma_0^{\tau+1}}{2x_0^{\tau+1}}\leq K\sigma_0^{\tau+1}\leq \kappa \sigma_0 \leq \kappa s_0\leq L_2^{-1-\zeta}s_0.
\end{equation}
This implies that $w_jL_2\leq u_jL_1$ and $w_0\leq L_2^{-1-\zeta}$.

Setting 
\begin{equation}
\label{cdefn}
r_0=c(n,\rho,\tau,\zeta)r
\end{equation}
and using \eqref{Radbound}, we obtain
\begin{equation}
r_0\leq cL_2^{-1-\zeta}<L_2^{-1-\zeta}
\end{equation}
and
\begin{equation}
r_0L_2 < cL_1^{-\zeta}\leq s_0L_1
\end{equation}
by choosing the $c$ in \eqref{cdefn} sufficiently small.

Finally, we need to choose $0<a(n,\rho,\tau,\zeta,\bar{R})\leq 1$ and $N(n,\rho,\tau)$ such that 
\begin{equation}
|P_0|_{U_0}\leq \tilde{\epsilon}_0
\end{equation}
and
\begin{equation}
|P_j-P_{j-1}|_{U_j}\leq \tilde{\epsilon}_j
\end{equation}
for $j\geq 1$. From \eqref{Pzerobound}, we obtain
\begin{equation}
\label{C0def}
|P_0|_{U_0}\leq C_0(\bar{R}^n+1) \langle P\rangle_r L_1^{2n+1}=\hat{\epsilon}\kappa r C_0(\bar{R}^n+1) L_1^{-N+3n+3}a. 
\end{equation}
On the other hand, the definition of $\sigma_0$ in \ref{sigmanaught} implies that
\begin{eqnarray}
\sigma_0^{\tau+1}\exp(-4B\sigma_0^{-1/(\rho-1)})&=& \sigma_0^{\tau+1}\exp(-4A_0(L_1\sigma_0)^{-1/(\rho-1)})\\
&=&L_1^{-\tau-1}\log(L_1+e)^{-(\rho-1)(\tau+1)}(L_1+e)^{-4A_0\sigma_0^{-1/(\rho-1)}}\\
&\geq& \label{Cdef} C(n,\rho,\tau)L_1^{-4A_0\sigma_0^{-1/(\rho-1)}-\tau-2}.
\end{eqnarray}
We now fix
\begin{equation}
a=CC_0^{-1}(\bar{R}^n+1)^{-1}r_0r^{-1}=CC_0^{-1}(\bar{R}^n+1)^{-1}c
\end{equation}
and
\begin{equation}
N=4A_0\sigma_0^{-1/(\rho-1)}+\tau+3n+5
\end{equation}
where the $C(n,\rho,\tau),C_0(n,\rho),c(n,\rho,\tau,\zeta),A_0$ are as in \eqref{Cdef}, \eqref{C0def}. This yields 
\begin{equation}
|P_0|_{U_0}\leq C_0(\bar{R}^n+1) \langle P\rangle_r L_1^{2n+1}\leq \tilde{\epsilon}_0.
\end{equation}
We can then insert this estimate into \eqref{Pjbound} to obtain
\begin{equation}
|P_j-P_{j-1}|_{U_j}\leq C_0(\bar{R}^n+1) \langle P\rangle_r L_1^{2n+1}\exp(-B_0\sigma_j^{-1/(\rho-1)})\leq \tilde{\epsilon}_j.
\end{equation}

We have now set up the necessary sequences for our iterative scheme in addition to the key ingredient of Theorem \ref{kamstep}.

We define the Hamiltonian 
\begin{equation}
H_j(\theta,I;\omega,t)=N_0(I;\omega)+P_j(\theta,I;\omega,t)=\langle \omega, I\rangle+P_j(\theta,I;\omega,t)
\end{equation}
which is real analytic in $U_j$.
For $j\geq 0$ we denote by $\mathcal{D}_j$ the class of real-analytic diffeomorphisms from $D_{j+1}\times O_{j+1}\times (-1,1)\rightarrow D_j\times O_j$ of the form
\begin{equation}
\mathcal{F}(\theta,I;\omega,t)=(\Phi(\theta,I;\omega),\phi(\omega;t))=(U(\theta;\omega,t),V(\theta,I;\omega,t),\phi(\omega;t))
\end{equation}
where $\Phi$ is affine in $I$ and canonical for fixed $(\omega,t)$ and the variable $t\in (-1,1)$ is regarded as a parameter.

\begin{prop}
\label{iterative}
Suppose $P_j$ is real analytic on $U_j$ for each $j\geq 0$, and that we have the estimates
\begin{equation}
|P_0|_{U_0}\leq \tilde{\epsilon}_0
\end{equation}
and
\begin{equation}
|P_j-P_{j-1}|_{U_j}\leq \tilde{\epsilon}_j
\end{equation}
for each $j\geq 1$.

Then for each $j\geq 0$, we can find a real-analytic normal form $N_j(I;\omega,t)=e_j(\omega,t)+\langle\omega,I\rangle$ and a real analytic map $\mathcal{F}^j$ given by
\begin{equation}
\mathcal{F}^{j+1}=\mathcal{F}_0\circ \ldots\circ \mathcal{F}_j: D_{j+1}\times O_{j+1}\times (-1,1) \rightarrow (D_0\times O_0)\cap U_j
\end{equation}
with the convention that the empty composition is the identity and where the $\mathcal{F}_j\in\mathcal{D}_j$ are such that
\begin{equation}
H_j\circ \mathcal{F}^{j+1}=N_{j+1}+R_{j+1}
\end{equation}
\begin{equation}
|R_{j+1}|_{j+1}\leq \epsilon_{j+1}
\end{equation}
\begin{equation}
\label{jacobian}
|\bar{W}_j(\mathcal{F}_j-id)|_{j+1},|\bar{W}_j(D\mathcal{F}_j-Id)\bar{W}_j^{-1}|<\frac{C\epsilon_j}{r_j h_j}
\end{equation}
\begin{equation}
\label{Sjbound}
|\bar{W}_0(\mathcal{F}^{j+1}-\mathcal{F}^j)|_{j+1}<\frac{C\epsilon_j}{r_j h_j}
\end{equation}
where the constants $C$ depend only on $n$ and $\rho$ and $\bar{W}_j=\textrm{diag}(\sigma_j^{-1}\textrm{Id},r_j^{-1}\textrm{Id},h_j^{-1}\textrm{Id})$.
\end{prop}
\begin{proof}
An immediate application of Theorem \ref{kamstep} provides us with $\mathcal{F}_0\in\mathcal{D}_0$ such that $H_0\circ \mathcal{F}_0=N_1+R_1$, where $|R_1|_1\leq \epsilon_1$. We proceed by induction, assuming that we have $H_{j-1}\circ \mathcal{F}^j=N_j+R_j$, where $N_j(I;\omega,t)=e_j(\omega,t)+\langle \omega,I\rangle$ is a real analytic normal form, $R_j$ is real analytic in $D_j\times O_j$, and $|R_j|_j\leq \epsilon_j$.

We now apply Theorem \ref{kamstep} to the Hamiltonian $N_j+R_j$ in order to find $\mathcal{F}_j\in \mathcal{D}_j$. From \eqref{Pplusbound}, we have $(N_j+R_j)\circ \mathcal{F}_j=N_{j+1}+\tilde{R}_{j+1} $ with the estimate
\begin{equation}
\label{rjbound}
|\tilde{R}_{j+1}|_{j+1}\leq \frac{1}{2}\hat{\epsilon}\kappa r_{j+1}\sigma_{j+1}^{\tau+1}c_1^{1/2}E_j^{3/2}=\frac{1}{2}\epsilon_{j+1}.
\end{equation}
Once we show that
\begin{equation}
\label{mappingfu}
\mathcal{F}^{j+1}:D_{j+1}\times O_{j+1}\times (-1,1)\rightarrow U_j
\end{equation}
we can also establish
\begin{equation}
|(P_j-P_{j-1})\circ \mathcal{F}^{j+1}|_{j+1}\leq |P_j-P_{j-1}|_{U_j}\leq \tilde{\epsilon}_j \leq \frac{1}{2}\epsilon_{j+1}.
\end{equation}

Using our inductive assumption, we can rewrite
\begin{eqnarray}
H_j\circ \mathcal{F}^{j+1}&=& (N_0+P_{j-1})\circ \mathcal{F}^{j+1}+(P_j-P_{j-1})\circ \mathcal{F}^{j+1}\\
&=& (H_{j-1}\circ \mathcal{F}^j)\circ \mathcal{F}_j+(P_j-P_{j-1})\circ \mathcal{F}^{j+1}\\
&=& (N_j+R_j)\circ \mathcal{F}_j+(P_j-P_{j-1})\circ \mathcal{F}^{j+1}.
\end{eqnarray}
This gives us
\begin{equation}
H_j\circ \mathcal{F}^{j+1}=N_{j+1}+R_{j+1}
\end{equation}
with
\begin{equation}
|R_{j+1}|_{j+1}\leq \epsilon_{j+1}.
\end{equation}

It remains to verify \eqref{mappingfu}, which we can do via the estimate \eqref{jacobian}. From \eqref{rjhj} we obtain
\begin{eqnarray}
\label{jacobian2}
|\bar{W}_0 D\mathcal{F}^{j+1}\bar{W}_{j}^{-1}|_{j+1}&=& |(\prod_{k=0}^{j-1}\bar{W}_{k}D\mathcal{F}_k \bar{W}_{k+1}^{-1})(\bar{W}_jD\mathcal{F}_j\bar{W}_{j}^{-1})|\\ &= & |(\prod_{k=0}^{j-1}\bar{W}_{k}D\mathcal{F}_k \bar{W}_k^{-1}(\bar{W}_k\bar{W}_{k+1}^{-1}))(\bar{W}_jD\mathcal{F}_j\bar{W}_{j}^{-1})|\\ 
&=& (\prod_{k=0}^{j-1} |\bar{W}_k\bar{W}_{k+1}^{-1}|)\prod_{k=0}^\infty (1+\frac{C\epsilon_k}{r_k h_k})\\
&\leq & 2(\prod_{k=0}^{j-1} |\bar{W}_k\bar{W}_{k+1}^{-1}|)\\
&\leq & 2\delta^j 
\end{eqnarray}
where 
\begin{equation}
|\bar{W}_k \bar{W}_{k+1}^{-1}|=\sup(s_{k+1}/s_k,r_{k+1}/r_k,h_{k+1}/h_k)=s_{k+1}/s_k=\delta
\end{equation}
follows from the definition of the sequences $s_j,r_j,h_j$.

Writing $x+iy=(\theta,I,\omega)\in D_{j+1}\times O_{j+1}$, we can Taylor expand $\mathcal{F}^{j+1}$ about $x$ to obtain
\begin{equation}
\mathcal{F}^{j+1}(x+iy)=\mathcal{F}^{j+1}(x)+i\bar{W}_0^{-1}\left(\int_0^1 \bar{W}_0 D\mathcal{F}^{j+1}(x+isy)\bar{W}_j^{-1}\, ds \right)\bar{W}_jy.
\end{equation}
From \eqref{jacobian2}, the integral expression is bounded by $2\delta^j$. Since $|\bar{W}_jy|\leq ((s_j/\sigma_j)^2+2)^{1/2}=(25(1-\delta)^{-2}+2)^{1/2}=M(\rho)$, and $\mathcal{F}^{j+1}(x)$ is real, we can conclude that $\mathcal{F}^{j+1}:D_{j+1}\times O_{j+1} \times (-1,1)\rightarrow U_j$. 
This completes the proof.
\end{proof}

The next step is to find Gevrey estimates for the $\mathcal{S}_j:=\mathcal{F}^{j+1}-\mathcal{F}^j$ so that we can show that this iterative scheme does indeed converge in the Gevrey class. We drop the dependence on $t$ from our notation, and remind ourselves that the resulting Gevrey estimates will be uniform in $t$. To this end we introduce the domains
\begin{equation}
\tilde{D}_j:=\{(\theta,I)\in D_j:|\textrm{Im}(\theta)|<s_j/2\},\; \tilde{O}_j:=\{\omega\in\mathbb{C}^n:\textrm{dist}(\omega,\Omega_\kappa)<h_j/2\}
\end{equation}
For multi-indices $\alpha,\beta$ with $|\beta|\leq m$, we also introduce the following notation for the $(m-|\beta|)$-th Taylor remainder in the frequency variable, centred at $\omega$.
\begin{equation}
R^m_\omega (\partial_\theta^\alpha\partial_\omega^\beta \mathcal{S}^j)(\theta,I,\omega'):=\partial_\theta^\alpha\partial_\omega^\beta \mathcal{S}^j-\sum_{|\gamma|\leq m-|\beta|}(\omega'-\omega)^\gamma \partial_\theta^\alpha\partial_\omega^{\beta+\gamma} \mathcal{S}^j(\theta,I,\omega)/\gamma!.
\end{equation}
We then have the following Gevrey estimates.
\begin{lem}
\label{gevest}
\begin{equation}
\label{Mgevbound}
|\bar{W}_0\partial_\theta^\alpha \partial_\omega^\beta \mathcal{S}^j(\theta,0,\omega)|\leq \hat{\epsilon}AC^{|\alpha|+|\beta|}L_1^{|\alpha|+|\beta|(\tau+1)+1}\kappa^{-|\beta|}\alpha!^\rho \beta!^{\rho'}E_j^{1/2}
\end{equation}
for all $(\theta,0;\omega,t)\in \tilde{D}_{j+1}\times \tilde{O}_{j+1}\times (-1,1)$, where $\rho'=\rho(\tau+1)+1$.
\begin{eqnarray}
\label{Lgevbound}
& &|\bar{W}_0 (R_\omega^m\partial_\theta^\alpha \partial_\omega^\beta \mathcal{S}^j)(\theta,0,\omega')|\\
&\leq & \nonumber\hat{\epsilon}AC^{m+|\alpha|+1}L_1^{|\alpha|+(m+1)(\tau+1)+1}\kappa^{-m-1}\frac{|\omega-\omega'|^{m-|\beta|+1}}{(m-|\beta|+1)!}\alpha!^\rho(m+1)!^{\rho'}E_j^{1/2}
\end{eqnarray}
for all $\theta\in\mathbb{T}^n$, $\omega,\omega'\in\Omega_\kappa$ and $|\beta|\leq m$, where the constants $A,C$ only depend on $n,\rho,\tau,\zeta.$
\end{lem}

\begin{proof}
For ease of notation, we define
\begin{equation}
M_{j,\alpha,\beta}(\theta,0,\omega):=|\bar{W}_0\partial_\theta^\alpha\partial_\omega^\beta\mathcal{S}^j(\theta,0,\omega)|
\end{equation}
and
\begin{equation}
L^m_{j,\alpha,\beta}(\theta,0,\omega'):=|\bar{W}_0(R_\omega^m\partial_\theta^\alpha\partial_\omega^\beta)\mathcal{S}^j(\theta,0,\omega')|.
\end{equation}

Now the Cauchy estimate for the derivative of an analytic function together with \eqref{Sjbound} yields
\begin{equation}
M_{j,\alpha,\beta}(\theta,0,\omega)\leq \frac{2^{|\alpha+\beta|}C\alpha!\beta!\epsilon_j}{r_jh_js_{j+1}^{|\alpha|}h_{j+1}^{|\beta|} }=\frac{2^{|\alpha+\beta|}C(n,\rho)\kappa \hat{\epsilon}E_j \sigma_j^{\tau+1}}{h_j s_{j+1}^{|\alpha|}h_{j+1}^{|\beta|}}
\end{equation}
in $\tilde{D}_{j+1}\times \tilde{O}_{j+1}$.

Now from the definitions at the beginning of this section, we have
\begin{equation}
s_j=5(1-\delta)^{-1}A_0^{\rho-1}L_1^{-1}(B\sigma_j^{-1/(\rho-1)})^{-(\rho-1)}
\end{equation}
and
\begin{equation}
h_j=\frac{\kappa \sigma_j^{\tau+1}}{2x_j^{\tau+1}}.
\end{equation}
Since \eqref{xjasymp} implies 
\begin{equation}
h_{j+1}^{-1}\leq C(\tau,\rho)\kappa^{-1}L_1^{\tau+1}(B\sigma_j^{-1/(\rho-1)})^{\rho(\tau+1)},
\end{equation}
we can bound
\begin{eqnarray}
& &M_{j,\alpha,\beta}\\
&\leq&\nonumber \hat{\epsilon}AC^{|\alpha+\beta|}L_1^{|\alpha|+|\beta|(\tau+1)+1}\kappa^{-|\beta|}\alpha!\beta!(B\sigma_j^{-1/(\rho-1)})^{(\rho-1)(\alpha-\tau)+\rho(\tau+1)(|\beta|+1)}\exp(-B\sigma_j^{-1/(\rho-1)})
\end{eqnarray}
By redefining $A$ and $C$ and absorption into the exponential, we arrive at
\begin{equation}
M_{j,\alpha,\beta}\leq \hat{\epsilon}AC^{|\alpha|+|\beta|}L_1^{|\alpha|+|\beta|(\tau+1)+1}\kappa^{-|\beta|}\alpha!^\rho \beta!^{\rho'}E_j^{1/2}
\end{equation}
where $A$ and $C$ depend only on $n,\rho,\tau$. This is precisely the claim \eqref{Mgevbound}.

To prove the second claimed estimate \eqref{Lgevbound}, we first assume that the $\omega,\omega'\in\Omega_\kappa$ satisfy $|\omega'-\omega|<h_{j+1}/8$.

From the elementary inequality
\begin{equation}
\frac{(\beta+\gamma)!}{\gamma!}\leq 2^{|\beta+\gamma|}\beta!\leq 2^{|\beta+\gamma|}\frac{(m+1)!}{(m-|\beta|+1)!},
\end{equation}
together with the established bound \eqref{Mgevbound}, we can estimate $L_{j,\alpha,\beta}^m$ by Taylor expansion.
\begin{eqnarray}
L^m_{j,\alpha,\beta}&\leq&\nonumber \sum_{|\gamma|\geq m-|\beta|+1}|\omega'-\omega|^{|\gamma|}M_{j,\alpha,\beta+\gamma}(\theta,0,\omega')/\gamma!\\
&\leq &\nonumber C\alpha!(m+1)!\frac{4^{|\alpha|+m+1}|\omega'-\omega|^{m-|\beta|+1}\epsilon_j}{(m-|\beta|+1)!r_jh_js_{j+1}^{|\alpha|}h_{j+1}^{m+1}}\sum_{|\gamma|\geq m-|\beta|+1}(\frac{4|\omega'-\omega|}{h_{j+1}})^{|\beta|+|\gamma|-m-1}\\
&\leq &\nonumber C\alpha!(m+1)!\frac{4^{|\alpha|+m+1}|\omega'-\omega|^{m-|\beta|+1}\epsilon_j}{(m-|\beta|+1)!r_jh_js_{j+1}^{|\alpha|}h_{j+1}^{m+1}}\\
&\leq &\nonumber \hat{\epsilon}AC^{m+|\alpha|+1}L_1^{|\alpha|+(m+1)(\tau+1)+1}\kappa^{-m-1}\frac{|\omega-\omega'|^{m-|\beta|+1}}{(m-|\beta|+1)!}\alpha!^\rho(m+1)!^{\rho'}E_j^{1/2}
\end{eqnarray}
where $A$ and $C$ depend only on $n,\rho$ and $\tau$.

For $|\omega'-\omega|\geq h_{j+1}/8$, we can obtain the same estimate by using \eqref{Mgevbound} to estimate $L_{j,\alpha,\beta}^m$ term by term directly from it's definition.
\end{proof}

The Cauchy estimate from Proposition \ref{cauchyappendix} immediately implies
\begin{cor}
\label{gevestcor}
\begin{equation}
\label{Mgevcorbound}
|\bar{W}_0\partial_\theta^\alpha \partial_\omega^\beta \partial_t^\gamma\mathcal{S}^j(\theta,0;\omega,t)|\leq \hat{\epsilon}AC^{|\alpha|+|\beta|+|\gamma|}L_1^{|\alpha|+|\beta|(\tau+1)+1}\kappa^{-|\beta|}\alpha!^\rho \beta!^{\rho'}\gamma!E_j^{1/2}
\end{equation}
for all $(\theta,0;\omega,t)\in \tilde{D}_{j+1}\times \tilde{O}_{j+1}\times (-3/4,3/4)$, where $\rho'=\rho(\tau+1)+1$.
\begin{eqnarray}
\label{Lgevcorbound}
& &|\bar{W}_0 (R_\omega^m\partial_\theta^\alpha \partial_\omega^\beta\partial_t^\gamma \mathcal{S}^j)(\theta,0,\omega',t)|\\
&\leq &\nonumber \hat{\epsilon}AC^{m+|\alpha|+|\gamma|+1}L_1^{|\alpha|+(m+1)(\tau+1)+1}\kappa^{-m-1}\frac{|\omega-\omega'|^{m-|\beta|+1}}{(m-|\beta|+1)!}\alpha!^\rho(m+1)!^{\rho'}\gamma!E_j^{1/2}
\end{eqnarray}
for all $\theta\in\mathbb{T}^n$,  $\omega,\omega'\in\Omega_\kappa$, $t\in (-3/4,3/4)$ and $|\beta|\leq m$, where the constants $A,C$ only depend on $n,\rho,\tau,\zeta.$
\end{cor}

From Proposition \ref{iterative} and Lemma \ref{gevestcor}, the rapid decay of $E_j$ implies that the limit
\begin{equation}
\label{jetnotation}
\partial_\theta^\alpha \partial_t^\gamma \mathcal{H}^\beta(\theta,\omega;t):=\lim_{j\rightarrow\infty} \partial_\theta^\alpha \partial_\omega^\beta\partial_t^\gamma(\mathcal{F}^j(\theta,0;\omega,t)-(\theta,0,\omega))
\end{equation}
exists for each $(\theta;\omega,t)\in \mathbb{T}^n\times \Omega_\kappa\times (-3/4,3/4)$, and each triple of multi-indices $\alpha,\beta,\gamma$. Convergence is uniform, and the limit is smooth in $\theta$ and $t$ and continuous in $\omega$, with $\partial_\theta^\alpha\partial_t^\gamma(\mathcal{H}^\beta)=\partial_\theta^\alpha\partial_t^\gamma \mathcal{H}^\beta$, justifying the notation in \eqref{jetnotation}. 

\section{Whitney extension}
\label{whitneysec}
We now need to use the jet $\mathcal{H}=(\partial_\theta^\alpha\partial_t^\gamma \mathcal{H}^\beta)$ of continuous functions $\mathbb{T}^n\times \Omega_\kappa\times (-3/4,3/4) \rightarrow \mathbb{T}^n\times D \times \Omega$ to obtain a Gevrey function on $\mathbb{T}^n\times \Omega \times (-3/4,3/4)$ by using a Gevrey version of the Whitney extension theorem.

To this end, we define
\begin{equation}
(R_\omega^m \partial_\theta^\alpha \partial_t^\gamma \mathcal{H})_\beta(\theta,\omega',t):=\partial_\theta^\alpha\partial_t^\gamma \mathcal{H}^\beta(\theta,\omega',t)-\sum_{|\delta|\leq m-|\beta|}(\omega'-\omega)^\delta \partial_\theta^\alpha\partial_t^\gamma\mathcal{H}^{\beta+\delta}(\theta;\omega,t)/\gamma!
\end{equation}

In this notation, the results of Corollary \ref{gevestcor} yield

\begin{equation}
\label{impgevest}
|\bar{W}_0\partial_\theta^\alpha\partial_t^\gamma \mathcal{H}^\beta(\theta;\omega,t)|\leq \hat{\epsilon}AL_1(CL_1)^{|\alpha|}(CL_1^{\tau+1}/\kappa)^{|\beta|}C^\gamma\alpha!^{\rho}\beta!^{\rho'}\gamma!
\end{equation}
and
\begin{equation}
\label{impgevest2}
|\bar{W}_0(R_\omega^m \partial_\theta^\alpha \partial_t^\gamma\mathcal{H})_\beta(\theta,\omega',t)|\leq \hat{\epsilon}AL_1(CL_1)^{|\alpha|}(CL_1^{\tau+1}/\kappa)^{m+1}C^\gamma\frac{|\omega-\omega'|^{m-|\beta|+1}}{(m-|\beta|+1)!}\alpha!^{\rho}(m+1)!^{\rho'}\gamma!
\end{equation}
for $|\beta|\leq m$, and $(\theta,\omega,\omega',t)\in \mathbb{T}^n\times \Omega_\kappa \times \Omega_\kappa\times (-3/4,3/4)$, where $A$ and $C$ depend only on $n,\rho,\tau.$

These estimates allow us to apply the following consequence of Theorem \ref{whitneymainthm}.

\begin{thm}
\label{whitneythm}
Suppose $K\subset \mathbb{R}^n$ is compact, and $1\leq \rho < \rho'$. 
If the jet $(f^{\alpha,\beta,\gamma})$ of functions $f^{\alpha,\beta,\gamma}:\mathbb{T}^n\times K\times (-3/4,3/4) \rightarrow \mathbb{R}$ is continuous on $\mathbb{T}^n\times K\times (-3/4,3/4)$ and is smooth in $(\theta,t)\in\T^n\times (-3/4,3/4)$ for each fixed $\omega\in K$ where
\begin{equation}
\partial_\theta^{\alpha'}\partial_t^{\gamma'}(f^{\alpha,\beta,\gamma})=f^{\alpha+\alpha',\beta,\gamma+\gamma'}
\end{equation}
and we have the estimates
\begin{equation}
| f^{\alpha,\beta,\gamma}(\theta;\omega,t)|\leq AC_1^{|\alpha|}C_2^{|\beta|}C_3^{|\gamma|}\alpha!^{\rho}\beta!^{\rho'}\gamma!
\end{equation}
and
\begin{equation}
|(R_\omega^m \partial_\theta^\alpha \partial_t^\gamma f)_\beta(\theta,\omega',t)|\leq AC_1^{|\alpha|}C_2^{m+1}C_3^{|\gamma|}\frac{|\omega-\omega'|^{m-|\beta|+1}}{(m-|\beta|+1)!}\alpha!^\rho (m+1)!^{\rho'}\gamma!
\end{equation}
then there exist positive constants $A_0,C_0$, dependent only on $(n,\rho,\tau)$ (in particular, independent of the set $K$) such that we can extend $f$ to $\tilde{f}\in G^{\rho,\rho',1}(\mathbb{T}^n\times \mathbb{R}^n\times (-3/4,3/4))$ such that $\partial_\theta^\alpha \partial_\omega^\beta\partial_t^\gamma \tilde{f}= f^{\alpha,\beta,\omega}$ on $\mathbb{T}^n\times K\times (-3/4,3/4)$ and
\begin{equation}
|\partial_{\theta}^\alpha \partial_\omega^\beta \partial_t^\gamma\tilde{f}(\theta,\omega)|\leq A_0 A\max(C_1,1)C_0^{|\alpha|+|\beta|+|\gamma|+n}C_1^{|\alpha|+n}C_2^{|\beta|}C_3^{|\gamma|}\alpha!^\rho \beta!^{\rho'}\gamma!
\end{equation}
\end{thm}

The proof of Theorem \ref{whitneythm} is an application (\cite{popovkam} Theorem 3.7) of a version of the Whitney extension theorem in anisotropic spaces of non quasi-analytic functions. 

We are now ready to prove Theorem \ref{kamwithparams}.

\begin{proof}[Proof of Theorem \ref{kamwithparams}]
From the estimates \eqref{impgevest} and \eqref{impgevest2} (taking $\gamma=0$), we can apply Theorem \ref{whitneythm} to extend the jet $\mathcal{H}$ to a Gevrey function
\begin{equation}
\mathcal{H}=(\mathcal{H}_1,\mathcal{H}_2,\mathcal{H}_3):\mathbb{T}^n\times \mathbb{R}^n\times (-3/4,3/4)\rightarrow \mathbb{T}^n\times \mathbb{R}^n\times \mathbb{R}^n.
\end{equation}
Moreover, we have the estimate 
\begin{equation}
\label{gevrefkam}
|\bar{W}_0\partial_\theta^\alpha \partial_\omega^\beta \mathcal{H}(\theta;\omega,t)|\leq \hat{\epsilon}AL_1^{n+2}(CL_1)^{|\alpha|}(CL_1^{\tau+1}/\kappa)^{|\beta|}\alpha!^\rho \beta!^{\rho'}
\end{equation}
where $\hat{\epsilon}=\langle P\rangle_r L_1^{N-n-2}(a\kappa r)^{-1}\leq 1$ and $A(n,\rho,\tau),C(n,\rho,\tau),a(n,\rho,\tau,\zeta,\bar{R})$ are positive constants.

We define 
\begin{equation}
\mathcal{F}=(\Phi,\phi)=(U,V,\phi):=(\mathcal{H}_1(\theta;\omega,t)+\theta,\mathcal{H}_2(\theta;\omega,t),\mathcal{H}_3(\omega,t)+\omega)
\end{equation}
Recalling that $r_0=cr$, where $c(n,\rho,\tau,\zeta)\ll 1$, and $h_0\leq \kappa \sigma_0^{\tau+1}<\kappa$, we can rewrite the Gevrey estimate \eqref{gevrefkam} as
\begin{eqnarray}
& &|\partial_\theta^\alpha\partial_\omega^\beta(U(\theta;\omega,t)-\theta)|+r^{-1}|\partial_\theta^\alpha \partial_\omega^\beta V(\theta;\omega,t)|+\kappa^{-1}|\partial_{\omega}^\beta (\phi(\omega,t)-\omega)|\\
&\leq & \frac{A\langle P\rangle_r L_1^N}{\kappa r}(CL_1)^{|\alpha|}(CL_1^{\tau+1}/\kappa)^{|\beta|}\alpha!^\rho \beta!^{\rho'}
\end{eqnarray}
where $A$ and $C$ depend only on $n,\rho,\tau,\zeta,\bar{R}$ and $N$ depends only on $n,\rho$, and $\tau$.

If we now choose $\epsilon< A^{-1}$ in the statement of Theorem \ref{kamwithparams}, then we obtain
\begin{equation}
|V(\theta;\omega)|\leq A\epsilon r <r \leq R
\end{equation}
and
\begin{equation}
\phi(\omega)\in\Omega
\end{equation}
for $\omega\in \Omega_\kappa$. Thus we have established the estimate \eqref{kamwithparamsbound}.

It remains to check that $\{\Phi(\theta;\omega,t):\theta\in\mathbb{T}^n\}$ is an embedded invariant Lagrangian torus for the Hamiltonian $H(\theta,I;\phi(\omega,t),t)$ when $\omega\in \Omega_\kappa$.

From the estimate
\begin{equation}
|H_{j-1}\circ \mathcal{F}^j-N_j|\leq C\epsilon_j
\end{equation}
that holds on $D_j\times O_j\times (-3/4,3/4)$, the Cauchy estimate yields
\begin{equation}
|W_j(J(D\Phi^j)^T\nabla H_{j-1}\circ \mathcal{F}^j-J\nabla N_j)|\leq \frac{C\epsilon_j}{r_j\sigma_j}
\end{equation}
for each $j\geq 0$ on 
\begin{equation}
\bigcap_{j\geq 0} D_j\times O_j\times (-3/4,-3/4)= \{|\textrm{Im}(\theta)|<s_0/2\}\times \{0\}\times \Omega_\kappa \times (-3/4,3/4)
\end{equation}
where 
\begin{equation}
W_j:=\textrm{diag}(\sigma_j^{-1},r_j^{-1}).
\end{equation}

Recalling that $J\nabla H=X_H$ and $JM_TJ=M^{-1}$ for symplectic matrices $M$ \eqref{sympmatrix}, we can rewrite the above estimate as
\begin{equation}
|W_j(D\Phi^j)^{-1}X_{H_{j-1}}\circ \mathcal{F}^j-X_{N_j}|\leq \frac{C\epsilon_j}{r_j\sigma_j}.
\end{equation}

From the uniform estimate \eqref{jacobian2}, we then obtain

\begin{equation}
|X_{H_{j-1}}\circ \mathcal{F}^j -D\Phi^j\cdot X_N|\leq \frac{C\epsilon_j}{r_j \sigma_j}
\end{equation}
on $\mathbb{T}^n\times \{0\}\times \Omega_\kappa\times (-3/4,3/4)$ where $N=\langle \omega,I\rangle$ which differs from $N_j$ by a term dependent only on $\omega$.

Since $\nabla H_j \rightarrow \nabla H$ uniformly, the rapid decay of $\epsilon_j$ allows us to conclude that 
\begin{equation}
X_{H(\cdot;\phi(\omega,t),t)}\circ \Phi=D\Phi\cdot X_{N}.
\end{equation}
Hence $\{\Phi(\theta;\omega,t):\theta\in\mathbb{T}^n\}$ is an embedded invariant Lagrangian torus for the Hamiltonian $H(\theta,I;\phi(\omega,t),t)$ with frequency $\omega\in \Omega_\kappa.$
This completes the proof.
\end{proof}

\begin{remark}
A similar result is also obtained for real analytic Hamiltonians in \cite{popovkam}.
\end{remark}

\section{Birkhoff normal form}
\label{mainresultssec}
We obtain a Birkhoff normal form for near-integrable Hamiltonians using a version of the KAM theorem that is a consequence of Theorem \ref{kamwithparams}. The Gevrey index $\rho(\tau+1)+1$ frequently appears in these results, and so we introduce $\rho':=\rho(\tau+1)+1$.

\begin{thm}
\label{kamcons}
Fix $0<\zeta \leq 1$ and let $H^0(I)$ be a real-valued non-degenerate  $G^\rho$ smooth Hamiltonian defined on $D^0$ and let $D$ be a subdomain with $\overline{D}\subset D^0$. We define $\Omega=\nabla H^0(D)$ and fix $L_2\geq L_1 \geq 1$ and $\kappa \leq L_2^{-1-\zeta}$ such that $L_2\geq L_0$ and $\Omega_\kappa \neq \emptyset$. Then there exists $N=N(n,\rho,\tau)$ and $\epsilon>0$ independent of $\kappa,L_1,L_2$ and $D\subset D^0$ such that for any $H\in G^{\rho,\rho,1}_{L_1,L_2,L_2}(\mathbb{T}^n\times D\times (-1,1))$ with norm
\begin{equation}
\label{epsilonhdef}
\epsilon_H:=\kappa^{-2}\|H-H^0\|_{L_1,L_2,L_2}\leq \epsilon L_1^{-N}
\end{equation}
there exists a map 
\begin{equation}
\bar{\Phi}=(\bar{U},\bar{V})\in G^{\rho,\rho',1}(\mathbb{T}^n\times \Omega\times (-3/4,3/4),\mathbb{T}^n \times D)
\end{equation}
such that 
\begin{enumerate}
\item For each $\omega\in \Omega_\kappa$ and each $t\in (-3/4,3/4)$, $\Lambda_\omega=\{\bar{\Phi}(\theta;\omega,t):\theta\in\mathbb{T}^n\}$ is an embedded invariant Lagrangian torus of $H$, and $X_H\circ \bar{\Phi}(\cdot;\omega,t)=D\bar{\Phi}(\cdot;\omega,t)\cdot \mathcal{L}_\omega$.
\item There exist constants $A,C>0$ independent of $\kappa,L_1,L_2$ and $D\subset D^0$ such that
\begin{eqnarray}
& &\nonumber|\partial_\theta^\alpha \partial_\omega^\beta(\bar{U}(\theta;\omega,t)-\theta)|+\kappa^{-1}|\partial_\theta^\alpha\partial_\omega^\beta(\bar{V}(\theta;\omega,t)-\nabla g^0(\omega))|\\
&\leq &\label{kamconsbound} A(CL_1)^{|\alpha|}(CL_1^{\tau+1}/\kappa)^{|\beta|}\alpha!^\rho \beta!^{\rho'}L_1^{N/2}\epsilon_H^{1/2}
\end{eqnarray}
uniformly in $\mathbb{T}^n\times \Omega\times (-3/4,3/4)$.
\end{enumerate}
\end{thm}
\begin{proof}
We begin by choosing the $\epsilon > 0$ guaranteed by Theorem \ref{kamwithparams} sufficiently small so that $r=R:=\kappa\sqrt{\epsilon_H}$ satisfies \eqref{Radbound} for
\begin{equation}
\sqrt{\epsilon_H}\leq \epsilon(1+\|P^0\|_{L_0,L_0})^{-1}L_1^{-N}
\end{equation}
where $N(n,\rho,\tau)$ is as in Theorem \ref{kamwithparams}.

Writing
\begin{equation}
H(\theta,I;\omega,t)=H^0(\psi_0(\omega))+\langle\omega, I\rangle+P(\theta,I;\omega,t)
\end{equation}
with
\begin{equation}
P(\theta,I;\omega,t)=\langle P^0(I;\omega)I,I\rangle+P^1(\theta,I;\omega,t) 
\end{equation}
as in \eqref{expnearfreq}, we can make use of Theorem \ref{kamwithparams}.

Indeed, we have
\begin{eqnarray}
\langle P\rangle_r &=&r^2\|P^0\|_{L_2,L_2}+\|P^1\|_{L_1,L_2,L_2}\\
&\leq & r^2\|P^0\|_{L_0,L_0}+\kappa^2\epsilon_H\\
&\leq & \kappa^2\epsilon_H (1+\|P^0\|_{L_0,L_0})\\
&\leq & r\kappa \sqrt{\epsilon_H}(1+\|P^0\|_{L_0,L_0})\\
&\leq & \epsilon\kappa r L_1^{-N}
\end{eqnarray}
and hence the assumption \eqref{kamwithparamasassump} is satisfied for the Hamiltonian $H(\theta,I;\omega,t)$.

Upon application of Theorem \ref{kamwithparams}, we obtain the family of transformations 
\begin{equation}
\Phi=(U,V)\in G^{\rho,\rho',1}(\T^n\times \Omega\times (-3/4,3/4),\T^n\times B_R).
\end{equation}

We can now define $\bar{\Phi}:\T^n\times\Omega\times (-3/4,3/4) \rightarrow \T^n\times \R^n$ by
\begin{equation}
\bar{\Phi}(\theta;\omega,t)=(U(\theta;\omega,t),V(\theta;\omega,t)+(\nabla g^0)(\phi(\omega;t))).
\end{equation}
By cutting off in $\omega$, we may assume that $V=0$ outside of the domain $\Omega'$ defined in \eqref{domains}.

Since $\nabla g^0(\phi(\omega;t))\in B_R$ and $R\leq \kappa/4 $, we obtain that $\bar{\Phi}$ maps $\T^n\times \Omega\times (-3/4,3/4)$ into $\T^n\times D$ as required.

Moreover, from Theorem \ref{kamwithparams}, it follows that $\{\bar{\Phi}(\theta;\omega,t):\theta\in \T^n\}$ is an invariant Lagrangian torus for the Hamiltonian $H(\cdot,\cdot,t)$ with frequency $\omega$. The estimates for $\bar{\Phi}$ follow readily from \eqref{kamwithparamsbound}.
\end{proof}
We can now use Theorem \ref{kamcons} to obtain the Birkhoff normal form as done by Popov in \cite{popovkam}.

\begin{thm}
\label{main1}
Suppose the assumptions of Theorem \ref{kamcons} hold. Then there exists $N(n,\rho,\tau)>0$ and $\epsilon>0$ independent of $\kappa,L_1,L_2,D$ such that for any $H\in G^{\rho,\rho,1}_{L_1,L_2,L_2}(\mathbb{T}^n\times D\times (-1,1))$ with 
\begin{equation}
\label{pertsmall2}
\epsilon_H\leq \epsilon L_1^{-N-2(\tau+2)}
\end{equation}
where $\epsilon_H$ is as in \eqref{epsilonhdef}, there is a family of $G^{\rho',\rho'}$ maps $\omega:D\times (-1/2,1,2)\rightarrow \Omega$ and a family of maps $\chi\in G^{\rho,\rho',\rho'}(\mathbb{T}^n\times D\times (-1/2,1,2),\mathbb{T}^n\times D)$ that are diffeomorphisms and exact symplectic transformations respectively for each fixed $t\in (-1/2,1,2)$.
Moreover, we can choose the maps $\omega$ and $\chi$ such that family of transformed Hamiltonians 
\begin{equation}
\tilde{H}(\theta,I;t):=(H\circ \chi)(\theta,I;t)
\end{equation}
is of Gevrey class $G^{\rho,\rho',\rho'}(\mathbb{T}^n\times D\times (-1/2,1,2))$ and can be decomposed as
\begin{equation}
\label{bnfeq}
K(I;t)+R(\theta,I;t):=\tilde{H}(0,I;t)+(\tilde{H}(\theta,I;t)-\tilde{H}(0,I;t))
\end{equation}
such that:
\begin{enumerate}
\item $\mathbb{T}^n\times \{I\}$ is an invariant Lagrangian torus of $\tilde{H}(\cdot,\cdot;t)$ for each $I\in E_\kappa(t)=\omega^{-1}(\tilde{\Omega}_\kappa;t)$ and each $t\in (-1/2,1,2)$. 
\item $
\IB (\nabla K(I;t)-\omega(I;t))=\IB R(\theta,I;t)=0\quad\textrm{for all }(\theta,I;t)\in \T^n\times E_\kappa(t)\times (-1/2,1,2),\beta\in\N^n.
$
\item There exist $A,C>0$ independent of $\kappa,L_1,L_2,$ and $D\subset D^0$ such that we have the estimates
\begin{eqnarray}
& &|\TA\IB\TD \phi(\theta,I;t)|+|\IB\TD(\omega(I;t)-\nabla H^0(I))|+|\TA\IB\TD (\tilde{H}(\theta,I;t)-H^0(I))|\nonumber\\
&\leq & \label{main1est}A\kappa C^{|\alpha|+|\beta|+|\delta|}L_1^{|\alpha|}(L_1^{\tau+1}/\kappa)^{|\beta|}\alpha!^\rho \beta!^{\rho'}\delta!^{\rho'}L_1^{N/2}\epsilon_H^{1/2}
\end{eqnarray}
uniformly in $\T^n\times D\times (-1/2,1,2)$ for all $\alpha,\beta$, where $\phi\in G^{\rho,\rho',\rho'}(\T^n\times D\times (-1/2,1,2))$ is such that $\langle \theta,I\rangle+\phi(\theta,I;t) $ generates the symplectic transformation $\chi$ in the sense of Proposition \ref{genfunc}.
\end{enumerate}
\end{thm}

\begin{remark}
For our purposes, high regularity in the $t$-parameter is not required, so we have dropped from analyticity to $G^{\rho'}$ regularity in $t$ at this point in order to simplify the proceeding arguments. I expect that analyticity in $t$ could be preserved by using a stronger variant of the Komatsu implicit function theorem than Corollary \ref{komatsucor}
\end{remark}

\begin{proof}
We begin by taking $\epsilon,N$ as in Theorem \ref{kamcons} and noting that $\epsilon_H\leq \epsilon L_1^{-N-2}$ by assumption. 
This implies that the factor $(ACL_1)L_1^{N/2}\sqrt{\epsilon_H}$ occurring in the Gevrey estimate \eqref{kamconsbound} can be bounded above by $AC\sqrt{\epsilon}$.

Hence, taking $\epsilon$ small enough that both the conclusion to Theorem \ref{kamcons} holds as well as $AC\sqrt{\epsilon}<1/2$, we can first apply the Cauchy estimate from Proposition \ref{cauchyappendix} to \eqref{kamconsbound} in $t$, and then apply a variant of the Komatsu implicit function theorem, Corollary \ref{komatsucor}, to obtain a solution $\theta(\gamma;\omega,t):\T^n\times \Omega\times (-1/2,1,2)\rightarrow\T^n$ to the implicit equation
\begin{equation}
\bar{U}(\theta;\omega,t)=\gamma.
\end{equation}
Moreover, this solution satisfies the Gevrey estimate
\begin{equation}
|\partial_\gamma^\alpha\OB\TD(\theta(\gamma;\omega,t)-\gamma)|\leq AC^{|\alpha|+|\beta|+|\delta|}L_1^{|\alpha|}(L_1^{\tau+1}/\kappa)^{|\beta|}\alpha!^\rho\beta!^{\rho'}\delta!^{\rho'}L_1^{N/2}\sqrt{\epsilon_H}
\end{equation}
uniformly on $\T^n\times\Omega\times (-1/2,1,2)$.

We set $F(\gamma;\omega,t):=\bar{V}(\theta(\gamma;\omega,t);\omega,t)$. In terms of $(\gamma;\omega,t)$, the Lagrangian torus $\Lambda_\omega$ is now given by $(\gamma,F(\gamma;\omega,t):\gamma\in\T^n)$ for each $\omega\in\Omega_\kappa$ and each $t\in (-1/2,1,2)$.
Moreover, Proposition \ref{gevcomp2} on the composition of Gevrey functions gives us the estimate
\begin{equation}
\label{Fest}
|\partial_\gamma^\alpha\OB\TD(F(\gamma;\omega,t)-\nabla g^0(\omega))|\leq A\kappa C^{|\alpha|+|\beta|+|\delta|}L_1^{|\alpha|}(L_1^{\tau+1}/\kappa)^{|\beta|}\alpha!^\rho\beta!^{\rho'}\delta!^{\rho'}L_1^{N/2}\sqrt{\epsilon_H}.
\end{equation}
We next construct functions $\psi\in G^{\rho,\rho',\rho'}(\R^n\times \Omega\times (-1/2,1,2))$ and $R\in G^{\rho',\rho'}(\Omega\times (-1/2,1,2))$ such that the function 
\begin{equation}
\label{Qdefn}
Q(x;\omega,t):=\psi(x;\omega,t)-\langle x,R(\omega,t)\rangle
\end{equation}
is $2\pi$-periodic in $x$ and satisfies
\begin{equation}
\nabla_x\psi(x;\omega,t)=F(p(x),\omega,t)
\end{equation}
in $\R^n\times \Omega_\kappa\times (-1/2,1,2)$ where $p:\R^n\rightarrow \T^n$ is the canonical projection as well as the estimate
\begin{eqnarray}
\label{constructest}
& &|\partial_x^\alpha\OB\TD Q(x;\omega,t)|+|\OB \TD (R(\omega,t)-\nabla g^0(\omega))|\\
&\leq & A\kappa C^{|\alpha|+|\beta|+|\delta|}L_1^{|\alpha|}(L_1^{\tau+1}/\kappa)^{|\beta|}\alpha!^\rho\beta!^{\rho'}\delta!^{\rho'}L_1^{N/2}\sqrt{\epsilon_H}
\end{eqnarray}
for $(x;\omega,t)\in \R^n\times\Omega\times (-1/2,1,2)$.

We do this by first integrating the canonical $1$-form $I\, dx$ over the chain
\begin{equation}
c_x:=\{(sx,F(p(sx);\omega,t)):0\leq s \leq 1\}\subset \R^n\times D.
\end{equation}
We define
\begin{equation}
\tilde{\psi}(x;\omega,t):=\int_{c_x}\sigma=\int_0^1 \langle F(p(sx);\omega,t),x\rangle \, ds
\end{equation}
in $\R^n\times \Omega\times (-1/2,1,2)$.
From the estimate \eqref{Fest} it follows that $\tilde{\psi}(x;\omega,t)-\langle\nabla g^0(\omega),x\rangle$ is bounded above by the right hand side of \eqref{constructest} in $[0,4\pi]^n\times \Omega\times (-1/2,1,2)$. Hence if we define $R_j(\omega,t)=(2\pi)^{-1}\tilde{\psi}(2\pi e_j;\omega,t)$, then $R-\nabla g^0$ satisfies the required estimates in \eqref{constructest}.

Since for $\omega\in\Omega_\kappa$ we know that $\Lambda_\omega$ is a Lagrangian torus, it follows that the  integral of the canonical $1$-form over any closed chain in $\Lambda_\omega$ is homotopy invariant. This means that such an integral is a homomorphism from the fundamental group of $\Lambda_\omega$ to $\R$. Hence
\begin{equation}
\tilde{\psi}(x+2\pi m;\omega,t)-\tilde{\psi}(x;\omega,t)=\langle 2\pi m,R(\omega,t)\rangle
\end{equation}
and so the function
\begin{equation}
\tilde{Q}(x;\omega,t):=\tilde{\psi}(x,\omega)-\langle x,R(\omega,t)\rangle
\end{equation}
both satisfies the Gevrey estimate in \eqref{constructest} and is $2\pi$-periodic in $x$ for $(\omega,t)\in \Omega_\kappa \times (-1/2,1,2)$.

To obtain the sought $Q$ in \eqref{Qdefn} from $\tilde{Q}$, we use an averaging trick.
Choosing $f\in G^{\rho}_C(\R^n)$ for some positive constant $C$ such that $f$ is supported in $[\pi/2,7\pi/2]^n$ and 
\begin{equation}
\sum_{k\in\mathbb{Z}^n}f(x+2\pi k)=1
\end{equation}
for each $x\in\R^n$, it then follows that
\begin{equation}
Q(x;\omega,t):=\sum_{k\in\mathbb{Z}^n}f(x+2\pi k)\tilde{Q}(x+2\pi k;\omega,t)
\end{equation}
is $2\pi$-periodic in $x$ for every $\omega\in\Omega$, and coincides with $\tilde{Q}$ for $\omega\in \Omega_\kappa$. Moreover, $Q$ satisfies the same Gevrey estimate \eqref{constructest} as $\tilde{Q}$. We define 
\begin{equation}
\psi(x;\omega,t):=Q(x;\omega,t)+\langle x, R(\omega,t)\rangle.
\end{equation}
Note that by multiplying $Q$ and $R-\nabla g^0$ by a cut-off function $h\in G^{\rho'}_{C/\kappa}$ which is equal to $1$ in a $\omega$-neighbourhood of $\Omega_\kappa$ and vanishes for $\textrm{dist}(\omega,\R^n\setminus\Omega)\leq \kappa/2$ where $C>0$ is independent of $\Omega\subset \Omega^0$, we can assume that $\psi(x;\omega,t)=\langle x, \nabla g^0(\omega)\rangle$ for $\textrm{dist}(\omega,\R^n\setminus\Omega)\leq \kappa/2$. This cutoff preserves the Gevrey estimates on $\psi$.

Now since $\epsilon_HL_1^{N+2(\tau+2)}\leq \epsilon$, we have that $\kappa A(CL_1)(CL_1^{\tau+1}/\kappa)L_1^{N/2}\sqrt{\epsilon_H}\leq AC^2\sqrt{\epsilon}$. By taking $\epsilon$ sufficiently small we have that $\omega\mapsto \nabla_x\psi(x;\omega,t)$ is a diffeomorphism for any fixed $x\in \R^n$ from the Gevrey estimate \eqref{constructest}. Hence we have a $G^{\rho,\rho'}$-foliation of $\T^n\times D$ by Lagrangian tori $\Lambda_\omega=\{(p(x),\nabla_x\psi(x,\omega)):x\in\R^n\}$ where $\omega\in \Omega$.

In the sought coordinate change, the action $I(\omega,t)$ of the Lagrangian torus $\Lambda_\omega$ will be  given by $R(\omega,t)$. Hence from \eqref{constructest} and Proposition \ref{komatsuprop}, it follows that for $\epsilon$ sufficiently small, the map 
\begin{equation}
(\omega,t)\mapsto (I(\omega,t),t)=(R(\omega,t),t)
\end{equation}
is a $G^{\rho',\rho'}$-diffeomorphism and we have the Gevrey estimate
\begin{eqnarray}
& &|\partial_I^\alpha\partial_t^\beta (\omega(I,t)-\nabla H^0(I))|\\
&\leq & A\kappa C^{|\alpha|+|\beta|}(L_1^{\tau+1}/\kappa)^{|\alpha|}\alpha!^{\rho'}\beta!^{\rho'}L_1^{N/2}\sqrt{\epsilon_H}
\end{eqnarray}
uniformly for $(\theta,I,t)\in \T^n\times D\times (-1/2,1,2)$.

We construct the sought symplectomorphism $\chi$ using the generating function $\Phi(x,I;t)$, setting
\begin{equation}
\Phi(x,I;t)=\psi(x,\omega(I;t);t)
\end{equation}
and noting that we have the required $2\pi$-periodicity of $\phi(x,I;t):=\Phi(x,I,t)-\langle x,I\rangle$, and from Proposition \ref{gevcomp2}, we also have the estimate
\begin{equation}
|\partial_x^\alpha \partial_I^\beta\TD(\Phi(x,I;t-\langle x,I\rangle))|\leq A\kappa C^{|\alpha|+|\beta|+|\delta|}L_1^{|\alpha|}(L_1^{\tau+1}/\kappa)^{|\beta|}\alpha!^\rho\beta!^{\rho'}\delta!^{\rho'}L_1^{N/2}\sqrt{\epsilon_H}.
\end{equation}
We can then apply Corollary \ref{komatsucor} to solve the implicit equation 
\begin{equation}
\partial_I\Phi(\gamma,I,t)=\theta
\end{equation}
for $\gamma$ with the estimate
\begin{equation}
|\partial_\theta^\alpha\partial_I^\beta\TD(\gamma(\theta,I,t)-\theta)|\leq  A\kappa C^{|\alpha|+|\beta|+|\delta|}L_1^{|\alpha|}(L_1^{\tau+1}/\kappa)^{|\beta|}\alpha!^\rho\beta!^{\rho'}\delta!^{\rho'}L_1^{N/2}\sqrt{\epsilon_H}.
\end{equation}
This completes the construction of a symplectomorphism $\chi$ satisfying
\begin{equation}
\chi(\partial_I\Phi(\theta,I,t),I)=(\theta,\partial_\theta \Phi(\theta,I,t)).
\end{equation} 
It follows that
\begin{equation}
(\theta,F(\theta;\omega,t))=\chi(\partial_I\Phi(\theta,I(\omega),t),I(\omega))=\chi(\theta,I(\omega),t)
\end{equation}
for $\omega\in\Omega_\kappa$ and so
\begin{equation}
\Lambda_\omega=\{\chi(\theta,I(\omega),t):\theta\in \T^n\}.
\end{equation}
for $(\omega,t)\in \Omega_\kappa\times (-1/2,1,2)$.

We now set $\tilde{H},K,R$ as in the theorem statement in terms of the symplectomorphism $\chi$.
Since $H$ is constant on $\Lambda_\omega$ for each $\omega\in \Omega_\kappa$, it follows that $R(\cdot,I;t)$ is identically zero for each $I=I(\omega)$ with $\omega\in\Omega_\kappa$. Hence $R$ is flat at $I\in E_\kappa(t)$, since each point in $E_\kappa(t)$ is of positive density in $I(\Omega_\kappa)$. 

Finally, the Gevrey estimate in \eqref{main1est} for $\tilde{H}(\theta,I,t)-H(I,t)$ follows from Proposition \ref{gevcomp2}.
This completes the proof.
\end{proof}

\begin{remark}
In addition to the quantum normal form we will construct in Chapter 5, Theorem \ref{main1} can also provide a short proof of effective stability of the Hamiltonian flow near $\Lambda$ (see \cite{popovkam} Corollary 1.3). 
\end{remark}

To conclude this chapter, we compute the integrable term $K$ of a Birkhoff normal form for the Hamiltonian $H(\theta,I)=H^0(I)+t H^1(\theta,I)$ to second order in $t$.

The key observation is that the initial KAM step leaves the completely integrable term unchanged if the perturbation has average zero over every torus $\T^n\times \{I\}$ in the action-angle variables of the completely integrable Hamiltonian.

In Section \ref{kamstep}, we proved a version of the KAM step that is localised by a frequency parameter. In this setting, the assertion is that
\begin{equation}
N_+(I;\omega,t)=N(I;\omega,t)
\end{equation}
if
\begin{equation}
\int_{\T^n} P(\theta,I;\omega,t)\, d\theta=0
\end{equation}
for each $I\in D$ and each $t\in (-1,1)$. This can be seen directly from \eqref{Nplus} for example. 

An analogous fact holds for a version of the KAM step without parameters as in Section 3 of \cite{galavotti}, and this version is most convenient for the proof of our final result in this chapter.

Using this result, we are able to change variables using two initial KAM step symplectic maps in order to increase the order in $t$ of the perturbation size before applying Theorem \ref{main2}. This improves the estimate in \eqref{main2}. 

As in the proof of Theorem \ref{kamwithparams}, tracking the value of the various decreasing sequences of positive constants is necessary to prove convergence of the iterative scheme generated by this step. Fortunately, we shall only require two applications of the KAM step of Galavotti.

We first state the the result for real-analytic Hamiltonians before using the approximation techniques of Section \ref{approxsec} to generalise to the Gevrey setting.

\begin{prop}
\label{galav}
Suppose $H(\theta,I;t)=H^0(I)+H^1(\theta,I;t)$ is a real analytic Hamiltonian in $\T^n\times D\times B_{\delta}$ that has an analytic extension to
\begin{equation}
W_{s,r}(D):= \{(\theta,I)\in \C^n/2\pi\Z \times \C^n:|\textrm{Im}(\theta)|<s,\textrm{dist}(I,D)<r \}.
\end{equation}

Suppose further that the conditions
\begin{equation}
\left|\frac{\partial H^0}{\partial I}\right|\leq E,
\end{equation}
\begin{equation}
\left|\left(\frac{\partial^2 H^0}{\partial I^2}\right)^{-1}\right|\leq \eta,
\end{equation}
and
\begin{equation}
\left(\left|\frac{\partial H^1}{\partial I}\right|+r^{-1}\left|\frac{\partial H^1}{\partial \theta}\right|\right)\leq \epsilon
\end{equation}
are satisfied. 

Taking
\begin{equation}
V_{C,N}=\{I\in D: |\langle \nabla H^0(I),k\rangle| \geq \frac{1}{C|k|^n}\quad \forall 0<|k|\leq N\},
\end{equation}
then for sufficiently small $\epsilon>0$ we consider the set
\begin{equation}
\tilde{D}_+=\{I\in D:\textrm{dist}(I,\partial D)>\tilde{r}/2\textrm{ and }I\in V_{C,N}\}
\end{equation}
and the smoother set
\begin{equation}
D_+:=\bigcup_{I\in \tilde{D}_+} B(I,\tilde{r}/2)
\end{equation}
as a new action domain, where $N$ is chosen in a way that depends on $C,\epsilon$ and $s$ and $\tilde{r}$ is proportional to $r$ and is otherwise dependent on $n,E,C,n$ with $\tilde{r}<r/2$.

There exists a family of real analytic symplectic maps
\begin{equation}
\label{bnfkamstep}
\chi:\T^n\times D_+ \times B_{\delta} \rightarrow \T^n\times D
\end{equation}
that analytically extend to a new domain of holomorphy
\begin{equation}
W_{s_+,r_+}(D_+)
\end{equation} 
such that
\begin{equation}
(H\circ \chi)(\theta,I;t)=H^0(I)+ (2\pi)^{-n}\int_{\T^n} H^1(\theta,I;t)\, d\theta +H^2(\theta,I;t).
\end{equation}
with
\begin{equation}
\sup_{D_{s_+,r_+}}|H^2| = O(\epsilon^{3/2})
\end{equation}
with constant depending only on $n,C$ and $E$.
\end{prop}

This result is essentially a restatement of the KAM step in Section 3 of \cite{galavotti}, in which more explicit details are given. We remark that the constant $C$ describes the family of nonresonant frequencies preserved by the iterative scheme, with $V_{C,\infty}$ having large measure for large $C$. 

We now consider the family of real analytic Hamiltonians of the form
\begin{equation}
H(\theta,I;t):=H^0(I)+tH^1(\theta,I)
\end{equation}
still satisfying the hypotheses of Proposition \ref{galav}.

Applying Proposition \ref{galav}, we obtain 
\begin{equation}
(H\circ \chi)(\theta,I;t)=H^0(I)+t\cdot (2\pi)^{-n}\int_{\T^n} H^1(\theta,I)\, d\theta +H^2(\theta,I;t).
\end{equation}

We will require a stronger estimate on the error term $H^2(\theta,I;t)$ that $O(\epsilon^{3/2})$, due to the presence of a square root in \eqref{main1est}. Hence we need another iteration of Proposition \ref{galav}. A second iteration of Proposition \ref{galav} yields a family of symplectic maps $\tilde{\chi}:\T^n\times D_{++}\rightarrow \T^n\times D$ that extend analytically to $D_{s_{++},r_{++}}$ and such that
\begin{equation}
(H\circ \tilde{\chi})(\theta,I;t)=H^0(I)+t\cdot (2\pi)^{-n}\int_{\T^n} H^1(\theta,I)\, d\theta+(2\pi)^{-n}\int_{\T^n} H^2(\theta,I;t)\, d\theta + H^3(\theta,I;t)
\end{equation}
with $|H^3|=O(t^{9/4})$.

This implies that the completely integrable part of $(H\circ\tilde{\chi})(0,I;t)$ is given by $H^0(I)+t\cdot (2\pi)^{-n}\int_{\T^n} H^1(\theta,I)\, d\theta+ O(t^{3/2})$, valid in the domain $\T^n \times D_{++}$, where $D_{++}$ is an open subset of $D$ given by
\begin{equation}
D_{++}:=\bigcup_{I\in \tilde{D}_{++}} B(I,\tilde{r}_+/2)
\end{equation}
where
\begin{equation}
\tilde{D}_{++}=\{I\in D_+:\textrm{dist}(I,\partial D_+)>\tilde{r}_+/2\textrm{ and }I\in V_{C_+,N_+}\}
\end{equation}
where $C_+$ and $N_+$ are defined inductively in \cite{galavotti}. In particular, we can ensure that the set $E_\kappa(t)=\omega^{-1}(\tilde{\Omega}_\kappa;t)$ is contained in $D_{++}$ by fixing $\tau=n$, taking $C$ in Proposition \ref{galav} sufficiently large. Thus we have
\begin{prop}
\label{bnf15}
Suppose $H(\theta,I;t)=H^0(I)+tH^1(\theta,I)$ is a real analytic Hamiltonian in $\T^n\times D\times B_{\delta}$ that has an analytic extension to
\begin{equation}
W_{s,r}(D):= \{(\theta,I)\in \C^n/2\pi\Z \times \C^n:|\textrm{Im}(\theta)|<s,\textrm{dist}(I,D)<r \}.
\end{equation}

Suppose further that the conditions
\begin{equation}
\left|\frac{\partial H^0}{\partial I}\right|\leq E,
\end{equation}
\begin{equation}
\left|\left(\frac{\partial^2 H^0}{\partial I^2}\right)^{-1}\right|\leq \eta,
\end{equation}
and
\begin{equation}
\left(\left|\frac{\partial H^1}{\partial I}\right|+r^{-1}\left|\frac{\partial H^1}{\partial \theta}\right|\right)\leq \epsilon
\end{equation}
are satisfied. 

Then for sufficiently small $t\in (-\delta,\delta)$, there exists a subdomain $\tilde{D}\subset D$ containing $E_\kappa(t)$ and a family of real analytic symplectic maps
\begin{equation}
\chi:\T^n\times \tilde{D} \times (-\delta,\delta) \rightarrow \T^n\times D
\end{equation}
that analytically extend to a new domain of holomorphy
\begin{equation}
W_{s_+,r_+}(\tilde{D})
\end{equation} 
such that
\begin{equation}
(H\circ \chi)(\theta,I;t)=\tilde{H}^0(I;t)+ \tilde{H}^1(\theta,I;t).
\end{equation}
with
\begin{equation}
\partial_t\tilde{H}^0(I;0)=(2\pi)^{-n}\int_{\T^n} H^1(\theta,I)\, d\theta
\end{equation} 
and
\begin{equation}
\sup_{D_{s_+,r_+}}|\tilde{H}^1| = O(t^{9/4})
\end{equation}
with constant depending only on $n,C$ and $E$.
\end{prop}

We can now generalise this result to the Gevrey setting.

\begin{prop}
\label{bnf175}
Suppose $H(\theta,I;t)=H^0(I)+t H^1(\theta,I)$ satisfies the assumptions of Theorem \ref{kamcons}. Then for sufficiently small $t\in (-\delta,\delta)$, there exists a subdomain $\tilde{D}\subset D$ containing $E_\kappa(t)$ and a $G^{\rho,\rho,\rho}$ family of  symplectic maps
\begin{equation}
\chi:\T^n\times \tilde{D} \times (-\delta,\delta) \rightarrow \T^n\times D
\end{equation}
such that
\begin{equation}
(H\circ \chi)(\theta,I;t)=\tilde{H}^0(I;t)+ \tilde{H}^1(\theta,I;t).
\end{equation}
with
\begin{equation}
\label{newh0}
\partial_t\tilde{H}^0(I;0)=(2\pi)^{-n}\int_{\T^n} H^1(\theta,I)\, d\theta
\end{equation} 
and
\begin{equation}
\label{why9on4}
\|\tilde{H}^1\|_{CL_1,CL_2,CL_2} = O(t^{9/4})
\end{equation}
with constant independent of $\kappa$ and with $C$ dependent only on $n$ and $\rho$.
\end{prop}
\begin{proof}
This result is established via the approximation of Gevrey functions by real-analytic functions. First, we use Theorem \ref{whitneythm} to extend $H^0$ and $H^1$ to the domain $\T^n\times \R^n\times (-1,1)$, before cutting off in $I$ to a ball $B_{\tilde{R}}$ with $D^0\subset B_{\tilde{R}-1}$ as done at the start of Section \ref{approxsec}.

From the same methods used in the proof of Proposition \ref{approxlemma}, we may then construct sequences of real analytic functions $P^0_j$ and $P^1_j$ on shrinking $j$ dependent complex domains $U_j$ containing $\T^n\times \R^n\times (-1,1)$ with a corresponding sequence $u_j\rightarrow 0$ such that 
\begin{equation}
\label{pkuj}
|P^k_{j+1}-P^k_j|_{U_{j+1}}\leq C(D^0,L_1,L_2) \exp\left(-\frac{3}{4}(\rho-1)(2L_1 u_j)^{-1/(\rho-1)}\right)\|H^k\|
\end{equation}
and
\begin{equation}
|\partial_x^\alpha (P^k_j-H^k)(\theta,I;t)|\leq C(D^0,L_1,L_2) \exp\left(-\frac{3}{4}(\rho-1)(2L_1 u_j)^{-1/(\rho-1)}\right)
\end{equation}
in $\mathbb{T}^n\times B_{\tilde{R}}\times (-1,1)$ for $|\alpha|\leq 1$. These sequences $P^k_j$ are convergent in $G^{\rho,\rho,1}(\T^n\times\R^n\times (-1,1))$, as is shown in \cite{gevapproxref} Proposition 2.2. (This fact can be readily obtained by applying Cauchy estimates to \eqref{pkuj}.)

Now for each $j\in \N$, we can apply Proposition \ref{bnfkamstep} to obtain a real analytic symplectic map 
\begin{equation}
\chi_j:\T^n\times D_+ \rightarrow \T^n\times
\end{equation}
defined in shrinking holomorphy domains that comprise the first KAM step for $H^0_j+H^1_j$.

Note that for an individual KAM step, the symplectic map $\chi_j$ is defined using a generating function $\Phi_j$ that is a weighted sum of finitely many Fourier components of $H^1_j$ (see \cite{galavotti} Equation 3.14).

This implies that as $P^0_j+tP^1_j\rightarrow H^0+tH^1$ in $G^{\rho,\rho,1}(\T^n\times D_+\times (-1,1))$, the generating functions $\Phi_j$ converges to some
\begin{equation}
\Phi\in G^{\rho,\rho,1}(\T^n\times D_+\times (-1,1))
\end{equation}
in the $G^{\rho,\rho,1}$ sense.

From the Komatsu implicit function theorem, Corollary \ref{komatsucor}, it follows that the corresponding symplectic maps $\chi_j$ converge to some
\begin{equation}
\chi^1\in G^{\rho,\rho,1}(\T^n\times D_+\times (-1,1))
\end{equation} 
in the Gevrey sense.

Similarly, the symplectic maps $\tilde{\chi}_j$ that arise from comprise a single KAM step for the Hamiltonians
\begin{equation}
(P^0_j+tP^1_j)\circ\chi_j
\end{equation}
can also be seen to converge to some
\begin{equation}
\chi^2\in G^{\rho,\rho,1}(\T^n\times D_{++},\T^n\times D_+).
\end{equation}

It follows that the family of symplectic maps $\chi_j\circ\tilde{\chi}_j$ whose existence is asserted by applying Proposition \ref{bnf15} to $H^0_j+H^1_j$ converge to $\chi:=\chi^1\circ\chi^2$ in the $G^{\rho,\rho,1}$-sense.

Moreover, if we write 
\begin{equation}
(P^0_j+tP^1_j)\circ \chi_j\circ \tilde{\chi}_j=\tilde{H}^0_j(I;t)+ \tilde{H}^1_j(\theta,I;t).
\end{equation}
in the notation of Proposition \ref{bnf15}, we have that $\tilde{H}^k_j$ are convergent sequences in $G^{\rho,\rho,1}$, and so it follows that their limits $\tilde{H}^0,\tilde{H}^1$ satisfy
\begin{equation}
\partial_t\tilde{H}^0(I;0)=(2\pi)^{-n}\int_{\T^n} H^1(\theta,I)\, d\theta
\end{equation}
and
\begin{equation}
\|\tilde{H}^1\|_{CL_1,CL_2,CL_2} = O(t^{9/4})
\end{equation}
as required.
\end{proof}

By applying Proposition \ref{bnf175} to $H(\theta,I;t)=H^0(I)+t H^1(\theta,I)$ with $t$ sufficiently small, we can then invoke Theorem \ref{main1} on the Hamiltonian 
\begin{equation}
\tilde{H}(\theta,I;t)=(H\circ\chi)(\theta,I;t)
\end{equation}
with an improved error term.

\begin{prop}
\label{bnf2}
Suppose the assumptions of Theorem \ref{kamcons} hold for the Hamiltonian
\begin{equation}
H(\theta,I;t)=H^0(I)+tH^1(\theta,I;t)\in G^{\rho,\rho,1}(\mathbb{T}^n\times D\times (-1,1)).
\end{equation} 
Then there exists $N(n,\rho,\tau)>0$ and $\epsilon>0$ independent of $L_1,L_2,D$ such that for any $H\in G^{\rho,\rho,1}_{L_1,L_2,L_2}(\mathbb{T}^n\times D\times (-1,1))$ with 
\begin{equation}
\label{pertsmall3}
\kappa^{-2}t\|H^1\|_{L_1,L_2,L_2}=\epsilon_H\leq \epsilon L_1^{-N-2(\tau+2)}
\end{equation}
there is a subdomain $\tilde{D}\subset D$ containing $E_\kappa(0)$ and a family of $G^{\rho',\rho'}$ maps $\omega:\tilde{D}\times (-1/2,1,2)\rightarrow \Omega$ and a family of maps $\chi\in G^{\rho,\rho',\rho'}(\mathbb{T}^n\times \tilde{D}\times (-1/2,1,2),\mathbb{T}^n\times \tilde{D})$ that are diffeomorphisms and exact symplectic transformations respectively for each fixed $t\in (-1/2,1,2)$.
Moreover, we can choose the maps $\omega$ and $\chi$ such that family of transformed Hamiltonians 
\begin{equation}
\tilde{H}(\theta,I;t):=(H\circ \chi)(\theta,I;t)
\end{equation}
is of Gevrey class $G^{\rho,\rho',\rho'}(\mathbb{T}^n\times \tilde{D}\times (-1/2,1,2))$ and can be decomposed as
\begin{equation}
K(I;t)+R(\theta,I;t):=\tilde{H}(0,I;t)+(\tilde{H}(\theta,I;t)-\tilde{H}(0,I;t))
\end{equation}
such that:
\begin{enumerate}
\item $\mathbb{T}^n\times \{I\}$ is an invariant Lagrangian torus of $\tilde{H}(\cdot,\cdot;t)$ for each $I\in E_\kappa(t)=\omega^{-1}(\tilde{\Omega}_\kappa)$ and each $t\in (-1/2,1,2)$. 
\item $
\IB (\nabla K(I;t)-\omega(I;t))=\IB R(\theta,I;t)=0\quad\textrm{for all }(\theta,I;t)\in \T^n\times E_\kappa(t)\times (-1/2,1,2),\beta\in\N^n.
$
\item There exist $A,C>0$ independent of $\kappa,L_1,L_2,$ and $D\subset D^0$ such that we have the estimates
\begin{eqnarray}
& &|\TA\IB\TD \phi(\theta,I;t)|+|\IB\TD(\omega(I;t)-\nabla \tilde{H}^0(I))|+|\TA\IB\TD (\tilde{H}(\theta,I;t)-\tilde{H}^0(I))|\nonumber\\
&\leq & \label{main1est2}A C^{|\alpha|+|\beta|+|\delta|}L_1^{|\alpha|}(L_1^{\tau+1}/\kappa)^{|\beta|}\alpha!^\rho \beta!^{\rho'}\delta!^{\rho'}L_1^{N/2}|t|^{9/8}
\end{eqnarray}
uniformly in $\T^n\times \tilde{D}\times (-1/2,1,2)$ for all $\alpha,\beta$, where $\phi\in G^{\rho,\rho',\rho'}(\T^n\times \tilde{D}\times (-1/2,1,2))$ is such that $\langle \theta,I\rangle+\phi(\theta,I;t) $ generates the symplectic transformation $\chi$ in the sense of Proposition \ref{genfunc} and $\tilde{H}^0,\tilde{H}^1$ are as in Proposition \ref{bnf175}.

\item \begin{equation}
\label{chp4tderiv}
\partial_t K(I;t)=(2\pi)^{-n}\int_{\T^n} H^1(\theta,I;t)+o(1)
\end{equation}
uniformly in $\T^n\times \tilde{D}\times (-1/2,1,2)$.
\end{enumerate}
\end{prop}
\begin{proof}
The only new claim in this Proposition is \eqref{chp4tderiv}, which follows from \eqref{main1est2} and the expression \eqref{newh0} for $\tilde{H}^0$. Note that the exponent $9/4$ in \eqref{main1est2} comes from \eqref{why9on4} and the square root in \eqref{main1est}.
\end{proof}


\chapter{Quantum Birkhoff normal form in KAM systems}
In this chapter, we continue from the construction of a Birkhoff normal form in Chapter \ref{chp4} and construct a quantum Birkhoff normal form for the quantisations of Gevrey KAM Hamiltonians, following the work of Popov in \cite{popovquasis}. The main consequence of this construction for our purposes is the construction of exponentially accurate quasimodes localising onto the invariant KAM tori.

\section{Quantum Birkhoff normal form}

We let $M$ be a compact $G^\rho$-smooth manifold of dimension $n\geq 2$ and consider the family of formally self-adjoint Schr\"{o}dinger type semiclassical pseudodifferential operator
\begin{equation}
\label{operatorP}
P_h:=h^2\Delta_g+V(x,hD)
\end{equation}
acting on half densities $f|dx|^{1/2}$ where each $V$ is a self adjoint Gevrey symbol in the class $S_\ell(T^*M)$ from Definition \ref{selldef} where $\ell=(\rho,\mu,\eta)$, with $\rho(\tau+n)+1>\mu>\rho'=\rho(\tau+1)+1$ and $\nu=\rho(\tau+n+1)$.

The paradigmal example here is the semiclassical Schr\"{o}dinger operator
\begin{equation}
P_h=h^2\Delta+V(x)
\end{equation}
for a smooth and compactly supported potential $V(x)\in G^\rho$ that is bounded below.

We suppose further that there exists an exact symplectomorphism $\chi_1:\T^n\times D\rightarrow U\subset \T^*M$ such that the transformed Hamiltonian $H(\phi,I)=P_0\circ \chi_1$  can be placed in a $G^{\rho,\rho'}$ Birkhoff normal form \eqref{bnfeq} about a family of invariant tori $\Lambda:=\{\Lambda_\omega:\omega\in \overline{\Omega}\}$ with $\overline{\Omega}\in \Omega_\kappa$.

From Theorem \ref{main1}, we have shown that this is the case if $P_0=H^0+tH^1$ is an analytic one-parameter family of small perturbations of a completely integrable and non-degenerate Hamiltonian $H^0$.

From this point on, we specialise from \eqref{operatorP} to considering one-parameter families of operators of the form
\begin{equation}
P_h(t):=H^0(x,hD)+tH^1(x,hD)
\end{equation}
where the Hamiltonian $H=H^0+tH^1$ satisfies the assumptions of Theorem \ref{main1}.

In this case, the family of maps $\chi_1(t)$ can be taken to be the transformation into ``action-angle" coordinates, the existence of which is guaranteed by the Liouville--Arnold theorem (\cite{arnoldmechanics} Chapter 10, Section 50).

In this chapter, our main goal is to obtain a quantum analogue to the Birkhoff normal form from Theorem \ref{main1} for the operator $P_h(t)$ in the Gevrey classes of semiclassical pseudodifferential operators introduced in Section \ref{gevsymbsec}.

\begin{thm}
\label{main2}
Suppose $P_h(t)$ is as in \eqref{operatorP}. Then there exists a uniformly bounded family of semiclassical Fourier integral operators 
\begin{equation}
U_h(t):L^2(\T^n;L)\times (-1,1)\rightarrow L^2(M)\quad (0<h<h_0)
\end{equation}
that are associated with the canonical relation graph of the Birkhoff normal form transformation $\chi(t)$ such that for each fixed $t\in (-1,1)$, we have
\begin{enumerate}
\item $U_h(t)^*U_h(t)-\textrm{Id}$ is a pseudodifferential operator with symbol in the Gevrey class $S_\ell(\T^n\times D)$ which restricts to an element of $S_\ell^{-\infty}(\T^n\times Y)$ for some subdomain $Y$ of $D$ that contains $E_\kappa(t).$
\item $P_h(t)\circ U_h(t)-U_h(t)\circ P^0_h(t)=\mathcal{R}_h(t)\in S_\ell^{-\infty}$, where the operator $P^0_h(t)$ has symbol 
\begin{equation}
\label{fioconjflatness}
p^0(\theta,I;t,h)=K^0(I;t,h)+R^0(\theta,I;t,h)=\sum_{j\leq \eta h^{-1/\nu}}K_j(I;t)h^j+\sum_{j\leq \eta h^{-1/\nu}}R_j(\theta,I;t)h^j
\end{equation}
with both $K^0$ and $R^0$ in the symbol class $S_\ell(\T^n\times D)$ from Definition \ref{selldef} where $\eta>0$ is a constant, $K_0(I;t),R_0(\theta,I;t)$ are the components of the Birkhoff normal form of the Hamiltonian $P_0\circ \chi_1$ as constructed in Theorem \ref{main1}, and
\begin{equation}
\label{qbnfflatness}
\partial_I^\alpha R_j(\theta,I;t)=0 
\end{equation}
for $(\theta,I;t)\in \T^n\times E_\kappa(t)\times (-1,1)$.
\end{enumerate}
\end{thm}

As a consequence of Theorem \ref{main2}, we will obtain a family of Gevrey quasimodes smoothly parametrised by $t\in (-1,1)$ in Section \ref{gevquasisec}. Moreover, for each fixed $t\in (-1,1)$ these quasimodes $\mathcal{Q}$ will microlocalise (in the sense of the Gevrey microsupport introduced in Definition \ref{gevmicro}) onto a family $\Lambda$ of the nonresonant invariant Lagrangian tori constructed in Chapter 4.

We sketch the details of the proof of Theorem \ref{main2} in this chapter, following the argument of Popov \cite{popovquasis}.

The construction of $U_h(t)$ can be broken into multiple steps. We begin by constructing a family of semiclassical Fourier integral operators $T_h$ that conjugate $P_h(t)$ to a family of pseudodifferential operators $P_h^1(t):\mathcal{C}^\infty(\T^n;\mathbb{L})$ with principal symbol equal to $K_0(I;t)+R_0(\theta,I;t)$, the Birkhoff normal form of $H$, and vanishing subprincipal symbol. This arises as the quantisation of the symplectomorphism that transforms $H$ into its Birkhoff normal form.

The symbol of the operator $P_h^1(t)$ satisfies the property $\eqref{fioconjflatness}$ to $O(h^2)$, and to improve this, we replace the conjugating Fourier integral operator $T_h$ with $T_hA_h$ for a suitable elliptic pseudodifferential operator $A_h$ whose symbol is determined iteratively on the family of Cantor-like set $\{(\theta,I;t)\in \T^n\times \R^n\times (-1,1):I\in E_\kappa(t)\}$ by solving equations of the form
\begin{equation}
\langle \nabla K_0,\partial_\theta\rangle f(\theta,I;t)=g(\theta,I;t)
\end{equation}
referred to in the literature as homological equations. In this manner the ``flatness condition" of \eqref{qbnfflatness} is obtained for $j>0$, where the $j=0$ statement was established by Theorem \ref{main1}.

\section{Conjugation by a h-FIO}
\label{fioconj}
In this section we first conjugate $P_h(t)$ by an Fourier integral operator to a semiclassical pseudodifferential operator on $\mathbb{T}^n\times D$. 

We do this by quantising the $G^\rho$ symplectic maps $\chi_1:\mathbb{T}^n\times D \rightarrow T^*M$ and $\chi_0:\mathbb{T}^n\times D \rightarrow \mathbb{T}^n\times D$ that transform the unperturbed Hamiltonian $H^0$ to action-angle variables and transform the perturbed Hamiltonian to Birkhoff normal form respectively.\\

We define
\begin{equation}
C_1=\{(\chi_1(y,\eta),y,\eta):(y,\eta)\in \mathbb{T}^n\times D\}
\end{equation}
and the flipped graph
\begin{equation}
C_1'=\{(x,y,\xi,\eta):(x,\xi)=\chi_1(y,-\eta)\}
\end{equation}
which is a Lagrangian submanifold of $T^*M$.

Because $\chi_1$ is exact-symplectic, we can quantise this map as a semiclassical Fourier integral operator which is a semiclassical Lagrangian distribution associated with $C_1'$ \cite{popov2} \cite{duistermaatpaper}.\\

In this construction, the Maslov class of the tori $\{\Lambda_\omega:\omega\in\Omega_\kappa\}$ (as defined in Section 3.4 of \cite{duistbook}) can be identified with elements of $\vartheta\in H^1(\T^n;\Z)$ via the symplectic map $\chi_0\circ\chi_1:\T^n\times D\rightarrow T^*M$.

Following \cite{popov2} and \cite{colin}, we can then associate a smooth line bundle $\L$ over $\T^n$ with the class $\vartheta$, such that smooth sections $f\in\mathcal{C}^\infty(\T^n,\L)$ can be canonically identified with smooth functions $\tilde{f}\in\mathcal{C}^\infty(\R^n,\C)$ satisfying the quasiperiodicity condition
\begin{equation}
\label{maslovcanonical}
\tilde{f}(x+2\pi p)=\exp\left(\frac{i\pi}{2}\langle\vartheta,p\rangle\right)\tilde{f}(x)
\end{equation}
for all $p\in \Z^n$.

We now construct this semiclassical Fourier integral operator microlocally. First we need to parametrise $C_1'$ locally by phase functions. To this end, we consider a fixed $\zeta_0=(x_0,y_0,\xi_0,\eta_0)\in C_1'\subset T^*(M\times \mathbb{T}^n)$.
We fix an analytic chart $U_0$ about $x$. Then the implicit function theorem shows that there exists a unique $\phi\in\mathcal{C}^\omega(U_1,U_2)$, where $U_1$ is a local chart in the torus and $U_1\times U_2$ is a rectangular neighbourhood of $(y_0,\xi_0)$ such that

\begin{equation}
C_1=\{(\partial_\xi \phi,\xi,y,\partial_y\phi)\}
\end{equation}
and
\begin{equation}
\det\left(\partial^2_{y\xi}\right)\neq 0
\end{equation}
and
\begin{equation}
\phi(y_0,\xi_0)=x_0\cdot \xi_0-f(\zeta_0)
\end{equation}

where $f$ is the function on $C_1'$ such that $df=i^*\alpha$, where $i:C_1'\rightarrow T^*(M\times\mathbb{T}^n)$ is the inclusion and $\alpha$ is he canonical one form.

Then $\Psi(x,y,\xi)=x\cdot \xi-\phi(y,\xi)$ parametrises $C_1'$ locally in the sense that on the set $O_\Psi=\{(x,y,\xi):\partial_\xi\Psi=0\}$, we have
\begin{itemize}
\item $\textrm{rank}(\partial_{(x,y,\xi)}\partial_\xi \Psi)=n$
\item $(x,y,\xi)\mapsto (x,y,\partial_x\Psi,\partial_y\Psi)\textrm{ is a local diffeomorphism}$.
\end{itemize}
Additionally, we have $\Psi(x_0,y_0,\xi_0)=f(\zeta_0)$.\\

We are now ready to define semiclassical Fourier integral operators associated to $C_1'$ mapping
\begin{equation}
\mathcal{C}^\infty(\mathbb{T}^n,\Omega^{1/2}\otimes \mathbb{L})\rightarrow \mathcal{C}_0^\infty(M,\Omega^{1/2}).
\end{equation}
We fix $\sigma>1,\ell=(\sigma,\sigma,2\sigma-1)$ and choose $a\in S_\ell((U_0\times U_1)\times U_2)$ as in Definition \ref{selldef}.\\

Motivated by equation \eqref{maslovcanonical}, we extend $a$ to
\begin{equation}
\tilde{a}(x,y+2\pi p,\xi;h)=e^{-\frac{i\pi}{2}\vartheta\cdot p}a(x,y,\xi;h)\textrm{ for }(x,y,\xi)\in U\times \mathbb{R}^n
\end{equation}
and extend $\Psi$ periodically to $U_0\times (U_1+2\pi \mathbb{Z}^n)\times U_2$.\\

Then given a section  $a\in\mathcal{C}^\infty(\mathbb{T}^n,\mathbb{L})$, we define
\begin{equation}
\label{fioint}
T_hu(h):=(2\pi h)^{-n}\int_{\mathbb{R}^n}\int_{U_1}e^{i\Psi(x,y,\xi)/h}\tilde{a}(x,y,\xi;h)\tilde{u}(y)\, dy\, d\xi
\end{equation}
noting the invariance of this integral under the translations
\begin{equation}
U_1\mapsto U_1+2\pi p \textrm{ where }p\in\mathbb{Z}^n.
\end{equation}

The class of semiclassical Fourier integral operators is then the set of finite sums of operators given microlocally by \ref{fioint}. As in the standard theory of H\"{o}rmander \cite{hormander} \cite{duistbook}, the above definition is independent of our choices of parametrising phase functions.\\

The principal symbol of $T_h$ is $e^{if(\zeta)}\Upsilon(\zeta)$, where $\Upsilon$ is a smooth section in 
\begin{equation}
\Omega^{1/2}(C_1')\otimes M_{C_1'}\otimes \pi_{2}^*(\mathbb{L}')
\end{equation}
where $\pi_2$ is the canonical projection onto the torus.\\

We can trivialise the half-density bundle by pulling back the canonical half-density on $\mathbb{T}^n\times D$ by the corresponding canonical projection, and we can canonically identity $\pi_2^*(\mathbb{L}')$ with the dual $M_{C_1'}'$ of the Maslov bundle as is done in \cite{colin}.\\

This allows us to canonically identify the principal symbol for a semiclassical Fourier interal operator in this setting with a function $b\in\mathcal{C}^\infty(C_1')$.\\

Additionally, we have
\begin{equation}
a_0(\partial_\xi\phi(y,\xi),y,\xi,-\partial_y\phi(y,\xi))=a_0(\partial_\xi'(y,\xi),y,\xi)|\det(\partial^2_{y\xi}\phi)|^{-1/2}
\end{equation}
where $a_0$ is the leading term in the amplitude corresponding to a given microlocal expression \ref{fioint} for the operator.\\

We are now in a position to state the main theorems in this section.

\begin{thm}
We can choose a semiclassical Fourier integral operator $T_{1h}:\mathcal{C}^\infty (\mathbb{T}^n,\Omega^{1/2}\otimes \mathbb{L})\rightarrow \mathcal{C}_0^\infty(M,\Omega^{1/2})$ with principal symbol equal to $1$ in a neighbourhood of the pullback of the the invariant tori of $\tilde{H}=H\circ\chi_1(\phi,I)$ under the canonical projection such that
\begin{enumerate}
\item $Q_h=T_{1h}^*T_{1h}$ is a semiclassical pseudodifferential operator in $\mathcal{C}^\infty(\mathbb{T}^n,\mathbb{L})$;
\item $Q_h$ has vanishing sub-principal symbol;
\item $P_h^1=T_{1h}^{*}P_hT_{1h}$ is a semiclassical pseudodifferential operator in $\mathcal{C}^\infty(\mathbb{T}^n,\mathbb{L})$;
\item $T_{1h}$ is microlocally invertible in a neighbourhood of the pullback of the union of invariant tori $\Lambda\subset T^*M$ by $\chi_1$, with $P_h^1=T_{1h}^{-1}P_hT_{1h}+h^2R_h$ where $R_h$ is a semiclassical pseudodifferential operator in $\mathcal{C}^\infty(\mathbb{T}^n,\mathbb{L})$;
\item The principal symbol of $P_h^1$ is equal to $H\circ\chi_1$ and its subprincipal symbol vanishes;
\end{enumerate}
and all of the involved pseudodifferential operators have symbols in $S_\ell(\mathbb{T}^n\times D)$.
\end{thm}

\begin{thm}
\label{fioconjthm}
Suppose $T_{2h}(t):L^2(\mathbb{T}^n,\mathbb{L})\rightarrow L^2(\mathbb{T}^n,\mathbb{L})$ is a semiclassical Fourier integral operator associated to the canonical relation graph of the Birkhoff normal form transformation $\chi(t):(y,I)\mapsto (x,\xi)$ with kernel given by
\begin{equation}
(2\pi h)^{-n}\int e^{i((x-y)\cdot I +\phi(x,I;t))/h}b(x,I;h)\, dI
\end{equation}
where
\begin{enumerate}
\item $\phi(x,I;t)=\Phi(x,I;t)-\langle x, I\rangle$;
\item $\Phi\in G^{1,s,s}(\mathbb{T}^n\times D\times (-1,1))$ is a generating function of the Birkhoff normal form canonical transformation;
\item $b\in S_{\tilde{\ell}}(\mathbb{T}^n\times D)$ where $\tilde{\ell}=(\sigma,\mu,\sigma+\mu-1)$ and $\mu>s=\tau'+1 >\sigma>1$;
\item The principal symbol is equal to $1$ in a neigbourhood of $\mathbb{T}^n\times D$.
\end{enumerate}
Then $T_{2h}(t)^*T_{1h}^*P_hT_{1h}T_{2h}(t)$ is a family of semiclassical pseudodifferential operator with symbol in $S_{\tilde{\ell}}$ smoothly varying in the parameter $t$. The principal symbol of $\tilde{P}(t)$ is $\tilde{H}(t)=H(t)\circ\chi(t)$ and its subprincipal symbol vanishes.

Moreover, the conjugating operator $T_h:=T_{1h}T_{2h}(t)$ is microlocally invertible in a neigbourhood of the pullback of the family $\Lambda$ of invariant tori by $\chi_1\circ\chi(t):\T^n\times D \rightarrow T^*M$, and we have
\begin{equation}
T_h^*=T_h^{-1}+O(h^2)
\end{equation}
for this microlocal inverse.
\end{thm}
Proof of these two results can be found in \cite{popov2} (Section 1). In light of Theorem \ref{fioconjthm}, we make the following definition.

\begin{equation}
\label{tildepdef}
\tilde{P}_h(t)=T_h^{-1}(t)P_h(t)T_h(t).
\end{equation}
\section{Conjugation by an elliptic h-PDO}
From Section \ref{fioconj}, we have conjugated $P_h(t)$ to a family of self-adjoint semiclassical operators $\tilde{P}_h(t)$ with symbol $\tilde{p}\in S_{\tilde{\ell}}(\mathbb{T}^n\times D)$ satisfying the flatness condition \eqref{qbnfflatness} to order $h^2$, where $\tilde{\ell}=(\rho,\rho',\rho+\rho'-1)$. That is to say, the formal summation of $\tilde{p}$
\begin{equation}
\sum_{j=0}^\infty \tilde{p}_j(\theta,I;t)h^j
\end{equation}
satisfies
\begin{equation}
\tilde{p}_0(\theta,I;t)=K_0(I;t)+R_0(\theta,I;t)
\end{equation}
and
\begin{equation}
\tilde{p}_1(\theta,I;t)=0.
\end{equation}




The next step of the proof of Theorem \ref{main2} is the improvement of the order of the flatness condition by composition with a suitable elliptic semiclassical pseudodifferential operator $A_h(t)=Id+O(h)$ with symbol
\begin{equation}
a(t)=\sum_{j=1}^\infty a_j(t).
\end{equation}
To motivate the method, we suppose that a quantum Birkhoff normal form $P_h^0$ exists in the sense of Theorem \ref{main2}. Our current operator $\tilde{P}_h$ is equal to $P_h^0$ up to order $h^2$ by construction. Hence, we have
\begin{eqnarray}
T_h(t)A_h(t)\tilde{P}_h(t)&=& T_h(t)\tilde{P}_h(t)A_h(t)+ T_h(t)[A_h(t),\tilde{P}_h(t)]\\
&=& P^1_h(t)T_h(t)A_h(t)+h^2T(t)B(t)A(t)+T_h(t)[A_h(t),\tilde{P}_h(t)].
\end{eqnarray}
for some semiclassical pseudodifferential operator $B_h(t)$ in the symbol class $S_{\tilde{\ell}}(\mathbb{T}^n\times D)$.
From \eqref{compform}, the symbol of the commutator is equal to
\begin{equation}
-(\partial_\theta^\alpha a_1 \partial_I^\alpha \tilde{p}_0)h^2=-\mathcal{L}_{\omega{I;t}}a_1
\end{equation}
where $\mathcal{L}_\omega=\langle \omega,\partial_\theta\rangle a_1(\theta,I;t)$. Thus to improve the order of the flatness condition, it suffices to choose $a_1$ solving the homological equation
\begin{equation}
\label{homeqex}
\mathcal{L}_{\omega(I;t)}a_1=b_0
\end{equation}
where $b_0$ denotes the principal symbol of $B_h(t)$. Indeed, if \eqref{homeqex} is solvable, then we have
\begin{equation}
T_h(t)A_h(t)P_h(t)=P_h^0(t)T_h(t)A_h(t)+O(h^3).
\end{equation}

Extending this idea, it is shown by Popov \cite{popov2} that we can choose higher order terms of the symbol $a$ in an iterative fashion by the solution of such a homological equation for each power of $h$ that we gain. The consequence is the following theorem.
\begin{thm}
\label{symbolconst}
There exists $a,K^0,r\in S_\ell(\mathbb{T}^n\times D)$ where $\ell=(\rho,\mu,\nu)$ such that
\begin{equation}
\label{aexpansion}
a(\theta,I;t,h)\sim \sum_{j=0}^\infty a_j(\theta,I;t)h^j
\end{equation}
\begin{equation}
\label{Kexpansion}
K^0(I;t,h)\sim \sum_{j=0}^\infty K_j(I;t)h^j
\end{equation}
and
\begin{equation}
r(\theta,I;t,h)\sim \sum_{j=0}^\infty r_j(\theta,I;t)h^j
\end{equation}
where $a_0=1,r_0=R_0,K_1=0,$ and 
\begin{equation}
\tilde{p}\circ a-a\circ K^0\sim r.
\end{equation}
where each $r_j(\theta,I;t)$ is flat in $I$ on $\T^n\times E_\kappa(t).$
\end{thm}

The symbol $K^0$ in the statement of theorem corresponds to the sought symbol $K^0$ in Theorem \ref{main2}, while the symbol $R^0$ is then constructed by solving $a\circ R^0=r$, which is possible by ellipticity.

We shall detail the technicalities of the solution of the homological equation in Gevrey classes in Section \ref{homeqsec}, and consequently outline the full construction of the full symbol of $a$. For now, we are in a position to explain how Theorem \ref{main2} follows from Theorem \ref{symbolconst} and Theorem \ref{fioconjthm}.

\begin{proof}[Proof of Theorem \ref{main2}]
From the definition \eqref{tildepdef} of $\tilde{P}_h(t)$, it follows that

\begin{eqnarray}
P_h(t)T_h(t)A_h(t)&=&T_h(t)T_h^{-1}(t)P_h(t)T_h(t)A_h(t)\\
&=& T_h(t)\tilde{P}_h(t)A_h(t)\\
&=& T_h(t)A_h(t)(K_h^0(t)+R_h^0(t))+O(h^\infty).
\end{eqnarray}

Thus, for
\begin{equation}
V_h(t):=T_h(t)A_h(t),
\end{equation}
$V_h(t)$ satisfies all of the desired properties of $U_h(t)$ in Theorem \ref{main2}, besides the condition of being microlocally unitary.\\

This $V_h$ will not in general be unitary, so we define the semiclassical pseudodifferential operator $W_h=V_h^*V_h$ with formal symbol
\begin{equation}
\sum_{j=0}^\infty w_j(\theta,I)h^j\in FS_\ell(\mathbb{T}^n\times D).
\end{equation}

Then $w_0=1$ and we have
\begin{lem}
\label{lemblah}
For each $j$, $p_j^0(I)$ is real valued on $E_\kappa(t)$ and $w_j(\theta,I)$ is $\theta$-independent for $I\in E_\kappa(t)$.
\end{lem}
A proof of this lemma is contained in \cite{popov2}.


If we now define $Q_h(t)=W_h^{-1/2}(t)$, Lemma \ref{lemblah} implies that its symbol $q\in S_\ell(\mathbb{T}^n\times D)$ satisfies
\begin{equation}
q(\theta,I;t,h)=q(0,I;t,h)
\end{equation}
to infinite order at $\mathbb{T}^n\times E_\kappa(t)$, and choosing $U_h(t)=V_h(t)\circ Q_h(t)$ completes the construction of a unitary operator satisfying the requirements in Theorem \ref{main2}.
\end{proof}

\section{Homological equations and the proof of Theorem \ref{symbolconst}}
\label{homeqsec}

In this section, we outline the proof of Theorem \ref{symbolconst}. The key ingredient is the following result which solves the homological equation \eqref{homeqn}.
\begin{thm}
\label{homeqthm}
Suppose $f(\cdot,\cdot;t)\in G^{\rho,\mu}(\mathbb{T}^n\times D)$ satisfies the estimate
\begin{equation}
\label{homeqhyp}
|\partial_\theta^\alpha\partial_I^\beta f(\theta,I;t)|\leq d_0 C^{|\alpha|+\mu|\beta|}\Gamma(\rho|\alpha|+\mu|\beta|+q)
\end{equation}
uniformly in the smooth parameter $t\in (-1,1)$ for some $q>0$ and some $C\geq 1$ and that for each $I\in D$, we have
\begin{equation}
\label{homeqavg}
\int_{\T^n}f(\theta,I;t)\, d\theta=0.
\end{equation}

Then for any smooth family $\omega(\cdot;t) \in G^{\rho'}_{L_0}(D,\Omega)$ there is a solution $u(\cdot,\cdot;t)\in G^{\rho,\mu}(\T^n\times D)$ to the equation
\begin{eqnarray}
\label{homeqn}
\mathcal{L}_\omega u(\theta,I;t)&=&f(\theta,I;t)\quad (\theta,I)\in \T^n\times E_\kappa(t)\\
u(0,I;t)&=&0\quad I\in D
\end{eqnarray}
where $\mathcal{L}_\omega=\langle \omega(I;t),\frac{\partial}{\partial \theta}\rangle $.
Moreover, $u$ is smooth in the parameter $t$ and satisfies the estimate
\begin{equation}
|\partial_\theta^\alpha \partial_I^\beta u(\theta,I;t)|\leq Ad_0C^{n+\tau+|\alpha|+\mu|\beta|+1 }\Gamma(\rho|\alpha|+\mu|\beta|+\rho(n+\tau+1)+q)
\end{equation}
where $A$ depends only on $n,\rho,\tau$ and $\mu$.
\end{thm}
\begin{proof}
We begin by taking the Fourier expansions of $f$ and $u$ in the $\theta$-variable. That is, we have
\begin{equation}
f(\theta,I;t)=\sum_{k\in \Z^n} e^{i\langle k,\theta\rangle} f_k(I;t)
\end{equation} 
and
\begin{equation}
\label{ufourierseries}
u(\theta,I;t)=\sum_{k\in \Z^n}e^{i\langle k, \theta\rangle} u_k(I;t).
\end{equation}
where $u_k$ is to be determined.

From \eqref{homeqavg}, we have that $f_0=0$. We set $u_0=0$ and assemble the solution $u(\theta,I;t)$ to the homological equation by choosing suitable Fourier coefficients.

We first fix $\delta>0$ so that 
\begin{equation}
\label{deltabound}
(2+\rho)\delta \leq \mu-\rho(\tau+1)-1
\end{equation}
and choose $\psi \in G^{1+\delta}(\R)$ such that $0\leq \psi \leq 1$ and
\begin{equation}
\psi(x)=\begin{cases} 0 &\mbox{if } x\geq \kappa/2 \\ 
1 & \mbox{if } x \leq \kappa/4 \end{cases}.
\end{equation}
We note that this is possible from the construction of Gevrey bump functions in \cite{gevbump}.

We then define 
\begin{equation}
u_k(I;t):=[\langle \omega(I;t),k\rangle+i\kappa |k|^{-\tau}\psi(|\langle\omega(I;t),k\rangle||k|^\tau)]^{-1}f_k(I;t)
\end{equation}
for $I\in D$ and nonzero $k$. The definition of $E_\kappa(t)=\omega^{-1}(\tilde{\Omega}_\kappa;t)$ immediately implies that
\begin{equation}
g_k(I;t)=\langle\omega(I;t),k\rangle\textrm{ for }I\in E_\kappa(t).\end{equation}
Moreover, since we either have
\begin{equation}
|\langle\omega(I;t),k\rangle||k|^\tau \geq \kappa/4
\end{equation}
or 
\begin{equation}
\kappa |k|^{-\tau}\psi(|\langle\omega(I;t),k\rangle||k|^\tau)=\kappa |k|^{-\tau}
\end{equation}
for any fixed $k\in \Z^n$ and $I\in D$, it follows that
\begin{equation}
\label{gfourierbound}
|\langle \omega(I;t),k\rangle+i\kappa |k|^{-\tau}\psi(|\langle\omega(I;t),k\rangle||k|^\tau)|\geq \kappa |k|^{-\tau}\geq \kappa |k|^{-\tau}/4.
\end{equation}
The rapid decay of Fourier coefficients together with the estimate \eqref{gfourierbound} imply that the Fourier series \eqref{ufourierseries} is convergent and that $u$ is smooth in $t$. It remains to estimate the derivatives of $u$.

Since $\omega(\cdot;t)\in G^{\rho'}=G^{\rho(\tau+1)+1}$, we have
\begin{equation}
|\partial_I^\alpha \langle \omega(I;t),k\rangle |\leq  C_1^{|\alpha|}|k|\alpha!(|\alpha|-1)!^{\rho(\tau+1)}
\end{equation}
for $\tilde{C}_1=L_0\max(c(n,\rho,\tau),\|\omega\|_{L_0})$.
We also have that the function $\psi_k(x)=\kappa |k|^{-\tau}\psi(x|k|^\tau)$ satisfies
\begin{equation}
|\partial_x^p \psi_k(x)|\leq \tilde{C}_2^{p+1}p!^{1+\delta}|k|^{\tau(p-1)}.
\end{equation}

Using the Faa di Bruno formula, we can now bound the derivatives of the second term in 
\begin{equation}
\langle \omega(I;t),k\rangle+i\kappa |k|^{-\tau}\psi(|\langle\omega(I;t),k\rangle||k|^\tau).
\end{equation}
Indeed, we obtain
\begin{eqnarray}
& &|\partial_I^\alpha(\kappa |k|^{-\tau}\psi(|\langle\omega(I;t),k\rangle||k|^\tau))|\\
&\leq & \sum_{p=1}^{|\alpha|}\sum \tilde{C}_1^{|\alpha|}\tilde{C}_2^{p+1}|k|^{\tau p+p-\tau}p!^\delta \alpha! \prod_{j=1}^p (|\beta_j|-1)!^{\rho(\tau+1)}\\
\end{eqnarray}
where the second summation is taken over all $p$-tuples $(\beta_1,\ldots,\beta_p)$ of nonzero multi-indices such that $\sum \beta_j=\alpha$.

From elementary combinatorics, we know that the number of $p$ element multi-indices of order $m$ is given by
\begin{equation}
\binom{m+p-1}{p-1}
\end{equation}
and so the number of such $p$-tuples is
\begin{equation}
\prod_{j=1}^n \binom{\alpha_j+p-1}{p-1}\leq 2^{|\alpha|}2^{n(p-1)}.
\end{equation}
Hence, by taking $C_1$ sufficiently large (dependent on $n,\rho,\tau,\|\omega\|$) and utilising the inequality
\begin{equation}
n^{|\alpha|}=\sum_{\alpha}\frac{|\alpha|!}{\alpha!}\geq \frac{|\alpha|!}{\alpha!}
\end{equation} 
we arrive at an estimate for $g_k$
\begin{equation}
\label{homeqngbound}
|\partial_I^\alpha g_k|\leq C_1^{|\alpha|}\alpha!^{1+\delta}\max_{1\leq p \leq |\alpha|}(|k|^{\tau \rho+\rho-\tau}(|\alpha|-p)^{\rho(\tau+1)})
\end{equation}
for nonzero $k\in \Z^n$ and nonzero multi-indices $\alpha$.

We can then apply the Faa di Bruna formula again together with the estimate \eqref{homeqngbound} to obtain 
\begin{eqnarray}
\label{ginvboundhom}
|\partial^\alpha (g_k(I;t))^{-1}| &\leq & \sum_{p=1}^{\alpha}\frac{\alpha!}{c^{p+1}}\sum \prod_{j=1}^p \frac{|\partial^{\beta_j}g(I;t)|}{\beta_j!}\\
&\leq & \sum_{p=1}^{|\alpha|} \frac{2^{|\alpha|+n(p-1)}}{c^{p+1}}\prod_{j=1}^p |\partial^{\beta_j}g|\\
& \leq & \tilde{C}_0^{|\alpha|+1}\alpha!^{1+\delta}\max_{1\leq j \leq |\alpha|}(|k|^{\tau j+j+\tau}(|\gamma|-j)!^{\rho(\tau+1)})
\end{eqnarray}
for sufficiently large $\tilde{C}_0$, dependent only on $n,\rho,\tau,\|\omega\|_{L_0}$ and our choice of cutoff function $\psi$.

We introduce the following quantity that simplifies the notation of the estimates to come 
\begin{equation}
\langle k\rangle_m:=1+\sum_{j=1}^n |k_j|^m
\end{equation}
for multi-indices $k$.

For $m=1$, we have $|k|\leq \langle k\rangle_1$, and for nonzero $k$ we additionally have $\langle k\rangle_1\leq 2|k|$. By splitting $\langle k\rangle_m$ into its $n+1$ summands and integrating \eqref{homeqhyp} by parts, the estimate $|k|^m\leq n^m\langle k \rangle_m$ yields
\begin{equation}
\label{ffourierboundhom}
|k^\beta\langle k\rangle_m \partial_I^\alpha f_k(I;t)|\leq (n+1)d_0C^{|\beta|+\mu|\alpha|+m}\Gamma(\rho|\beta+\mu|\alpha|+\rho m +q)
\end{equation}
for $m\in\N$ and any multi-indices $\alpha,\beta$. This estimate is uniform in $I\in D$.

We will need to apply \eqref{ffourierboundhom} multiple times with different choices of $m$, so we define
\begin{equation}
m(j)=\lfloor (\tau+\delta)j+\tau\rfloor+j+n+1
\end{equation}
for $j\in \N$. By choosing $\delta<1+\lfloor \tau\rfloor-\tau$, we have the following bound for $m(j)$.
\begin{equation}
\tau j + j + \tau +n+\delta<m(j)\leq (\tau+\delta+1)j+\tau+n+1.
\end{equation}
Setting $W(k):=\langle k\rangle_1^{-n-\delta}$, we proceed to bound
\begin{equation}
A_{k,\alpha,\beta}:=|\langle k\rangle_1^{n+\delta}k^\alpha \partial_I^\beta u_k(I;t)|
\end{equation}
using \eqref{ffourierboundhom}, \eqref{ginvboundhom}. This bound will be made $k$-independent, and so summation will complete the proof.

The Leibniz rule implies
\begin{eqnarray}
\nonumber & &A_{k,\alpha,\beta}\\
\nonumber &\leq & \langle k\rangle_1^{n+\delta}\sum_{\gamma \leq \beta}\binom{\beta}{\gamma}\tilde{C}_0^{|\gamma|+1}\gamma!^{1+\delta}\max_{0\leq j \leq |\gamma|}|k|^{\tau j +\tau +j}(|\gamma-j|)!^{\rho(\tau+1)}\langle k \rangle_{m(j)}^{-1}\left|k^\alpha \langle k \rangle_{m(j)} \partial_I^{\beta-\gamma}(f_k(I))\right|\\
\nonumber &\leq & A\sum_{\gamma \leq \beta}\binom{\beta}{\gamma}\tilde{C}_0^{|\gamma|+1}\gamma!^{1+\delta}\max_{0\leq j \leq |\gamma|}|k|^{\tau j +\tau +j+n+\delta}(|\gamma-j|)!^{\rho(\tau+1)}\langle k \rangle_{m(j)}^{-1}\left|k^\alpha \langle k \rangle_{m(j)} \partial_I^{\beta-\gamma}(f_k(I))\right|\\
\nonumber &\leq & A\sum_{\gamma \leq \beta}\binom{\beta}{\gamma}\tilde{C}_0^{|\gamma|+1}\gamma!^{1+\delta}\max_{0\leq j \leq |\gamma|}\left((|\gamma-j|)!^{\rho(\tau+1)}n^{m(j)}\right)\left|k^\alpha \langle k \rangle_{m(j)} \partial_I^{\beta-\gamma}(f_k(I))\right|\\
\nonumber &\leq & A\sum_{\gamma \leq \beta}\frac{\beta!}{(\beta-\gamma)!}\tilde{C}_0^{|\gamma|+1}\gamma!^{\delta}\max_{0\leq j \leq |\gamma|}\hat{C}_0^{|\gamma|+1}\left((|\gamma-j|)!^{\rho(\tau+1)}\right)\left|k^\alpha \langle k \rangle_{m(j)} \partial_I^{\beta-\gamma}(f_k(I))\right|\\
\nonumber &\leq & Ad_0 \sum_{\gamma \leq \beta}\frac{\beta!}{(\beta-\gamma)!}\hat{C}_0^{|\gamma|+1}\gamma!^\delta\max_{0\leq j \leq |\gamma|}C^{|\alpha|+\mu|\beta-\gamma|+m(j)}(|\gamma|-j)!^{\rho(\tau+1)}\Gamma(s)
\end{eqnarray}
where
\begin{equation}
s:=\rho|\alpha|+\mu|\beta-\gamma|+\rho m(j)+q\leq \rho |\alpha|+\mu |\beta-\gamma|+\rho(\tau+\delta+1)j+\rho (\tau+n+1)+1.
\end{equation}
and $A,\hat{C}_0$ are constants only dependent on $n,\rho,\tau,\delta$.

Stirling's formula implies that $\gamma!^\delta \leq \Gamma(\delta |\gamma|) $ and $(|\gamma|-j)!^{\rho(\tau+1)}\leq \Gamma (\rho(\tau+1)(|\gamma|-j))$. Since $s\geq 1+|\beta-\gamma|\geq 1$, we can estimate
\begin{eqnarray}
& &B_{j,\beta,\gamma}\\
&:= &\frac{\beta!}{(\beta-\gamma)!}\gamma!^\delta(|\gamma|-j)!^{\rho(\tau+1)}\Gamma(s)\\
&\leq & \frac{\beta!}{(\beta-\gamma)!}(C_2C_3)^{|\gamma|} \Gamma(s+\rho(\tau+1)(|\gamma|-j)+\delta|\gamma|)\\
&\leq & s(s+1)\ldots (s+|\gamma|-1)(C_2C_3)^{|\gamma|}\Gamma(s+\rho(\tau+1)(|\gamma|-j)+\delta|\gamma|)\\
&\leq & \label{preveqhom} (C_2C_3)^{|\gamma|}\Gamma(s+\rho(\tau+\delta+1)(|\gamma|-j)+(1+\delta)|\gamma|)
\end{eqnarray}
where the last inequality follows from replacing the first $|\gamma|$ factors in the previous line with larger factors that can be absorbed into the Gamma function.
From \eqref{deltabound}, we have $\mu > 2\delta+\rho(\tau+\delta+1)+1$. Together with \eqref{preveqhom}, we can further estimate
\begin{eqnarray}
B_{j,\beta,\gamma}&\leq & (C_2C_3)^{|\gamma|}\Gamma(\rho|\alpha|+\mu|\beta|+\rho(\tau+n+1)+1+q-\delta|\gamma|)\\
&\leq & \frac{(C_2C_3)^{|\gamma|}\Gamma(\rho|\alpha|+\mu|\beta|+\rho(\tau+n+1)+1+q)}{\Gamma(\delta|\gamma|)}
\end{eqnarray}
for nonzero $\gamma$.

Combining our estimates for $A_{k,\alpha,\beta}$ and $B_{j,\beta,\gamma}$, and redefining $A(n,\rho,\tau)$ appropriately, we obtain
\begin{eqnarray}
& &\nonumber A_{k,\alpha,\beta}\\
&\leq &\nonumber Ad_0\sum_{\gamma\leq \beta}\hat{C}_0^{|\gamma|+1}C^{|\alpha|+\mu|\beta-\gamma|+(\tau+\delta+1)|\gamma|+(\tau+n+1)}\max_j(B_{j,\beta,\gamma})\\
\nonumber&\leq & Ad_0C^{|\alpha|+\mu|\beta|+\tau+n+1}\Gamma(\rho|\alpha|+\mu|\beta|+\rho(\tau+n+1)+1+q))\left(1+\sum_{\gamma\neq 0} \frac{\hat{C}_0^{|\gamma|+1}(C_2C_3)^{|\gamma|}}{\Gamma(\delta|\gamma|)}\right)\\
&\leq & Ad_0C^{|\alpha|+\mu|\beta|+\tau+n+1}\Gamma(\rho|\alpha|+\mu|\beta|+\rho(\tau+n+1)+1+q)).
\end{eqnarray}

Hence, we can conclude that
\begin{eqnarray}
& &\nonumber|\partial_\theta^\alpha\partial_I^\beta u|\leq \sum_{k\in \Z^n} |k^\alpha \partial_I^\beta u_k(I)|\\
&\leq &\nonumber\left[\sum_{k\in \Z^n} \langle k\rangle_{n+\delta}^{-1}\right]\cdot Ad_0C^{|\alpha|+\mu|\beta|+\tau+n+1}\Gamma(\rho|\alpha|+\mu|\beta|+\rho(\tau+n+1)+1+q))\\
&=& Ad_0C^{|\alpha|+\mu|\beta|+\tau+n+1}\Gamma(\rho|\alpha|+\mu|\beta|+\rho(\tau+n+1)+1+q))
\end{eqnarray}
where $A$ depends only on $n,\rho,\tau$ as required.
\end{proof}

\begin{proof}[Completion of proof of Theorem \ref{symbolconst}]
To prove Theorem \ref{symbolconst}, we need to find $a,K^0\in S_\ell(\mathbb{T}^n\times D)$ such that $r=p\circ a - a\circ K^0$ is flat on $E_\kappa(t)$, where $\ell=(\rho,\mu,\nu)$ by continuing the iterative procedure outlined in the discussion preceding the theorem.

This can be done by using the composition formula \eqref{compform} and induction using Theorem \ref{homeqthm} to assemble the symbols $a,K^0$ in terms of their semiclassical expansions \eqref{aexpansion} and \eqref{Kexpansion}.

We begin by setting $a_0=0,K_1=0,r_0=R_0$ and recalling that $p_0(I;t)=K_0(I;t)+R_0(\theta,I;t)$ and $p_1=0$ from the construction of the semiclassical pseudodifferential operator $\tilde{P}_h$ in Theorem \ref{fioconjthm}.

The composition formula \eqref{compform} then implies that 
\begin{equation}
r_0(\theta,I;t)=R_0(\theta,I;t)
\end{equation}
and
\begin{equation}
r_1(\theta,I;t)=a_1(\theta,I;t)R_0(\theta,I;t)
\end{equation}
which are both flat in $I$ at $E_\kappa(t)$ since $R_0$ is.

For $j\geq 2$, the composition formula \eqref{compform} yields
\begin{equation}
\label{rjayexp}
r_j(\theta,I;t)=-i(\mathcal{L}_\omega a_{j-1})(\theta,I;t)+p_j(\theta,I;t)-K_j(I;t)+F_{j1}(\theta,I;t)-F_{j2}(\theta,I;t)
\end{equation}
where
\begin{equation}
F_{j1}(\theta,I;t)=\sum_{s=1}^{j-2}\sum_{r+|\gamma|=j-s}\frac{1}{\gamma !}\partial_I^\gamma p_r(\theta,I;t)\partial_\theta^{\gamma}a_s(\theta,I;t)
\end{equation}
and
\begin{equation}
F_{j2}(\theta,I;t)=\sum_{s=1}^{j-2} a_s(\theta,I;t)K_{j-s}^0(I;t).
\end{equation}

We solve the equations $r_j=0$ by using the expression \eqref{rjayexp} and Theorem \ref{homeqthm}.

Setting
\begin{equation}
K_j(I;t):=(2\pi)^{-n}\int_{\T^n}(p_j(\theta,I;t),+F_j(\theta,I;t))\, d\theta
\end{equation}
we can apply Theorem \ref{homeqthm} to obtain $a_{j-1}\in G^{\rho,\mu}(\T^n\times D)$
such that
\begin{equation}
-i\mathcal{L}_\omega a_{j-1}(\theta,I;t)=K_j(I;t)-p_j(\theta,I;t)-F_{j1}(\theta,I;t)+F_{j2}(\theta,I;t)
\end{equation}
on $\T^n\times E_\kappa(t)$ and
\begin{equation}
\int_{\T^n}a_{j-1}(\theta,I;t)\, d\theta=0
\end{equation}
on $D$ for each $t\in (-1,1)$.

Since our estimates are uniform in $t$, the argument in Section 2.3 of \cite{popovquasis} shows that for each $t\in (-1,1)$, the estimate \eqref{homeqhyp} is inductively satisfied for each application of Theorem \ref{homeqthm}. This allows us to complete the construction of the symbols of $a,K^0$ and $r$ as required.
\end{proof}

\section{Quasimode construction}
\label{gevquasisec}
Having completed the proof of Theorem \ref{main2}, we now complete Popov's construction of a family of Gevrey quasimodes for $P_h(t)$ that are semiclassically supported on a family of nonresonant tori. Moreover, these quasimodes are smooth in the parameter $t\in (-1,1)$.

\begin{defn}
\label{gevmodes}
A $G^\rho$ family of quasimodes $\mathcal{Q}(t)$ for $P_h(t)$ is a family
\begin{equation}
\{(u_m(x;t,h),\lambda_m(t,h)):m\in\mathcal{M}_h(t)\} \subset \mathcal{C}^\infty(M\times \mathcal{D}_h(m))\times \mathcal{C}^\infty(\mathcal{D}_h(m))
\end{equation}
parametrised by $h\in (0,h_0]$ where
\begin{itemize}
\item $\mathcal{M}_h(t)\subset \Z^n$ is a $h$-dependent finite index set;
\item $\mathcal{D}_h(m)=\{t\in (0,1):m\in \mathcal{M}_h(t)\}$
\item each $u(\cdot;t,h)$ is uniformly of class $G^\rho$;
\item $\|P_h(t)u_m(\cdot;t,h)-\lambda_m(t;h)u_m(\cdot;t,h)\|_{L^2}\leq Ce^{-c/h^{1/\rho}} \quad \forall m\in\mathcal{M}_h(t) $;
\item $|\langle u_m(\cdot;t,h),u_l(\cdot;t,h)\rangle -\delta_{ml}|\leq Ce^{-c/h^{1/\rho}} \quad \forall m,l\in\mathcal{M}_h(t) $.
\end{itemize}
\end{defn}

\begin{thm}
\label{gevquasisthm}
Suppose now that $t\in (-1,1)$ is fixed and $S\subset E_\kappa(t)$ is a closed collection of nonresonant actions. For an arbitrary constant $L>1$, we define the index set
\begin{equation}
\label{indexsetref}
\mathcal{M}_h:=\{m\in \Z^n: \textrm{dist}(S,h(m+\vartheta/4))< Lh\}
\end{equation}
where $\vartheta\in \Z^n$ is the Maslov class of any Lagrangian tori $\{\chi(\T^n\times \{I\})\}$ with $I\in S$. Note that this class is independent of choice of torus by the local constancy of the Maslov class.

Then 
\begin{equation}
\{(u_m(x;t,h),\lambda_m(t;h)):m\in\mathcal{M}_h(t)\}:= (U_h(t)e_m,K^0(h(m+\vartheta/4);t,h)
\end{equation}
defines a $G^\rho$ family of quasimodes for $P_h(t)$ that has Gevrey microsupport on the family of tori 
\begin{equation}
\displaystyle
\Lambda_S=\bigcup_{I\in S} \Lambda_{\omega(I;t)}=\bigcup_{I\in S}\chi(\T^n\times \{I\})\subset T^*M
\end{equation}
where $\{e_m\}_{m\in\Z^n}$ is the orthonormal basis of $L^2(\T^n;\mathbb{L})$ associated to the quasiperiodic functions
\begin{equation}
\tilde{e}_m(x):=\exp(i\langle m+\vartheta/4,x \rangle)
\end{equation}
\end{thm}
\begin{proof}
From the definition of the functions $e_m$, it follows that
\begin{eqnarray}
P_h^0(t)(e_m)(\theta)&=& \sigma(P_h^0(t))(\theta,h(m+\vartheta/4))e_m(\theta)\\
&=& (K^0(h(m+\vartheta/4);t,h)+R^0(\theta,h(m+\vartheta/4);t,h))e_m(\theta)\\
&=& (\lambda_m(t;h)+R^0(\theta,h(m+\vartheta/4))e_m(\theta).
\end{eqnarray}

From the definition \eqref{indexsetref} of the index set $\mathcal{M}_h(t)$ and from $\eqref{flatissmall}$, it thus follows that 
\begin{equation}
P_h(t)(U_h(t)e_m)=U_h(t)P_h^0(t)e_m=O(e^{-c/h^{1/\rho}})
\end{equation}
upon an application of Theorem \ref{main2}. This establishes that the $U_h(t)e_m$ are exponentially accurate quasimodes.

The almost-orthogonality of the $U_h(t)e_m$ then follows from the fact that $U_h(t)$ is almost unitary from Theorem \ref{main2}, and the $e_m$ are exactly orthogonal by construction.

This completes the proof.
\end{proof}

These quasimodes are as numerous as we could hope for, indeed the index set $\mathcal{M}_h(t)$ satisfies the local Weyl asymptotic
\begin{equation}
\label{indexsetasymp}
\lim_{h\rightarrow 0}(2\pi h)^n \# \mathcal{M}_h =m(\T^n\times S)= \mu(\Lambda_S)
\end{equation}
where $m$ denotes the $(2n)$-dimensional Lebesgue measure and $\mu$ denotes the symplectic measure $d\xi\,dx$. To see this, we can denote by $U$ the union of $n$-cubes centred at the lattice points in $\mathcal{M}_h$ with side length $h$. The containment
\begin{equation}
\label{latter}
S\subset U \subset \{I:\textrm{dist}(I,S)<\tilde{L}h\}
\end{equation}
for a constant $\tilde{L}$ then yields the claim by monotone convergence of measures, noting that since $S$ is closed we have
\begin{equation}
S=\overline{S}=\cap_{h>0} \{I:\textrm{dist}(I,S)<\tilde{L}h\}.
\end{equation}

In the special case of $S=\{I\}$, we have a family of $G^\rho$ quasimodes with microsupport on an individual torus $\chi(\T^n\times \{I\})$. 




\chapter{Eigenvalue Localisation Results}
\section{Introduction}
\label{6intro}
In this section, we consider a one-parameter family of elliptic semiclassical pseudodifferential operators obtained by quantising a one-parameter family of KAM Hamiltonian systems. As an application of the quasimode construction in Chapter 5, we prove that the quantum dynamics generated by this family of operators can only be quantum ergodic for a Lebesgue null subset of the possible parameter values.\\

Suppose that $M$ is an $n$-dimensional compact boundaryless Riemannian manifold of regularity $G^\rho$. The family of Hamiltonian systems in question is then given by
\begin{equation}
\label{formofhamiltonian}
H(x,\xi;t)=\|\xi\|_g^2+V(x,\xi)+tQ(x,\xi)
\end{equation}
where $H(x,\xi;0)\in G^\rho(T^*M)$ is a completely integrable system satisfying the nondegeneracy condition \eqref{nondegen} for a suitable choice of action-angle coordinates $(\theta,I)\in \T^n\times D$, and $Q\in G^\rho(T^*M)$ is a non-negative and compactly supported symbol with positive quantisation. In particular, these symbols live in $S_\ell(T^*M)$ in the notation of Definition \ref{selldef} with $\ell=(\rho,\mu,\eta)$, with $\rho(\tau+n)+1>\mu>\rho'=\rho(\tau+1)+1$ and $\nu=\rho(\tau+n+1)$ as in Chapter 5.

The quantisation of this Hamiltonian is given by the elliptic operator
\begin{equation}
\label{semschrod}
\mathcal{P}_h(t):=h^2\Delta_g+V(x)+tQ(x,hD).
\end{equation}
The operator $P_h(t)$ then acts on $L^2(M)$ with domain $H^2_h(M)$, the semiclassical Sobolev space defined as in \cite{zworski} Chapter 8. 
The inverse of $P_h(t)$ is compact and self-adjoint, and so the eigenfunctions of $P_h(t)$ form an orthonormal basis for $L^2(M)$ with an unbounded purely discrete spectrum consisting of the eigenvalues $0\leq E_1(t;h)<\ldots.$\\

If the Hamiltonian flow corresponding to the Hamiltonian $H(x,\xi,t)=\sigma(P_h)=|\xi|_g^2+V(x)+tQ(x,\xi)$ were ergodic on the energy surface $\Sigma_E=H^{-1}(E)\subset T^*M$ for a regular value $E$ of the Hamiltonian, then the quantum ergodicity theorem \cite{HRM} asserts that $P_h(t)$ is quantum ergodic on $\Sigma_E$ in the sense of Definition \ref{qesurface}.

However, by rewriting \eqref{formofhamiltonian} locally in the action angle variables $(\theta,I)$ for the completely integrable Hamiltonian $H(x,\xi;0)$, we see that for fixed $Q\in G^\rho_{L_1,L_2}(\T^n\times D)$ with $L_2\geq L_1\geq 1$, the results of Section \ref{mainresultssec} imply that $H(x,\xi;t)$ generates KAM dynamics for every $t\in [0,\delta]$, for sufficiently small $\delta>0$, dependent on $Q$. As KAM dynamics are far from ergodic dynamics in character, I expect that $P_h(t)$ is  typically \emph{not} quantum ergodic, and that there could even exist sequences of eigenfunctions for $P_h(t)$ with semiclassical mass entirely supported on individual invariant tori.\\

As a first result in this area of investigation, we prove that $P_h(t)$ is not quantum ergodic by arguing by contradiction. A key assumption is the existence of what we shall call a \emph{slow torus} for our family of perturbations.
\begin{defn}
\label{slowtorusdef}
A \emph{slow torus} in the energy band $[a,b]$ for the Hamiltonian 
\begin{equation}
H(\theta,I;t)=H^0(I)+tH^1(\theta,I)
\end{equation}
in action-angle coordinates, is a Lagrangian invariant torus $\Lambda_{\omega_0}$ with nonresonant frequency $\omega_0\in \tilde{\Omega}_\kappa$ and energy ${H^0(I(\omega_0))=H^0(\nabla g^0(\omega_0))\in [a,b]}$ in the notation of Theorem \ref{main1} that satisfies
\begin{equation}
\label{slowtorus}
(2\pi)^{-n}\int_{\T^n} H^1(\theta,I(\omega_0))\, d\theta < \inf_{E\in [a,b]}\frac{1}{\mu_E(\Sigma_E)}\int_{\Sigma_E} Q(x,\xi)\, d\mu_E
\end{equation}
at $t=0$.
\end{defn}

\begin{remark}
It should be mentioned that the Lagrangian tori that foliate the completely integrable space $\mathcal{D}_t\setminus U_t$ in Chapter 3 are analogously \emph{slow} with respect to the perturbation parameter $t$, making the spectral non-concentration arguments in Section 6.3 and Section 3.5 rather similar in spirit. 

However, in the KAM setting we do not have a result analogous to Theorem \ref{galkowski} that precisely describes the restriction of semiclassical measures to the non-integrable region. This makes it more difficult to prove results on eigenfunction localisation in KAM systems compared to mixed systems, at least via the methods used in this thesis.
\end{remark}

The assumption \eqref{slowtorus} holds in a small energy band $[a,b]$ whenever the completely integrable Hamiltonian system $H(x,\xi;0)$ posesses a nonresonant torus on some energy surface $\Sigma_E$ on which the perturbation $Q$ is smaller than it is on average on $\Sigma_E$. We call such a torus a \emph{slow torus} to emphasise the key feature that the quasi-eigenvalues associated to such a torus will increase at a slower rate than the typical increase of eigenvalues at the same energy. The intuition behind this stems from the Hadamard variational formula \eqref{qequantity}, and the fact that the associated quasimodes localise onto $\Lambda_{\omega_0}$ in phase space.\\

This condition is rather mild, and will typically be satisfied for some torus provided that $Q$ is not constant on energy surfaces. An example of a Hamiltonian system with a slow torus could be explicitly constructed by taking a completely integrable surface of revolution, and choosing $Q(x,\xi)$ to be small near a collar containing a stable geodesic.

We are now in a position to state our main theorem.
\begin{thm}
\label{nonqe}
Suppose $M$ is a compact boundaryless $G^\rho$-smooth Riemannian manifold, and the Hamiltonian $H(x,\xi;t)$ is regular on the energy interval $[a,b]$ and can be written as 
\begin{equation}
H(x,\xi;t)=H^0(I)+tH^1(\theta,I)\in G^{\rho,\rho,1}_{L_1,L_2,L_2}(\T^n\times D \times (-1,1))
\end{equation}
in a subdomain $\T^n\times D\times (-1,1)\subset \T^n\times D^0\times (-1,1)$ of local action-angle variables for the non-degenerate completely integrable Hamiltonian $H^0\in G^\rho_{L_0}(D^0)$ with Legendre transform $g^0\in G^\rho_{L_0}(\Omega^0)$.

Suppose further that $a<b$ and there exists a \emph{slow torus} $\Lambda_{\omega_0}$ in the energy band $[a,b]$.




Then there exist $\epsilon,\delta>0$ such that for almost all $t\in[0,\delta]$ the quantisation $P_h(t)$ of $H(x,\xi;t)$ is quantum ergodic over the energy surface $\Sigma_E$ for at most a Lebesgue  measure $(1-\epsilon)(b-a)$ subset of $E\in [a,b]$.
\end{thm}
We prove Theorem \ref{nonqe} by using an argument based on controlling the flow speed of eigenvalues in the parameter $t$ as in Chapter 3. 

We begin by using the slow torus condition and choosing $\epsilon_1>0$ sufficiently small so that
\begin{equation}
(2\pi)^{-n}\int_{\T^n} H^1(\theta,I(\omega_0))\, d\theta < \inf_{E\in [a,b]}\frac{1}{\mu_E(\Sigma_E)}\int_{\Sigma_E} Q(x,\xi)\, d\mu_E-2\epsilon_1
\end{equation}
is satisfied at $t=0$.

We next observe that it suffices to prove the theorem for a small interval $[a,b]$ containing the energy $H^0(I(\omega_0))$, and so we can scale our interval $[a,b]$ by a small factor $\lambda$ to ensure that the condition
\begin{equation}
\label{qplusqminus}
\sup_{E\in[a,b]}\frac{1}{\mu_E(\Sigma_E)}\int_{\Sigma_E} Q\, d\mu_E - \inf_{E\in [a,b]}\frac{1}{\mu_E(\Sigma_E)}\int_{\Sigma_E} Q\, d\mu_E =: Q_+(0)-Q_-(0) < \epsilon_2<\epsilon_1.
\end{equation}
is satisfied.

Now Theorem \ref{bnf2} applies, and we obtain a family of symplectic maps $\chi\in G^{\rho,\rho',\rho'}(\T^n\times D\times (-1/2,1/2), \T^n\times D)$ and a family of diffeomorphisms $\omega\in G^{\rho',\rho'}(D\times (-1/2,1/2),\Omega)$ such that 
\begin{equation}
\label{BNF6}
H(\chi(\theta,I;t);t)=K(I;t)+R(\theta,I;t) 
\end{equation}
where $R$ is flat in $I$ at the nonresonant actions $I\in E_\kappa(t)$.

Defining the action map $I\in G^{\rho',\rho'}(\Omega\times (-1/2,1/2))$ implicitly by 
\begin{equation}
\omega:=\omega(I(\omega;t) ; t)
\end{equation}
we can specify the action coordinates of a nonresonant torus at $t\in [0,\delta]$ in the Birkhoff normal form furnished by $\chi(\cdot,\cdot;t)$.

We begin by expanding the nonresonant slow torus $\Lambda_{\omega_0}$ to a positive measure family of slow tori.
\begin{prop}
\label{fattori}
For any $0<c<\min(\frac{b-a}{2},\textrm{dist}(H^0(I(\omega_0)),[a,b]^c)$, there exists $r>0$ and $\delta>0$ such that for any $\omega\in \overline{\Omega}:= B(\omega_0,r)\cap \tilde{\Omega}_\kappa$, the torus $\Lambda_\omega=\chi(\T^n\times \{I(\omega,t)\})$ has energy  
\begin{equation}
\label{inenergywindow}
K(I(\omega;t),t)\in[K(I(\omega_0;t);t)-c,K(I(\omega_0;t);t)+c]
\end{equation}
for all $t\in [0,\delta]$ and the torus $\Lambda_{\omega_0}(t)$ has energy 
\begin{equation}
K(I(\omega_0;t);t)\in[a+c,b-c]
\end{equation}
for all $t\in[0,\delta]$.

In particular, the family of tori
\begin{equation}
\Lambda(t):=\bigcup_{\omega\in \overline{\Omega}} \Lambda_{\omega}
\end{equation}
is a positive measure family of tori entirely contained within the energy band $[a,b]$.

Moreover, $r$ and $\delta$ can be chosen small enough to ensure
\begin{eqnarray}
\nonumber (2\pi)^{-n}\int_{\T^n} H^1(\theta,I(\omega;t))\, d\theta &<&  (2\pi)^{-n}\int_{\T^n} H^1(\theta,I(\omega))\, d\theta+\epsilon_1\\ &<& \label{slowpostori}Q_--\epsilon_1.
\end{eqnarray}
for each $\omega\in \overline{\Omega}$ and each $t\in [0,\delta]$.

We can also choose $\delta>0$ small enough to ensure that
\begin{equation}
\label{ergodicboundsclose}
Q_+(t)-Q_-(t)<2\epsilon_2.
\end{equation}
for all $t\in[0,\delta]$.
\end{prop}
\begin{proof}
From the regularity of $\chi,I,$ and $K$ established in Theorem \ref{main1}, it follows that we can take $r<Lc$ for some constant $L>0$ to ensure that \eqref{inenergywindow} is satisfied at $t=0$ for sufficiently small $c>0$. Similarly, we can ensure that 
\begin{equation}
(2\pi)^{-n}\int_{\T^n} H^1(\theta,I)\, d\theta <  (2\pi)^{-n}\int_{\T^n} H^1(\theta,I(\omega_0))\, d\theta+\epsilon_1/2
\end{equation}
holds for $|I-I(\omega_0)|$ sufficiently small.
\eqref{slowpostori} is satisfied at $t=0$. In particular this implies that 
\begin{equation}
(2\pi)^{-n}\int_{\T^n} H^1(\theta,I(\omega))\, d\theta <  (2\pi)^{-n}\int_{\T^n} H^1(\theta,I(\omega_0))\, d\theta+\epsilon_1/2
\end{equation}
for all $\omega\in \overline{\Omega}=B(\omega_0,r)\cap \tilde{\Omega}_\kappa$ upon taking $r$ sufficiently small.

The regularity of $\chi,I$ and $K$ in the parameter $t$ then allow us to then deduce that \eqref{inenergywindow} and \eqref{slowpostori} are satisfied for $t\in[0,\delta]$, for sufficiently small $\delta>0$ and for each $\omega \in \overline{\Omega}$.

Finally, the estimate \eqref{ergodicboundsclose} for small $\delta$ follows from the regularity of 
\begin{equation}
\frac{1}{\mu_E(\Sigma_E)}\int_{\Sigma_E} Q d\mu_E
\end{equation}
in $t$.
\end{proof}

From the Birkhoff normal form \eqref{BNF6}, we construct a quantum Birkhoff normal form as in Chapter 5. The constructions of Section \ref{gevquasisec} then provide us with a family of quasimodes localising onto $\Lambda(t)$ by taking $S(t)=\{I(\omega;t):\omega\in \overline{\Omega}\}$ and defining the index set $\mathcal{M}_h(t)$ as in \eqref{indexsetref}.

We can now define $t$-dependent energy bands
\begin{equation}
\mathcal{I}(t):=[K_0(I(\omega_0,t),t)-c,K_0(I(\omega_0,t),t)+c]\subseteq [a,b]
\end{equation}
and the union of $h^\alpha$-width quasi-eigenvalue windows
\begin{equation}
\label{wthdef}
W(t;h):=\mathcal{I}(t)\cap \bigcup_{m\in \mathcal{M}_h(t)} [K^0(h(m+\vartheta/4),t;h)-h^\alpha,K^0(h(m+\vartheta/4),t;h)+h^\alpha]
\end{equation}
where $\alpha>2n$ and $K_0=K$ and $K^0$ are as in Theorem \ref{main2}.

For the sake of brevity, we introduce the notation 
\begin{equation}
\mu_m(t;h):=K^0(h(m+\vartheta/4),t;h)
\end{equation} 
for the quasi-eigenvalues under consideration.

The number of exact eigenvalues lying in the union of quasi-eigenvalue windows $W(t;h)$ is given by
\begin{equation}
\label{Ndef}
N(t;h):=\#\{E_k(t;h)\in W(t;h)\}.
\end{equation}

Next we obtain asymptotic estimates for the number of eigenvalues and the number of quasi-eigenvalues contained in the interval $\mathcal{I}(t)$ as $h\rightarrow 0$.

\begin{prop}
We have the asymptotic estimate
\begin{equation}
\label{indexsetcount}
\#\mathcal{M}_h(t)\sim (2\pi h)^{-n}\mu(\T^n\times \{I(\omega,t):\omega\in \overline{\Omega}\}).
\end{equation}
for each $t\in [0,\delta]$.

Furthermore, we have
\begin{eqnarray}
\label{eigencount}
& &\limsup_{h\rightarrow 0}(2\pi h)^n\#\{k\in \N: E_k(t;h)\in [a,b]\textrm{ for some }t\in [0,\delta]\}\\
&\leq& \mu(\{(x,\xi):H(x,\xi;0)\in [a-M\delta,b]\})
\end{eqnarray}
where $M=\|Q(x,hD)\|_{L^2\rightarrow L^2}$.

Here $\mu$ denotes the symplectic measure $d\xi\, dx$ on $T^*M$.
\end{prop}

\begin{proof}
The estimate \eqref{indexsetcount} is a consequence from \eqref{indexsetasymp}, and \eqref{eigencount} follows from \eqref{globalboundgevrey} and an application of the Weyl law \eqref{zworskiweyl}.
\end{proof}

A crucial fact is that if we scale the interval $[a,b]$ and $c$ about $K_0(I(\omega_0,0),0)$ by a factor of $\lambda>0$, then  
the $r>0$ chosen in Proposition \ref{fattori} can be chosen to be scaled by the factor $\lambda>0$ as well. Doing this for small $\lambda>0$ will make the quantity 
\begin{equation}
\sup_{t\in [0,\delta]}Q_+(t)-\inf_{t\in [0,\delta]}Q_-(t)< 3\epsilon_2=o_\lambda(1).
\end{equation}
However we will still have 
\begin{equation}
\limsup_{h\rightarrow 0}\frac{\#\{k\in \N:E_k(t)\in [a,b]\textrm{ for some }t\in [0,\delta]\}}{\inf_{t\in [0,\delta]}\#\mathcal{M}_h(t)}
\end{equation}
uniformly bounded in $t$ and $\lambda$.


\section{Eigenvalue and quasieigenvalue variation}
\label{flowspeedsec}

For each fixed $h>0$, the operators $P_h(t)$ comprise an holomorphic family of type A in the sense of \cite{kato} and so we can choose eigenvalues and corresponding eigenprojections holomorphic in the parameter $t$. Thus if at each time $t$ we order our eigenpairs $E_k(t;h)$ in order of increasing energy, by holomorphy it follows that $E_k$ will be continuous, piecewise smooth, and have multiplicity $1$ for all but finitely many $t\in [0,\delta]$. On this cofinite set, we have
\begin{eqnarray}
\dot{E}_k(t;h)&=&\nonumber\langle\dot{P}_h(t)u_k(t;h),u_k(t;h)\rangle\\
&=& \label{qequantity} \langle Q(x,hD)u_k(t;h),u_k(t;h) \rangle 
\end{eqnarray}
from $(u_k)$ being an orthonormal basis. We will control \eqref{qequantity} using an assumption of quantum ergodicity.

Suppose, for the sake of contradiction, that for every $\epsilon>0$ there exists a positive measure subset $\mathcal{B}_\epsilon$ of $[0,\delta]$ such that for each $t\in \mathcal{B}_\epsilon$, $P_h(t)$ is quantum ergodic on a set of energies $\mathcal{E}_t\subset [a,b]$ with $m(\mathcal{E}_t)>(b-a)(1-\epsilon/2)$.

Analogously to Section \ref{eigflowsec}, we can use quantum ergodicity to deduce bounds on the speed at which the eigenvalues monotonically increase in $t$ under the assumption of eigenfunction equidistribution.

We first note that we have a global in time bound 
\begin{equation}
\label{globalboundgevrey}
E_k'(t)\leq M=\|Q\|_{L^2\rightarrow L^2}<\infty
\end{equation}
from differentiation of the expression
\begin{equation}
E_k(t)=\langle P_h(t)u_k(t),u_k(t) \rangle.
\end{equation}

Now, our first result estimate on eigenvalue variation is one localised to individual energy surfaces on which we have quantum ergodicity.

\begin{prop}
\label{pointwisesurface}
For each $t\in \mathcal{B}_\epsilon$, each $E\in \mathcal{E}_t$, and each $\hat{\epsilon}>0$, there exists an $h$-dependent subset $S_E(t;h)\subset \{k\in \N: E_k(t;h)\in [E-h,E+h]\}$ and $h_E>0$ such that
\begin{equation}
\label{densityatsurface}
\frac{\#S_E(t;h)}{\#\{E_k(t;h)\in [E-h,E+h]\}}>1-\hat{\epsilon}
\end{equation}
and
\begin{equation}
\label{flowspeedatsurface}
\dot{E}_k(t;h)\in [Q_--\epsilon_2,Q_++\epsilon_2]
\end{equation} 
for all $h<h_E$ and all $k\in S_E(t;h)$.
\end{prop}
\begin{proof}
This claim follows immediately from the Hadamard variational formula $\eqref{qequantity}$, Definition \eqref{qesurface}, and \eqref{qplusqminus}.
\end{proof}

Proposition \ref{pointwisesurface}, together with the definition of $\mathcal{B}_\epsilon$ allows us to obtain a similar result over the whole energy band $[a,b]$.

\begin{prop}
\label{pointwiseband}
For each $t\in \mathcal{B}_\epsilon$, there exists an $h$-dependent subset $S(t;h)\subset \{k\in \N: E_k(t;h)\in [a,b]\}$ and $h_0>0$ such that
\begin{equation}
\label{firstC}
\frac{\#S(t;h)}{\#\{E_k(t;h)\in [a,b]\}}>1-\frac{C(n,Q,b-a)\epsilon}{\epsilon_1^2}
\end{equation}
and
\begin{equation}
\dot{E}_k(t;h)\in [Q_--\epsilon_2,Q_++\epsilon_2]
\end{equation} 
for all $h<h_0$ and all $k\in S(t;h)$.
\end{prop}
\begin{proof}
Noting that this result is pointwise in time, we omit the parameter $t$ in our notation and we define 
\begin{equation}
f(E,k;h):= 1_{|E-E_k(h)|<h}\cdot h^{n-1}\left|\langle Q(x,hD)u_k(h),u_k(h)\rangle - \frac{1}{\mu_E(\Sigma_E)}\int_{\Sigma_E} Q \, d\mu_E \right|^2.
\end{equation}
From Definition \ref{qesurface} of quantum ergodicity, we have 
\begin{equation}
\sum_{k\in \mathbb{N}}f(E,k;h)\rightarrow 0
\end{equation}
as $h\rightarrow 0$ for each $E\in \mathcal{E}$.
Hence, for any $\hat{\epsilon}>0$, we can find an $h_0>0$ such that
\begin{equation}
\int_a^b \left(\sum_{k\in\N} f(E,k;h)\right)\, dE < \hat{\epsilon}(b-a)(1-\epsilon)+N(b-a)\epsilon
\end{equation}
where $N$ is an upper bound for
\begin{eqnarray}
& &h^{n-1}\sum_{E_k(h)\in [E-h,E+h]} \left|\langle Q(x,hD)u_k(h),u_k(h)\rangle-\frac{1}{\mu_E(\Sigma_E)}\int_{\Sigma_E}Q d\mu_E\right|^2\\
&\leq &\nonumber 4\max\left(\sup_{H^{-1}([a,b])}(Q),\|Q(x,hD)\|_{L^2\rightarrow L^2}\right)^2 \cdot h^{n-1} \#\{k\in \N: E_k(h)\in [E-h,E+h]\}
\end{eqnarray}
which exists due to the uniform Weyl law in $h$-energy bands proven in \cite{robert}.
Taking $\hat{\epsilon}=N\epsilon $, we obtain
\begin{equation}
\int_a^b \left(\sum_{k\in\N} f(E,k;h)\right)\, dE < 2N(b-a)\epsilon
\end{equation}
for $0<h<h_0$.

Now, by interchanging the order of integration and noting that 
\begin{equation}
\left|\frac{1}{\mu_{E_k(h)}(\Sigma_{E_k(h)})}\int_{\Sigma_{E_k(h)}} Q \, d\mu_{E_k(h)} - \frac{1}{\mu_E(\Sigma_E)}\int_{\Sigma_E} Q \, d\mu_E\right|=O(h)
\end{equation}
uniformly for $E_k(h)\in [E-h,E+h]$, we have
\begin{equation}
\label{fubinieqn}
4h^n\sum_{E_k(h)\in [a,b]} \left|\langle Qu_k(h),u_k(h)\rangle-\frac{1}{\mu_{E_k(h)}(\Sigma_{E_k(h)})} \int_{\Sigma_{E_{k}(h)}} Q d\mu_{E_k(h)}\right|^2\leq 2N(b-a)\epsilon
\end{equation}
for $h<h_0$, where $h_0$ has been redefined.

Denoting by $S(t;h)$ the collection of all indices $k\in \N$ with
\begin{equation}
\dot{E}_k(t;h)\in [Q_--\epsilon_2,Q_++\epsilon_2],
\end{equation}
we can deduce from \eqref{fubinieqn} that
\begin{equation}
2h^n\sum_{k\in S(t;h)^c}\epsilon_2^2\leq N(b-a)\epsilon
\end{equation}
which implies
\begin{equation}
\frac{\#S(t;h)}{\#\{E_k(t;h)\in [a,b]\}}>1-\frac{C(n,Q,b-a)\epsilon}{\epsilon_2^2}
\end{equation}
from Weyl's law, as required.









\end{proof}

From the outer regularity of the Lebesgue measure, we can find an open interval $J=(t_1,t_2)\subset [0,\delta]$ such that
\begin{equation}
\frac{m(\mathcal{B}_\epsilon\cap J)}{m(J)}>1-\tilde{\epsilon}.
\end{equation}
for any $\tilde{\epsilon}>0$.
We can strengthen the pointwise bound on eigenvalue flow speed from Proposition \ref{pointwiseband} for $t\in \mathcal{B}_\epsilon$ to an almost-uniform bound on the interval $J$.

\begin{prop}
\label{almostuniformgevrey}
There exists a subset $\tilde{\mathcal{B}}\subseteq \mathcal{B}_\epsilon\cap J$ and a $h_0>0$ such that
\begin{equation}
\frac{m(\tilde{\mathcal{B}})}{m(\mathcal{B}_\epsilon\cap J)}>1-\tilde{\epsilon}
\end{equation}
and for any $h<h_0$ and any $t\in \tilde{\mathcal{B}}$, there exists a subset $Z(t,h)\subset \{k\in \N: E_k(t,h)\in [a,b]\}$ such that
\begin{equation}
\frac{\#Z(t,h_*)}{\#\{k\in \N:E_k(t,h_*)\in [a,b]\}} >  1- C\epsilon\quad \forall 0<h_* < h_0
\end{equation}
and 
\begin{equation}
\dot{E}_k(t,h)\in [Q_--\epsilon_2,Q_++\epsilon_2]\quad \forall k \in Z(t,h).
\end{equation}
where $C$ is the same constant as in \eqref{firstC}.
\end{prop}
\begin{proof}
We set
\begin{eqnarray}
G(h)&:=&\nonumber \{t\in \mathcal{B}_\epsilon\cap J:\frac{\#\{k\in\N:E_k(t,h_*)\in [a,b]\textrm{ and }\dot{E}_k(t,h_*)\in [Q_--\epsilon_2,Q_++\epsilon_2]\}}{\#\{k\in \N:E_k(t,h_*)\in [a,b]\}}\\
& \geq& 1-C\epsilon\quad \forall 0<h_*<h\}.
\end{eqnarray}
From Proposition \ref{pointwiseband}, we know that
\begin{equation}
\mathcal{B}_\epsilon\cap J=\bigcup_{h>0} G(h)=\bigcup_{n\in\mathbb{N}^+} G(1/n)
\end{equation}
and so the claim follows from monotone convergence.
\end{proof}

On the other hand, we also have an upper bound for the variation of the quasi-eigenvalues.
\begin{prop}
\label{quasispeedprop}
For all sufficiently small $\delta>0$ and all $t\in[0,\delta]$, we have 
\begin{equation}
\limsup_{h\rightarrow 0}\partial_t \mu_m(t;h)\leq Q_- - \epsilon_1/2.
\end{equation}
for all $m\in \cup_{t\in [0,\delta]}\mathcal{M}_h(t)$ uniformly in $t$.
\end{prop}
\begin{proof}
From Proposition \ref{bnf2}, we have 
\begin{equation}
K_0(I;t)=H^0(I)+t\cdot (2\pi)^{-n}\int_{\T^n} H^1(\theta,I)\, d\theta +O(t^{9/8})
\end{equation}
for any $I\in D$.
Hence we have
\begin{equation}
\partial_t(K_0(h(m+\vartheta/4);t))< (2\pi)^{-n}\int_{\T^n} H^1(\theta,h(m+\vartheta/4)) \, d\theta + \epsilon_1/2
\end{equation}
for all $t\in[0,\delta]$, taking $\delta$ sufficiently small.

From the definition of $\mathcal{M}_h(t)$, we know that $|h(m+\vartheta/4)-I(\omega;t)|< Lh$ for some $\omega\in \overline{\Omega}$, and so from the regularity of $I$ in $t$ it follows that
\begin{equation}
\partial_t(K_0(h(m+\vartheta/4);t)) < (2\pi)^{-n}\int_{\T^n} H^1(\theta,I(\omega;t)) \, d\theta+\epsilon_1/2+O(h)
\end{equation}
for some $\omega\in \overline{\Omega}$. This allows us to use \eqref{slowpostori}.

Indeed, we have 
\begin{eqnarray}
\partial_t\mu_m(t;h)&=&\partial_t(K^0(h(m+\vartheta/4);t,h))\\
&=& \partial_t(K_0(h(m+\vartheta/4);t))+O(h)\\
\Rightarrow \limsup_{h\rightarrow 0} \partial_t\mu_m(t;h)& \leq & Q_- - \epsilon_1/2.
\end{eqnarray}
\end{proof}
In particular, we can choose $B>0$ and $h_0>0$ such that 
\begin{equation}
\label{Bdefn}
\partial_t \mu_m(t;h)<B < Q_- -\epsilon_1/4.
\end{equation}
for all $t\in J$ and all $h<h_0$.

\begin{remark}
We have abused notation slightly here by writing $\mu_m(t;h)$ even when $m\notin \mathcal{M}_h(t)$. That is, we track the behaviour of $K^0(h(m+\vartheta/4),t;h)$ even for $t\in[0,\delta]$ such that this does not correspond to a quasi-eigenvalue in our family. This is a necessity due to the rough nature of the set $\{I(\omega;t):\omega\in\overline{\Omega}\}$ of nonresonant actions. Indices $m\in \Z^n$ will typically be elements of $\mathcal{M}_h(t)$ for only $O(h)$-sized $t$-intervals at a time.
\end{remark}

\section{Spectral non-concentration}

We can now complete proof of Theorem \ref{nonqe} by proving a spectral non-concentration result that follows from the results of Section \ref{flowspeedsec}.

\begin{prop}
\label{specnonconckamprop}
There exists sufficiently small $\epsilon$, such that 
\begin{equation}
\label{specnonconckam}
\frac{N(t_*;h)}{\#\mathcal{M}_h(t_*)}<1/2
\end{equation}
for some $t_*\in J$ and all $0<h<h_0$, where $N$ is as in \eqref{Ndef}.
\end{prop}
\begin{proof}
We begin by defining
\begin{equation}
A(t;h)=\{k\in\N : E_k(t;h)\in W(t;h)\}.
\end{equation}
where $W(t;h)$ is as in \eqref{wthdef}.
It then suffices to show that
\begin{equation}
\frac{\#A(t;h)}{\#\mathcal{M}_h(t)}
\end{equation}
can be made arbitrarily small for some $t\in J$ by taking $\epsilon$ arbitrarily small in the definition of $\mathcal{B}_\epsilon$ given in the introduction of Section \ref{flowspeedsec}.

We do this by averaging in $t$ and exploiting Proposition \ref{almostuniformgevrey}. First we define
\begin{equation}
B_k(h):=\{t\in J: k\in Z(t;h)\}
\end{equation} 
where $Z$ is as defined in Proposition \ref{almostuniformgevrey}, and consider the problem of bounding 
\begin{equation}
\sum_{k\in \N} \int_J 1_{B_k}(t)\, dt
\end{equation}
noting that only finitely many terms of the sum are nonzero. This quantity controls the average amount of time that the individual eigenvalues that meet $W(t;h)$ for some $t\in J$ spend travelling at the approximately the ergodic rate given by the small interval $[\tilde{Q}_-,\tilde{Q}_+]:=[Q_--\epsilon_2,Q_++\epsilon_2]$.

By definition, this can be re-written as 
\begin{eqnarray}
& &\int_{\tilde{\mathcal{B}}} \#Z(t;h)\, dt\\
&>& (1-C\epsilon)\cdot \int_{\tilde{\mathcal{B}}}\#\{k\in \N:E_k(t;h)\in [a,b]\} \, dt\\
&>& (1-C\epsilon)(1-\tilde{\epsilon})m(J)\cdot \#\{k\in \N: E_k(0;h)\in [a,b-M\delta]\}
\end{eqnarray}
where $\tilde{\mathcal{B}}$ is the set constructed in Proposition \ref{almostuniformgevrey}.

Hence the average amount of time spent by eigenvalues that meet $W(t;h)$ for some $t\in J$ is bounded below by 
\begin{equation}
(1-C\epsilon)(1-\tilde{\epsilon})m(J)\cdot \frac{\#\{k\in \N: E_k(0;h)\in [a,b-M\delta]\}}{\#\{k\in\N:E_k(0;h)\in [a-M\delta,b]\}}=:(1-\eta)m(J).
\end{equation}
where $\eta\rightarrow 0$ if $\epsilon,\tilde{\epsilon}$, and $\delta$ do.

It follows that at least $(1-\eta^{1/2})$ proportion of the eigenvalues $E_k(t;h)$ that meet $W(t;h)$ for some $t\in J$ have speed $E_k'(t;h)\in [\tilde{Q}_-,\tilde{Q}_+]$ for at least $m(J)(1-\eta^{1/2})$ time. We denote the collection of such indices $k$ by $\mathcal{F}$.

Taking $E(t;h):=E_k(t;h)$ for some $k\in \mathcal{F}$, we have
\begin{equation}
\label{eigenjumplower}
E(t_2;h)-E(t_1;h)> m(J)(1-\eta^{1/2})\tilde{Q}_-.
\end{equation}

On the other hand, we now bound $E(t_2;h)-E(t_1;h)$ above.
To do this, we define $\tilde{E}(t;h)=E(t;h)-Bt$ and $\tilde{\mu}_m(t;h)=\mu_m(t;h)-Bt$
where $B$ was the upper bound in \eqref{Bdefn}.

Then the transformed quasi-eigenvalue windows $\tilde{\mu}_m(t;h)$ are non-increasing. From this it follows that if $\tilde{E}(s;h)\in [\tilde{\mu}_m(s;h)-h^\alpha,\tilde{\mu}_m(s;h)+h^\alpha]$ and $m\in\mathcal{M}_h(s)$ for some $s\in J$, then $\tilde{E}(s';h)-\tilde{E}(s;h)<2h^{\alpha}$, where $s'$ is the final time $t\in J$ such that $m\in\mathcal{M}_h(t)$ and $\tilde{E}(t;h)\in [\tilde{\mu}_m(t;h)-h^\alpha,\tilde{\mu}_m+h^\alpha]$. This implies that $E(s';h)-E(s;h)<2h^\alpha+B(s'-s)$.

Generalising this idea, we can cover $\{t\in J:E(t;h)\in W(t;h)\}$ as a finite union of almost-disjoint intervals $\cup_jI_j$ with $I_j=[s_j,s_j']$ defined as follows:

\begin{enumerate}
\item We define $s_0:=\inf\{t\in J:E(t;h)\in W(t;h)\}$, and we choose an $m(0)\in\mathcal{M}_h(s_0)$ such 
that $E(t;h)\in [\mu_{m(0)}(t;h)-h^\alpha,\mu_{m(0)}(t;h)+h^\alpha]$ and $m(0)\in \mathcal{M}_h(t)$ for 
all sufficiently small $t-s_0>0$.
\item We then define $s_0':=\sup\{t\in J:E(t;h)\in [\mu_{m(0)}(t;h)-h^\alpha,\mu_{m(0)}(t;h)+h^\alpha]\}$.
\item If $t\in \{J: t>s_{j-1}' \textrm{ and } E(t;h)\in W(t;h)\}$ is empty, we terminate the inductive 
process, otherwise we proceed inductively by defining $s_j:=\inf\{t\in J: t>s_{j-1}' \textrm{ and } E
(t;h)\in W(t;h)\}$ and choosing a corresponding $m(j)\in\mathcal{M}_h(s_j)$ such that $E(t;h)\in [\mu_
{m(j)}(t;h)-h^\alpha,\mu_{m(j)}(t;h)+h^\alpha]$ and $m(j)\in \mathcal{M}_h(t)$ for all sufficiently 
small $t-s_{j-1}>0$.
\item We then define $s_j':=\sup\{t\in J:E(t;h)\in [\mu_{m(j)}(t;h)-h^\alpha,\mu_{m(j)}(t;h)+h^\alpha]\}$.
\end{enumerate}
From the Weyl asymptotics, this procedure must terminate after finitely many iterations.

\begin{remark}
In the case that $E(t;h)$ is still in a quasi-eigenvalue window after the window corresponding to $\mu_{m(j)}$, we will have $s_{j+1}=s_j'$. This is the only kind of overlap possible between the intervals $I_j$. We also remark that the $m(j)$ are necessarily distinct, by the nature of this construction.
\end{remark}
For each such interval $I_j=[s_j,s_j']$, we have that $E(s_j';h)-E(s_j;h)<2h^\alpha+B(s_j'-s_j)$.

As there can be at most $O(h^{-n})$ intervals $I_j$, we obtain:
\begin{equation}
\label{coolbound}
\sum_{j}E(s_j';h)-E(s_j;h)< B\sum_j (s_j'-s_j)+O(h^{\alpha-n}).
\end{equation}

For such eigenvalues, we thus obtain the upper bound
\begin{eqnarray}
& &\nonumber E(t_2;h)-E(t_1;h)\\
&<&\nonumber \sum_j (E(s_j';h)-E(s_j;h))+\left(m(J)(1-\eta^{1/2})-\sum_j (s_j'-s_j)\right)\tilde{Q}_++m(J)\eta^{1/2}M\\
&<& \label{eigenjumpupper} (B-\tilde{Q}_+)\sum_j (s_j'-s_j)+m(J)\eta^{1/2}M+m(J)(1-\eta^{1/2})\tilde{Q}_+
\end{eqnarray}
in the limit $h\rightarrow 0$.

Rearranging \eqref{eigenjumpupper} and using \eqref{eigenjumplower}, we arrive at
\begin{equation}
(\tilde{Q}_+-B)\sum_j (s_j'-s_j) < m(J)\eta^{1/2}M+m(J)(1-\eta^{1/2})(\tilde{Q}_+-\tilde{Q}_-).
\end{equation}
By taking $\epsilon, \tilde{\epsilon}$ and $\delta$ small and then passing to sufficiently small $h>0$ we can thus bound 
\begin{equation}
\frac{1}{|J|}\int_J 1_{A_k}(t;h)\, dt
\end{equation}
by an arbitrarily small positive constant $\gamma$ for all $k\in\mathcal{F}$.

Hence we have
\begin{eqnarray}
& &\frac{1}{|J|}\int_J \frac{1}{\#\mathcal{M}_h(t)}\sum_{k\in\N} 1_{A(t;h)}(k)\, dt\\
&\leq & \frac{1}{\inf_{t\in [0,\delta]}\#\mathcal{M}_h(t)}\sum_{k\in \N} \frac{1}{|J|}\int 1_{A(t;h)}(k)\, dt\\
& <& (\eta^{1/2}+(1-\eta^{1/2})\gamma)\cdot \frac{\#\{k\in\N:E_k(t;h)\in[a,b]\textrm{ for some }t\in [0,\delta]\}}{\inf_{t\in[0,\delta]}\#\mathcal{M}_h(t)} \\
&<& \frac{1}{2}
\end{eqnarray}
by using the remark at the end of Section \ref{6intro}, taking $\eta$ small and then passing to sufficiently small $h>0$. 

This average being less than $1/2$ implies the existence of a $t_*\in J$ satisfying the claims of the proposition (possibly dependent on the arbitrarily small parameter $h>0$).
\end{proof}

We proceed as in the beginning of Proposition \ref{spectral}. We denote by $U$, the $h$-dependent span of all eigenfunctions with eigenvalues in $\mathcal{I}(t_*,h)$.

\begin{prop}
\label{spectralcont}
For sufficiently small $h>0$, the projections
\begin{equation}
w_m(t_*,h)=\pi_U(v_m(t_*,h))
\end{equation}
are linearly independent.
\end{prop}
\begin{proof}
First, we show that the estimate from Definition \ref{gevmodes} on the error of quasimodes implies that the projections $\pi_U(v_m(t_*,h))$ are small. In particular, for $m\in \mathcal{M}_h(t_*)$, we have
\begin{eqnarray*}
\|(P_h(t_*)-\mu_m(t_*,h))\sum_{j\in\mathbb{N}}\langle v_m(t_*,h),u_j(t_*,h)\rangle u_j\|^2 & < & Ce^{-ch^{-1/\rho}}\\ 
\Rightarrow \sum_{|E_j-\mu_m|>h^\alpha}^\infty |E_j(t,h)-\mu_m(t,h)|^2|\langle v_m(t_*,h),u_j(t_*,h)\rangle|^2 &<& Ce^{-ch^{-1/\rho}} \\
\Rightarrow \sum_{E_j\notin \mathcal{I}(t_*,h)}^\infty |\langle v_m(t_*,h),u_j(t_*,h)\rangle|^2 &<& \frac{Ce^{-ch^{-1/\rho}}}{h^{2\alpha}}.
\end{eqnarray*}
Hence there exist constants $c,C>0$ such that for sufficiently small $h$, we have 
\begin{equation}
\|\pi_U(v_m(t_*,h))\|^2\geq 1-Ce^{-ch^{-1/\rho}}.
\end{equation}
It follows that the $\pi_U(v_m(t,h))$ are almost orthogonal for distinct $m,k\in\mathcal{M}_h(t)$.
\begin{eqnarray*}
|\langle \pi_U(v_m(t_*,h)),\pi_U(v_k(t_*,h))\rangle|&\leq& |\langle v_m(t_*,h),v_k(t_*,h)\rangle|+|\langle\pi_{U^\perp}(v_m(t_*,h)),\pi_{U^\perp}(v_k(t,h))\rangle|\\
& < & Ce^{-ch^{-1/\rho}} + \sqrt{(1-\|\pi_U(v_m(t,h))\|^2)(1-\|\pi_U(v_k(t,h))\|^2)}.
\end{eqnarray*}

Hence there exist constants $c,C>0$ such that we have 
\begin{equation}
|\langle\pi_U(v_m(t_*,h)),\pi_U(v_k(t_*,h))\rangle-\delta_{k,m}|\leq Ce^{-ch^{-1/\rho}}.
\end{equation}
for all sufficiently small $h$.

If we enumerate the quasimodes $v_m(t_*,h)$ by positive integers rather than $m\in \Z^n$, we can then form the Gram matrix $M(h)\in \textrm{Mat}(\#\mathcal{M}_h(t_*),\R)$, with entries given by
\begin{equation}
M_{ij}(h)=\langle \pi_U(v_i(t_*,h)),\pi_U(v_j(t_*,h))\rangle.
\end{equation}

Since
\begin{equation}
\label{approxid2}
\|M-I\|^2_{HS}\leq C(\#\mathcal{M}_h(t_*))^2e^{-ch^{-1/\rho}}
\end{equation}
for some constants $c,C>0$, it follows as in Proposition \ref{spectral} that we can invert $M=I+(M-I)$ as a Neumann series. Since $M$ is nonsingular, we can therefore conclude that the collection of functions
\begin{equation}
\{\pi_U(v_m(t_*,h)):m\in \mathcal{M}_h(t_*)\}
\end{equation}
is linearly independent.
\end{proof}


We are now in a position to complete the proof of Theorem \ref{nonqe}.

\begin{proof}[Completion of proof of Theorem \ref{nonqe}]
Having chosen $\delta>0$ in Proposition \ref{quasispeedprop}, we have shown in Proposition \ref{specnonconckamprop} that if there does not exist an $\epsilon>0$ satisfying the claim of Theorem \ref{nonqe}, then these exists a $t_*\in [0,\delta]$ at which we have the spectral non-concentration result \eqref{specnonconckam}.

On the other hand, we have shown in Proposition \ref{spectralcont} that the projections $\pi_U(v_m(t_*,h))$ are $\#\mathcal{M}_h(t_*)$ linearly independent vectors in a vector space of dimension $\dim(U)=N(t_*,h)<\#\mathcal{M}_h(t_*)/2$.

This contradiction completes the proof.
\end{proof}

\newpage

\counterwithin{thm}{chapter}
\counterwithin{equation}{chapter}

\appendix

\chapter{Estimates for analytic functions}
\counterwithin{thm}{chapter}
\counterwithin{equation}{chapter}

In this appendix we prove several elementary but important estimates for analytic functions.

\begin{prop}
\label{cauchyappendix}
Suppose $\tilde{\Omega}_j \subset \C$ are open sets and $\Omega_j\subset \tilde{\Omega}_j$ are such that $\textrm{dist}(\Omega_j,\C\setminus \tilde{\Omega_j})<r_j$. 

Define
\begin{equation}
\Omega=\prod_{j=1}^n \Omega_j
\end{equation}
and
\begin{equation}
\tilde{\Omega}=\prod_{j=1}^n \tilde{\Omega}_j.
\end{equation}

If the analytic function $f:\tilde{\Omega}^n\rightarrow \mathbb{C}$ satisfies 
\begin{equation}
\|f\|_{\Omega}=A < \infty
\end{equation}
then we have
\begin{equation}
\|\partial^\alpha_zf\|_\Omega \leq Ar^{-\alpha} \alpha!
\end{equation}
for each multi-index $\alpha$.
\end{prop}
\begin{proof}
From the Cauchy integral formula, we have
\begin{equation}
f(z)=\frac{1}{(2\pi i)^n}\oint_{\partial B(z_1,r_1)}\oint_{\partial B(z_2,r_2)}\ldots \oint_{\partial B(z_n,r_n)} \frac{f(w)}{w-z}\, dw_1\, dw_2 \ldots \, dw_n.
\end{equation}
which yields
\begin{equation}
\partial^\alpha_z f(z)=\frac{\alpha!}{(2\pi i)^n}\oint_{\partial B(z_1,r_1)}\oint_{\partial B(z_2,r_2)}\ldots \oint_{\partial B(z_n,r_n)} \frac{f(w)}{(w-z)^{\alpha+1}}\, dw_1\, dw_2 \ldots \, dw_n.
\end{equation}
upon repeated differentiation, where $1$ denotes the multi-index $(1,1,\ldots,1)$.
Hence 
\begin{equation}
\|\partial^\alpha_zf\|_\Omega \leq Ar^{-\alpha} \alpha!
\end{equation}
as required.
\end{proof}

\begin{prop}
\label{fourierappendix}
Suppose $f$ is a $1$-periodic analytic function on the complex strip 
\begin{equation}
S_\sigma=\{z:|\textrm{Im}(z)|<\sigma\}\subset \C^n
\end{equation}
with the estimate 
\begin{equation}
\|f\|_{S_\sigma}=A<\infty.
\end{equation}
Then for $m\in \mathbb{Z}^n$, the $m$-th Fourier coefficient
\begin{equation}
\label{cauchyintegral}
\hat{f}(m)=\int_{[0,1]^n} e^{-2\pi i m\cdot x}f(x)\, dx
\end{equation}
satisfies the estimate
\begin{equation}
|\hat{f}(m)|\leq e^{-2\pi |m| \sigma}A
\end{equation}
\end{prop}
\begin{proof}
From Cauchy's theorem and the periodicity of $f$, we can replace the integral \eqref{cauchyintegral} with
\begin{equation}
\int_0^1 e^{-2\pi i m\cdot(x-i\delta q)}f(x-i\delta q)\, dx
\end{equation}
for any $0<\delta < \sigma$, where $q_j=\textrm{sgn}(m_j)$.

The desired estimate then follows immediately from the triangle inequality and letting $\delta\rightarrow\sigma$.
\end{proof}

As a consequence of Proposition \ref{fourierappendix}, it is a straightforward matter to control the error incurred by truncating the Fourier series of an analytic function.

\begin{prop}
\label{truncateappendix}
Suppose $f$ is a $1$-periodic analytic function on the complex strip 
\begin{equation}
S_\sigma=\{z:|\textrm{Im}(z)|<\sigma\}\subset \C^n
\end{equation}
with the estimate 
\begin{equation}
\|f\|_{S_\sigma}=A<\infty.
\end{equation}

Then we have
\begin{equation}
\left\|f(x)-\sum_{|m|\leq K}\hat{f}(m)e^{2\pi i m\cdot x}\right\|_{S_{\sigma-\delta}}\leq C(n)K^ne^{-K\delta}\|f\|_{S_{\sigma}}.
\end{equation}
for any $0<\delta < \sigma$.
\end{prop}

We also have an implicit function theorem for real analytic functions. Defining
\begin{equation}
O_h=\{\omega\in \C^n:\textrm{dist}(\omega,\Omega)<h\}
\end{equation}
where distances in $\C^n$ are taken with the sup-norm, we have the following.
\begin{prop}
\label{popovlemma}
Suppose $f:O_h\times (-1,1)\rightarrow \C^n$ is real analytic, and we have the estimate
\begin{equation}
|f|_{h}<\infty,
\end{equation}
then for any $0<v<1/6$ such that
\begin{equation}
|f-id|_h\leq vh
\end{equation}
the function has a real analytic inverse $g:O_{(1/2-3v)h}\times (-1,1)\rightarrow O_{(1-4v)h}$ that satisfies the estimate
\begin{equation}
\label{appendixnormderiv}
\max(|g-id|_{(1/2-3v)h},3vh|D\phi-Id|_{(1/2-3v)h})\leq |f-id|_h
\end{equation}
uniformly in $t\in (-1,1).$
The matrix norm in \eqref{appendixnormderiv} is the norm induced by equipping $\C^n$ with the sup-norm.
\end{prop}
Proposition \ref{popovlemma} can be proven in the same way as in Lemma 3.4 of \cite{popovkam}. The only difference is that we need to work on domains of the form $O_{\lambda h}\times B^\mathbb{C}_1$, and invert maps of the form 
\begin{equation}
\tilde{f}(\omega,t):=(f(\omega,t),t)
\end{equation}
for given $f$ satisfying the assumptions of the proposition uniformly in $t$.

\chapter{Properties of anisotropic Gevrey classes}
\counterwithin{thm}{chapter}
\counterwithin{equation}{chapter}

In this appendix, we collect several results on anisotropic Gevrey classes from the appendix of \cite{popovkam}. The first of these is an implicit function theorem of Komatsu.

\begin{prop}
\label{komatsuprop}
Suppose that $F\in G^{\rho,\rho'}_{L_1,L_2}(X\times \Omega^0,\R^n)$ where $X\subset \R^n$, $\Omega^0\subset \R^m$ and  $L_1\|F(x,\omega)-x\|_{L_1,L_2}\leq 1/2$. Then there exists a local solution $x=g(y,\omega)$ to the implicit equation
\begin{equation}
F(x,\omega)=y
\end{equation} 
defined in a domain $Y\times \Omega$.
Moreover, there exist constants $A,C$ dependent only on $\rho,\rho',n,m$ such that $g\in G^{\rho,\rho'}_{CL_1,CL_2}(Y\times \Omega,X)$ with $\|g\|_{CL_1,CL_2}\leq A\|F\|_{L_1,L_2}$.
\end{prop}

A consequence of this theorem is established by Popov in \cite{popovkam}.

\begin{cor}
\label{komatsucor}
Suppose $F\in G^{\rho,\rho'}_{L_1,L_2}(\T^n\times \Omega,\T^n)$ where $\Omega^0\subset \R^m$ and  $L_1\|F(\theta,\omega)-\theta\|_{L_1,L_2}\leq 1/2$. Then there exists a local solution $x=g(y,\omega)$ to the implicit equation
\begin{equation}
F(x,\omega)=y
\end{equation} 
defined on $\T^n\times \Omega$.
Moreover, there exist positive constants $A,C$ dependent only on $\rho,\rho',n,m$ such that $g\in G^{\rho,\rho'}_{CL_1,CL_2}(\T^n\times \Omega)$ with $\|g\|_{CL_1,CL_2}\leq A\|F\|_{L_1,L_2}$.
\end{cor}

Finally, we have two results on the composition of functions of Gevrey regularity, which can also be found in \cite{popovkam}.

\begin{prop}
\label{gevcomp1}
Let $X\subset \R^n$, $Y\subset\R^m$, and $\Omega\subset \R^k$ be open sets.
Suppose $g\in G^{\rho'}_{L_1}(\Omega,Y)$ with $\|g\|_{L_1}=A_1$ and $f\in G^{\rho,\rho'}_{B,L_2}(X\times Y)$ with $\|f\|_{B,L_2}=A_2$.
Then the composition $F(x,\omega):=f(x,g(\omega))$ is in $G^{\rho,\rho'}_{B,L}(X\times \Omega)$, where $$L=2^{l+\rho'}l^{\rho'}L_1\max(1,A_1L_2)$$ with $l=\max(k,m,n)$.
Moreover we have the Gevrey norm estimate $$\|F\|_{B,L}\leq A_2$$.
\end{prop}

\begin{prop}
\label{gevcomp2}
Let $X\subset \R^n$, $Y\in\R^m$, and $\Omega\subset \R^k$ be open sets.
Suppose $g\in G^{\rho,\rho'}_{B_1,L_1}(X\times \Omega,Y)$ with $\|g\|_{B_1,L_1}=A_1$ and $f\in G^{\rho,\rho'}_{B_2,L_2}(Y\times \Omega)$.
Then the composition $F(x,\omega):=f(g(x,\omega),\omega)$ is in $G^{\rho,\rho'}_{B,L}(X\times \Omega)$, where $$B=4^l(4l)^\rho B_1 \max(1+A_1B_2)$$ and $$L=L_2+4^l(4l)^\rho L_1 \max(1,A_1B_2)$$ with $l=\max(k,m,n)$.
Moreover we have the Gevrey norm estimate $$\|F\|_{B,L}\leq A_2.$$
\end{prop}

\chapter{Whitney extension theorem}
\counterwithin{thm}{chapter}
\counterwithin{equation}{chapter}

In this appendix, we prove a version of the Whitney extension theorem for anisotropic Gevrey classes. The proof is adapted from the work of Bruna \cite{bruna} in the case without in the non-anisotropic case.

\begin{defn}
\begin{equation}
\label{anisspace}
\Cm=\{f\in\mathcal{C}^\infty(X\times Y,\R): \sup_{(x,y)\in X\times Y}\sup_{\alpha,\beta}\left(\frac{|(\partial_x^\alpha \partial_y^\beta f)(x,y)|}{L_1^{|\alpha|}L_2^{|\beta|}M_{|\alpha|}\tilde{M}_{|\beta|}}\right) < \infty \textrm{ for some }L_j>0\}
\end{equation}
\end{defn}
where $X,Y$ are open sets in Euclidean spaces of possibly differing dimension, $\alpha,\beta$ are multi-indices of the appropriate dimension, and $M$ and $\tilde{M}$ are positive sequences satisfying
\begin{enumerate}
\item $M_0=1$
\item $M_k^2\leq M_{k-1}M_{k+1}$
\item $M_k\leq A^k M_j M_{k-j}$
\item $M_{k+1}^k \leq A^k M_k^{k+1}$
\item $M_{k+1}/(kM_k)$ is increasing
\item $\sum_{k\geq 0}M_k/M_{k+1}\leq ApM_p/M_{p+1}$ for $p>0$
\end{enumerate}
where $A>0$ is a positive constant.

In the Gevrey case of interest to us, $M_k=k!^{\rho_1},\tilde{M}_k=k!^{\rho_2}$.

For fixed $L_j>0$, the supremum in \eqref{anisspace} defines a norm which equips a subspace of $\Cm$ with a Banach space structure.

The space $\Cm$ is then the inductive limit of these spaces as $L=L_1=L_2\rightarrow\infty$, which identifies it a Silva space.

For $f\in \Cm$, and $z=(z_1,z_2)\in X\times Y, x\in X$ we define
\begin{defn}
\label{mixedtaylor}
\begin{equation}
(T^{m}_xf)(z):=\sum_{|\alpha|\leq m}\frac{(\partial_x^{\alpha} f)(x,z_2)}{\alpha!}(z_1-x)^{\alpha}
\end{equation}
\end{defn}

\begin{defn}
\label{mixedremainder}
\begin{equation}
(R^{m}_xf)(z):=f(z)-(T^{m}_xf)(z).
\end{equation}
\end{defn}

To slightly generalise this notation, for a jet $f^{\alpha,\beta}$ of continuous functions, we write

\begin{defn}
\label{jetremainder}
\begin{equation}
(R^{m}_{x}f)_{\alpha,\beta}(z):=f^{\alpha,\beta}(z)-(T^{m-|\alpha|}_{x}f^{\alpha,\beta})(z)
\end{equation}
\end{defn}

We can now pose the central question:

\medskip 
\fbox{\parbox{\textwidth}{
Given a compact set $K\subset X$, under what conditions is it true that an arbitrary continuous jet $(f^{\alpha,\beta}):K\times Y\rightarrow \R$ is the jet of a function $\tilde{f}\in\Cm$?
}}
\medskip

We assume without loss of generality here that the set $X$ is a full Euclidean space $\R^d$, rather than just an open subset thereof.

This question is the anisotropic non quasi-analytic analogue of Whitney's extension theorem from classical analysis, which deals with the $\mathcal{C}^\infty$ case.

We begin by finding \emph{necessary} conditions for the existence of such an extension, before proving that these conditions are indeed sufficient.

\begin{prop}
\label{anistaylor}
Suppose $f\in \Cm$ with Gevrey constants $L_1,L_2$. Then there exists a constant $A$ dependent only on the dimensions of $X,Y$ and on $M,\tilde{M}$ such that the jet $f^{\alpha,\beta}=\partial_z^{(\alpha,\beta)}f$ satisfy
\begin{equation}
\label{gevcond}
|f^{\alpha,\beta}|\leq AL_1^{|\alpha|}L_2^{|\beta|}M_{|\alpha|}\tilde{M}_{|\beta|}
\end{equation}
and
\begin{equation}
\label{gevcond2}
|(R^n_xf)_{k,l}(z)|\leq A\tilde{L}_1^{n+1}M_{n+1}L_2^{|l|}\tilde{M}_{|l|}\cdot \frac{|z_1-x|^{n+1}}{(n+1)!}
\end{equation}
for all non-negative integers $m,n$ and all multi-indices $|k|\leq m,|l|\leq n$, where $\tilde{L}_1=CL_1$ with the $C$ dependent only on the dimension of $X$.
\end{prop}
\begin{proof}
The first estimate \eqref{gevcond} follows immediately from the definition of $\Cm$.

We prove the second claim \eqref{gevcond2} by making use of the estimate \eqref{gevcond} on the jet $f^{\alpha,\beta}=\partial_x^\alpha \partial_y^\beta f$ and Taylor expansion.
\begin{eqnarray}
R_x^nf(z)&=& \sum_{|\alpha|=n+1}\frac{n+1}{\alpha!} (z_1-x)^{\alpha}\int_0^1 (1-t)^nf^{\alpha,0}(x+t(z_1-x),z_2)\, dt\\
&\leq & \left(\sup_{|\alpha|=n+1}\sup_{z\in X\times Y} |f^{\alpha,0}(z)|\right)\cdot \sum_{|\alpha|=n+1}\left|\frac{(z_1-x)^\alpha}{\alpha!}\right|\\
&\leq & \left(\sup_{|\alpha|=n+1}\sup_{z\in X\times Y} |f^{\alpha,0}(z)|\right)\cdot \frac{C^{n+1}|z_1-x|^{n+1}}{(n+1)!}\\
\end{eqnarray}

Hence
\begin{equation}
|(R^n_xf)_{k,l}(z)|=|(R_x^{n-|k|}f)(z)|\leq A\tilde{L}_1^{n+1}M_{n+1}L_2^{|l|}\tilde{M}_{|l|}\cdot \frac{|z_1-x|^{n+1}}{(n+1)!}
\end{equation}
as required.
\end{proof}

Subsequently, for simplicity of notation, we omit the tilde in $\tilde{L}_1$ with the understanding that we are allowed to absorb constants that are dependent only on the dimensions of $X,Y$ and on the sequences $M,\tilde{M}$.

\begin{thm}
\label{whitneymainthm}
Suppose $(f^{\alpha,\beta}):K\times Y\rightarrow \R$ is a jet of continuous functions smooth in $y$ that satisfies
\begin{equation}
\partial_y^\gamma(f^{\alpha,\beta})=f^{\alpha,\beta+\gamma}
\end{equation} 
as well as the conditions \eqref{gevcond} and \eqref{gevcond2} on $K\times Y$. Then there exists a function $f\in \Cm$ such that $\partial^{\alpha,\beta}_x f=f^{\alpha,\beta}$ on $K\times Y$.

Moreover, there exist constants $C_0,C_1$ dependent only on the dimensions of $X$ and $Y$ and the weight sequences $(M_k),\tilde{M}_k$ such that
\begin{equation}
\|f\|_{C_1L_1,L_2}\leq C_0A.
\end{equation}
\end{thm}

Before proving Theorem \ref{whitneymainthm}, we need to collect some lemmas, the proofs of which can be found in \cite{bruna}.

\begin{prop}
\label{cubesprop}
Suppose $K\subset \mathbb{R}^d$ is compact. Then there exists a collection of closed cubes $\{Q_j\}_{j\in\mathbb{N}}$ with sides parallel to the axes such that
\begin{enumerate}
\item $\R^d\setminus K=\cup_j Q_j$;
\item $\textrm{int}(Q_j)$ are disjoint;
\item $\delta_j:=\textrm{diam}(Q_j)\leq d_j:= d(Q_j,K)\leq 4\delta_j$;
\item For $0<\lambda < 1/4$, $d(z,K)\sim\delta_j$ for $z\in Q_j^*:= (1+\lambda)Q_j$;
\item Each $Q_i^*$ intersects at most $D=(12)^{2d}$ cubes $Q_j^*$;
\item $\delta_i\sim\delta_j$ if $Q_i^*\cap Q_j^*\neq \emptyset$.
\end{enumerate}
\end{prop}

\begin{prop}
\label{pou}
For each $\eta >0$, there exists a family of functions $\phi_i\in \mathcal{C}^\infty_M(\R^d)$ such that
\begin{enumerate}
\item $0\leq \phi_i$;
\item $\textrm{supp}(\phi_i)\subset Q_i^*$;
\item $\sum_i \phi_i(z)=1$ for $z\in \R^d$;
\item $|\partial^\alpha \phi_i(z)|\leq Ah(B\eta d(z,K))\eta^{|\alpha|} M_{|\alpha|}$ for $z\in Q_i^*$.
\end{enumerate}
where $A,B>0$ are constants and 
\begin{equation}
h(t):=\sup_k\frac{k!}{t^kM_k}.
\end{equation}
\end{prop}

\begin{prop}
\label{openmapping}
Suppose $T\in \mathcal{L}(E,F)$ is a continuous linear surjection between Silva spaces. Then for any bounded set $B\subset F$, there exists a bounded set $C\subset E$ with $T(C)=B$.
\end{prop}

We also require an anisotropic version of Carleman's theorem, which is the special case of \ref{whitneymainthm} with $K=\{0\}$, and Gevrey analogue of Borel's theorem from classical analysis.

\begin{prop}
\label{carlemanvariant}
Let $(g_\alpha)_{\alpha\in \N^{d}}$ be a multisequence of functions in $\mathcal{C}^\infty_{\tilde{M}}(Y)$ such that
\begin{equation}
|\partial_y^l g_\alpha(y)|\leq KL_1^{|\alpha|}L_2^{|l|}M_{|\alpha|}\tilde{M}_{|l|}.
\end{equation}
for some constant $K>0$.

Then there exists a function $f\in \Cm$ such that $g_\alpha(y)=\partial_x^{\alpha}f(0,y)$ for all $y\in Y$. Moreover, $\|f\|_{CL_1,L_2}\leq AK$ for some constants $A,C>0$ independent of $f,L_1,$ and $L_2$.
\end{prop}
\begin{proof}
We adapt the solution of \cite{pet} of the classical Carleman problem to this setting. Key is that the assumptions on $M$ imply that the hypotheses of \cite{pet} are satisfied. 
Hence as in the proof of \cite{pet} Theorem 2.1 (ai), we can construct compactly supported $\chi_p(x)\in \mathcal{C}^\infty_{M_p}(\R)$ for each non-negative integer $p$ such that
\begin{equation}
\chi^{(k)}_p(0)=\delta(k,p)
\end{equation}
and
\begin{equation}
\|\chi_p\|_{L(2+A^{-1})}\leq \frac{1}{M_p}\cdot \left(\frac{Ae}{L}\right)^p
\end{equation}
for some dimensional constant $A$ and any $L>0$.

Hence we can define
\begin{equation}
\chi_\alpha(x):=\prod_{j=1}^d \chi_{\alpha_j}(x_j)
\end{equation}
for $\alpha\in\N^d$ which satisfies
\begin{equation}
\chi_\alpha^{(\beta)}(0)=\delta(\beta,\alpha).
\end{equation}
Moreover, we have the estimate
\begin{eqnarray}
|\chi_\alpha^{(\beta)}|&= & \prod_{j=1}^d |\chi_{\alpha_j}^{\beta_j}|\\
&\leq & \prod_{j=1}^d \frac{1}{M_{\alpha_j}}\left(\frac{Ae}{L}\right)^{\alpha_j}(L(2+A^{-1}))^{\beta_j}M_{\beta_j}\\
&\leq & \left(\frac{Aec(d,M)}{L}\right)^{|\alpha|}\cdot M_{|\alpha|}^{-1}(L(2+A^{-1}))^{|\beta|}M_{|\beta|}.
\end{eqnarray}

By taking $L=2CL_1=2Aec(d,M)L_1$, we can estimate
\begin{eqnarray}
|\partial_x^k\partial_y^l (\chi_\alpha(x)g_\alpha(y))|&\leq & K((C/L)^{|\alpha|}M_{|\alpha|}^{-1}(L(2+A^{-1}))^{|k|}M_{|k|})\cdot(L_1^{|\alpha|}L_2^{|l|}M_{|\alpha|}\tilde{M}_{|l|})\\
&\leq & K\cdot 2^{-|\alpha|}(2CL_1(2+A^{-1}))^{|k|}L_2^{|l|}M_{|k|}\tilde{M}_{|l|}.
\end{eqnarray}

Where $A,C,$ and $K$ are constants independent of $f,L_1,$ and $L_2$.

Hence we have that $\|\chi_\alpha(x)g_\alpha(y)\|_{2CL_1(2+A^{-1}),L_2}\leq K\cdot 2^{-|\alpha|}$. It follows that 
\begin{equation}
f(x,y):=\sum_{\alpha\in \N^d}\chi_\alpha(x)g_\alpha(y)
\end{equation}
converges in the $\Cm$ sense, and satisfies $\partial_x^\alpha f(0,y)=g_\alpha(y)$ as required.
\end{proof}

Equipped with these tools, we are ready to prove Theorem \ref{whitneymainthm}.

\begin{proof}[Proof of Theorem \ref{whitneymainthm}]
We begin by estimating the difference in Taylor expansions about different points in $K$.

Using the identity
\begin{equation}
(T_x^{n}f)(z)-(T_y^{n}f)(z)=\sum_{|\alpha|\leq n} \frac{(z_1-x)^\alpha}{\alpha!}(R_y^nf)_{\alpha,0}(x,z_2)
\end{equation}
we can estimate
\begin{eqnarray}
& &\partial_z^{k,l}((T_x^{n}f)(z)-(T_y^{n}f)(z))\\
&=& \sum_{|\alpha|\leq n-|k|}\frac{(z_1-x)^\alpha}{\alpha!}(R_y^{n}f)_{k+\alpha,l}(x)
\end{eqnarray}
using the assumed estimate \eqref{gevcond2} for $(R_y^{m,n}f)_{k,l}$. 

This yields
\begin{equation}
\label{diffoftaylor}
|\partial_z^{k,l}((T_x^{n}f)(z)-(T_y^{n}f)(z))|\leq AL_1^{n+1}M_{n+1}L_2^{|l|}\tilde{M}_{|l|}\frac{(|z_1-x|+|z_1-y|)^{n-|k|+1}}{(n-|k|+1)!}.
\end{equation}

We now invoke Proposition \ref{carlemanvariant}.

For $x\in X$ consider the map $T_x:\Cm\rightarrow G_x$ given by $(T_xf)_\alpha(y):=f^{\alpha,0}(x,y)$ where the space $G_x$ consists of all multisequences of analytic functions $f_\alpha:Y\rightarrow \R$ satisfying $|f_\alpha|\leq AL_1^{|\alpha|}L_2^{|\beta|}M_{|\alpha|}\tilde{M}_{|\beta|}$ for some $A>0$.

From the assumed estimate \eqref{gevcond} on $f^{\alpha,\beta}$, Proposition \ref{carlemanvariant} applies, and for each $x\in K$, we can find a function $f_x\in \Cm$ such that 




\begin{equation}
\label{fxsolvescarleman}
\partial^{\alpha,\beta}_zf_x(x,z_2)=f^{\alpha,\beta}(x,z_2)
\end{equation}
for each $\alpha,\beta$. Moreover, the conclusion of Proposition \ref{carlemanvariant} implies that there exist constants $B=C_0A,K_1=C_1L_1,K_2=L_2>0$ such that the estimate
\begin{equation}
\label{boundedcarleman}
|(\partial^{\alpha,\beta}_zf_x)(z)|\leq BK_1^{|\alpha|}K_2^{|\beta|}M_{|\alpha|}\tilde{M}_{|\beta|}
\end{equation} 
holds uniformly, where $C_j$ depend only on the dimensions of $X$ and $Y$ and the weight sequences $M_k,\tilde{M}_k$.


Hence we can bound 
\begin{equation}
\partial_z^{k,l}(f_x(z)-(T_x^{m,n}f_x)(z))=(R^{m,n}f_x)_{k,l}(z)
\end{equation}
using the same calculation as in Proposition \ref{anistaylor}.
We obtain
\begin{eqnarray}
|\partial_x^{k,l}(f_x(z)-(T_x^nf)(z))|&=&|(R^nf_x)_{k,l}(z)|\\
&\leq& A(C_1L_1)^{n+1}M_{n+1}L_2^{|l|}\tilde{M}_{|l|}\frac{|z_1-x|^{n-|k|+1}}{(n-|k|+1)!}.
\end{eqnarray}

The upshot of this estimate is that we can replace $T_x^nf$ and $T_y^nf$ in \eqref{diffoftaylor} with $f_x$ and $f_y$ respectively. 

That is, we have
\begin{equation}
\label{fxfyn}
|\partial_z^{k,l}(f_x(z)-f_y(z))|\leq A(C_1L_1)^{n+1}M_{n+1}L_2^{|l|}\tilde{M}_{|l|}\frac{(|z_1-x|+|z_1-y|)^{n-|k|+1}}{(n-|k|+1)!}.
\end{equation}

We now fix $k,l$ and vary $n\geq k$ in order to optimise the upper bound \eqref{fxfyn}.

By defining the quantity 
\begin{equation}
h(t):=\sup_{k\geq 0} \frac{k!}{t^k M_k}
\end{equation}
as in \cite{bruna} we obtain
\begin{equation}
\label{fxfy}
|\partial_z^{k,l}(f_x(z)-f_y(z))|\leq A(C_1L_1)^{|k|}M_{|k|}L_2^{|l|}\tilde{M}_{|l|}h((C_1L_1)(|z_1-x_1|+|z_1-y|))^{-1}.
\end{equation}
by using property (3) following Definition \ref{anisspace}.

The next step in the construction is to use Proposition \ref{pou} to piece together the functions $f_x$ using a $\mathcal{C}^\infty_{M}$ partition of unity subordinate to the cover arising from the decomposition of $X \setminus K$ by cubes in Proposition \ref{cubesprop}.

Taking the collection $\{Q_j\}_{j\in\N}$ of cubes in $X=\mathbb{R}^d$ constructed by Proposition \ref{cubesprop}, we choose $x_j\in K$ such that $d(x_j,Q_j)=d(Q_j,K)$.

Note that the conclusion of Proposition \ref{cubesprop} implies that 
\begin{equation}
|z-x_j|\sim d(z,K)
\end{equation}
for all $z\in Q_j^*$.

Now taking $\phi_j$ as in Proposition \ref{pou}, we define:

\begin{equation}
\tilde{f}(z):= \begin{cases} f(z) &\mbox{if } z_1\in K \\ 
\sum_i \phi_i(z_1)f_{x_j}(z) & \mbox{if } z_1\in X\setminus K.\end{cases}
\end{equation}

Note that since the partition of unity $\{\phi_j\}$ is locally finite, the function $\tilde{f}(z)$ is smooth in $(X\setminus K)\times Y$.

It remains to check that $\tilde{f}$ is smooth elsewhere, and moreover that $\tilde{f}\in \Cm$.

To this end, for $x\in K$ and $z_1\in X\setminus K$, we estimate
\begin{equation}
\partial_z^{\alpha,\beta}(\tilde{f}(z)-f_x(z))=\sum_{k\leq \alpha}\binom{\alpha}{k}\sum_i (\partial^k\phi_i)(z_1)\cdot \partial_z^{\alpha-k,\beta}(f_{x_i}(z)-f_x(z)).
\end{equation}

First we estimate the $k=0$ term. If $z_1\in \textrm{spt}(\phi_i)=Q_i^*$, we have 
\begin{equation}
d(z_1,x_i)\sim d(z_1,K)\leq d(z_1,x)
\end{equation}
and hence we have 
\begin{equation}
\label{kzero}
|\sum_i \phi_i(z_1)\cdot \partial_z^{\alpha,\beta}(f_{x_i}(z)-f_x(z))|\leq A(C_1L_1)^{|\alpha|}M_{|\alpha|}L_2^{|\beta|}\tilde{M}_{|\beta|}h((C_1L_1)|z_1-x|)^{-1}
\end{equation}
from \eqref{fxfy}.

We now estimate the terms with $|k|>0$. For $x\in X\setminus K$, we choose $\bar{x}\in K$ with $d(x,\bar{x})=d(x,K)$.

Since $\sum_i \partial^k\phi_i=0$, we have
\begin{equation}
\sum_i (\partial^k\phi_i)(z_1)\cdot \partial_z^{\alpha-k,\beta}(f_{x_i}(z)-f_x(z))=\sum_i (\partial^k\phi_i)(z_1)\cdot \partial_z^{\alpha-k,\beta}(f_{x_i}(z)-f_{\bar{z_1}}(z)).
\end{equation}

Now as before, we exploit the fact that $d(z_1,x_i)\sim d(z_1,K)$ to bound
\begin{equation}
|\partial_z^{\alpha-k,\beta}(f_{x_i}(z)-f_{\bar{z_1}}(z))|\leq A(C_1L_1)^{|\alpha|-|k|}M_{|\alpha|-|k|}L_2^{|\beta|}\tilde{M}_{|\beta|}h((C_1L_1)d(z_1,K))^{-1}.
\end{equation}

Since $\log(M_j)$ is an increasing convex sequence with first term $0$, it is also superadditive, and we have $M_{|k|}M_{|l|}\leq M_{|k|+|l|}$. Hence for $|k|\geq 1$, we can use property (4) in Proposition \ref{pou} to conclude that
\begin{equation}
\left|\sum_i (\partial^k\phi_i)(z_1)\cdot \partial_z^{\alpha-k,\beta}(f_{x_i}(z)-f_x(z))\right|\leq AM_{|\alpha|}\tilde{M}_{|\beta|}(C_1L_1)^{|\alpha|-|k|}L_2^{|\beta|}\eta^{|k|}\frac{h(B\eta d(z_1,K))}{h((C_1L_1)d(z_1,K))}
\end{equation}
where $\eta$ remains to be chosen.

Equation (15) from \cite{bruna} implies the existence of a constant $c>0$ such that
\begin{equation}
\frac{h(t)}{h(ct)}\leq \frac{A}{h(t)}
\end{equation}
for some $A>0$.

Hence we choose $\eta =(C_1L_1)/cB$ to arrive at the estimate 
\begin{equation}
\label{knotzero}
\left|\sum_i (\partial^k\phi_i)(z_1)\cdot \partial_z^{\alpha-k,\beta}(f_{x_i}(z)-f_x(z))\right| \leq A(C_1L_1)^{|\alpha|-|k|}L_2^{|\beta|}M_{|\alpha|}\tilde{M}_{|\beta|}\eta^{|k|}h((C_1L_1)|z_1-x|)^{-1}.
\end{equation}

Combining \eqref{kzero} and \eqref{knotzero}, we arrive at
\begin{equation}
\label{whitneykey}
|\partial_z^{\alpha,\beta}(\tilde{f}(z)-f_x(z))|\leq AL_2^{|\beta|}M_{|\alpha|}\tilde{M}_{|\beta|}((C_1L_1)+\eta)^{|\alpha|}h((C_1L_1)|z_1-x|)^{-1}
\end{equation}
for $z\in (X\setminus K) \times Y$.

The estimate \eqref{whitneykey} is key to proving $\tilde{f}\in\mathcal{C}^\infty(X\times Y)$ (and that the derivatives coincide with the those given by the jet $f^{\alpha,\beta}$), as well as the subsequent deduction of $\mathcal{C}^\infty_{M,\tilde{M}}$ regularity.

We write
\begin{equation}
\tilde{f}^{\alpha,\beta}(z):=\begin{cases} \partial_z^{\alpha,\beta}\tilde{f}(z) &\mbox{if } z_1\in X\setminus K \\ 
f^{\alpha,\beta}(z) & \mbox{if } z_1\in K.\end{cases}
\end{equation}

The smoothness of each $\tilde{f}^{\alpha,\beta}:X\times Y\rightarrow \R$ readily follows from the fact that  each $f^{\alpha,\beta}:K\times Y\rightarrow \R$ is smooth in $y$, together with the estimate
\begin{equation}
\label{smoothestimate}
|\tilde{f}^{\alpha,\beta}(z)-\partial_z^{\alpha,\beta}T_x^mf(z)|=o(|z_1-x|^{m-|\alpha|}).
\end{equation}

For $z$ with $z_1\in K$, the estimate \eqref{smoothestimate} comes immediately from \eqref{gevcond2} on $K\times Y$. Otherwise, it is a consequence of the estimate \eqref{whitneykey}, the defining property  \eqref{fxsolvescarleman} of the functions $f_x$, and the fact that the function $h(t)$ increases faster than any polynomial in $t^{-1}$ as $t\rightarrow 0$.

Finally, we need to check $\mathcal{C}^\infty_{M,\tilde{M}}$ regularity. That is, we need to verify that the Gevrey estimate
\begin{equation}
\label{endofwhitney}
\|f\|_{C_1L_1,L_2}\leq C_0A.
\end{equation}
for some constants $C_0,C_1$ dependent only on the dimensions of the spaces $X$ and $Y$ and the weight sequences $M_k,\tilde{M}_k$.

In light of \eqref{gevcond}, it only remains to prove \eqref{endofwhitney} on $(X\setminus K)\times Y$, and by multiplication by a cutoff function we may assume $d(z_1,K)$ is bounded.

Then, by applying \eqref{whitneykey} with $x=\bar{z}_1$ we can further reduce the problem to verifying \eqref{endofwhitney} for $f_x$, uniformly in $x\in K$. However this was established earlier in \eqref{boundedcarleman}. 

Hence, the proof is complete.
\end{proof}

\chapter{Miscellaneous}
\counterwithin{thm}{chapter}
\counterwithin{equation}{chapter}
In this section, we prove the following abstract lemma that we have used several times to assemble full density subsequences along which a given function has limit $0$.

\begin{lem}
\label{densitylemma}
If there exists a function $g:\mathbb{N}\rightarrow \R^+$ and a family of subsets $S_j\subset \mathbb{N}$ such that
\begin{equation}
\liminf_{n\rightarrow\infty} \frac{\#\{k\leq n:k\in S_j\}}{n}>d-\epsilon_j
\end{equation}
and 
\begin{equation}
\limsup_{n\in S_j\rightarrow\infty} g(n)<\epsilon_j'
\end{equation}
where $\epsilon_j,\epsilon_j'\searrow 0$, then there exists a subset $S\subset \mathbb{N}$ such that
\begin{equation}
\liminf_{n\rightarrow\infty} \frac{\#\{k\leq n:k\in S\}}{n}\geq d
\end{equation}
and
\begin{equation}
\label{gconv}
\lim_{n\in S \rightarrow\infty} g(n)=0.
\end{equation}
\end{lem}
\begin{proof}
For ease of notation, we define
\begin{equation}
d_n(A)=\frac{\#\{k\leq n:k\in A\}}{n}
\end{equation}
for $A\subseteq\mathbb{N}$ and $n\in\mathbb{N}$.

We have $g(n)< 2\epsilon_j'$ for cofinitely many elements of $S_j$, and we denote these sets by $S_j'$.
Now let
\begin{equation}
B_j=\{k\in\mathbb{N}:g(k)\geq 2\epsilon_j'\}\subseteq \mathbb{N}\setminus S_j'.
\end{equation}
Since each $d_n$ respects the partial ordering of set inclusion and is additive with respect to disjoint unions, we can construct a strictly increasing sequence $(N_j)_{j\in\mathbb{N}}$ such that $N_1=1$ and $d_n(B_j)< 1-d+2\epsilon_j$ for all $n\geq N_j$.

We define
\begin{equation}
B=\bigcup_{j\in \mathbb{N}}B_j\cap [N_j,\infty).
\end{equation}

If $n\in [N_j,N_{j+1})$, then any $k\in [1,n]\cap B$ must lie in $B_i$ for some $i\leq j$ and hence in $B_j$.

This implies that for $n\in [N_j,N_{j+1})$ we have $d_n(B)\leq d_n(B_j) <1-d+2\epsilon_j$ and consequently, that $\displaystyle \limsup_{n\rightarrow\infty} d_n(B)\leq 1-d$.

We now take $S:=\mathbb{N}\setminus B$, with the required density bound 

\begin{equation}
\liminf_{n\rightarrow\infty }d_n(S)\geq d.
\end{equation}

To complete the proof we observe that if $n\in [N_j,\infty) \cap S$, then $n\in \mathbb{N}\setminus B_i$ for each $i\leq j$, and hence $g(n)<2\epsilon_j'$. This establishes \eqref{gconv}.
\end{proof}

For sequences without well-defined natural densities, an analogue of Lemma \ref{densitylemma} holds for the notion of upper density, with an easier proof.

\begin{lem}
\label{densitylemma2}
If there exists a function $g:\mathbb{N}\rightarrow \R^+$ and a family of subsets $S_j\subset \mathbb{N}$ such that
\begin{equation}
\limsup_{n\rightarrow\infty} \frac{\#\{k\leq n:k\in S_j\}}{n}>d-\epsilon_j
\end{equation}
and 
\begin{equation}
\limsup_{n\in S_j\rightarrow\infty} g(n)<\epsilon_j'
\end{equation}
where $\epsilon_j,\epsilon_j'\searrow 0$, then there exists a subset $S\subset \mathbb{N}$ such that
\begin{equation}
\limsup_{n\rightarrow\infty} \frac{\#\{k\leq n:k\in S\}}{n}\geq d
\end{equation}
and
\begin{equation}
\label{gconv2}
\lim_{n\in S \rightarrow\infty} g(n)=0.
\end{equation}
\end{lem}
\begin{proof}
For each $j\in\N$, we have $d_n(S_j)>d-\epsilon_j$ for infinitely many $N$. We define $N_0=0$ and inductively choose an increasing sequence $(N_k)$ of positive integers such that $d_{N_k}(S_k)>d-\epsilon_k$ and such that $g(n)<\epsilon_k'$ for $n>N_k$ in $S_k$.

We can then construct the set
\begin{equation}
S:=\bigcup_{k=1}^\infty \left(\bigcup_{j\geq k} S_j\right)\cap (N_{k-1},N_k].
\end{equation}

Since $d_{N_k}(S)\geq d_{N_k}(S_k)>d-\epsilon_k$, we thus obtain
\begin{equation}
\limsup_{n\rightarrow\infty}d_n(S)\geq d.
\end{equation}

Moreover, for $n\in S$ with $n>N_k$, we have $g(n)<\epsilon_{k}'$, and hence $g(n)\rightarrow 0$ along $S$. This completes the proof.
\end{proof}

\bibliographystyle{plain}
\bibliography{gomesthesis}

\end{document}